\documentclass[letterpaper,11pt,reqno]{amsart}

\usepackage[margin=1in]{geometry}

\usepackage{hyperref}
\hypersetup{
     colorlinks   = true,
     citecolor    = black,
     linkcolor    = blue,
     urlcolor     = black
} 
\usepackage[alphabetic]{amsrefs} 
\usepackage{graphicx,amsmath,amssymb,amsthm,paralist,color,tikz-cd}
\usepackage{mathrsfs}
\usepackage{mathtools}
\usepackage{setspace}
\usepackage{bbm}

\usepackage[T1]{fontenc}
\fontencoding{T1}
\usepackage[utf8]{inputenc}

\usepackage[color=blue!30]{todonotes}
\usepackage{amscd}
\usepackage[all]{xy}
\usepackage{bbm} 

 \usepackage{ifthen}%% for conditionals
 \newcommand{\omicron}{o}

\newcommand{\compact}{\mrm{C}_{c}}

\newcommand{\Qp}{\mathbb{Q}_{p}}
\newcommand{\Qv}{\mathbb{Q}_{v}}
\newcommand{\Zp}{\mathbb{Z}_{p}}

\newcommand{\QS}{\mathbb{Q}_{S}}
\newcommand{\ZS}{\mathbb{Z}[S^{-1}]}

\newcommand{\Sf}{S_{\mathrm{f}}}%% finite places

\newcommand{\Zf}{\mathbb{Z}_{\mathrm{f}}}

\newcommand{\GL}{\mathrm{GL}}
\newcommand{\PGL}{\mathrm{PGL}}
\newcommand{\Mat}{\mathrm{Mat}}

\newcommand{\SL}{\mathrm{SL}}

\newcommand{\PO}{\mathrm{PO}}

\newcommand{\Stab}{\mathrm{Stab}}
\newcommand{\Ad}{\mathrm{Ad}}
\newcommand{\GammaS}{\Gamma_{S}}%% G(Z[S^{-1}])
\newcommand{\liepgl}{\mathfrak{pgl}}

\newcommand{\ad}{\mathrm{ad}}

\newcommand{\norm}[1]{\left\lVert {#1} \right\rVert}

\newcommand{\tr}{\mathrm{tr}}

\renewcommand{\mod}{\,\mathrm{mod}\,}

\newcommand{\der}[1][]{% differential
  \ifthenelse{ \equal{#1}{} }%
  {\ensuremath{\mathrm{d}}}%
  {\ensuremath{\mathrm{d}_{#1}\!}}%
}
\newcommand{\quotient}[2]{\mathchoice
  {% \displaystyle
    \raisebox{.6ex}{\small \newline ${#2}$}\!\big/\!\raisebox{-.6ex}{\small ${#1}$}
  }
  {% \textstyle
    {#2}/{#1}%\raisebox{.4ex}{\small \newline ${#2}$}\!/\!\raisebox{-.4ex}{\small ${#1}$}
  }
  {% \scriptstyle
    {#2}/{#1}%\raisebox{.4ex}{\tiny \newline ${#2}$}\!/\!\raisebox{-.4ex}{\tiny ${#1}$}
  }
  {% \scriptscriptstyle
    {#2}/{#1}%\raisebox{.3ex}{\tiny \newline ${#2}$}\!/\!\raisebox{-.3ex}{\tiny ${#1}$}
  }
}
\newcommand{\rquotient}[2]{\mathchoice
  {% \displaystyle
    \raisebox{-.6ex}{\small \newline ${#1}$}\!\big\backslash\!\raisebox{.6ex}{\small ${#2}$}
  }
  {% \textstyle
    {#1}\backslash{#2}%\raisebox{-.4ex}{\small \newline ${#1}$}\!\backslash\!\raisebox{.4ex}{\small ${#2}$}
  }
  {% \scriptstyle
    \raisebox{-.4ex}{\tiny \newline ${#1}$}\!\backslash\!\raisebox{.4ex}{\tiny ${#2}$}
  }
  {% \scriptscriptstyle
    \raisebox{-.3ex}{\tiny \newline ${#1}$}\!\backslash\!\raisebox{.3ex}{\tiny ${#2}$}
  }
}
\newcommand{\biquotient}[3]{\mathchoice
  {% \displaystyle
    \raisebox{-.6ex}{\small \newline ${#3}$}\!\big\backslash\!\raisebox{.6ex}{\small ${#2}$}\!\big/\!\raisebox{-.6ex}{\small ${#1}$}
  }
  {% \textstyle
    {#3}\backslash{#2}/{#1}%\raisebox{-.4ex}{\small \newline ${#1}$}\!\backslash\!\raisebox{.4ex}{\small ${#2}$}\!/\!\raisebox{-.4ex}{\small ${#3}$}
  }
  {% \scriptstyle
    \raisebox{-.4ex}{\tiny \newline ${#3}$}\!\backslash\!\raisebox{.4ex}{\tiny ${#2}$}\!/\!\raisebox{-.4ex}{\tiny ${#1}$}
  }
  {% \scriptscriptstyle
    \raisebox{-.3ex}{\tiny \newline ${#3}$}\!\backslash\!\raisebox{.3ex}{\tiny ${#2}$}\!/\!\raisebox{-.3ex}{\tiny ${#1}$}
  }
}

\newcommand{\vol}[1]{\mathrm{Vol}( {#1} )}

\newcommand{\supp}{\mathrm{supp}}

\newcommand{\bA}{\mathbb{A}}

\newcommand{\bC}{\mathbb{C}}

\renewcommand{\L}{\mathbb{L}}
\newcommand{\N}{\mathbb{N}}
\newcommand{\bN}{\mathbb{N}}

\newcommand{\Q}{\mathbb{Q}}
\newcommand{\bQ}{\mathbb{Q}}
\newcommand{\R}{\mathbb{R}}
\newcommand{\bR}{\mathbb{R}}

\newcommand{\Z}{\mathbb{Z}}
\newcommand{\bZ}{\mathbb{Z}}

\newcommand{\Dcal}{\mathcal{D}}
\newcommand{\Ecal}{\mathcal{E}}
\newcommand{\Fcal}{\mathcal{F}}
\newcommand{\Gcal}{\mathcal{G}}

\newcommand{\Ical}{\mathcal{I}}
\newcommand{\Jcal}{\mathcal{J}}
\newcommand{\Kcal}{\mathcal{K}}
\newcommand{\Lcal}{\mathcal{L}}

\newcommand{\Rcal}{\mathcal{R}}
\newcommand{\Ocal}{\mathcal{O}}
\newcommand{\Pcal}{\mathcal{P}}
\newcommand{\Qcal}{\mathcal{Q}}
\newcommand{\Scal}{\mathcal{S}}

\newcommand{\Vcal}{\mathcal{V}}

\newcommand{\Zcal}{\mathcal{Z}}

\newcommand{\Xfrak}{\mathfrak{X}}

\newcommand{\Bbf}{\mathbf{B}}

\newcommand{\Gbf}{\mathbf{G}}
\newcommand{\Hbf}{\mathbf{H}}

\newcommand{\ybf}{\mathbf{y}}

\newcommand{\vbf}{\mathbf{v}}

\newcommand{\pbf}{\mathbf{p}}

\newcommand{\xbf}{\mathbf{x}}

\newtheorem{theorem}{Theorem}[section]
\newtheorem{lem}[theorem]{Lemma}

\newtheorem{theoremAlph}{Theorem}

\newtheorem{question}[theorem]{Question}

\newtheorem{proposition}[theorem]{Proposition}
\newtheorem{prop}[theorem]{Proposition}

\newtheorem{corollary}[theorem]{Corollary}

\theoremstyle{definition}
\newtheorem{definition}[theorem]{Definition}
\newtheorem*{definition-nono}{Definition}

\newtheorem{remark}[theorem]{Remark}

\newtheorem*{acknowledgement}{Acknowledgements}

\newtheoremstyle{step}{}{}{}{}{}{:}{ }{}
\theoremstyle{step}
\newtheorem{step}{\textbf{Step}}

\newcommand{\mc}{\mathcal}

\newcommand{\mf}{\mathfrak}
\newcommand{\mrm}{\mathrm}
\newcommand{\bfm}{\mathbf}

\renewcommand{\a}{\alpha}
\renewcommand{\b}{\beta}
\newcommand{\g}{\gamma}

\renewcommand{\d}{\delta}
\newcommand{\e}{\varepsilon}
\renewcommand{\l}{\lambda}
\renewcommand{\L}{\Lambda}
\newcommand{\w}{\omega}
\newcommand{\s}{\sigma}
\newcommand{\vp}{\varphi}
\renewcommand{\t}{\tau}
\renewcommand{\th}{\theta}

\renewcommand{\k}{\kappa}

\newcommand{\Ga}{\Gamma}

\renewcommand{\r}{\rightarrow}

\def\multiset#1#2{\ensuremath{\left(\kern-.3em\left(\genfrac{}{}{0pt}{}{#1}{#2}\right)\kern-.3em\right)}}
\newcommand{\seti}[1]{\left\{#1\right\}}
\newcommand{\dx}[1]{\;\mrm{d}#1}

\newcommand{\cutoff}{0.839}

\newcommand{\lval}{3} 
\newcommand{\miss}{5} 

\numberwithin{equation}{section}

\title[Khintchine's Theorem on Fractals]{Random Walks, Spectral Gaps, and Khintchine's Theorem on Fractals}
\author{Osama Khalil}
 \address{Department of Mathematics, University of Utah, Salt Lake City, UT}
 \email{khalil@math.utah.edu}
\author{Manuel Luethi}
\address{Department of Mathematics, Tel Aviv University}
\email{manuelluthi@mail.tau.ac.il}

\date{}

\begin{document}

\begin{abstract}
     This work addresses problems on simultaneous Diophantine approximation on fractals, motivated by a long standing problem of Mahler regarding Cantor's middle $1/3$ set.
    We obtain the first instances where a complete analogue of Khintchine's Theorem holds for fractal measures. 
    Our results apply to fractals which are self-similar by a system of rational similarities of $\mathbb{R}^d$ (for any $d\geq 1$) and have sufficiently small Hausdorff co-dimension. 
    A concrete example of such measures in the context of Mahler's problem is the Hausdorff measure on the ``middle $1/\miss$ Cantor set''; i.e.~the set of numbers whose base $\miss$ expansions miss a single digit.

    The key new ingredient is an effective equidistribution theorem for certain fractal measures on the homogeneous space $\mathcal{L}_{d+1}$ of unimodular lattices; a result of independent interest. The latter is established via a new technique involving the construction of $S$-arithmetic operators possessing a spectral gap and encoding the arithmetic structure of the maps generating the fractal. As a consequence of our methods, we show that spherical averages of certain random walks naturally associated to the fractal measures effectively equidistribute on $\mathcal{L}_{d+1}$.
\end{abstract}

\maketitle

\section{Introduction}

    Given a function $\psi: \N\r\R_+$, we say that $\xbf = (x_1,\dots,x_d) \in \R^d$ is $\mathbf{\psi}$\textbf{-approximable} if for infinitely many $q\in \N$, we have
    \begin{equation}\label{eq: psi-approx}
    \max_{1\leqslant i\leqslant d} |  qx_i  -p_i | < \psi(q),\quad \text{ for some }
    \bfm{p} =(p_1,\dots,p_d)\in \Z^d.
    \end{equation}
    We denote by $W(\psi) \subseteq \R^d$ the set of $\psi$-approximable vectors.
    A much studied example is the function $\psi_\t(q):= q^{-\t-1/d}$.
    The union over $\t>0$ of $W(\psi_\t)$
    comprises the set of Very Well Approximable (VWA) vectors.
     Khintchine's Theorem~\cite{Khintchine} in its modern formulation asserts that if $\psi$ is non-increasing and $\mrm{Leb}$ is the Lebesgue measure on $\R^d$, then
    \begin{align}\label{eq:Khinchine}
        \mrm{Leb}(W(\psi))= \begin{cases}
        0 & \text{ if }\sum_{q\geqslant 1} \psi^d(q) < \infty,\\
        \mrm{FULL} & \text{ if } \sum_{q\geqslant 1} \psi^d(q) = \infty.
        \end{cases}
    \end{align}

    Motivated by the study of approximation of real numbers by algebraic numbers, Mahler conjectured in 1932 that the Veronese curve $\Vcal$ parametrized by $x\mapsto (x,x^2,\dots,x^d)$ is \emph{extremal}, i.e., the Lebesgue measure on $\Vcal$ assigns zero mass to the set of VWA vectors.
    This conjecture set forth a field of study aiming to understand the prevalence of $\psi$-approximable vectors with respect to measures which may be singular with respect to Lebesgue on $\R^d$, e.g., volume measures on manifolds and affine subspaces as well as fractal measures. 
    In particular, in 1984, Mahler asked:
    \begin{question}[Section 2, \cite{Mahler}]
    \label{q:Mahler}
    How close can irrational elements of Cantor's set be approximated by rational numbers not in Cantor's set?
    \end{question}
    
    Mahler's conjecture on the extremality of $\Vcal$ was settled by Sprind\v zuk who conjectured that every non-degenerate submanifold (e.g., analytic submanifolds which are not contained in a proper affine subspace) of $\R^d$ is extremal.
    This latter conjecture was resolved by Kleinbock and Margulis in~\cite{KleinbockMargulis-Sprinzuk}.
    Subsequently, Weiss proved in~\cite{Weiss-Cantor} that a large class of measures on $\R$, which includes the canonical measure on Cantor's set, is extremal.
    This result, along with the work of~\cite{KleinbockMargulis-Sprinzuk}, was generalized in~\cite{KleinbockLindenstraussWeiss} to show the extremality of a wide class of measures on $\R^d$ which the authors called \emph{friendly measures}.
    This class includes volume measures on non-degenerate manifolds as well as measures which are self-similar by an irreducible iterated function system (IFS for short); cf.~Section~\ref{section: prelim} for the corresponding definitions.
    Naturally, the authors posed the following problem.
    \begin{question}[Question 10.1, \cite{KleinbockLindenstraussWeiss}]
    \label{q:KLW}
    Suppose $\mu$ is a friendly measure on $\R^d$. Does the analogue of~\eqref{eq:Khinchine} hold with $\mrm{Leb}$ replaced by $\mu$?
    \end{question}
    Questions~\ref{q:Mahler} and~\ref{q:KLW} have generated intense activity in recent years.
    We refer the reader to~\cites{BeresnevichDickinsonVelani,Beresnevich-divergence,VaughanVelani,BerEtal-convergence,Huang} and references therein for recent breakthroughs on this problem for non-degenerate manifolds.
    The question for fractals remains wide open however.

    Finally, we refer the reader to Question~\ref{Q: BQ} and Theorem~\ref{cor:BQ} below for another motivation for our work beyond Diophantine approximation, regarding the equidistribution of ``spherical averages" of random walks on homogeneous spaces.
    
    \subsection{Statement of the results}
    The goal of this article is to obtain a complete analogue of Khintchine's Theorem for certain self-similar fractal measures, Theorem~\ref{thm:main simplified}.
    This answers Question~\ref{q:KLW} in the affirmative for those measures.
    The class of measures for which our results hold includes Hausdorff measures on missing digit Cantor sets of sufficiently small Hausdorff co-dimension, Theorem~\ref{thm:intro Cantor}. This provides the first evidence that a similar result is to be expected for Cantor's set in the setting of Question~\ref{q:Mahler}.
    The main ingredient in our proof is a new effective equidistribution result for fractal measures on the space of unimodular lattices, Theorem~\ref{intro-thm:equidist}.

     We introduce notation to be used throughout the rest of the introduction.
    We refer the reader to Section~\ref{section: prelim} for detailed definitions.

    Let $\Fcal =\seti{f_i: i\in \L }$ be an IFS consisting of a finite collection of contractive similarities of $\R^d$ with respect to some inner product.
   We say that $\Fcal$ is \textbf{rational} if $f_i=\rho_iO_i+b_i$ with 
   \begin{equation}\label{def:rational IFS}
    0<\rho_i<1,\qquad    \rho_i \in\Q,\qquad O_i\in \mrm{SO}_d(\R)\cap\mrm{SL}_d(\Q),\qquad b_i\in \Q^d, \qquad  \forall i\in\L,
   \end{equation}
   and $\mrm{SO}_d(\R)$ is the special orthogonal group of some inner product on $\R^d$.
    
    Let $(\l_i)_{i\in\L}$ be a probability vector, i.e., $\l_i>0$ for all $i$ and $\sum \l_i =1$. Denote by $\mu$ the unique self-similar probability measure on $\R^d$ determined by $\Fcal$ and $\l$ and
    by $s$ the Hausdorff dimension (denoted $\dim_H$) of the attractor of $\Fcal$.
    Set 
    \begin{equation*}
        \rho_{\min} := \min\seti{\rho_i:i\in\L}, \qquad \rho_{\max}:= \max \seti{\rho_i: i\in\L}.
    \end{equation*}
    We define $\l_{\min}$ and $\l_{\max}$ similarly.
     The following is the main result of this article.
    \begin{theoremAlph}
    \label{thm:main simplified}
    There exists an explicit $\epsilon_0>0$, depending only on $d$, such that the following holds.
    Suppose $\Fcal$ is a rational IFS satisfying the open set condition.
    Assume further that
     \begin{equation}\label{eq:thickness hypothesis}
        \left(\frac{d\log\rho_{\min}}{\log \l_{\max}}-1\right) \frac{\log\l_{\min}}{s \log\rho_{\max} }  < \epsilon_0.
    \end{equation}
    Let $\psi:\N\r \R_+$ be any non-increasing function. Then,
        \begin{align}\label{eq:khinchine mu}
        \mu(W(\psi))= \begin{cases}
        0 & \text{ if }\sum_{q\geqslant 1} \psi^d(q) < \infty,\\
        1 & \text{ if } \sum_{q\geqslant 1} \psi^d(q) = \infty.
        \end{cases}
        \end{align}
    \end{theoremAlph}
    An explicit choice of $\epsilon_0$ is stated in~\eqref{eq:epsilon0}. We note that even the convergence part of Theorem~\ref{thm:main simplified} is new.
    \begin{remark}
      In the special case of equal contraction ratios and $\l$ being the uniform probability vector, Condition~\eqref{eq:thickness hypothesis} amounts to requiring that the Hausdorff dimension of the fractal is sufficiently close to that of the ambient Euclidean space. 
    \end{remark}

     The key ingredient in the proof of Theorem~\ref{thm:main simplified} is the following dynamical theorem.
    Let $G=\mrm{SL}_{d+1}(\R)$ and $\Ga= \mrm{SL}_{d+1}(\Z)$.
    For $t>0$ and $\xbf \in \R^d$, define the following elements of $G$:
    \begin{equation}\label{eq: simultaneous parametrization}
    	g_t = \begin{pmatrix} e^{t/d} \mrm{Id}_d & \bfm{0}\\ \bfm{0} &  e^{-t}\end{pmatrix}, \qquad
    	u(\xbf) = \begin{pmatrix}
    		\mrm{Id}_d & \xbf \\ \mathbf{0} & 1
    	\end{pmatrix},        
    \end{equation}
    where $\mathrm{Id}_d$ denotes the $d\times d$ identity matrix.
    \begin{theoremAlph}\label{intro-thm:equidist}
    Under the hypotheses of Theorem~\ref{thm:main simplified}, there exists $\d>0$ and $\ell\in\N$ such that for every $\vp\in \mrm{C}_c^\ell(\quotient{\Gamma}{G} )$, and $t>0$,
    the following holds:
    \begin{equation*}
        \int \vp\left(g_tu(\xbf)\Ga \right)\;d\mu(\xbf) = \int \vp \; d m_{\quotient{\Gamma}{G}}
       + O\big( \Scal_{\infty,\ell}(\vp)  e^{-\d t} \big),
    \end{equation*}
    where $m_{G/\Ga}$ is the unique $G$-invariant measure on $G/\Ga$ and the Sobolev norm $\Scal_{\infty,\ell}$ is defined in~\eqref{sec: sobolev}.
    \end{theoremAlph}
    
    It is worth noting that Theorem~\ref{intro-thm:equidist} is new even in its qualitative form.
    The reader is referred to Theorem~\ref{thm: effective single equidist} for a more precise statement.
   In the special case of missing digit Cantor sets, we obtain the following sharper statement.
   
   \begin{theoremAlph}\label{thm:intro Cantor} 
   Theorems~\ref{thm:main simplified} and~\ref{intro-thm:equidist} hold when $\mu$ is the Hausdorff measure on a missing digit Cantor set $\Kcal$ in a prime base, cf. Definition~\ref{def: p cantor}, satisfying
   \begin{equation}\label{eq:intro Cantor cutoff}
       \dim_H(\Kcal)> \cutoff.
   \end{equation}
   In particular, these results hold for $\Kcal$ the set of numbers whose base $\miss$ expansions miss a single digit.
   \end{theoremAlph} 
   
   \begin{remark}
     For comparison, we note that Theorem~\ref{thm:main simplified} implies that Khintchine's Theorem holds for the Hausdorff measure on a missing digit set $\Kcal$ whenever $\dim_H(\Kcal)\geq 0.9992$.
   \end{remark}

    \subsection{Random walks}

To demonstrate the scope of the methods introduced in this article, we establish the equidistribution of certain random walks on $G/\Ga$, motivated by the breakthroughs of Benoist-Quint and Bourgain-Furman-Lindenstrauss-Mozes. 
Given $\Fcal$ and $\l$ as above, let
\begin{equation}\label{eq:BQ notation}
    c_i=\rho_i^{-1/(d+1)}, \qquad
    s_i = \begin{pmatrix} c_i O_i &c_i b_i\\ \mathbf{0} & c_i^{-d}
    \end{pmatrix},
    \qquad \nu= \sum_{i\in \L} \l_i \d_{s_i},
\end{equation}
where we regard $\nu$ as a probability measure on $G$.
Then, the IFS induces a random walk on $G/\Ga$ with law $\nu$.
The methods used to establish Theorem~\ref{intro-thm:equidist} yield the following:

\begin{theoremAlph}\label{cor:BQ}
There exists $\varrho_0>0$ such that the following holds.
Suppose $\Fcal$ is a missing digit IFS on $\R$, cf.~Definition~\ref{def: p cantor}, with attractor $\Kcal$ and $\l$ is the uniform probability measure on $\L$.
Assume that $\dim_H(\Kcal)>1-\varrho_0$. Let $\nu$ be as in~\eqref{eq:BQ notation}. Then,
\begin{equation*}
    \nu^{\ast n} \ast\d_e \to m_{G/\Ga},
\end{equation*}
where $e\in G/\Ga$ is the coset of the identity element in $G$.
The speed of convergence is exponential in $n$ for sufficiently smooth functions on $G/\Ga$.
\end{theoremAlph}

In fact, our methods apply to certain more general basepoints and IFS; cf.~Eq.~\eqref{eq:BQ general basepoints} and Remark~\ref{rem:random walks}.
 This result is motivated by the following well-known open problem. 

\begin{question}[Question 3,~\cite{BQ-takagi}]
\label{Q: BQ}
Suppose $\nu$ is a compactly supported measure on $G$ and let $\Ga_\nu$ denote the subsemigroup generated by its support. Assume that the Zariski closure of $\Ga_\nu$ is semisimple without compact factors. Let $x\in G/\Ga$. As $n\to\infty$, do the measures $\nu^{\ast n}\ast\d_x$ converge towards the unique homogeneous probability measure on $\overline{\Ga_\nu\cdot x}$?
\end{question}

Note that the measures $\nu$ in~\eqref{eq:BQ notation} do not fall under Question~\ref{Q: BQ}.
In the setting of Question~\ref{Q: BQ}, Benoist and Quint showed that $\frac{1}{N}\sum_{1}^N \nu^{\ast n}\ast\d_x$ converge to the expected limit~\cites{BenoistQuint,BQ-II,BQ-III}.
Question~\ref{Q: BQ} was previously resolved in~\cite{BFLM} in the setting of random walks on the torus (under certain additional hypotheses). In that result, a rate of equidistribution was also provided.

Generalizing the work of Benoist-Quint, Simmons and Weiss~\cite{SimmonsWeiss} studied random walks with law $\nu$ as in~\eqref{eq:BQ notation} arising from IFS and proved that for all $x\in G/\Ga$,
\begin{equation}\label{eq:averaged conjecture}
    \frac{1}{N}\sum_{n=1}^N \nu^{\ast n} \ast\d_x \to m_{G/\Ga}.
\end{equation}

Motivated by Question~\ref{Q: BQ}, it is natural to ask whether the Cesaro averaging can be removed in~\eqref{eq:averaged conjecture}. Theorem~\ref{cor:BQ} solves this problem in the cases considered.

    \subsection{Generalizations}
    It is worth noting that we do not require the probability vector to be rational nor the contraction ratios be equal. Our results also apply to fractals in all dimensions.
    
    In order to keep the article to a manageable length, we have not included the most general statements that can be obtained with our methods. We describe below several generalizations of our results we hope to address in forthcoming work.
    
    \begin{enumerate}
    
        \item Jarn\'ik-Besicovitch Theorem:
        H.~Yu recently proved, using Fourier analytic techniques, that the set of VWA numbers have full dimension inside missing digit Cantor sets $\Kcal\subset \R$ whose Hausdorff dimensions are close to $1$ \cite{Han-fractals}. Recalling the notation at the beginning of the introduction, his methods also yield the exact value of the dimension of $W(\psi_\t)\cap \Kcal $ when $\t$ is sufficiently small. 
        We expect our results can be used to provide an alternative proof of those facts. We hope to provide a more complete Hausdorff measure theory of the intersections $W(\psi)\cap\Kcal$ and to address more general fractals $\Kcal$ in future work.

        \item Gallagher's Theorem:
        Our proof of Theorem~\ref{intro-thm:equidist} extends with minor modifications to more general diagonal flows, which commute with the IFS in a suitable sense. In ongoing work, we are studying the application of such extensions to obtain generalizations of Gallagher's Theorem in multiplicative Diophantine approximation~\cite{Gallagher} for fractal measures.
        The reader is referred to~\cite{ChowYang} for related recent developments.
        
        \item Khintchine-Groshev Theorem: Our proof of Theorem~\ref{intro-thm:equidist} also extends to cover rational self-similar measures on the space of systems of linear forms under a suitable analogue of Hypothesis~\eqref{eq:thickness hypothesis}. In particular, the convergence case of Theorem~\ref{thm:main simplified} holds for those measures as well. We leave the divergence case for those measures to future work.
    \end{enumerate}

\subsection{Related work}

The best known result towards Question~\ref{q:KLW} for fractal measures was obtained by Pollington and Velani in~\cite{PollingtonVelani} (cf.~\cite{Weiss-parameter} for the case of Cantor sets on the line).
    They show that for an absolutely friendly measure $\mu$ which is $(C,\a)$-absolutely decaying for some constants $C,\a>0$ (cf.~\cite{PollingtonVelani} and~\eqref{eq:absolute decay} for definitions), the following holds for non-increasing functions $\psi$:
    \begin{align}\label{eq:pollington velani}
        \sum_{q\geq 1} q^{\frac{\a}{d}-1} \psi^\a(q) <\infty \Longrightarrow \mu(W(\psi)) =0.
    \end{align}

   Simmons-Weiss~\cite{SimmonsWeiss} recently proved that for $\mu$-almost every $\xbf$ the measures 
   \begin{equation*}
        \frac{1}{T}\int_0^T \d_{g_t u(\xbf) \Ga}\;\der t 
   \end{equation*}
   converge towards the Haar measure under the minimal necessary hypotheses, i.e., that the IFS is irreducible. For comparison with Theorem~\ref{intro-thm:equidist}, this implies that the averaged measures 
   \begin{equation}\label{eq:averaged mu_t}
        \frac{1}{T}\int_0^T \int \d_{g_t u(\xbf) \Ga}\;\der\mu(\xbf)\;\der t
   \end{equation}
   converge to the Haar measure.
   At the heart of their proof is a generalization of the measure classification results of Benoist and Quint~\cite{BenoistQuint}.
   A weaker equidistribution result for the measures in~\eqref{eq:averaged mu_t} was obtained earlier in~\cite{EinsiedlerFishmanShapira} in the case where $\mu$ is a $\times n$-invariant measure on the circle or an invariant measure for a hyperbolic toral automorphism. Their method relies on the measure classification results of Lindenstrauss~\cite{Lindenstrauss-AQUE}.
   In particular, the methods in both instances are inherently non-effective.
   Moreover, the additional averaging in $t$ is necessary in both cases.
   
   In~\cite{ChowYang}, Chow and Yang showed that the translates by certain diagonal flows (in the interior of the standard positive Weyl chamber) of the Lebesgue measure on a straight line (with Diophantine parameters) become effectively equidistributed on $\mrm{SL}_3(\R)/\mrm{SL}_3(\Z)$. They applied this result to obtain refinements of Gallagher's Theorem~\cite{Gallagher}. Their methods are of a completely different nature to ours and build on an effective equidistribution theorem by Str\"ombergsson on the space of affine lattices in $\R^2$.

   The reader is also referred
   to~\cites{LevesleySalpVelani-Cantor,AllenBarany} for related results on intrinsic Diophantine approximation on fractals and to~\cite{AllenChowYu} for results on the approximation of points on Cantor's middle thirds set by dyadic rationals.

    \subsection{Outline of the proof}
    We first describe the deduction of Theorem~\ref{thm:main simplified} from the equidistribution theorem\footnote{
    Here and throughout, we refer to Theorem~\ref{thm: effective single equidist}, which is the more precise form of Theorem~\ref{intro-thm:equidist}, as the equidistribution theorem.}.
    In view of the connection between $\psi$-approximability and cusp excursions,
    we show in Section~\ref{sec: convergence} that
    the convergence part of Theorem~\ref{thm:main simplified} holds for any (not necessarily self-similar) measure $\mu$ satisfying
     the conclusion of Theorem~\ref{intro-thm:equidist}.
     In Section~\ref{sec:divergence}, we show that the divergence part holds for any (not necessarily rational) self-similar measure with the open set condition satisfying a stronger form of Theorem~\ref{intro-thm:equidist}; namely Corollary~\ref{cor:SL equidist}.
    
    The main difficulties in deducing the divergence part from Corollary~\ref{cor:SL equidist} arise from the fact that our error terms are in terms of Sobolev norms of $\mrm{L}^\infty$-type and are not uniform over certain basepoints associated to the fractal, even as they vary in a fixed compact set in $G/\Ga$. This complicates the independence arguments, especially when the approximation function $\psi$ has slowly diverging partial sums.
    We remark that these issues do not arise in~\cite{KleinbockMargulis-Loglaws}; for instance the error terms in \textit{loc.~cit.~}are in terms of $\mrm{L}^2$-Sobolev norms.
    In particular, for smooth approximations of shrinking cusp neighborhoods, these $\mrm{L}^2$-Sobolev norms provide additional decay in the error terms due to the decay of the measure of the support of such functions.

    To overcome these issues, we prove a converse to the classical Borel-Cantelli lemma, Proposition~\ref{lem: BC divergence}, which is adapted to our problem.
    This result requires as input two quasi-independence estimates of different nature.
    To explain the idea, suppose $E_n$ is a sequence of events in a probability space $(\Omega,\mu)$ such that $\sum\mu(E_n)=\infty$.
    The first such independence estimate roughly takes the form
    \begin{equation}\label{eq:indep long}
        \mu(E_n \cap E_m) \leq C\mu(E_n) \mu(E_m) + O(e^{-\d |n-m|}),
    \end{equation}
    whenever
    \begin{equation}
        n\geq C_\ast m \qquad \text{or} \qquad m\leq n\leq (1+\e_\ast)m
    \end{equation}
    for some constants $C,C_\ast \geq 1$ and $0<\d,\e_\ast<1$ and for all $n,m\in\N$.
    This estimate is most useful when $n\geq m$ are sufficiently separated; namely when $n\gtrsim m-\log \mu(E_m)$.
    To account for close-by pairs of $n$ and $m$, we use an estimate roughly of the form
    \begin{equation}\label{eq:indep short}
        \mu(E_n\cap E_m) \ll \mu(E_m)\mu(E_n)^\e,
    \end{equation}
    for some $\e>0$ and for all $n> m$.
    Additionally, this result requires control over the failure of monotonicity of the measures $\mu(E_n)$; cf. Proposition~\ref{lem: BC divergence}(\ref{item:weak monotone}).

    The first estimate~\eqref{eq:indep long} is deduced in Proposition~\ref{prop: double equidist} from a stronger version of the equidistribution theorem, Corollary~\ref{cor:SL equidist}, which holds for more general basepoints besides the identity cosets.
    Proposition~\ref{prop: double equidist} can be viewed as a substitute for mixing of the flow $g_t$.
    The reason we cannot establish~\eqref{eq:indep long} for all pairs $n$ and $m$ is explained below.

    The second estimate~\eqref{eq:indep short} is proved in Proposition~\ref{prop: weak quasi-independence}.
    The proof relies on the simplex lemma and self-similarity and is similar in spirit to some proofs of the classical Khintchine Theorem.
    It also requires our equidistribution theorem.
    The proof of the divergence part of Theorem~\ref{thm:main simplified} is completed in Section~\ref{sec:divergence}.

    The key ingredient in verifying all the above estimates is the equidistribution theorem. To explain its proof, define
    \begin{equation}\label{eq:mu_t}
        \mu_t =\int \d_{g_t u(\xbf) \Ga}\;d\mu(\xbf), \qquad t\in \R.
    \end{equation}
    The key idea is to construct a random walk which \emph{commutes} with the flow $g_t$ and which leaves each $\mu_t$ stationary.
    To do this, we lift the problem to a suitable $S$-arithmetic cover $G_S/\Ga_S$ of $G/\Ga$.
    The set of finite primes used to define the cover comes from the rational parameters of the IFS.
    The random walk is supported on a finite set $\seti{\g_i:i\in\L} \subset G_S$, defined in~\eqref{def: gamma_w}, and satisfying the following key identity
    \begin{equation*}
        \g_i \cdot g_tu(\xbf)\Ga_S = g_t u(f_i(\xbf))\Ga_S, \qquad (i\in\L),
    \end{equation*}
    for all $t\in\R$ and $\xbf\in \R^d$.
    
    Denote by $\Pcal$ the averaging operator associated to this random walk and the probability vector $\l$.
    A key step in the proof is to show that $\Pcal$ has a spectral gap, in a suitable sense, as an operator on $\mrm{L}^2(G_S/\Ga_S)$.
    This is Proposition~\ref{prop: spectral gap} where we give an explicit estimate on the size of the spectral gap.
    The essential observation used in the proof is that the subsemigroup generated by the support of the random walk remains at a uniformly bounded distance from the lattice $\Ga_S$.
    This allows us to use the fact that the matrix coefficients of $\Ga_S$, acting on $\mrm{L}^2(G_S/\Ga_S)$, belong to $\ell^p(\Ga_S)$ to deduce that $\Pcal$ has a spectral gap.
    At some stage in the proof, we use the fact that this subsemigroup is free and hence our proof is valid for all IFS without exact overlaps, cf.~\eqref{eq:exact overlaps}.
    
    In Section~\ref{sec:summability}, we find an explicit choice of $p$ so that the matrix coefficients belong to $\ell^p(\Ga_S)$.
    In Appendix~\ref{sec:Cantor spectral gap}, we give sharper estimates on the spectral gap of $\Pcal$ in the special case of missing digit Cantor sets, using more elementary techniques; cf.~Proposition~\ref{prop:cantor spectral gap}.
    It is desirable to generalize these methods to more general fractals.
    
    The proof of Theorem~\ref{intro-thm:equidist} is carried out in Section~\ref{sec:equidist}.
    Using the fact that all the maps in the IFS are contractions, we approximate $\mu$---with an explicit bound on the approximation error---by an absolutely continuous probability measure on $\R^d$ (Theorem~\ref{thm: convergence in L-metric}).
    As the approximation happens along the unstable manifold of $g_t$, the approximation errors blow up with $t$.
    A crucial Cauchy-Schwarz step allows us to bring the spectral properties of $\Pcal$ into the argument, cf.~\eqref{eq: Cauchy Schwarz}. 
    When our assumption in~\eqref{eq:thickness hypothesis} holds, the spectral gap of $\Pcal$ is stronger than the approximation error allowing us to obtain the result in this case.

    Over the course of the proof, we apply effective equidistribution of translates of absolutely continuous measures by $g_t$ to functions of the form $\Pcal^n(\vp)$, where $\vp$ is the lift of a smooth function from $G/\Ga$ to $G_S/\Ga_S$. As $n\to \infty$, the functions $\Pcal^n(\vp)$ become less smooth in the $S$-arithmetic sense, i.e., they correspond to functions which live on suitable (congruence) covers of $G/\Ga$. In Proposition~\ref{prop:banana}, we verify the needed equidistribution statements, with uniform error rates and \textit{uniform implied constants} over the family of covers in question.

    Using suitable conjugation of the operator $\Pcal$, along with the above arguments, allows us to prove an equidistribution statement of translates of fractal measures anchored at certain rational basepoints (cf.~\eqref{eq:unimodular-rep}) in $G/\Ga$ which are naturally associated with the IFS. This more general statement is crucial for the independence result in Proposition~\eqref{prop: double equidist} which is a key ingredient in the divergence part of Theorem~\ref{thm:main simplified}.
    However, the index of the congruence cover on which we apply Proposition~\ref{prop:banana} depends on the conjugation of $\Pcal$; i.e. on the basepoint.
    As the error terms in Proposition~\ref{prop:banana} depend on the index of the congruence cover, this causes non-uniformity of our error terms for Corollary~\ref{cor:SL equidist} over basepoints in a compact set in $G/\Ga$.
    This is the reason we are not able to prove the estimate~\eqref{eq:indep long} for all pairs $n$ and $m$.

    Theorem~\ref{cor:BQ} is proved in Section~\ref{sec:BQ} using a similar strategy to the proof of the equidistribution theorem. In this case, we appeal to the equidistribution of rational points instead of absolutely continuous measures (Proposition~\ref{prop:rational-points}); cf. Remark~\ref{rem:random walks} for a discussion of the reason for this difference. 
    In Appendix~\ref{sec:Cantor spectral gap}, we provide the needed modifications on the proofs to obtain Theorem~\ref{thm:intro Cantor}. 
     
    \begin{acknowledgement}
        We would like to thank Jon Chaika and Samantha Fairchild for generously sharing their version of Proposition~\ref{lem: BC divergence} and, in particular, for explaining how an estimate like~\eqref{eq:indep short} can be used for short range correlations. We further thank Manfred Einsiedler, Nimish Shah, Andreas Str\"ombergsson, and Barak Weiss for earlier discussions surrounding this project.  
        Both authors thank the Hausdorff Research Institute for Mathematics at the Universit\"at Bonn for its hospitality during the trimester program ``Dynamics: Topology and Numbers''. M.~L.~thanks the Ohio State University for their hospitality during his visit where this project was started.
        M.~L.~acknowledges the financial support of the ISF through grant 1483/16.
        The authors would like to thank the referees for numerous corrections and suggestions that improved the exposition.
    \end{acknowledgement}

\section{Preliminaries} \label{section: prelim}

    In this section, we recall several facts regarding self similar measures.
    
\subsection{Iterated Function Systems}
    We fix an arbitrary inner product on $\R^d$ and denote by $\mrm{SO}_d(\R)$ the subgroup of $\mrm{SL}_d(\R)$ which preserves it.
	A finite collection of maps $\Fcal= \seti{f_i: i\in \L}$ on $\R^d$ is said to be an \textbf{iterated function system} (IFS for short) if each $f_i$ is a contractive similarity of $\R^d$ relative to our chosen inner product, i.e., $f_i$ has the form
    \begin{equation*}
         f_i = \rho_i O_i + b_i,
    \end{equation*}
    where $0<\rho_i <1$, $O_i \in \mrm{SO}_d(\R)$, and $b_i \in \R^d$. Let
    \begin{equation*}
        \L^\ast = \bigcup_{k\geq 0} \L^k,
    \end{equation*}
    where we use the convention $\L^0 = \seti{\emptyset}$ and $f_\emptyset $ is the identity mapping.
    In particular, $\rho_\emptyset = 1$, $b_\emptyset = \mathbf{0}$, and $O_\emptyset$ is the identity element of $\mrm{SO}_d(\R)$.

    Using a fixed point theorem, it is shown in~\cite{Hutchinson} that there exists a unique compact set $\mc{K} \subset \R^d$ which is invariant by $\Fcal$ in the sense that
    \begin{equation}\label{eq:limit set}
    	\mc{K} = \bigcup_{i\in \L} f_i(\mc{K}).
    \end{equation}
    We refer to the set $\mc{K}$ as the \textbf{attractor} of $\Fcal$.
    Given $\omega=(\omega_i)\in\L^{k}$, we let
\begin{equation*}
  f_{\omega}=f_{\omega_1}\circ\cdots\circ f_{\omega_{k}}.
\end{equation*}
Denoting by $\mathbf{0}$ the origin in $\bR^{d}$, the maps $f_\w$ take the form $\rho_\w O_\w + b_\w$, where
    \begin{equation}\label{eq: composition parameters}
     \rho_\w = \prod_{i=0}^{k} \rho_{\w_i},
     \quad O_\w = O_{\w_{1}} \cdots O_{\w_k},
     \quad b_\w = f_\w(\mathbf{0}).
    \end{equation}
    Hence, by induction, for all $k\in\N$, we have
    \begin{equation*}
        \Kcal=\bigcup_{\w\in \L^k} \Kcal_\w, \qquad \Kcal_\w := f_\w(\Kcal).
    \end{equation*}
    
    \begin{definition}\label{def: markov operator}
Given an IFS $\Fcal=\seti{f_i:i\in \L}$ and probability vector $\l$ on $\L$, define the operator $P_\l$ on $\mrm{C}(\bR^{d})$ as follows. For all $\vp\in \mrm{C}(\R^d)$ let
    \begin{equation*}
        P_\l(\vp)(\xbf) =  \sum_{i\in\L}\l_i \vp(f_i(\xbf))\quad(\xbf\in\bR^{d}).
    \end{equation*}
    The dual operator on measures, also denoted $P_\l$, is defined similarly by  
 \begin{equation*}
        P_\l(\nu) =  \sum_{i\in\L}\l_i (f_i)_\ast\nu,
\end{equation*}
    for all Borel measures $\nu$ on $\R^d$. We say a compactly supported probability measure $\mu$ on $\R^d$ is \textbf{self-similar} if
    \begin{equation} \label{eqn: self similar measure}
    	\mu = P_\l(\mu).
    \end{equation}
\end{definition}

Given a probability vector $\l$, induction applied to~\eqref{eqn: self similar measure} shows that
\begin{equation}\label{eq: iterated P_l}
    \mu = P_\l^k(\mu)= \sum_{\w\in \L^k} \l_\w (f_\w)_\ast \mu,
\end{equation}
whenever $\mu$ is a self-similar measure, where $\l_\w = \prod_{i=1}^{k}\l_i$.
We also note that given a probability measure $\nu$ and $\psi\in\mrm{C}_c(\R^d)$, we have
    \begin{equation}\label{eq: MF on functions}
        \int \psi(\xbf) \;\mrm{d} P_\l^{k}(\nu)(\xbf) = \sum_{\w \in \L^k} \l_\w \psi(f_\w(\xbf))\;\mrm{d}\nu(\xbf).
    \end{equation}

We say a map $P:X\r X$ of a metric space $X$ is a contraction with ratio $r\in (0,1)$ if for every $x_1,x_2\in X$, $d(P(x_1),P(x_2)) \leq r d(x_1,x_2)$. We need the following elementary lemma.

\begin{lem}\label{lem: rate of convergence in fixed point theorem}
  Suppose $P:X\r X$ is a contraction mapping of a metric space with contraction ratio $r \in (0,1)$. Let $y, x_0 \in X$ be such that $P^{n_k}(y)\to x_0$ along a sequence $n_k\in\N$. Then, $P^n(y)\to x_0$, $x_0$ is the unique fixed point of $P$, and for all $n\in\bN$ we have
  \begin{equation}\label{eq:contraction in metric spaces}
      d(x_0, P^n(y)) \leqslant r^n \frac{d(y,P(y))}{1-r} . 
  \end{equation}
\end{lem}

\begin{proof}
    For each $m\in \N$, let $y_m = P^m(y)$ and let $y_0 = y$.
    Then, for every $m>n$,
  \begin{align*}
  d(y_{m},y_{n}) &\leqslant d(y_{m},y_{m-1})+d(y_{m-1},y_{m-2})+\cdots +d(y_{n+1},y_{n})\\
  &\leqslant r^{m-1}d(y_{1},y_{0})+r^{m-2}d(y_{1},y_{0})+\cdots +r^{n}d(y_{1},y_{0})\\
  &=r^{n}d(y_{1},y_{0})\sum _{k=0}^{m-n-1}r^{k}
  \leqslant r^{n}d(y_{1},y_{0})\sum _{k=0}^{\infty }r^{k}
  =r^{n}{\frac {d(y_{1},y_{0})}{1-r}}.
  \end{align*}
  Since $y_{n_k} \r x_0$, then~\eqref{eq:contraction in metric spaces} follows by continuity of the distance function. It follows that $y_n\to x_0$. 
  In particular, $x_0 = \lim_n y_n = P(\lim_n y_{n-1}) = P(x_0)$ and hence $x_0$ is fixed by $P$. Uniqueness follows since $P$ is a contraction.
\end{proof}

Given a measure space $(X,\mu)$ and $\vp\in \mrm{L}^1(\mu)$, we use the notation
\begin{equation*}
    \mu(\vp) := \int\vp\;\der\mu.
\end{equation*}

Given a Lipschitz function $\vp$ on $\R^d$, we let $\mrm{Lip}(\vp)$ denote its Lipschitz constant. We use $\norm{\vp}_\infty$ to denote the sup-norm of $\vp$. Denote by $\mrm{Prob}_\mrm{c}(\R^d)$ the space of compactly supported probability measures on $\R^d$.
Following~\cite{Hutchinson}, we define the $L$-metric on $\mrm{Prob}_\mrm{c}(\R^d)$ as follows:
\begin{equation*}
    L(\mu,\nu) = \sup \left|\mu(\vp) - \nu(\vp)\right|,
\end{equation*}
for all $\mu,\nu \in \mrm{Prob}_\mrm{c}(\R^d)$, where the supremum is taken over all Lipschitz functions $\vp$ with Lipschitz constant at most $1$.

\begin{theorem}[Theorem 4.4.1 (ii), \cite{Hutchinson}]
\label{thm: convergence in L-metric}
    Let $\Fcal$ be an IFS and $\l$ be a probability vector.
    There exists a unique probability measure $\mu\in \mrm{Prob}_c(\R^d)$ satisfying~\eqref{eqn: self similar measure}.
    Moreover, for every Lipschitz function $\vp$ on $\R^d$, $\nu\in\mrm{Prob}_\mrm{c}(\R^d)$, and all $m\in \N$, we have
  \begin{equation*}
    \lvert 
    \mu(\vp)- P_\l^m(\nu)(\vp)\rvert\ll_\nu r^{m}\mrm{Lip}(\vp),
  \end{equation*}
  where $r$ is given by
  \begin{equation*}
      r = \sum_{i\in \L} \l_i \rho_i.
  \end{equation*}
\end{theorem}
\begin{proof}

  First, we show that $P_\l$ is a contraction in the $L$-metric on $\mrm{Prob}_\mrm{c}(\R^d)$ with ratio $r$. Indeed,
  we have for all $\nu_1,\nu_2 \in \mrm{Prob}_\mrm{c}(\R^d)$,
  \begin{align*}
      \left| P_\l(\nu_1)(\vp) - P_\l(\nu_2)(\vp) \right| &\leqslant
      \sum_{i\in \L} \l_i \rho_i \left| \nu_1(\rho_i^{-1} \vp\circ f_i) - \nu_2(\rho_i^{-1} \vp\circ f_i)  \right| \\
      &\leqslant r L(\nu_1,\nu_2) \mrm{Lip}(\vp), 
  \end{align*}
  where on the second line, we used the fact that $\mrm{Lip}(\vp \circ f_i)\leq\rho_i\mrm{Lip}(\vp)$.
  
  In order to apply Lemma~\ref{lem: rate of convergence in fixed point theorem}, it remains to check that $P_\l^n(\nu)$ converges along a subsequence to some $\mu\in\mrm{Prob}_c(\R^d)$ in the $L$-metric.
  Since all the maps in $\Fcal$ are contractions, there is a closed ball $B\subset \R^d$ around $\mathbf{0}$ containing the supports of the measures $P_\l^n(\nu)$ for all $n$.
  By compactness of the weak-$\ast$ topology on $\mrm{Prob}(B)$, we can find $\mu\in \mrm{Prob}(B)\subset \mrm{Prob}_c(\R^d)$ and a sequence $n_k$ such that $P_\l^{n_k}(\nu)\to \mu$ in the weak-$\ast$ topology.
  
  We note that this implies convergence in the $L$-metric on $\mrm{Prob}_c(\R^d)$. Indeed, for every $\vp\in \mrm{C}(B)$ with $\mrm{Lip}(\vp)\leq 1$, and for $\nu_1,\nu_2\in \mrm{Prob}(B)$, we have 
  \[|\nu_1(\vp)-\nu_2(\vp)| = |\nu_1(\vp-\vp(\mathbf{0})) - \nu_2(\vp-\vp(\mathbf{0}))|. \]
  Hence, it suffices to check convergence in the $L$-metric where the supremum is restricted to functions which vanish at $\mathbf{0}$.
  For such functions, we have $\norm{\vp}_\infty \leq |\vp(\mathbf{0})|+R=R$, where $R$ is the radius of $B$ since $\mrm{Lip}(\vp)\leq 1$. This set of function is pre-compact in the uniform norm on $\mrm{C}(B)$ in view of the Arzel\`a-Ascoli theorem. One then verifies that this implies convergence in the $L$-metric. Applying Lemma~\ref{lem: rate of convergence in fixed point theorem} completes the proof.

\end{proof}

    Finally, we record the following simple lemma concerning averages of multiplicative cocycles.
    \begin{lem}\label{lem: average of submul coc is subadditive}
    Suppose a tuple $(\t_i)_{i\in \L} \in \R^{|\L|} $ and a probability vector $\l$ on $\L$ are given.
    For $\w=(\w_1,\dots ,\w_k)\in \L^k$, let  $\t_\w = \t_{\w_1}\cdots \t_{\w_k}$. 
    Then, for all $n\in \N$,
    
      \begin{equation*}
        \sum_{\w\in\L^n}\l_\w \t_\w = \left(\sum_{i\in\L} \lambda_{i}\t_i \right)^n.
      \end{equation*}

    \end{lem}
    
    \begin{proof}
    Let $a_n = \sum_{\w\in\L^n}\l_\w \t_\w$.
    Given two words $\a$ and $\w$, let $\a\w$ denote the word obtained by concatenating $\w$ to the end of $\a$.
    We then note that $\l_{\a\w}=\l_\a\l_\w$ and $\t_{\a\w}=\t_\a\t_\w$.
    It follows that, for all $m,n \in \N$, we have
    \begin{align*}
        a_{m+n} = \sum_{\w\in\L^{m+n}}\l_\w \t_\w
        &= \sum_{\a\in \L^{m}} \sum_{\w\in \L^n} \l_{\a\w}\t_{\a\w}
         \sum_{\a\in \L^{m}} \sum_{\w\in \L^n} \l_{\a}\l_\w \t_{\a}\t_\w = a_m a_n.
    \end{align*}
    \end{proof}

   \subsection{Notational Convention}\label{sec: IFS}
      For the remainder of this article, we will denote by $(\Fcal,\lambda)$ a tuple consisting of an IFS and a probability vector $\lambda$. With such a tuple, we implicitly fix a choice of an inner product relative to which $\Fcal$ consists of similarity maps and denote the resulting norm simply by $\norm{\cdot}$.
       We denote by $\mrm{SO}_d(\R)$ the orthogonal group of this inner product.
      We extend this norm to $\R^{d+1}$ as follows:
      \begin{equation*}
          \norm{(x_1,\dots,x_{d+1})} = \max\seti{\norm{(x_1,\dots,x_d)}, |x_{d+1}| }.
      \end{equation*}
      We will denote by $\Kcal=\Kcal_{\Fcal}$ the attractor, which is completely determined by $\Fcal$, and by $\mu=\mu_{(\Fcal,\lambda)}$ the unique self-similar measure for the operator $P_{\lambda}$ provided by Theorem~\ref{thm: convergence in L-metric}.

   \subsection{The open set condition and null overlaps}
   \label{sec:osc}
   In general, serious problems in the analysis occur if the images of the fractal by distinct maps of the IFS overlap.
   We recall several conditions under which such overlap becomes negligible in a precise sense.

    We say $\Fcal$ satisfies the \textbf{open set condition} if there exists an open set $U\subset \R^d$ such that for all $i\neq j\in \L$,
    \begin{equation}\label{eq:OSC}
    	f_i (U) \subseteq U, \qquad
        f_i(U) \cap f_j(U) =\emptyset.
    \end{equation}
    We say that a self-similar measure $\mu$ has \textbf{null overlaps} if for all $i\neq j\in \L$,
    \begin{equation}\label{def:null overlaps}
        \mu(\Kcal_i \cap \Kcal_j) = 0.
    \end{equation}
    Note that by self similarity, the null overlaps property implies that $\mu(\Kcal_\a \cap \Kcal_\w) =0$ for all $\a\neq \w \in \L^n$ and for all $n\in\N$.

    \begin{lem}[Lemma 3.3,~\cite{Graf}]
    \label{lem: null overlap}
    Suppose $\mc{F}$ is an IFS satisfying the open set condition, $\lambda$ is a probability vector and $\mu$ is a self-similar measure for $(\mc{F},\lambda)$. Then, $\mu$ has null overlaps.
    \end{lem}

    For a Borel set $A$ and a Borel measure $\mu$, we denote by $\mu\vert_A$ the restriction of $\mu$ to $A$. That is for every Borel set $B$, $\mu\vert_A(B) = \mu(B\cap A)$.
    The following consequence of null overlaps will be useful for us.
    
    \begin{lem}\label{lem: transformation of self-similar measures}
    Suppose $\mc{F}$ is an IFS, $\l$ is a probability vector and $\mu$ is a self-similar measure for $(\mc{F},\l)$ having null overlaps. Then, for every $k\in \N$ and all $\w\in \L^k$,
    \begin{equation*}
         \mu \vert_{f_\w(\Kcal)}= \l_\w (f_\w)_\ast\mu
    \end{equation*}
    and, in particular, $\mu(\Kcal_{\w})=\l_{\w}$.
    \end{lem}
    \begin{proof}
    The main assertion follows from the null overlaps property and equation~\eqref{eq: iterated P_l}.
    \end{proof}
    
    The weakest notion of separation for an IFS is having \textbf{no exact overlaps}, where $\Fcal$ has no exact overlaps if 
    \begin{equation}\label{eq:exact overlaps}
        \a\neq \w \Longrightarrow f_\a \neq f_\w, \qquad \forall \a,\w\in\L^\ast.
    \end{equation}
    Having no exact overlaps turns out to be sufficient for spectral gap arguments, cf.~\textsection~\ref{sec:spectral gap}.
    
    \subsection{A Zero-Full law for fractals}
    The following lemma will be used in the proof of the divergence part of Theorem~\ref{thm:main simplified} to upgrade positivity of the measure of the set of $\psi$-approximable vectors to the statement that they have full measure. 
    Note that Cassel's Zero-Full law~\cite{Cassels-Zerofull} does not apply directly to fractal measures; cf.~Remark~\ref{rem:ergodicity}.
    
       \begin{lem}\label{lem:Lebesgue-density}
         Let $(\Fcal,\lambda)$ as in Section~\ref{sec: IFS} and suppose that the associated self-similar measure $\mu$ has null overlaps.
         Let $A \subseteq \R^d$ be a Borel measurable set and suppose that there exists $c > 0$ such that for every $\a\in\L^{\ast}$, we have that $\mu(A \cap  \Kcal_\a) \geq c\mu(\Kcal_\a)$. Then, $\mu(A) = 1$.
       \end{lem}
       \begin{proof}
       Let $\Sigma = \L^\N$ and endow it with the product topology induced from the discrete topology on $\L$.
       Denote by $\l^{\bN}$ the product measure on $\Sigma$ defined by $\l$.
       Let $\pi: \L^\N \to \Kcal$ be the coding map defined by $\pi(\a) = \lim_{n\to\infty} (f_{\a_1}\circ\cdots\circ f_{\a_n})(\mathbf{0})$. Then, $\pi$ is continuous and $\mu = \pi_\ast \l^\N$; cf.~\cite[Theorems 3.1(3) and 4.4(4)]{Hutchinson}.
       In particular, $\mu(A)=\l^\N(\pi^{-1}(A))$.
       For $\a\in \L^\ast$, let $\Sigma_\a$ denote the cylinder set determined by $\a$. For all $x$, let $\Sigma(x,k)$ denote the unique cylinder set of length $k$ containing $x$.
       
       By our null overlaps hypothesis, the symmetric difference between $\pi^{-1}(\Kcal_\a)$ and $\Sigma_\a$ has measure $0$.
       Hence, our hypothesis shows that 
       \begin{equation}\label{eq:dense in cylinders}
           \l^\N(\pi^{-1}(A)\cap\Sigma_\a) \geq c \l^{\N}(\Sigma_\a).
       \end{equation}
       
       On the other hand, if $B = \Sigma \setminus \pi^{-1}(A)$, then by a version of the Lebesgue density theorem for $(\Sigma,\l^\N)$, cf. Lemma~\ref{lem:Leb density for cylinders}, for almost every $x\in B$, we have
       \begin{equation*}
           \lim_{k\to\infty}\frac{\l^\N(\Sigma(x,k)\cap B)}{\l^\N(\Sigma(x,k))}=1.
       \end{equation*}
       It follows that if $B$ has positive measure, then we can find some cylinder $\Sigma(x,k)$ so that $\l^\N(\Sigma(x,k)\cap B)> (1-c)\l^\N(\Sigma(x,k))$. 
       This contradicts~\eqref{eq:dense in cylinders}.
        
       \end{proof}

    \subsection{Absolutely decaying measures}
    For a subset $\Lcal \subset \R^d$ and $\e>0$, we denote
    \begin{equation*}
        \Lcal^{(\e)} = \big\{\xbf\in \R^d: d(\xbf,\Lcal) <\e\big\}.
    \end{equation*}

    The IFS $\Fcal$ is said to be \textbf{irreducible} if no proper affine subspace of $\R^d$ is invariant under all the maps in $\Fcal$.
    
    The following absolute decay property was established in~\cite{KleinbockLindenstraussWeiss} for a wide class of natural measures on $\R^d$.
    We recall their result in our setting.

\begin{prop}[Theorem 2.3,~\cite{KleinbockLindenstraussWeiss}]
\label{prop: friendly}
Let $(\Fcal,\lambda)$ as in Section~\ref{sec: IFS} and suppose that $\Fcal$ is irreducible and satisfies the open set condition.
Let $\mu$ be the associated self-similar measure.
Then, there exist constants $C\geq 1$ and $\g >0$ such that for every word $\a \in  \Lambda^\ast$, 
for every proper affine subspace $\Lcal$, and for every $\e>0$, we have
\begin{equation}\label{eq:absolute decay}
    \mu(\Kcal_\a \cap \Lcal^{(\e)}) \leqslant C \left(\frac{\e}{\rho_\a}\right)^\g \mu(\Kcal_\a).
\end{equation}
\end{prop}
\begin{proof}
Let $s=\dim_H\Kcal$.
 The proposition follows from the argument establishing the $(C,\a)$-absolutely decaying property for $\mu$ in~\cite[Theorem 2.3]{KleinbockLindenstraussWeiss}, under the open set condition, in the case $\mu$ is the Hausdorff measure supported on the fractal, i.e., $\mu$ is the self-similar measure for the probability vector $(\rho_i^s)_{i\in\L}$.
 The proof adapts verbatim to general self-similar measures. Indeed, the only place in the proof in~\cite[Section 8]{KleinbockLindenstraussWeiss} where the fact that $\mu$ is the Hausdorff measure is used is to show that for any Borel set $A\subset \R^d$,
 \begin{equation}\label{eq: scaling mu}
     \mu(f_\w(A\cap\Kcal)) = \rho^s_\w \mu(A),
 \end{equation}
 for all $\w\in \L^\ast$; cf.~proof of (8.5) and (8.6) in \textit{loc.~cit}.
 For general probability vectors $\l$,~\eqref{eq: scaling mu} holds with $\l_\w$ in place of $\rho_\w^s$ by Lemmas~\ref{lem: null overlap} and~\ref{lem: transformation of self-similar measures}.
\end{proof}

%--------------------------------------------------

\section{Congruence quotients of PGL}

The goal of this section is to introduce notation for the $S$-arithmetic groups and homogeneous spaces we work with in our proofs. We also discuss several technical connectedness issues which arise naturally in equidistribution problems and which stem from the failure of Strong Approximation for $\PGL$. We also recall bounds on matrix coefficients of $\PGL$ which we use in later sections.

\subsection{S-arithmetic Setup} \label{sec: s-arith}

We let $\PGL_{d+1}$ denote the automorphism group of the algebra $\Mat_{d+1}$. Throughout this article we fix the $\bQ$-structure (and the integral structures) to be the one induced by the faithful representation $\PGL_{d+1}\hookrightarrow\GL_{(d+1)^2}$ induced by the standard basis of $\Mat_{d+1}$.

In what follows, we write $\Gbf=\PGL_{d+1}$.
Given a finite set $S$ of places of $\bQ$ possibly containing the archimedean place, we denote 
\[
\Sf=S\setminus\{\infty\},\quad
\Z_{\infty}=\R,\quad
\QS=\prod_{v\in S}\Qv,\quad 
\Zf=\prod_{p\in\Sf}\Zp,\quad
\ZS=\Z[\tfrac{1}{p};p\in \Sf].
\]
Accordingly, we define the following:
\begin{align*}
  \Gbf_S &= \Gbf(\QS), \quad
  \Gbf_\infty = \Gbf(\R),\quad
  \Gbf_{\mrm{f}} = \Gbf(\Q_{\Sf}),\\
  \GammaS &= \Gbf(\ZS)=\Gbf(\Q_{S})\cap\GL_{(d+1)^{2}}(\ZS),
  \\
  K_{\mrm{f}} &= \Gbf(\Zf)=\Gbf(\Q_{\Sf})\cap\GL_{(d+1)^{2}}(\Zf).
\end{align*}
If $\infty\in S$, then $\GammaS$ embedded diagonally in $\Gbf_S$ is a lattice and we denote
\begin{equation*}
X_S=\Gbf_S/\Gamma_S.     
\end{equation*}
We also use the following notation:
\begin{equation*}
    M_v = \begin{cases} 
    \PO_{d+1}(\bR) & \text{if } v=\infty,\\
    \Gbf(\bZ_{v}) & \text{if } v \text{ is finite.}
    \end{cases}
\end{equation*}

We will denote by $\Phi:\GL_{d+1}\to\GL_{(d+1)^{2}}$ the rational representation defined by mapping an element $x\in\GL_{d+1}$ to the automorphism $\Phi(x)$ of $\Mat_{d+1}$ defined by
\begin{equation}\label{eq:representation-of-GL}
  \Phi(x)w=xwx^{-1}\quad(w\in\Mat_{d+1}).
\end{equation}
By the Skolem-Noether theorem, for every field $k$ we have $\Gbf(k)=\Phi(\GL_{d+1}(k))$. We also note that $\ker\Phi=k^{\times}$ and therefore $\Gbf(k)\cong\GL_{d+1}(k)/k^{\times}$.

We will denote by $\lVert\cdot\rVert_{v}$ the operator norm on $\GL(\Mat_{d+1}(\bQ_{v}))$ given by the sup-norm with respect to the standard basis if $v$ is finite, and induced by the inner product
\begin{equation*}
  \langle L,M\rangle=\mrm{tr}({}^{t}LM)\quad\big(L,M\in\GL(\Mat_{d+1}(\bQ_{v}))\big)
\end{equation*}
if $v$ is infinite. Abusing notation, we denote
\begin{equation}\label{def:adjoint norm}
  \lVert g\rVert_{v}=\lVert\Ad_{g}\rVert_{v}\quad(g\in\Gbf(\bQ_{v})).
\end{equation}
We note that for all $g\in\Gbf(\bQ_{v})$ and for all $k_{1},k_{2}\in M_{v}$ we have (cf. Corollary~\ref{cor:trivialcorollary})
\begin{equation}\label{eq:invariance of norm}
  \lVert k_{1}gk_{2}\rVert_{v}=\lVert g\rVert_{v}.
\end{equation}

\begin{lem}\label{lem:identification-space-of-lattices}
  Let $\SL_{d+1}(\bR)$ act on $\Gbf_{\infty}/\Gbf(\bZ)$ via the representation $\Phi$. This action is transitive and the stabilizer of the identity coset is $\SL_{d+1}(\bZ)$. In particular, we have
  \begin{equation*}
    \quotient{\SL_{d+1}(\bZ)}{\SL_{d+1}(\bR)}\cong\quotient{\Gbf(\bZ)}{\Gbf_{\infty}}
  \end{equation*}
  as $\SL_{d+1}(\bR)$-spaces.
\end{lem}
\begin{remark}
The above statement is a well-known fact in the theory of lattices in $\bR^{d+1}$ once we know that $\Gbf(\Z)$ agrees with the image of $\GL_{d+1}(\bZ)$ under projection with respect to the center, which is not apriori clear.
\end{remark}
\begin{proof}
  We first show that $\Gbf(\bZ_{p})=\Phi(\GL_{d+1}(\bZ_{p}))$ for every prime $p$. In fact, for every $x\in\GL_{d+1}(\bQ_{p})$ satisfying $\Phi(x)\in\Gbf(\bZ_{p})$ there is some $n_{p}\in\bZ$ such that $p^{n_{p}}x\in\GL_{d+1}(\bZ_{p})$. Indeed, choose $n_{p}\in\bZ$ such that $\lVert p^{n_{p}}x\rVert_{p}=1$. Using the KAK-decomposition on $\GL_{d+1}(\bQ_{p})$, let $p^{n_{p}}x=k_{1}ak_{2}$, where $k_{1},k_{2}\in\GL_{d+1}(\bZ_{p})$ and $a$ is a diagonal matrix whose entries are decreasingly ordered with respect to the $p$-adic valuation. As $\GL_{d+1}(\bZ_{p})$ preserves $\Mat_{d+1}(\bZ_{p})$ under left- and right-multiplication, we find that that
  \begin{equation*}
    a\Mat_{d+1}(\bZ_{p})a^{-1}=\Mat_{d+1}(\bZ_{p}).
  \end{equation*}
  In particular, the automorphism of $\Mat_{d+1}(\bZ_{p})$ given by conjugation by $a$ maps the standard basis to a $\bZ_{p}$-basis of $\Mat_{d+1}(\bZ_{p})$. This implies that $a=\mrm{Id}_{d+1}$ and in particular $p^{n_{p}}x\in\GL_{d+1}(\bZ_{p})$. As $\Phi(\lambda x)=\Phi(x)$ for all $\lambda\in\bQ_{p}^{\times}$, the initial claim then follows.

  We next deduce that $\Gbf(\bZ)=\Phi(\GL_{d+1}(\bZ))$. To this end let $g\in\Gbf(\bZ)$ and let $x\in\GL_{d+1}(\bQ)$ such that $g=\Phi(x)$. As of the previous argument, we know that for all primes $p$ there is $n_{p}\in\bZ$ such that $p^{n_{p}}x\in\GL_{d+1}(\bZ_{p})$ and we note that $n_{p}=0$ for almost all $p$. Therefore the number
  \begin{equation*}
    c=\prod_{p}p^{n_{p}}\in\bQ^{\times}
  \end{equation*}
  is well-defined and $cx\in\GL_{d+1}(\bZ_{p})$ for all primes $p$. As
  \begin{equation*}
    \GL_{d+1}(\bZ)=\bigcap_{p}\big(\GL_{d+1}(\bQ)\cap\GL_{d+1}(\bZ_{p})\big),
  \end{equation*}
  we get $g=\Phi(cx)\in\Phi(\GL_{d+1}(\bZ))$.

    Two lattices $\Delta_{1},\Delta_{2}\subseteq\bR^{d+1}$ are homothetic if there exists a scalar $c\in\bR$ such that $\Delta_{1}=c\Delta_{2}$. The natural action of $\GL_{d+1}(\bR)$ on lattices in $\bR^{d+1}$ induces a transitive action of $\Gbf_{\infty}$. Moreover, as every homothety class admits a unique unimodular representative, this induces a transitive action of $\Gbf_{\infty}$ on $\SL_{d+1}(\bR)/\SL_{d+1}(\bZ)$ and the stabilizer of the identity coset is $\Gbf(\bZ)$. The identification is clearly $\SL_{d+1}(\bR)$-equivariant and hence the claim follows.    
\end{proof}

Given $N\in\bN$ we also denote by $\Gamma(N)\leq\Gbf(\bZ)$ the principal congruence subgroup of level $N$, i.e., the kernel of the homomorphism defined by coordinate-wise reduction mod $N$.
For the remainder of the article, we use the notation
\begin{equation*}
  X_\infty(N) := \Gbf_\infty/\Ga(N).
\end{equation*}
We abuse notation and let $\Gamma(1)=\Gbf(\bZ)$. A subgroup $\Delta\leq\Gamma(1)$ is called a congruence subgroup if it contains a principal congruence subgroup.

We note that for the chosen representation of $\Gbf$, the group has class number one (cf. Proposition~\ref{prop:classnumber} and Corollary~\ref{cor:adelic-to-S-arithmetic}), i.e., 
\begin{equation}\label{eq: class number one}
  (\Gbf_{\infty}\times K_{\mrm{f}})\GammaS=\Gbf_{S}.
\end{equation}

Denote by $\N_0$ the set of non-negative integers.
Given $\mathbf{m} = (m_p)_{p\in \Sf}\in\N_{0}^{\Sf}$, let $\Sf^{\mathbf{m}}=\prod_{p\in\Sf}p^{m_{p}}$ and denote by $K_{\mrm{f}}[\Sf^{\mathbf{m}}]$ the kernel of the canonical homomorphism $\Gbf(\Zf)\to\prod_{p\in\Sf}\Gbf(\quotient{p^{m_{p}}\Z}{\Z})$. It will be useful to abuse notation and let $K_{\mrm{f}}[1]=K_{\mrm{f}}$. Note that $\Gamma(\Sf^{\mathbf{m}})=\Gbf(\bQ)\cap K_{\mrm{f}}[S^{\mathbf{m}}]$.

As $\Gbf$ has class number one, there is $r(\mathbf{m})\in\bN$ such that
\begin{equation}\label{eq:congruence}
  \biquotient{\GammaS}{\Gbf_S}{K_{\mrm{f}}[\Sf^{\mathbf{m}}]}\cong\bigsqcup_{i=1}^{r(\mathbf{m})}X_\infty(\Sf^{\mathbf{m}}),
\end{equation}
as $\Gbf_{\infty}$-spaces; cf.~Proposition ~\ref{prop:correspondence}.

Let $\Gbf_{S}^{+}$ denote the image of $\SL_{d+1}(\bQ_{S})$ in $\Gbf_{S}$ and similarly $\Gbf_{v}^{+}$ denotes the image of $\SL_{d+1}(\Q_v)$, for each $v\in S$. These are normal subgroups of finite index in $\Gbf_{S}$ and $\Gbf_{v}$ respectively.
The number $r(\mathbf{m})$ of cosets in~\eqref{eq:congruence} is bounded by the index of $\Gbf_S^+$ inside $\Gbf_S$, cf. Proposition~\ref{prop:components-as-quotients}.

In what follows, we will call the copies of $X_{\infty}(\Sf^{\mathbf{m}})$ the components of $K_{\mrm{f}}[\Sf^{\mathbf{m}}]\backslash\Gbf_{S}/\GammaS$. We want to point out that these components are in general not connected. In fact, the connected components are precisely the $\Gbf_{\infty}^{+}$ orbits.

%--------------------------------------------------

    \subsection{Sobolev norms} \label{sec: sobolev}
    We introduce a family of Sobolev norms to be used throughout the article.
    We fix a basis $\mathcal{G}_{d}$ of the Lie algebra of $\Gbf_\infty$ and, given $\ell\in\N_{0}$, we denote by $\Xfrak_{\ell}$ the set of all monomials in the elements of $\mathcal{G}_{d}$ of degree at most $\ell$. 
    The elements of $\Xfrak_{\ell}$ act as differential operators on $\mrm{C}^{\infty}(\quotient{\Gamma}{\Gbf_\infty})$, for any lattice $\Ga$ in $\Gbf_\infty$.
    Given $\varphi\in \mrm{C}^{\infty}(\Gbf_\infty/\Gamma)$, we denote by $\Scal_{p,\ell}$ the $\mrm{L}^{p}$-Sobolev norm of degree $\ell$ defined by
    \begin{equation}\label{eq:SobolevNorms}
      \Scal_{p,\ell}(\varphi)=\sum_{\Dcal\in\Xfrak_{\ell}}\lVert \Dcal\varphi\rVert_{p},
    \end{equation}
    where the $\mrm{L}^{p}$-norm is defined with respect to the $\Gbf_{\infty}$-invariant probability measure. Given $p\in[1,\infty]$, we let
    \begin{equation}\label{eq:smooth vectors}
     \mrm{B}_{p,\ell}^\infty(\Gbf_\infty/\Gamma) = \seti{\vp\in \mrm{C}^\infty(\Gbf_\infty/\Gamma):\Scal_{p,\ell}(\vp)<\infty }.
    \end{equation}
    Note that in case $\ell=0$, $\mrm{B}_{p,\ell}^{\infty}(\Gbf_\infty/\Gamma)=\mrm{C}^{\infty}(\Gbf_\infty/\Gamma)\cap \mrm{L}^{p}(\Gbf_\infty/\Gamma)$. 

We write $\mrm{C}^\infty(X_S)$ for the space of functions on $X_S$ which are invariant by some compact open subgroup of $K_\mrm{f}$ and which are smooth along the $\Gbf_\infty$ directions.
 Accordingly, we can extend~\eqref{eq:SobolevNorms} and~\eqref{eq:smooth vectors} to $\mrm{C}^\infty(X_{S})$, where the $\mrm{L}^{p}$-norms are defined with respect to the $\Gbf_S$-invariant probability measure.
    We shall use the following basic estimates on these norms.
    Throughout the article, for a function $\vp$ on $X_S$ and $g\in \Gbf_S$, we use $\vp\circ g$ to denote the composition of $\vp$ with the left multiplication action of $g$ on $X_S$.
    \begin{lem}\label{lem:Sobolev}
    For all $\ell\in\N$, the following holds with implicit constants depending on $\ell$:
    
      \begin{enumerate}
        
        \item \label{item:scaling of Sobolev norms}
        For all $\vp \in \mrm{B}_{\infty,\ell}^\infty(X_S)$ and $g\in \Gbf_S$,
        $\Scal_{\infty,\ell}(\vp\circ g) \ll \norm{\mrm{Ad}(g)}_\infty^{\ell} \Scal_{\infty,\ell}(\vp)$, where $\norm{\mrm{Ad}(g)}_\infty$ denotes the operator norm of the adjoint action of the Archimedean component of $g$ on the Lie algebra of $\Gbf_\infty$.
         \item\label{item:Sobolev of product} For any $p\geq 1$ and $\vp,\psi\in \mrm{B}_{p,\ell}^\infty(X_S)$, $\Scal_{p,\ell}(\vp\psi)\ll_\ell \Scal_{2p,\ell}(\vp)\Scal_{2p,\ell}(\psi)$.
         In particular, we have $\Scal_{\infty,\ell}(\vp^2) \ll \Scal_{\infty,\ell}(\vp)^2$.
        \item \label{item:Sobolev embedding diff}For each non-zero $\Dcal$ in the Lie algebra of $\Gbf_\infty$, $\norm{\Dcal\vp}_\infty \ll \norm{\Dcal} \Scal_{\infty,1}(\vp)$, where $\norm{\Dcal}$ is taken with respect to any fixed choice of norm on the Lie algebra and the implicit constant depends on such choice.
        \item \label{item:Sobolev convolution}
        For $\th\in \mrm{C}_\mrm{c}^\infty(\Gbf_\infty)$ and $\vp\in \mrm{L}^\infty(X_S)$, $\Scal_{\infty,\ell}(\th\ast \vp) \ll
         \mrm{m}_{\Gbf_\infty}(\mrm{supp}(\th))
        \Scal_{\infty,\ell}(\th)
        \norm{\vp}_\infty$, where $\th\ast \vp$ denotes the convolution of the two functions and $\mrm{m}_{\Gbf_\infty}$ is the Haar measure on $\Gbf_\infty$.
      \end{enumerate}
    \end{lem}  
    \begin{proof}
    The independence of the estimate in Item~\eqref{item:scaling of Sobolev norms} from the non-Archimedean component of $g$ follows from the fact that the differential operators in the definition of our Sobolev norms commute with $\Gbf_{\mrm{f}}$.
    Item~\eqref{item:Sobolev of product} follows by Cauchy-Schwarz; cf.~\cite[Lemma 8.1]{Venkatesh}.
    Item~\eqref{item:Sobolev embedding diff} follows from expressing $\Dcal$ in terms of the basis $\mc{G}_d$ with coefficients bounded by $\norm{\Dcal}$. 
    To show Item~\eqref{item:Sobolev convolution}, it suffices to note that invariance of the Haar measure implies that $\Dcal(\th\ast \vp) = (\Dcal \th) \ast \vp$ for any differential operator $\Dcal$ on $\Gbf_\infty$.
    \end{proof}

    \begin{remark}
      Throughout the remainder of the article, we omit the dependence of implicit constants in our estimates on the order $\ell$ of the Sobolev norms in question for the sake of readability. 
    \end{remark}
 
 %--------------------------------------------------
\subsection{Uniform bounds on matrix coefficients}\label{sec:GMO}
The main reference for the material in this section is~\cite[Section 3]{GMO}. Much of the general discussion there is a lot simpler in the setting where $\Gbf=\PGL_{d+1}$, and we restrict ourselves to this case.

We let $v$ denote a possibly infinite place of $\bQ$, $A_{v}$ the image of the diagonal subgroup of $\GL_{d+1}(\bQ_{v})$ in $\Gbf$, and $\Sigma^+_v$ the system of positive roots of $\Gbf(\Q_v)$ relative to $A^{+}_{v}\leq A_{v}$, where
\begin{equation*}
  A^+_\infty=\Phi\left\{\begin{pmatrix}
      a_{1} & & \\
      & \ddots & \\
      & & a_{d+1}
    \end{pmatrix}:a_{i}\in\bR,a_{1}\geq\cdots\geq a_{d+1}=1\right\}
\end{equation*}
and for any finite rational prime $p$
\begin{equation*}
  A^+_p=\Phi\left\{\begin{pmatrix}
      p^{-n_{1}} & & \\
      & \ddots & \\
      & & p^{-n_{d+1}}
    \end{pmatrix}:n_{i}\in\bN,n_{1}\geq\cdots\geq n_{d+1}=0\right\}.
\end{equation*}
In what follows we will usually identify elements in $A_{v}^{+}$ with their representative in the sets on the right hand side of the above equations.

Recall that $M_{v}=\PO_{d+1}(\bR)$ if $v=\infty$ and $M_{v}=\Gbf(\bZ_{v})$ if $v$ is finite. Using the polar decomposition for infinite $v$ or the $p$-adic analogue for finite $v$, we have $\Gbf(\bQ_{v})=M_{v}A_{v}^{+}M_{v}$; cf. Appendix~\ref{sec:KAK decomposition}.

Choose a maximal strongly orthogonal system $\Scal_v$ in $\Sigma^+_v$; cf.~\cite{Oh}. Define a bi-$M_v$ invariant function $\xi_v$ on $\Gbf(\Q_v)$ as follows: for every $g = m_1am_2 \in \Gbf(\Q_v)$,
    \begin{equation}\label{eq:Harish-Chandra}
        \xi_v(g) = \prod_{\a\in\Scal_v} \Xi_{\mrm{PGL}_2(\Q_v)}
        \begin{pmatrix} \a(a) & 0 \\ 0 & 1 \end{pmatrix},
    \end{equation}
    where $\Xi_{\mrm{PGL}_2(\Q_v)}$ is the Harish-Chandra function on $\mrm{PGL}_2(\Q_v)$. We then define $\xi_\Gbf$ on $\Gbf_S$ by
    \begin{equation} \label{def: xi_G}
        \xi_\Gbf (g) = \prod_{v\in S} \xi_v(g_v)\quad(g\in\Gbf_S).
    \end{equation}

    Denote by $\mrm{L}^2_{00}(X_S)$ the closed subspace of $\mrm{L}^2(X_S)$ which is orthogonal to the subspace spanned by $\Gbf(\Q_v)^+$-invariant functions for all $v$.
    We note that the normality of $\Gbf(\Q_v)^+$ implies that $\mrm{L}^2_{00}(X_S)$ is a $\Gbf_S$-invariant subspace. The following lemma will be of importance in Section~\ref{sec:equidist}.
    \begin{lem}\label{lem: character spectrum}
      Assume that $\vp\in \mrm{L}^{2}(X_{S})$ is $K_{\mrm{f}}$-invariant and has mean zero. Then,  $\vp$ belongs to $\mrm{L}_{00}^{2}(X_{S})^{K_{\mrm{f}}}$.
    \end{lem}
    \begin{proof}
      As $\vp$ is by assumption $K_{\mrm{f}}$-invariant, we only have to show that the average of $\vp$ over every $\Gbf_{S}^{+}$-orbit in $X_{S}$ vanishes. As $\Gbf$ has class number one, cf.~\eqref{eq: class number one}, we know that $K_{\mrm{f}}\backslash X_{S}\cong X_{\infty}(1)$. Hence transitivity of the action of $\Gbf_{\infty}^{+}$ on $X_{\infty}(1)$, cf.~Lemma~\ref{lem:identification-space-of-lattices}, implies that the average of $\vp$ over any $\Gbf_{\infty}^{+}\times K_{\mrm{f}}$-orbit vanishes. As shown in the proof of Proposition~\ref{prop:components-as-quotients}, the orbits of $\Gbf_{\infty}^{+}\times K_{\mrm{f}}$ in $X_{S}$ agree with the orbits of $\Gbf_{S}^{+}K_{\mrm{f}}$ in $X_{S}$. Hence $K_{\mrm{f}}$-invariance of $\vp$ implies that the average of $\vp$ on every $\Gbf_{S}^{+}$-orbit in $X_{S}$ vanishes.
    \end{proof}
    
    We recall the following bound on matrix coefficients given in~\cite{GMO}.
    \begin{theorem}[Theorems 3.10 and 3.20,~\cite{GMO}] 
    \label{thm: matrix coeffs}
    For any compact open subgroup $W_{\mrm{f}}$ of $\Gbf_{\mrm{f}}$, there exists $C_{W_{\mrm{f}}} > 0$ such that the following holds:
    for all $W_{\mathrm{f}}$-invariant $M_{\infty}$-finite vectors $v,w\in \mrm{L}_{00}^{2}(X_S)$ and for all $g\in\Gbf_S$, 
\begin{equation}\label{eq:GMO}
  \lvert\langle gv,w\rangle\rvert\leq C_{W_{\mathrm{f}}}\dim\langle M_{\infty}\cdot v\rangle \dim\langle M_{\infty}\cdot w\rangle\lVert v\rVert_{2}\lVert w\rVert_{2}\xi_{\Gbf}(g)^{\epsilon(d)},
\end{equation}
    where $\langle M_{\infty}\cdot v\rangle$ is the span of the orbit of $v$ under $M_\infty$
    and
    \begin{equation}\label{eq:epsilon}
        \epsilon(d) =\begin{cases} 1/2 & \text{if } d=1,\\
        1 & \text{otherwise}.
        \end{cases}
    \end{equation}
    
    \end{theorem}

    Theorem~\ref{thm: matrix coeffs} has the following corollary for smooth functions, which are not necessarily $M_\infty$-finite.
    
    \begin{corollary} \label{cor:matrix coeffs smooth}
    Let $W_{\mrm{f}}<\Gbf_f$ be a compact open subgroup.
    For all $g\in\Gbf_S$ and $W_{\mathrm{f}}$-invariant smooth functions $\psi_1,\psi_2\in \mrm{L}_{00}^{2}(X_S)$ such that $\Scal_{2,\ell}(\psi_i)<\infty$, $i=1,2$, we have
    \begin{equation*}
      \lvert\langle g\psi_1,\psi_2\rangle\rvert\ll_{W_{\mrm{f}}} \Scal_{2,\ell}(\psi_1)\Scal_{2,\ell}(\psi_2)\xi_{\Gbf}(g)^{\epsilon(d)},
    \end{equation*}
    where $\ell=\dim(M_\infty)$.
    \end{corollary}
    
    \begin{proof}
        As $\psi_1$ and $\psi_2$ admit an orthogonal decomposition in terms of $M_\infty$-finite vectors, the argument in~\cite[Section 6.2.1]{EMV} applies with minimal changes to deduce the corollary where one replaces Eq.~(6.7) in \textit{loc.~cit.}~with~\eqref{eq:GMO}. Note that Eq.~(6.7) in~\cite{EMV} involves $\dim\langle M_\infty\cdot\rangle^{1/2} $. The argument goes through however and the resulting Sobolev norm is of order $\ell=\dim M_\infty$ (instead of $\ell = \left\lceil \dim M_\infty/2\right\rceil$ as in~\cite{EMV}).
        
    \end{proof}

%--------------------------------------------------

\section{Spectral Gap of S-Arithmetic Random Walks}
\label{sec:spectral gap}

The goal of this section is to introduce certain $S$-arithmetic operators which are naturally associated to a rational IFS and which leave fractal measures invariant.
Moreover, we prove that these operators possess a spectral gap and we provide an estimate on their operator norms, Proposition~\ref{prop: spectral gap}.
The results of this section are key ingredients in our equidistribution theorems.

\subsection{The S-arithmetic random walk}
\label{sec: QIFS}

    Given $(t,s)\in \R^\times \times \Q_{S_{\mrm{f}}}^\times$ and $(\xbf,\mathbf{y}) \in \R^d \times \Q_{S_{\mrm{f}}}^d$ regarded as a pair of column vectors, we define the following elements of $\Gbf_S$:
\begin{equation*}
     a(t,s) = \left(\begin{pmatrix} t\mrm{Id}_{d} & \mathbf{0} \\ \mathbf{0} & 1
        \end{pmatrix} , \begin{pmatrix} s\mrm{Id}_{d} & \mathbf{0} \\ \mathbf{0} & 1
        \end{pmatrix}\right), \qquad
        u(\xbf,\mathbf{y})  =\left( \begin{pmatrix}\mrm{Id}_{d} & \xbf \\ \mathbf{0} & 1
        \end{pmatrix}, \begin{pmatrix} \mrm{Id}_{d}& \mathbf{y} \\ \mathbf{0} & 1
        \end{pmatrix} \right),
\end{equation*}
where $\mrm{Id}_{d}$ is the identity matrix in dimension $d$.
    We also use the following notation:
    \begin{equation*}
        a(t) = \begin{pmatrix} t\mrm{Id}_{d} & \mathbf{0} \\ \mathbf{0} & 1
        \end{pmatrix}, \qquad 
         u(\xbf)  =\begin{pmatrix}\mrm{Id}_{d} & \xbf \\ \mathbf{0} & 1
        \end{pmatrix}.
    \end{equation*}
    The above matrices are regarded as elements of $\Gbf_\infty$ (resp. $\Gbf_f$) whenever their parameters belong to $\R^\times$ (resp. $\Q_{S_{\mrm{f}}}^\times$). We will denote by $\mrm{Id}$ the identity element in $\Gbf$.

    Throughout the remainder of this section, we fix a tuple $(\Fcal,\lambda)$ as in Section~\ref{sec: IFS} such that $\Fcal$ is rational.
    We will define an analogue of the operator $P_\l$ on the homogeneous space $\quotient{\GammaS}{\Gbf_S}$; cf. Definition~\ref{def: markov operator}. 
    Using the rationality of $\Fcal$, let $S(\Fcal)$ denote the smallest set of primes such that 
    \begin{equation*}
       \rho_i \in \Z[S(\Fcal)^{-1}]^{\times}, \quad b_i \in \Z[S(\Fcal)^{-1}]^d, \qquad  O_i \in \mrm{GL}(d,\Z[S(\Fcal)^{-1}]), 
    \end{equation*}
    for all $ i\in \L$. In addition to the prime factors of the numerators and denominators of $\rho_i$, $S(\Fcal)$ need only contain primes in the denominators of the components of $b_i$ and the entries in $O_i$, but not the numerators.
    We let $S=S(\Fcal)\cup\{\infty\}$, omitting the dependence on $\Fcal$.

    Given $\a\in \L^\ast$, let $k_{\a}\in\GammaS$ be given by
    \begin{equation}\label{eq:k_omega}
        k_\a = \begin{pmatrix} O_\a & \mathbf{0} \\ \mathbf{0} & 1\end{pmatrix},
    \end{equation}
    where $O_\a$ is defined as in~\eqref{eq: composition parameters}.
    The following elements of $\Gbf_S$ will be central to the analysis to follow: 
    \begin{equation} \label{def: gamma_w}
      \gamma_\a = u(\mathbf{0}, -b_\a) a(\rho_\a, \rho_\a)(k_\a,k_\a).
    \end{equation}
    The crucial property of $\gamma_{\a}$ is that for all $\mathbf{x}\in\bR^{d}$ we have
      \begin{equation}
        \label{eq:crucialpropertygamma}
        \gamma_{\a}u(\mathbf{x},\mathbf{0})\gamma_{\a}^{-1}u(b_{\a},\mathbf{0})=u(f_{\a}(\mathbf{x}),\mathbf{0}).
      \end{equation}
    We note that if $\a = (\a_i) \in \L^k$, then the following equality holds:
    \begin{equation}\label{eq:order of multiplication gamma}
        \gamma_\a = \gamma_{\a_1} \cdots \gamma_{\a_k}.
    \end{equation}
    Given a probability vector $\l$, we define an operator $\Pcal_\l$ by
    \begin{equation}\label{eq: padic MF}
        \Pcal_\l(\vp)(x) = \sum_{i \in \L} \l_i \vp( \gamma_i \cdot x),
    \end{equation}
    for all $\vp\in\mrm{C}\left(X_S\right)$, and all $x\in X_S$.
    In fact, we shall need a more general family of operators which we now define.
    Given $\a,\w\in\L^\ast$, we set
    \begin{equation}\label{eq: gamma_omega^alpha}
        \g^\a_\w := (k_\a^{-1},\mrm{Id})\g_\a\gamma_\w \g_\a^{-1}(k_\a,\mrm{Id}),
    \end{equation}
    and define $\a\cdot\Pcal_\l$ by
    \begin{equation}\label{eq: padic MF alpha}
        (\a\cdot\Pcal_\l)(\vp)(x) = \sum_{i \in \L} \l_i 
        \vp( \g_i^\a\cdot x),
    \end{equation}
    where $\a\cdot\Pcal_\l = \Pcal_\l$ if $\a$ is the empty word.
    For all $n\in\N$ we have
    \begin{equation*}
        (\a\cdot\Pcal_\l)^n(\vp)(x) = \sum_{\w \in \L^n} \l_\w 
        \vp\big( \g_\a\gamma_\w^\a\cdot x)\big).
    \end{equation*}
    These operators are among the main objects of study in this article.
    \begin{remark}
      In Appendix~\ref{sec:Cantor spectral gap}, we introduce a variant of the above operators which can be used to obtain sharper cutoffs in the case when the contraction ratios of the IFS are all equal.
      
    \end{remark}
    
    We need the following lemma.
   \begin{lem}\label{lem:index estimate}
     There exist positive constants $A$ and $L$ such that the following holds.
      Suppose $\vp$ is a function on $X_\infty(1)$. Then, for all $n\in\N$ and $\a\in\L^\ast$, there exists $N\geq 1$ such that $(\a\cdot\Pcal_\l)^n(\vp)$ is invariant under $K_{\mrm{f}}[N]\lhd K_{\mrm{f}}$ and
      \begin{equation*}
        [\Ga(1):\Ga(N)] \ll_{S,d} \rho_\a^{-A} \rho_{\min}^{-n L},
      \end{equation*}
      where $\rho_{\min}=\min\seti{\rho_i:i\in\L}$. When $\Fcal$ is a missing digit IFS, cf. Definition~\ref{def: p cantor}, we may take $A=6$, $L=3$.
    \end{lem}

    \begin{proof}
    For $g\in \Gbf_S$, let $\mrm{con}(g)(K_\mrm{f} )= g^{-1}K_{\mrm{f}}g\cap K_{\mrm{f}}$.
    Note that $g^{-1}K_{\mrm{f}}g\subseteq\Gbf_{\mrm{f}}$ is a compact open subgroup for every $g=(g_{\infty},g_{\mrm{f}})\in\Gbf_{S}$, as $g^{-1}K_{\mrm{f}}g=g_{\mrm{f}}^{-1}K_{\mrm{f}}g_{\mrm{f}}$.
    Denote by $W \subseteq K_{\mrm{f}}$ the compact open subgroup given by
    \begin{equation*}
      W = \bigcap_{\w\in \L^n } \mrm{con}(\g_\w^\a)(K_{\mrm{f}}).
    \end{equation*}
    Since $\vp$ is $K_{\mrm{f}}$-invariant, the function $(\a\cdot \Pcal_\l)^n(\vp)$ is invariant under $W$.
    We will find $N$ so that $W$ contains $K_{\mrm f}[N]$. 
      In view of Proposition~\ref{prop:congruenceviacompactopen-S-arith}, to bound the index $[\Ga(1):\Ga(N)]$, it will suffice to bound the index $[K_{\mrm{f}}:K_{\mrm{f}}[N]]$. We will obtain bounds on the latter by reducing the question to bounds on the index of the kernel of reduction mod $p^{\nu}$ for $\SL_{d+1}(\bZ_{p})$ for appropriate $\nu$. By~\cite[Corollary 2.8]{GL-ring}, we have
    \begin{equation}\label{eq:index-congruence-SL}
        \lvert\SL_{d+1}(\bZ/p^{\nu}\bZ)\rvert=p^{(d^{2}+2d)\nu}\prod_{k=2}^{d+1}\left(1-\frac{1}{p^{k}}\right).
    \end{equation} 
    Since $\SL_{d+1}(\bZ/p^{\nu}\bZ)$ is generated by unipotents~\cite[Theorem 4.3.9]{Hahn1989}, the reduction mod $p^{\nu}$ is surjective and therefore the right hand side of~\eqref{eq:index-congruence-SL} gives the desired bound at place $p$.
    
    Let $p\in S$ be a finite prime and recall that $M_{p}=\Gbf(\bZ_{p})$.
       Let $g_{p}\in\Gbf(\bQ_{p})$. By the KAK-decomposition (cf. Appendix~\ref{sec:KAK decomposition}), we can write $g_{p}=k_{1}ak_{2}$ for $k_1,k_2\in M_p$ and $a\in A_p^+$. Then,
          \begin{equation*}
            g_{p}M_{p}g_{p}^{-1}\cap M_{p}=k_{1}(aM_{p}a^{-1}\cap M_{p})k_{1}^{-1}.
          \end{equation*}

          Let $\lVert g_{p}\rVert_{p}=p^{\nu}$ denote the operator norm of the adjoint action of $g_{p}$. We claim that $aM_{p}a^{-1}\cap M_{p}$ (and thus also $g_{p}M_{p}g_{p}^{-1}\cap M_{p}$) contains $M_{p}[\lVert g_{p}\rVert_{p}]$, i.e.~the kernel of reduction mod $p^{\nu}=\lVert g_{p}\rVert_{p}$. Denote $D=(d+1)^{2}$ and let $n_{1}\geq\cdots\geq n_{D}=0$ be so that $a=\mrm{diag}(p^{-n_1},\dots,p^{-n_{D}})$. Note that $n_1=\nu$ by Corollary~\ref{cor:trivialcorollary}. Let $h\in M_{p}[p^{\nu}]$.
          The $(i,j)$-coordinate of $a^{-1}ha$ is given by multiplying the corresponding coordinate of $h$ by $p^{n_{i}-n_{j}}\geq p^{-\nu}$. As the diagonal entries of $h$ are preserved and the off-diagonal entries of $h$ are multiples of $p^{\nu}$, we obtain that $a^{-1}ha\in M_{p}$, i.e.~$h\in aM_{p}a^{-1}\cap M_{p}$. It follows that
          \begin{equation*}
              W=\prod_{p}(g_{p}M_{p}g_{p}^{-1}\cap M_{p})
          \end{equation*}
          contains the subgroup $K_{\mrm{f}}[N]=\prod_{p}M_{p}[\lvert N\rvert_{p}]$, where
          \begin{equation*}
            N=\prod_{p}\max\{\lVert\g_\a \g_\w \g_\a^{-1}\rVert_{p}:\w\in\L^n\}
          \end{equation*}
          where $\lVert \g_\a\g_\w\g_\a^{-1}\rVert_{p}$ denotes the operator norm for the adjoint action of the $p$-adic component.
          It thus remains to bound $\big[M_{p}:M_{p}[p^{\nu}]\big]$.
     
      As shown in the proof of Lemma~\ref{lem:identification-space-of-lattices}, we have $M_{p}=\Phi(\GL_{d+1}(\bZ_{p}))$. Hence $M_{p}=FU_{p}$, where $U_{p}=\Phi(\SL_{d+1}(\bZ_{p}))$ and $F$ is the image of
          \begin{equation}\label{eq:F}
            T_{p}=\left\{\begin{pmatrix}
                x & 0 \\
                0 & \mrm{Id}_{d}
              \end{pmatrix}:x\in\bZ_{p}^{\times}\right\}<\GL_{d+1}(\bZ_{p}),
          \end{equation}
          under $\Phi$. Given $r\in\bN$, let $\bZ_{p}^{\times r}\subseteq \bZ_{p}^{\times}$ denote the set of elements admitting an $r$-th root in $\bZ_{p}$. We have that $F/(F\cap U_{p})\cong\bZ_{p}^{\times}/\bZ_{p}^{\times(d+1)}$ is finite with cardinality depending only on $d$ and $p$. Letting $U_{p}[p^{\nu}]=U_{p}\cap M_{p}[p^{\nu}]$, we obtain
          \begin{equation*}
            \big[M_{p}:M_{p}[p^{\nu}]\big]\ll_{p,d} \big[U_{p}:U_{p}[p^{\nu}]\big].
          \end{equation*}
          
    As the kernel of the reduction mod $p^{\nu}$ in $\SL_{d+1}(\bZ_{p})$ is mapped into $U_{p}[p^{\nu}]$ under $\Phi$,~\eqref{eq:index-congruence-SL} yields
    \begin{equation}\label{eq:index-congruence-sccover}
        \big[U_{p}:U_{p}[p^{\nu}]\big]\leq p^{(d^{2}+2d)\nu}\prod_{k=2}^{d+1}\left(1-\frac{1}{p^{k}}\right).
    \end{equation}

    Hence, we conclude that
          \begin{equation*}
            [K_{\mrm{f}}:W]\leq\big[K_{\mrm{f}}:K_{\mrm{f}}[N]\big]\ll_{S,d}\prod_{p}\max\{\lVert\g_\a \g_\w \g_\a^{-1}\rVert_{p}:\w\in\L^n\}^{d^{2}+2d}.
          \end{equation*}
    By Lemmas~\ref{lem:padic supnorm} and~\ref{lem:KAKvsOpNorm}, the norm $\norm{\cdot}_p$ is submultiplicative and satisfies $\norm{g_p}_p=\norm{g_p^{-1}}_p$.
    Hence, we find that
          \begin{equation*}
            \big[K_{\mrm{f}}:K_{\mrm{f}}[N]\big]\ll_{S,d}\prod_{p}\lVert\gamma_\a\rVert_{p}^{2(d^{2}+2d)}\max\{\lVert\g_\w\rVert_{p}:\w\in\L^n\}^{d^{2}+2d}.
          \end{equation*}
        \item Given $p\in\Sf$, we let $c_{p}=\max\{\lVert\gamma_{i}\rVert_{p}:i\in\L\}$. Note that $c_{p}\geq 1$, and hence there are $L_{p},A_{p}>0$ such that
          \begin{equation*}
            \rho_{\min}^{-L_{p}}=\rho_{\max}^{-A_{p}}=c_{p}.
          \end{equation*}
          We define
          \begin{equation*}
            L=(d^{2}+2d)\sum_{p\in\Sf}L_{p},\quad A=2(d^{2}+2d)\sum_{p\in\Sf}A_{p}.
          \end{equation*}
          Using submultiplicativity and~\eqref{eq:order of multiplication gamma} and denoting by $\lvert\a\rvert$ the length of $\a$, we get
          \begin{equation}\label{eq:compact index}
            \big[K_{\mrm{f}}:K_{\mrm{f}}[N]\big]\ll_{S,d}\prod_{p}c_{p}^{(d^{2}+2d)(2\lvert\a\rvert+n)}=\rho_{\max}^{-\lvert\a\rvert A}\rho_{\min}^{-nL}\leq\rho_{\a}^{-A}\rho_{\min}^{-nL}.
          \end{equation}
        \item Recall that $\Gamma(N)=\Gamma_{S}\cap K_{\mrm{f}}[N]$; cf.~Proposition~\ref{prop:congruenceviacompactopen-S-arith}. Applying the second isomorphism theorem with ambient group $K_{\mrm{f}}$ and subgroups $\GammaS\cap K_{\mrm{f}}$ and $K_{\mrm{f}}[N]\lhd K_{\mrm{f}}$, we get
          \begin{equation*}
            [\Gamma(1):\Gamma(N)]=\big[\GammaS\cap K_{\mrm{f}}:\GammaS\cap K_{\mrm{f}}[N]\big]\leq\big[K_{\mrm{f}}:K_{\mrm{f}}[N]\big].
          \end{equation*}
      Combining all of the above, one obtains
          \begin{equation*}
            [\Gamma(1):\Gamma(N)]\ll_{S,d}\rho_{\a}^{-2A}\rho_{\min}^{-nL}.
          \end{equation*}
      
      This completes the proof in the general case. For missing digit Cantor sets, assume that $q\in\bN$ is at least two and $\Lambda\subseteq\{0,\ldots,q-1\}$ such that $\#\Lambda \geq 2$. Then $\rho_{i}=1/q$, $k_{i}=\mrm{Id}$, and $b_{i}=i/q$ for all $i\in\Lambda$. 
      For $\a\in \L^\ast$, using~\eqref{eq:order of multiplication gamma} and denoting by $\gamma_{\a,p}$ the $p$-adic component of $\gamma_{\a}$, we have
      \begin{align*}
      \g_{\a,p} = \begin{pmatrix}
        \rho_\a & -b_{\a} \\
        0 & 1
      \end{pmatrix}.
      \end{align*}
      Let $x=\rho_\a$, $t=-b_{\a}$, and fix a prime divisor $p$ of $q$.
      If $\lvert t\rvert_{p}\leq\lvert x\rvert_{p}$, we have
      \begin{equation*}
        \g_{\a,p}=\begin{pmatrix}
          x & 0 \\
          0 & 1
        \end{pmatrix}\begin{pmatrix}
          1 & x^{-1}t \\
          0 & 1
        \end{pmatrix}
      \end{equation*}
      and hence $\lVert \g_{\a}\rVert_{p}=|x|_p$. If $\lvert x\rvert_{p}<\lvert t\rvert_{p}$, then we note that
      \begin{equation*}
        M_{p}\g_{\a,p} M_{p}=M_{p}\g_{\a,p} \begin{pmatrix} 1 & 0 \\ 1 & 1\end{pmatrix}M_{p}
      \end{equation*}
      and hence it suffices to calculate the norm of
      \begin{equation*}
        \g_{\a,p}\begin{pmatrix}
          1 & 0 \\
          1 & 1
        \end{pmatrix}
        =\begin{pmatrix}
          x+t & t \\
          1 & 1
        \end{pmatrix}.
      \end{equation*}
      As argued in the proof of Lemma~\ref{lem:KAK2x2}, there are unipotent elements $u_{1},u_{2}\in \GL_{2}(\bZ_{p})$ such that
      \begin{equation*}
        u_{1}\begin{pmatrix}
          x+t & t \\
          1 & 1
        \end{pmatrix}u_{2}=
        \begin{pmatrix}
          x+t & 0 \\
          0 & \frac{x}{x+t}
        \end{pmatrix}.
      \end{equation*}
     Therefore, we obtain that
      \begin{equation*}
        \g_{\a,p}\in M_{p}\begin{pmatrix}
          \lvert t\rvert_{p} & 0 \\
          0 & \lvert x\rvert_{p}/\lvert t\rvert_{p}
        \end{pmatrix}M_{p}=M_{p}\begin{pmatrix}
          \lvert t\rvert_{p}^{2}/\lvert x\rvert_{p} & 0 \\
          0 & 1
        \end{pmatrix}M_{p}.
      \end{equation*}
    Hence, $\norm{\g_\a}_p\leq \max\{|x|_p,|t|^2_p/|x|_p\}$.
    Now, note that $|x|_p = |q|_p^{|\a|}$ and $|t|_p \leqslant |q|_p^{|\a|}$. 
    It follows that
    \begin{align*}
        \prod_{p} \norm{\g_\a}_p \leqslant \prod_{p} |q|_p^{|\a|} = q^{|\a|}.
    \end{align*}
    Since $\rho_{\min} =\rho_{\max} = q^{-1}$ and $d=1$ in this case, it follows by~\eqref{eq:compact index} that we may take $A=6$ and $L=3$.

     \end{proof}
    
%--------------------------------------------------

\subsection{Spectral gap for the averaging operator}

    We wish to estimate the operator norm of $\a\cdot \Pcal_\l$ on the subspace of $\mrm{L}^2_{00}(X_S)$ consisting of functions which are invariant by a compact open subgroup of $W_{\mrm{f}}<\Gbf_{\Sf}$.

    The main difficulty is that the subgroup generated by $\seti{\g_\w:\w\in\L^\ast}$ is not discrete or free in general (although the subsemigroup is).
    This causes difficulty in controlling the separation of the $\g_\w$'s.
    To explain the idea, let us focus on the case $\a=\emptyset$. Observe that for each word $\w$, 
    \begin{equation}
        \tilde{\g}_\w := u(-b_\w,\mathbf{0})\g_\w  \in \GammaS.
    \end{equation}
    In particular, the subgroup generated by $\seti{\tilde{\g}_\w:\w\in\L^\ast}$ is discrete since it is contained in the lattice.
    Moreover, all the elements of the form $u(-b_\w,\mathbf{0})$ belong to a compact neighborhood of identity (recall the first coordinate corresponds to the Archimedean place).
    This allows us to relate the spectral properties of $\Pcal_\l$ to an operator which is supported on the lattice $\Ga_S$.

    The following is one of the key results of this article. The reader is referred to Proposition~\ref{prop:cantor spectral gap} for sharper bounds for missing digit Cantor sets.
    \begin{proposition}\label{prop: separate group from rep}
    \label{prop: spectral gap}
    
    Assume that $\Fcal$ has no exact overlaps; cf.~\eqref{eq:exact overlaps}.
    Let $W_{\mrm{f}}<\Gbf_{\mrm{f}}$ be a compact open subgroup and $\ell =\dim(M_\infty)$.
    For every $\a\in \L^\ast$ and every $k\in\N$, there exists a finite set $\Delta_k\subset \GammaS$, which is determined by $\a$ and the IFS, such that the following holds.
    Suppose $p> 1$ is given and let $\th$ be the H\"older conjugate of $p$.
    Let $q = 2\th/(\th+1)$.
    Then, for every smooth $W_{\mrm{f}}$-invariant $\vp\in \mrm{L}^2_{00}(X_S)$,
    \begin{equation*}
        \big\lVert(\a\cdot\Pcal_\l)^{k}(\vp)\big\rVert^2_{L^2} \ll_{\Fcal,W_{\mrm{f}}} \Scal_{2,\ell}(\vp)^2 \left( \sum_{i\in\L}\l_i^q \right)^{2k/q} \left(\sum_{g\in\Delta_k} \xi_\Gbf^{p\epsilon(d)}( g)  \right)^{\frac{1}{p}},
    \end{equation*}
    where $\epsilon(d)$ is defined in~\eqref{eq:epsilon}.
    In particular, if $\xi_\Gbf^{\epsilon(d)}\in \ell^{p}(\GammaS)$, then
    \begin{equation*}
        \big\lVert(\a\cdot\Pcal_\l)^{k}(\vp)\big\rVert^2_{L^2} \ll_{\Fcal,W_{\mrm{f}}} \Scal_{2,\ell}(\vp)^2 \left( \sum_{i\in\L}\l_i^q \right)^{2k/q} \big\lVert \xi_\Gbf^{\epsilon(d)}\big\rVert_{\ell^{p}(\GammaS)}.
    \end{equation*}
    \end{proposition}

    \begin{proof}
    
    In order to simplify notation, let
    \begin{equation*}
  \eta = \xi_\Gbf^{\epsilon(d)},\qquad \Scal = \Scal_{2,\ell}.      
    \end{equation*}
    For $g\in \Gbf_S$, we write $g\vp$ to denote $\vp\circ g$.
    Fix $\a\in \L^\ast$. Given $\w\in \L^\ast$, let $\g_\w^\a$ be as in~\eqref{eq: gamma_omega^alpha}.
    A direct computation shows that the Archimedean component of $\g^\a_\w$ is $k_\w a(\rho_\w)$ while its non-Archimedean component is given by
    \begin{equation}\label{eq:Langlands decomp}
        u\big(-f_\a f_\w f_\a^{-1}(\mathbf{0})\big) \cdot 
    k_\a k_\w k_\a^{-1}  \cdot a(\rho_\w).
    \end{equation}
    Let $\t_\w^\a = u\big(-f_\a f_\w f_\a^{-1}(\mathbf{0})\big) 
    k_\a k_\w k_\a^{-1} k_\w^{-1}$, so that $\tilde{\g}^\a_\w : = (\t^\a_\w,\mrm{Id}) \g_\w^\a\in\Ga_{S}$.
    It follows by~\eqref{eq:order of multiplication gamma} that
    \begin{equation}\label{eq:gamma tilde multiplicative}
        \tilde{\g}^\a_\w = \tilde{\g}^\a_{\w_1}\cdots \tilde{\g}^\a_{\w_k}
    \end{equation}
    for all $\a,\w=(\w_1,\dots,\w_k)\in\L^\ast$.
    
    By calculating the translation vector of the similarities $f_\a f_\w f_\a^{-1}$, one sees that $u(-f_\a f_\w f_\a^{-1}(\mathbf{0}))$ is uniformly bounded in $\Gbf_\infty$, independently of $\a$ and $\w$.
    Moreover, at the Archimedean place the elements $k_\bullet$ are all contained in a compact subgroup. Hence, it follows that $\{(\t^\b_\sigma,\mrm{Id}):\b,\sigma\in \L^\ast\}$ is contained in a bounded set $\Ocal\subseteq\Gbf_S$, which depends only on the IFS $\Fcal$. 
    By Lemma~\ref{lem:Sobolev}, there exists $C_\Ocal \geq 1$ such that for any smooth function $\psi$ and for any $g\in\Ocal\cup\Ocal^{-1}$,
    \begin{equation}\label{eq:sobolev scaling}
         \Scal(g\psi) \leq C_\Ocal \Scal(\psi).
    \end{equation}

    Let $\overline{\tau}_{\w}^{\a}=(\tau_{\w}^{\a},\mrm{Id})^{-1}$. It follows from Corollary~\ref{cor:matrix coeffs smooth} that
    \begin{align*}
        \big\lVert(\a\cdot\Pcal_\l)^{k}(\vp)\big\rVert^2  &= 
        \sum_{u,\w\in \L^k} \l_u\l_\w \langle \g^\a_u \vp,\g^\a_\w \vp  \rangle
        = \sum_{u,\w\in \L^k} \l_u\l_\w \big\langle \tilde{\g}^\a_u (\overline{\t}^\a_u\vp),\tilde{\g}^\a_\w (\overline{\t}^\a_\w\vp)   \big\rangle\\
        &\ll_{W_{\mrm{f}}} \sum_{u,\w\in \L^k} \l_u\l_\w \eta\big(\tilde{\g}^\a_\w(\tilde{\g}^\a_u)^{-1}\big)
        \Scal\left(\overline{\t}_{u}^{\a}\vp\right)
        \Scal\left(\overline{\t}_{\w}^{\a}\vp\right).
    \end{align*}
    By~\eqref{eq:sobolev scaling}, we get
    \begin{equation}\label{eq:remove tau}
         \big\lVert(\a\cdot\Pcal_\l)^{k}(\vp)\big\rVert^2 \ll_{\Fcal,W_{\mrm{f}}} \Scal(\vp)^2 
         \sum_{u,\w\in \L^k} \l_u\l_\w \eta\big(\tilde{\g}^\a_\w(\tilde{\g}^\a_u)^{-1}\big).
    \end{equation}
    
    Denote by $\nu_\a$ the measure supported on $\seti{\tilde{\g}^\a_i:i\in\L}$ such that $\nu_\a(\tilde{\g}^\a_i) = \l_i$.
    In particular, $\nu_\a$ is supported on $\GammaS$.
    Moreover, in view of~\eqref{eq:gamma tilde multiplicative}, we have for every $g\in \Gbf_S$ and $k\in \N$ that
    \begin{equation}\label{eq:convolution}
        \nu_\a^{*k}(g) = \sum_{\w\in\L^k:\tilde{\g}^\a_\w = g} \l_\w,
    \end{equation}
    where $\nu_\a^{*k}$ denotes the $k^{\text{th}}$ convolution power of $\nu_\a$.
    Denote by $\check{\nu}_\a$ the adjoint of $\nu_{\a}$, defined as the push-forward of $\nu_\a$ under the map $g\mapsto g^{-1}$.
    For $k\in \N$, let $\Delta_k \subset \GammaS$ be the (finite) support of the measure $\nu_\a^{*k} * (\check{\nu}_\a)^{*k}$.
    With this notation, we can rewrite the upper bound in~\eqref{eq:remove tau} as
    \begin{equation*}
        \sum_{u,\w\in \L^k} \l_u\l_\w \eta(\tilde{\g}^\a_\w(\tilde{\g}^\a_u)^{-1} ) 
         =  \int_{\Gbf_S} \eta(g) \;\der \big(\nu_\a^{*k} * (\check{\nu}_\a)^{*k} \big)(g).
    \end{equation*}
    Recall that $\th$ denotes the H\"older conjugate of $p$. 
    By H\"older's inequality, we obtain:
    \begin{align*}
    \int_{\Gbf_S} \eta(g) \;\der \big(\nu_\a^{*k} * (\check{\nu}_\a)^{*k} \big)(g) &= 
        \sum_{g\in \GammaS} \big(\nu_\a^{*k} * (\check{\nu}_\a)^{*k}\big) (g)  \eta( g) \\
        &\leqslant
        \Bigg(\sum_{g\in \GammaS} \big(\big(\nu_\a^{*k} * (\check{\nu}_\a)^{*k} \big) (g)\big)^{\th} \Bigg)^{\frac{1}{\th}} \Bigg(\sum_{g\in\Delta_k} \eta^{p}( g)  \Bigg)^{\frac{1}{p}}.
    \end{align*}
    By Young's inequality, applied with $q = 2\th/(1+\th)$, 
    \begin{equation*}
    \big\lVert\nu_\a^{*k} * (\check{\nu}_\a)^{*k}\big\rVert_{\ell^\th(\GammaS)} \leqslant
    \big\lVert\nu_\a^{*k}\big\rVert_{\ell^q(\GammaS)} \big\lVert(\check{\nu}_\a)^{*k}\big\rVert_{\ell^q(\GammaS)} = \big\lVert\nu_\a^{*k}\big\rVert^2_{\ell^q(\GammaS)}.
    \end{equation*}
    
    Since $\Fcal$ has no exact overlaps, the sub-semigroup generated by $\seti{\tilde{\g}^\a_i:i\in\L}$ is free.
    Indeed, this can be seen directly from the decomposition in~\eqref{eq:Langlands decomp} of the elements $\tilde{\g}^\a_\w$.
    In particular, for all $u,\w\in \L^k$, 
    \begin{equation*}
    \tilde{\g}^\a_u = \tilde{\g}^\a_\w \Longleftrightarrow u =\w.    
    \end{equation*}
    Combined with~\eqref{eq:convolution}, it follows that
     \begin{equation*}
        \big\lVert\nu_\a^{*k}\big\rVert^2_{\ell^q(\GammaS)} =  
        \Bigg( \sum_{\w\in\L^k}\l_\w^q \Bigg)^{2/q} = \Bigg( \sum_{i\in\L}\l_i^q \Bigg)^{2k/q},
     \end{equation*}
     where the last equality follows by Lemma~\ref{lem: average of submul coc is subadditive}, applied with $\t_i = \l_i^{q-1}$.
     This completes the proof.
    \end{proof}

%--------------------------------------------------
\subsection{Summability of Matrix Coefficients}
\label{sec:summability}
 We show that the matrix coefficients of $\Ga_S$ acting on $\mrm{L}^2_{00}(\Gbf_S/\Ga_S)$ belong to $\ell^p(\GammaS)$ for an explicit choice of $p$, Proposition~\ref{prop:summable}.
This verifies the hypothesis of the last assertion of Proposition~\ref{prop: spectral gap} for this value of $p$, thus completing the proof of the bound on the norm of the operators $\Pcal_\l$.
\begin{proposition}\label{prop:summable}
  Let
  \begin{equation*}
    v(d)=\left(\left\lfloor\frac{d}{2}\right\rfloor+1\right)\left\lceil\frac{d}{2}\right\rceil, \qquad p=2v(d).
  \end{equation*}
  Then, $\xi_{\Gbf}\in \ell^{p+\varepsilon}(\GammaS)$ for all $\varepsilon>0$.
\end{proposition}
%--------------------------------------------------

As a first step, we bound the functions $\xi_{v}$ in terms of the operator norm of the adjoint action.
We recall that for any place $v$ of $\bQ$ the group $\Gbf(\bQ_{v})$ admits a so-called KAK-decomposition; cf. Appendix~\ref{sec:KAK decomposition}. More precisely, for any $g\in\Gbf(\bQ_{v})$ there are $k_{1},k_{2}\in M_{v}$ and $a\in A_{v}^{+}$ such that
\begin{equation}\label{eq:KAK}
  g=k_{1}a k_{2}.
\end{equation}

\begin{proposition}\label{prop:boundHarishChandra}
  Let $D=\lfloor \frac{d+1}{2}\rfloor$.
  For all sufficiently small $\varepsilon>0$ there exists a  constant $\delta(\varepsilon)>0$ such that for all $g\in\Gbf(\Q_{v})$,
  \begin{equation}\label{eq:boundHarish-Chandra}
    \delta_{1}\lVert g\rVert_{v}^{-\frac{1}{2}D}\leq\xi_{v}(g)\leq\delta_{2}(\varepsilon)\lVert g\rVert_{v}^{-\frac{1}{2}+\varepsilon}.
  \end{equation}
\end{proposition}
\begin{remark}
  Note that $D=1$ if $d=1$ or $d=2$.
\end{remark}
\begin{proof}
  We start with the case where $v=p$ is a finite place.
  Given $g\in\Gbf(\Qp)$, let $a(g)\in A_{p}^{+}$ denote the Cartan element defined by~\eqref{eq:KAK} normalized so that the bottom right entry equals $1$. Using~\cite[Thm.~5.9]{Oh} and~\cite[Prop.~2.3]{OhSystems}, cf.~\cite[\textsection6.1]{Oh}, we find $\delta_{1},\delta_{2}(\varepsilon)>0$ such that
  \begin{equation}\label{eq:boundsHarishChandrap-adic}
    \delta_{1}\left(\prod_{i=1}^{\lfloor\frac{d+1}{2}\rfloor}\frac{\lvert a(g)_{i}\rvert_{p}}{\lvert a(g)_{d+1-(i-1)}\rvert_{p}}\right)^{-\frac{1}{2}}\leq\xi_{p}(g)\leq\delta_{2}(\varepsilon)
    \left(\prod_{i=1}^{\lfloor\frac{d+1}{2}\rfloor}\frac{\lvert a(g)_{i}\rvert_{p}}{\lvert a(g)_{d+1-(i-1)}\rvert_{p}}\right)^{-\frac{1}{2}+\varepsilon}.
  \end{equation}
  It remains to bound the product appearing in~\eqref{eq:boundsHarishChandrap-adic} in terms of $\lVert\Ad_{g}\rVert_{p}$. Using \eqref{eq:KAK} assume without loss of generality that
  \begin{equation*}
    g=\mathrm{diag}(p^{-n_{1}},\ldots,p^{-n_{d}},1)
  \end{equation*}
  for integers $n_{1}\geq\cdots\geq n_{d}\geq 0$. Then,
  \begin{equation*}
    \prod_{i=1}^{\lfloor\frac{d+1}{2}\rfloor}\frac{\lvert a(g)_{i}\rvert_{p}}{\lvert a(g)_{d+1-(i-1)}\rvert_{p}}=p^{\eta(g)},
  \end{equation*}
  where
  \begin{equation*}
    \eta(g)=n_{1}-n_{d+1}+n_{2}-n_{d}+\cdots+n_{\lfloor\frac{d+1}{2}\rfloor}-n_{d+2-\lfloor\frac{d+1}{2}\rfloor}
  \end{equation*}
  and in particular
  \begin{equation}\label{eq:boundexponent}
    n_{1}\leq\eta(g)\leq Dn_{1}.
  \end{equation}
  Therefore, since $\lVert g\rVert_{p}=\lVert\Ad_{g}\rVert_{p}=p^{n_{1}}$, we obtain the claim for finite places of $\bQ$.
  
  If $v$ is the infinite place the argument is very similar. We again have 
  \begin{equation}\label{eq:boundsHarishChandraarchimedean}
    \delta_{1}\left(\prod_{i=1}^{\lfloor\frac{d+1}{2}\rfloor}\frac{a(g)_{i}}{a(g)_{d+1-(i-1)}}\right)^{-\frac{1}{2}}\leq\xi_{\infty}(g)\leq\delta_{2}(\varepsilon)\left(\prod_{i=1}^{\lfloor\frac{d+1}{2}\rfloor}\frac{a(g)_{i}}{a(g)_{d+1-(i-1)}}\right)^{-\frac{1}{2}+\varepsilon}
  \end{equation}
  and
  \begin{equation}\label{eq:archimedeanintermediatebound}
    a(g)_{1}\leq\prod_{i=1}^{\lfloor\frac{d+1}{2}\rfloor}\frac{a(g)_{i}}{a(g)_{d+1-(i-1)}}\leq a(g)_{1}^{D}.
  \end{equation}

  We recall that $\lVert\Ad_{g}\rVert_{\infty}=\lVert\Ad_{a(g)}\rVert_{\infty}$ and $\Ad_{a(g)}$ is diagonalizable with eigenvalues
  \begin{equation*}
    \sigma(\Ad_{a(g)})=\left\{\frac{a(g)_{i}}{a(g)_{j}}:1\leq i,j\leq d+1\right\}.
  \end{equation*}
By definition, we get
  
  \begin{equation*}
    \lVert\Ad_{g}\rVert_{\infty}^{2}=\sum_{1\leq i,j\leq d+1}\frac{a(g)_{i}^{2}}{a(g)_{j}^{2}}
  \end{equation*}
  and, in particular,
  \begin{equation}\label{eq:squeeze}
    a(g)_{1}\leq\lVert\Ad_{g}\rVert_{\infty}\leq(1+\dim\Gbf)^{\frac{1}{2}}a(g)_{1}.
  \end{equation}
  Combining this with~\eqref{eq:boundsHarishChandraarchimedean},~\eqref{eq:archimedeanintermediatebound} and $\lVert g\rVert_{\infty}=\lVert\Ad_{g}\rVert_{\infty}$, the claim follows.
\end{proof}
We record for later reference that~\eqref{eq:squeeze} implies that, given $t\in(0,\infty)$ and $a(t)$ as in Section~\ref{sec: QIFS}, we have
\begin{equation}\label{eq: norm flow}
    \lVert\Ad_{a(t)}\rVert_{\infty}=\lVert a(t)\rVert_{\infty}\asymp_{d}\max\{t,t^{-1}\}.
\end{equation}

\subsubsection{Volume growth for norm balls in $\Gbf(\Q_{p})$}\label{sec:volumegrowth}
In preparation of the proof of Proposition~\ref{prop:summable} we derive bounds on the volume of norm balls in $\Gbf(\Q_{p})$ for $p$ a finite place of $\Q$. We fix a choice of a Haar measure $m_{p}$ on $\Gbf(\Qp)$ such that $m_{p}(\Gbf(\Zp))=1$.
\begin{lem}\label{lem:padicshells}
  Let $n\in\N\cup\{0\}$ and set
  \begin{equation*}
    v(d)=\left(\left\lfloor\frac{d}{2}\right\rfloor+1\right)\left\lceil\frac{d}{2}\right\rceil.
  \end{equation*}
  Then for all $\varepsilon>0$ we have
  \begin{equation*}
    p^{v(d)n}\ll_{d}\vol{\{g\in\Gbf(\Qp):\lVert g\rVert_{p}=p^{n}\}}\ll_{d,\varepsilon}p^{(v(d)+\varepsilon)n}.
  \end{equation*}
  If $d=1$, the latter bound remains valid for $\varepsilon=0$.
\end{lem}
\begin{proof}
  Note that $\lVert g\rVert_{p}=1$ if and only if $g\in\Gbf(\bZ_{p})$. Hence we can assume that $n\in\bN$. Let
  \begin{equation*}
    I_{n}^{d}=\{(n_{1},\ldots,n_{d})\in(\N\cup\{0\})^{d}:n=n_{1}\geq\cdots\geq n_{d}\}.
  \end{equation*}
  Given $\mathbf{n}\in I_{n}^{d}$, we let $a_{\mathbf{n}}=\mathrm{diag}(p^{-n_{1}},\ldots,p^{-n_{d}},1)$. Then
  \begin{equation}\label{eq:decompositionshells}
    \{g\in\Gbf(\Qp):\lVert g\rVert_{p}=p^{n}\}=\bigsqcup_{\mathbf{n}\in I_{n}^{d}}Ka_{\mathbf{n}}K.
  \end{equation}
  It therefore remains to determine the cardinality of the set $I_{n}^{d}$ and the Haar measure of sets of the form $Ka_{\mathbf{n}}K$.
  It is known, cf.~\cite[Lem.~4.1.1]{Silberger}, that
  \begin{equation}\label{eq:volumebymodularfunction}
    \vol{Ka_{\mathbf{n}}K}\asymp\delta_{\Bbf}(a_{\mathbf{n}}),
  \end{equation}
  where $\delta_{\Bbf}$ is the modular character on the image $\Bbf\leq\Gbf$ of the upper triangular subgroup in $\GL_{d+1}$.
   Let $\mathbf{n}\in I_{n}^{d}$. One calculates (cf.~Appendix~\ref{sec:modular char})
  \begin{equation}\label{eq:determineshellexponent}
    \delta_{\Bbf}(a_{\mathbf{n}})=\exp\left(\sum_{i=1}^{d}n_{i}(d+2-2i)\log p\right).
  \end{equation}
  As $n_{i}\geq 0$, the right-hand side attains its maximum at $\mathbf{n}^{\ast}$ given by
  \begin{equation*}
    n_{i}^{\ast}=\begin{cases}
      n & \text{if }2i\leq d+2,\\
      0 & \text{else}.
    \end{cases}
  \end{equation*}
  One calculates
  \begin{align*}
    \sum_{i=1}^{d}n_{i}^{\ast}(d+2-2i)=n\left(\left\lfloor\frac{d}{2}\right\rfloor+1\right)\left\lceil\frac{d}{2}\right\rceil.
  \end{align*}
  In particular, combining~\eqref{eq:decompositionshells},~\eqref{eq:volumebymodularfunction} and~\eqref{eq:determineshellexponent}, we obtain
  \begin{equation*}
    p^{v(d)n}\ll_{d}\vol{\{g\in\Gbf(\Qp):\lVert g\rVert_{p}=p^{n}\}}\ll_{d}\lvert I_{n}^{d}\rvert p^{v(d)n}.
  \end{equation*}
  
  We next determine the cardinality of $I_{n}^{d}$. We first note that for $d=1$ we clearly have $\lvert I_{n}^{d}\rvert=1$ and therefore the last part of the lemma follows immediately. For general $d\in\bN$ we note that~$I_{n}^{d}$ is precisely the set of ordered $d-1$-tuples of non-negative integers at most equal to $n$ or, put differently,~$I_{n}^{d}$ identifies with the collection of multisets of cardinality $d-1$ with elements in $\{0,\ldots,n\}$. Therefore, we find
  \begin{equation*}
    \lvert I_{n}^{d}\rvert= \binom{n+d-1}{d-1}
    \asymp_{d} n^{d-1}.
  \end{equation*}
  It follows that
  \begin{equation*}
    \vol{\{g\in\Gbf(\Qp):\lVert g\rVert_{p}=p^{n}\}}\ll_{d}n^{d-1}p^{v(d)n}\ll_{d,\varepsilon}p^{(v(d)+\varepsilon)n}
  \end{equation*}
  and the lemma is proven.
\end{proof}
\begin{corollary}\label{cor:integrability-norm}
  Let $\eta_{S}:\Gbf_{S}\to(0,\infty)$ denote the function
  \begin{equation*}
    \eta_{S}(g)=\prod_{v\in S}\lVert g_{v}\rVert_{v}^{-1}\quad(g\in\Gbf_{S}).
  \end{equation*}
  Then, $\eta_{S}\in \mrm{L}^{v(d)+\varepsilon}(\Gbf_{S})$.
\end{corollary}
\begin{proof}
  This follows relatively easily from the description of the Haar measure in terms of the KAK-decomposition for the infinite place, cf.~\cite[Prop.~5.28]{Knapp1986}, and from the bound in Lemma~\ref{lem:padicshells} for finite places; cf.~Appendix~\ref{sec:integrability-eta}.
\end{proof}
%--------------------------------------------------

\subsection{Proof of Proposition~\ref{prop:summable}}
Note that for all $v\in S$ and for all $g_{v}\in\Gbf(\bQ_{v})$ we have $\lVert g_{v}^{-1}\rVert_{v}=\lVert g_{v}\rVert_{v}$. Hence submultiplicativity implies that for all $\gamma_{v}\in\Gbf(\bQ_{v})$ we have
\begin{equation*}
  \lVert g_{v}\rVert^{-1}_v\lVert\gamma_{v}\rVert_{v}\leq\lVert g_{v}\gamma_{v}\rVert_v.
\end{equation*}
Using the folding-unfolding technique and Corollary~\ref{cor:integrability-norm} we find that for $q>v(d)$ we have
\begin{equation*}
  \infty>\int_{\Gbf_{S}}\eta_{S}(g)^{q}\der g=\int_{X_{S}}\sum_{\gamma\in\GammaS}\eta_{S}(g\gamma)^{q}\der g\GammaS
\end{equation*}
and hence for almost all $g\in\Gbf_{S}$ we have
\begin{equation*}
  \sum_{\gamma\in\GammaS}\eta_{S}(g\gamma)^{q}<\infty
\end{equation*}
by Fubini's theorem. By submultiplicativity,
we have $\eta_{S}(g\gamma)^{q}\geq\eta_{S}(g)^{q}\eta_{S}(\gamma)^{q}$ and therefore
\begin{equation*}
  \sum_{\gamma\in\Gamma_S}\eta_{S}(\gamma)^{q}<\infty.
\end{equation*}
Using Proposition~\ref{prop:boundHarishChandra}, it follows that
\begin{equation*}
  \sum_{\gamma\in\GammaS}\xi_{\Gbf}(\gamma)^{\frac{2q}{1-2\varepsilon}}\ll_{\varepsilon,d}\sum_{\gamma\in\Gamma}\eta_{S}(\gamma)^{q}<\infty
\end{equation*}
and in particular $\xi_{\Gbf}\in\ell^{2v(d)+\varepsilon}(\GammaS)$.
\begin{remark}
  Note that the above argument works for any unimodular subgroup of $\Gbf_{S}$ in place of the integer lattice $\GammaS$.
\end{remark}

%--------------------------------------------------

\section{Expanding Horospheres and Congruence Covers}
\label{sec:horospheres}

    The goal of this section is to show that $a(t)$-translates of absolutely continuous measures on the horospherical group of $a(t)$ become equidistributed, in a suitable sense, towards the Haar measure on quotients of $\Gbf_\infty$ by principal congruence subgroups of $\Gbf(\Z)$, with an emphasis on obtaining a uniform error rate and implied constants, independently of the congruence level.
    The main result of this section is Proposition~\ref{prop:banana}.

Recall the notation introduced in Section~\ref{sec: s-arith}. 
     We fix a right-invariant Riemannian metric on $\Gbf_{\infty}$.
    This metric induces a right invariant metric on the connected component $\Gbf_{\infty}^{+}$ of $\Gbf_{\infty}$. For any lattice $\Gamma\leq\Gbf_{\infty}$, this induces a Riemannian metric on $\Gbf_{\infty}/\Gamma$ such that the canonical projection is a local isometry. Given $x\in\Gbf_{\infty}/\Gamma$, we denote by $\mrm{inj}_{\Gamma}(x)$ the supremum over all radii $R$ such that, for all $g\in\Gbf_{\infty}^{+}$ contained in the ball of radius $R$ at the origin, the map $g\mapsto gx$ is injective.  
    
    For a Lipschitz function $\vp$, we write $\mrm{Lip}(\vp)$ for its Lipschitz constant and we denote the space of Lipschitz functions on $\Gbf_{\infty}/\Ga$ by $\mrm{Lip}(\Gbf_\infty/\Ga)=\{\vp\in \mrm{C}(\Gbf_\infty/\Ga):\mrm{Lip}(\vp)<\infty\}$.

Note that $X_\infty(N)$ is in general not connected for $N>1$ and hence some care is needed in formulating equidistribution statements.
    The connected components of $X_\infty(N)$ correspond to the distinct orbits of $\Gbf_\infty^+$.
    Each such component supports a unique $\Gbf_\infty^+$-invariant Haar probability measure.
    Moreover, a function is \textbf{orthogonal to $\Gbf^+_\infty$-invariant functions} in $\mrm{L}^2(X_\infty(N))$ if and only if it has integral $0$ on each connected component of $X_\infty(N)$.

    The following is the main result of this section. 
    
\begin{prop}\label{prop:banana}
    
  There exist $\k>0, \ell\in\N$ and $C\geq 1$ such that the following holds.
  Let $\psi\in \mrm{C}_c^\ell(\R^d)$ be a non-negative function of integral $1$.
  Then, for every $N\in\N$, $\vp\in \mrm{B}_{2,\ell}^\infty(X_\infty(N))\cap \mrm{Lip}(X_\infty(N))$, and for every $t\geq1$, the following holds for all $x\in X_\infty(N)$:
  \begin{align*}
    &\int \vp\big(a(t)u(\xbf) x \big) \psi(\xbf) \dx{\xbf}
     \\
    &=\int \vp \;\der m_{\Gbf_\infty^+ \cdot x}
    + 
    O_\psi\Big(\sqrt{[\Gamma(1):\Gamma(N)]} \big(\Scal_{2,\ell}(\vp)+\mrm{Lip}(\vp)\big) \max\{\mrm{inj}_{\Ga(N)}(x)^{-C},1\}  \cdot t^{-\k} \Big),
  \end{align*}
  where $m_{\Gbf_\infty^+ \cdot x}$ is the unique $\Gbf_\infty^+$-invariant probability measure on $\Gbf_\infty^+\cdot x$ and
  $\mrm{inj}_{\Ga(N)}(x)$ denotes the injectivity radius at $x$.
  
\end{prop}

    Without any attempt to optimize the exponents, we show Proposition~\ref{prop:banana} holds for any integer $\ell>d(d+1)/4$ and that $\k$ can be chosen as follows:
        \begin{equation}\label{eq:explicit kappa and ell}
            \k =  \frac{\k'-\e}{2+ 2d+6\ell+d^2 },
        \end{equation}
        for any $\e>0$ (the implicit constant depends on $\e$ and $\ell$), 
        where
        \begin{equation}\label{eq:kappa-prime}
          \k'=\begin{cases}
            25/64 & \text{if }d=1,\\
            1/2 & \text{otherwise}.
          \end{cases}
        \end{equation}
        The value of $\k'$ comes from known bounds towards Selberg's eigenvalue conjecture~\cite{KimSarnak}; cf.~Proposition~\ref{prop:uniformspectralgap} below.
        It is possible to obtain much better values for $\k$ (possibly at the cost of worse values of $\ell$) via more analytic techniques similar to those in~\cites{Sarnak1981,Burger-effectiveHorocycles,FlaminioForni,Strombergsson,Strombergsson-JMD,Edwards}.

    \begin{remark}\label{remark:value of kappa and ell}
     The main point of Proposition~\ref{prop:banana} is the explicit dependence of the implied constant on $N$. This statement is well-known but we include a proof as we could not locate it in the literature.
        We note also that the implied constant depends on $\norm{\psi}_{\mrm{C}^\ell}$ and the radius of the smallest ball around the origin containing its support.
      
    \end{remark}

%--------------------------------------------------

\subsection{Uniform Spectral Gap}
We start with a standard result which is a crucial ingredient to the proof of Proposition~\ref{prop:banana}.
    
\begin{proposition}[Uniform spectral gap]\label{prop:uniformspectralgap}
  Let $\varepsilon>0$. For all $N\in\N$, for all $t\geq1$, and for all $\varphi,\psi\in \mrm{B}_{2,\ell}^\infty(X_\infty(N))$ which are orthogonal to the $\Gbf^+_\infty$-invariant functions, we have 
  \begin{equation*}
    \lvert\langle a(t)\varphi,\psi\rangle_{2}\rvert\leq C\Scal_{2,\ell}(\varphi)\Scal_{2,\ell}(\psi)t^{-\kappa^{\prime}+\e}.
  \end{equation*}
  where $\k'$ is as in~\eqref{eq:kappa-prime}, $\ell$ is any integer larger than half the dimension of the maximal compact subgroup $M_{\infty}\leq\Gbf_{\infty}$ and $C\geq 1$ depending only on $d,\e$ and $\ell$.
\end{proposition}

\begin{proof}[Sketch of the Proof] 
    This result is well-known and we only emphasize the fact that $C$ is independent of $N$.
    We recall that $\mrm{L}_{00}^{2}(X_\infty(N))$ denotes the orthogonal complement to the subspace of $\Gbf_{\infty}^{+}$-invariant vectors in $\mrm{L}^{2}(X_\infty(N))$ and that $\Gbf_\infty^+$ is the image of $G=\mrm{SL}_{d+1}(\R)$ in $\Gbf_\infty$. As $\Gbf_\infty^+$ has index $2$ inside $\Gbf_\infty$, $X_\infty(N)$ consists of at most two connected components, each of which is isomorphic to $G/\Delta$, where $\Delta$ is a congruence lattice in $\mrm{SL}_{d+1}(\Z)$. 
    Moreover, when $d\geq 2$, Vogan's classification of the unitary dual of $\mrm{GL}_{d+1}$ implies that matrix coefficients of a dense subset of vectors of any non-trivial, irreducible, unitary $G$-representation belong to $\mrm{L}^{2d+\e}$ for every $\e>0$; cf.~\cite[Corollary C]{OhSystems}.
    This in particular applies to the $G$-representations $\mrm{L}^{2}_{00}(X_\infty(N))$.
    It then follows by~\cite[Corollary on pg. 108]{CowlingHaagerupHowe}
    that for any two $M_{\infty}$-finite vectors $v,w\in\mrm{L}_{00}^{2}(X_{\infty}(N))$, we have
  \begin{equation}\label{eq:K-finite-coeffs}
    \lvert\langle a(t)v,w\rangle\rvert\leq (\dim\langle
    M_{\infty}v\rangle\dim\langle M_{\infty}w\rangle)^{\frac{1}{2}}\lVert
    v\rVert\lVert w\rVert \Xi\big(a(t)\big)^{\frac{1}{d}},
  \end{equation}
  where $\Xi$ denotes the Harish-Chandra spherical function on $G$.
  Using~\cite[Thm.~4.5.3]{Wallach} and denoting by $\mf{u}$ the Lie algebra of
  the subgroup $U=\{u(\xbf)\colon \xbf\in\bR^{d}\}$, we know that there is some
  $\delta>0$ depending on $d$ such that for $\lVert \cdot\rVert_\infty$ as
  in~\ref{def:adjoint norm}, we have
  \begin{equation*}
    \Xi\big(a(t)\big) \ll_{d} e^{-\frac{1}{2}\operatorname{tr}\ad \log
    a(t)|_{\mf{u}}}(1 + \log \lVert
    a(t) \rVert_{\infty} )^{\delta} = t^{-\frac{d}{2}}(1+\log t)^{\delta}.
  \end{equation*}
  In particular, we obtain that 
  \begin{equation*}
    \lvert\langle a(t)v,w\rangle\rvert\ll_{\varepsilon,d} (\dim\langle
    M_{\infty}v\rangle\dim\langle M_{\infty}w\rangle)^{\frac{1}{2}}\lVert
    v\rVert\lVert w\rVert t^{-\frac{1}{2}+\varepsilon}.
  \end{equation*}
    The statement for smooth vectors follows by the argument in~\cite[Section 6.2.1]{EMV} (with $C$ depending only on $\ell$).
    
    In the case $d=1$, it is shown in~\cite[Proposition 2]{KimSarnak} that the smallest non-zero eigenvalue of the Laplacian on $\mathbb{H}^2/\Delta$ is $\geq 975/4096$ for any congruence lattice $\Delta$ in $\mrm{SL}_2(\R)$.
    Using~\cite[Theorem 2]{Ratner-mixing} and the formula for the Haar measure, this implies that smooth matrix coefficients of $G$ belong to $\mrm{L}^{64/25+\e}$ for all $\e>0$.
    Hence,~\eqref{eq:K-finite-coeffs} follows in this case by~\cite[Theorem 2.1]{Shalom}.
    The statement for smooth vectors follows upon combining~\eqref{eq:K-finite-coeffs} with~\cite{EMV} as above.
    
\end{proof}

%----------------------------------------------------------
%----------------------------------------------------------

\subsection{Proof of Proposition~\ref{prop:banana}}

    Without loss of generality, we will assume that $\vp$ is identically $0$ on all connected components of $X_\infty(N)$, except the one containing $x$.
    By further replacing $\vp$ with $\vp-\big(\int \vp \;\der m_{\Gbf_\infty^+\cdot x}\big) \chi_{\Gbf_{\infty}^{+}\cdot x}$, we may assume it is orthogonal to $\Gbf^+_\infty$-invariant functions in $\mrm{L}^2(X_\infty(N))$.

We use the standard thickening technique to deduce Proposition~\ref{prop:banana} from Proposition~\ref{prop:uniformspectralgap}. More precisely, using Proposition~\ref{prop:uniformspectralgap}, one can deduce the following Proposition~\ref{prop:thickening} which was originally obtained in~\cite{KleinbockMargulis-EffectiveHorospheres}.

We abuse notation and denote by $\Scal_{2,\ell}$ the $(2,\ell)$-Sobolev norm on $\R^d$.

\begin{prop}[Theorem 2.3,~\cite{KleinbockMargulis-EffectiveHorospheres}]
  \label{prop:thickening}
  There exists a constant $r_0>0$, depending only on $\Gbf_{\infty}$, such that the following holds. Let $\psi\in \mrm{C}_c^\infty(\R^d)$, $0<r<r_0/2$, and $x\in X_\infty(N)$ for some $N\geq 1$.
  Suppose that $\psi$ is supported in the ball of radius $r$ around the origin in $\R^d$ and that the injectivity radius at $x$ is at least $2r$.
  Let $\ell$ and $\k'$ be as in Proposition~\ref{prop:uniformspectralgap}.
  Then, for every $\vp\in \mrm{B}_{2,\ell}^\infty(X_\infty(N))\cap \mrm{Lip}(X_\infty(N))$ such that $\vp$ is orthogonal to $\Gbf^+_\infty$-invariant functions and for all $t\geq 1$ and $\e>0$,
  \begin{align*}
    \bigg|\int \vp(a(t) u(\xbf)x) \psi(\xbf)& \;d\xbf\bigg| 
    \\
    &\ll_\e
    V_N
\big(\Scal_{2,\ell}(\vp) + \mrm{Lip}(\vp)\big)
    \left( r \int_{\R^d} |\psi| + r^{-(2\ell+k/2)} \Scal_{2,\ell}(\psi) t^{-\k'+\e} \right),
  \end{align*}
  where $V_N=\sqrt{[\Gamma(1):\Gamma(N)]}$ and $k=\dim\Gbf_{\infty}-d = d^2+d$.
\end{prop}

    \begin{proof}
    This result was obtained in~\cite[Theorem 2.3]{KleinbockMargulis-EffectiveHorospheres} in the case $N=1$. We give a sketch of the required modifications.
    For general $N$, one replaces Theorem 2.1 in \textit{loc.~cit.~}with Proposition~\ref{prop:uniformspectralgap} above.
    
    The factor $ V_N$ arises as follows. Let $\Ga^+(N) = \Gbf_\infty^+\cap \Ga(N)$ and $I_N^+=[\Ga^+(1):\Ga^+(N)]$.
    We assume that the Haar measure on $\Gbf_\infty^+$ is normalized so that it projects to a probability measure on $\Gbf_{\infty}^{+}/\Ga^+(1)$.
    Each connected component of $X_{\infty}(N)$ is isomorphic to $\Gbf_{\infty}^{+}/\Ga^+(N)$. Hence, in order to locally decompose the Haar measure on $\Gbf_{\infty}^{+}$ into the product of the Lebesgue measure on $\bR^{d}$ and the non-expanding subgroup in a compatible manner, cf.~\cite[Eq.~(2.3)]{KleinbockMargulis-EffectiveHorospheres}, the measure on the non-expanding subgroup needs to be scaled by $1/I_N^+$.
    This scaling implies that the Sobolev norm of
    the bump function on the non-expanding subgroup is scaled by the square root of $I_N^+$. Using that  $[\Gbf_{\infty}:\Gbf_{\infty}^{+}]=2$, one shows 
    \begin{equation*}
      I_N^+\ll[\Gamma(1):\Gamma(N)].
    \end{equation*}
    We leave the details to the reader.
    Finally, one uses that $X_{\infty}(N)$ has at most two connected components to bound the Sobolev norms of the restriction of $\vp$ to $\Gbf_{\infty}^{+}\cdot x$, as they occur in the proof of~\cite[Thm.~2.3]{KleinbockMargulis-EffectiveHorospheres}, by $\Scal_{2,\ell}(\vp)$.
  \end{proof}

    Let $\k>0$ be a parameter to be chosen later.
Let $\psi$ be as in the statement of Proposition~\ref{prop:banana}.
    Let $r_0>0$ be the number provided by Proposition~\ref{prop:thickening}.
    Let $R>0$ denote the infimum over all $\tilde{R}>0$ such that $\psi$ is supported in the ball $B(\mathbf{0},\tilde{R})$
    of radius $\tilde{R}$ around the origin.
    Define $r$ by
    \begin{equation*}
        r := \frac{1}{2}\min\seti{1,r_0, \inf_{\mathbf{y}\in B(\mathbf{0},4R)} \mrm{inj}_{\Ga(N)}(u(\mathbf{y})x) }.
    \end{equation*}
    Note that for all $\mathbf{y}\in B(\mathbf{0},4R)$, $\mrm{inj}_{\Ga(N)}(x) \ll_R \mrm{inj}_{\Ga(N)}(u(\mathbf{y})x)$, where the implicit constant is independent of $N$.
    It follows that
      \begin{equation}\label{eq:r bounds inj}
          r^{-1} \ll_R \max\seti{1, \mathrm{inj}_{\Ga(N)}(x)^{-1}}.
      \end{equation}

        Let $V_N=\sqrt{[\Gamma(1):\Gamma(N)]}$.
      If $2R\leq t^{-\k} r$, using $t\geq 1$ we apply Proposition~\ref{prop:thickening} with $t^{-\k}r$ in place of $r$ to get 
      \begin{align}\label{eq:small R}
        \int\vp(a(t)u(\mathbf{x})x)\psi(\mathbf{x})\der\mathbf{x}&\ll V_N(\Scal_{2,\ell}(\vp)+\lVert\vp\rVert_{\mathrm{Lip}})
        \left(t^{-\k}r\int\psi+ (t^{-\k}
        r)^{-(2\ell+\frac{k}{2})}\Scal_{2,\ell}(\psi)t^{-\k'+\e}\right)
        \nonumber\\
        &\leq V_N (\Scal_{2,\ell}(\vp)+\lVert\vp\rVert_{\mathrm{Lip}})
        \left(\max\seti{1,\Scal_{2,\ell}(\psi)} \right)
        \left(t^{-\k}r+ (t^{-\k}r)^{-(2\ell+\frac{k}{2})}t^{-\k'+\e}
        \right),
      \end{align}
      where we used our assumption that the integral of $\psi$ is $1$.
      By equating the terms involving $t$ in the error above, we choose $\k$ to satisfy
      \begin{equation}\label{eq:gamma for small R}
          \k \leq  \frac{\k'-\e}{1+2\ell+\frac{k}{2} }.
      \end{equation}
      Since $r<1$ and, in view of~\eqref{eq:r bounds inj} and letting $c=2\ell + k/2$, we obtain
      \begin{equation}\label{eq:bound for small R}
          \int\vp(a(t)u(\mathbf{x})x)\psi(\mathbf{x})\der\mathbf{x}
          \ll V_N(\Scal_{2,\ell}(\vp)+\lVert\vp\rVert_{\mathrm{Lip}})
        \left(\max\seti{1,\Scal_{2,\ell}(\psi)} \right) 
        \max\seti{1,\mathrm{inj}_{\Ga(N)}(x)^{-c}} t^{-\k}.
      \end{equation}

    If $2R> t^{-\k} r$, we choose a smooth partition of unity $\seti{\phi_i}$ of $B(\mathbf{0},2R)$ with the following properties:
    \begin{enumerate}
        \item Each $\phi_i$ is supported in a ball of radius $t^{-\k}r$ and satisfies $0\leq \phi_i \leq 1$.
        \item For every $i$, $\Scal_{2,\ell}(\phi_i)\ll t^{\k(\ell-d/2)}r^{-(\ell-d/2)}$ .
        \item The cardinality of the set of indices $j\neq i$ such that the supports of $\phi_i$ and $\phi_j$ intersect non-trivially is bounded above by a constant $C_d$, depending only on $d$.
        \item  $\#\seti{\phi_i} = O((R/t^{-\k} r)^{d})$.
    \end{enumerate}
    
    Denote by $\mathbf{y}_i\in B(\mathbf{0},4R)$ the center of a  $t^{-\k}r$-ball containing the support of $\phi_i$ and by $\psi_i$ the function $\phi_i\psi$.
    Let $x_i=u(\mathbf{y}_i)x$.
    It follows from the properties of the norm $\Scal_{2,\ell}(\cdot)$ that (cf.~\cite[Lemma 2.2]{KleinbockMargulis-EffectiveHorospheres})

    \begin{equation*}
        \Scal_{2,\ell}(\psi_i) \ll \Scal_{2,\ell}(\phi_i) \norm{\psi}_{\mrm{C}^\ell} 
        \ll t^{\k(\ell-d/2)}r^{-(\ell-d/2)} \norm{\psi}_{\mrm{C}^\ell}.
    \end{equation*}

    Fix $i$ and apply Proposition~\ref{prop:thickening} with $t^{-\k}r$, $\psi_i(\cdot+\mathbf{y}_i)$ and $x_i$ in place of $r$, $\psi$ and $x$ respectively to get
    \begin{align*}
          \int \vp(a(t) &u(\xbf)x) \psi_i(\xbf) \;d\xbf
          =\int \vp(a(t) u(\xbf-\mathbf{y_i})x_i) \psi_i(\xbf) \;d\xbf
          =\int \vp(a(t) u(\xbf)x_i) \psi_i(\xbf+\mathbf{y}_i) \;d\xbf 
          \\
        &=O\left( V_N
          \big(\Scal_{2,\ell}(\vp) + \mrm{Lip}(\vp)\big)
          \left( t^{-\k} r \int_{\R^d} \psi_i +  r^{-(2\ell-k/2)} \Scal_{2,\ell}(\psi_i)_\ell t^{\k(2\ell-k/2)-\k'+\e} \right)
        \right)
        \\
     &=O\left(V_N
       \big(\Scal_{2,\ell}(\vp) + \mrm{Lip}(\vp)\big)
       \left( t^{-\k} r \int_{\R^d} \psi_i +  r^{-(3\ell+k/2-d/2)} \lVert\psi\rVert_{\mrm{C}^{\ell}} t^{\k(3\ell+k/2-d/2)-\k'+\e} \right)
     \right),
        \end{align*} 
      where $k$ and $\k$ are as in the proposition.
      Here we used the non-negativity of $\psi_i$.
      Using the fact that ${\phi_i}$ is a partition of unity and that $\int_{\R^d}\psi=1$, we thus obtain
  \begin{align}\label{eq:eq1}
    \int \vp(a(t) &u(\xbf)x) \psi(\xbf) \;d\xbf
    = 
    \sum_i \int \vp(a(t) u(\xbf)x) \psi_i(\xbf) \;d\xbf
    \\
    &= \sum_i O\left(V_N
      \big(\Scal_{2,\ell}(\vp) + \mrm{Lip}(\vp)\big)
      \left(t^{-\k} r \int_{\R^d} \psi_{i} +   
        r^{-(3\ell+k/2-d/2)} \lVert \psi\rVert_{\mrm{C}^{\ell}} t^{\k(3\ell+k/2-d/2)-\k'+\e}
      \right) \right)\label{eq:eq2}\\
    &=O\left(V_N
      \big(\Scal_{2,\ell}(\vp) + \mrm{Lip}(\vp)\big)
      \left(t^{-\k} r +   
        R^{d}r^{-(3\ell+k/2+d/2)} \lVert\psi \rVert_{\mrm{C}^{\ell}} t^{\k(3\ell+k/2+d/2)-\k'+\e}
      \right) \right).\label{eq:eq3}
  \end{align}
Equating the powers of $t$ in the two terms above, we obtain
\begin{equation*}
  \k= \frac{\k'-\e}{1+ d/2+3\ell+k/2 },
\end{equation*}
which also satisfies~\eqref{eq:gamma for small R}.
Using~\eqref{eq:r bounds inj} it follows that 
\begin{align*}
  \int \vp(a(t) u(\xbf)x) \psi(\xbf) \;d\xbf
  =   O_R\left(V_N (\Scal_{2,\ell}(\vp)+\mrm{Lip}(\vp))
    \max\seti{1,\norm{\psi}_{\mrm{C}^\ell} }t^{-\k}
    \max\seti{\mrm{inj}_{\Ga(N)}(x)^{-c},1}
  \right),
\end{align*}
where $c=3\ell+k/2+d/2$. As $\min\{\mrm{inj}_{\Gamma(N)},1\}\leq 1$, letting $C=3\ell+k/2+d/2$, the claim follows.

%----------------------------------------------------------   
\section{Effective Equidistribution of Fractal Measures}
\label{sec:equidist}

The goal of this section is to prove that translates of certain self-similar measures become effectively equidistributed on the space of unimodular lattices, Theorem~\ref{thm: effective single equidist}.
This result is one of the main contributions of this article and constitutes the main ingredient in our proof of the analogue of Khintchine's theorem for fractal measures.

We fix a tuple $(\Fcal,\lambda)$ as in Section~\ref{sec: IFS} and we assume that $\Fcal$ is rational (cf.~\eqref{def:rational IFS}) and satisfies the open set condition. We denote by $\mu=\mu_{(\Fcal,\lambda)}$ the associated self-similar probability measure.
Throughout this section, we use $r$ to denote the average contraction ratio
\begin{equation}\label{eq: r}
  r := \sum_{i\in \L} \l_i \rho_i.
\end{equation}

Recall the definition of the spaces $\mrm{B}_{\infty,\ell}^\infty$ in~\eqref{eq:smooth vectors} and the definitions of $O_\a, \rho_\a$ and $b_\a$ in~\eqref{eq: composition parameters} for $\a\in\L^\ast$.
We will also use the notation introduced in Sections~\ref{sec: s-arith} and~\ref{sec: QIFS} pertaining to the algebraic group $\Gbf=\mrm{PGL}_{d+1}$.

    Let $\k>0,$ and $\ell \in \N$ be constants satisfying Proposition~\ref{prop:banana}.
    Recall that $\ell$ can be chosen to be any integer with $\ell>d(d+1)/4$ (cf.~discussion following Proposition~\ref{prop:banana}).
    In particular, we choose 
    \begin{equation}\label{eq:value of ell}
         \ell=\dim M_\infty=\frac{d(d+1)}{2}
    \end{equation}
    so that the orders of the Sobolev norms in Proposition~\ref{prop: spectral gap} and~\ref{prop:banana} match.
    We assume without loss of generality that
    \begin{equation}\label{eq:kappa is not big}
        \k \leq d(d+1).
    \end{equation}
  Let $p$ be the constant provided by Proposition~\ref{prop:summable}.
  For $\e>0$, denote by $\th_\e$ the H\"older conjugate of $(p+\e)/\epsilon(d)$; cf.~\eqref{eq:epsilon}, i.e.~$\th_\e$ is the constant satisfying
  \begin{equation*}
      \frac{1}{\th_\e} + \frac{\epsilon(d)}{p+\e} = 1.
  \end{equation*}
  Let $q_\e= 2\th_\e/(\th_\e+1)>1$. Note that $p/\epsilon(d)>1$ and in particular $1<q_\e<\infty$.
  Let $A$ and $L$ be the constants satisfying Lemma~\ref{lem:index estimate}. 

  Define positive constants $\sigma,\omicron_\e,\upsilon$ by the following equations:
  \begin{align}\label{def: a, b,c}
    r^{-\sigma} = \left( \sum_{i\in\L}\l_i^2 \rho_i^{-d}\right)^{1/2}, \quad
    r^{\omicron_\e} = \left(\sum_{i\in\L}\l_i^{q_\e}\right)^{1/q_\e}, \quad
    r^{-\upsilon} = \rho_{\min}^{-L/4} \sum_{i\in \L} \l_i 
    \rho_i^{-\ell}.
  \end{align}
  The positivity of $\upsilon$ follows from the open set condition; cf.~Remark~\ref{remark:positivity of a}.
  Given $\a\in\L^\ast$, we use the following notation:
    \begin{equation*}
        h_\a := k_\a^{-1} a(1/\rho_\a)u(b_\a), \qquad
        x_\a := h_\a \Gamma(1).
    \end{equation*}

\begin{theorem}\label{thm: effective single equidist}
  Let $\mu$ be as above and suppose that 
  \begin{equation}\label{eq:hypothesis on a,b,c}
       \frac{2\sigma(\omicron_\e+\upsilon)}{\omicron_\e+\sigma} <\k,
  \end{equation}
  for some $\e>0$. Then, there exist $\d>0$ and $A_\ast\geqslant 1$ such that for every word $\a\in  \L^\ast$,
  $\vp\in \mrm{B}_{\infty,\ell}^{\infty}(X_\infty(1) )$,
  and $t>1$,
  the following holds:
  \begin{equation*}
    \int \vp\left(a(t)u(\xbf)x_\a \right)\;d\mu(\xbf) = \int_{X_\infty(1)} \vp 
    + O\left( \rho_\a^{-A_\ast} \Scal_{\infty,\ell}(\vp) \cdot t^{-\d} \right), 
  \end{equation*}
  where $\ell$ is as in~\eqref{eq:value of ell}. The implicit constant depends on $\Fcal, \mu$ and $\e$.
\end{theorem}

\begin{remark}
    We refer the reader to Appendix~\ref{sec:Cantor spectral gap}, where we show that the above result holds under the hypothesis $\dim_H(\Kcal)>\cutoff$, in place of~\eqref{eq:hypothesis on a,b,c}
    in the special case of missing digit Cantor sets $\Kcal$.
    In particular, in that case, we remove the dependence of the cutoffs on the number of derivatives in the Sobolev norm in the parameter $\upsilon$ above. That part of the arguments in Appendix~\ref{sec:Cantor spectral gap} is in fact valid for any IFS with equal contraction ratios. 
\end{remark}

    \begin{remark}
        Recall that for all $\a,\w\in\L^\ast$, we have $\l_{\a\w}=\l_\a\l_\w$ and $\rho_{\a\w}=\rho_\a\rho_\w$. This implies that for any $k\in\bN$ the IFS $\Fcal_{k}=\{f_{\a}:\a\in\L^k\}$, given by considering the $k$-iterates of the IFS $\Fcal$ with the probability vector $\l^{\ast k}$ on $\L^{k}$, has average contraction ratio $r^k$; cf.~Lemma~\ref{lem: average of submul coc is subadditive}. The same lemma implies that the constants $\sigma,o_{\e},\upsilon$ in~\eqref{def: a, b,c}  for $(\Fcal_{k},\lambda^{\ast k})$ are the same as the constants determined by $(\Fcal,\lambda)$, i.e., Hypothesis~\eqref{eq:hypothesis on a,b,c} is invariant under iteration of the IFS.
    
    \end{remark}
 
    Theorem~\ref{thm: effective single equidist} yields the following: 
    \begin{corollary}\label{cor:SL equidist}
    Let $G=\mrm{SL}_{d+1}(\R)$, $\Ga=\mrm{SL}_{d+1}(\Z)$.
    For $\a\in\L^\ast$, let $k_{\alpha}\in\SL_{d+1}(\bR)$ be defined as in~\eqref{eq:k_omega} and let $t_{\a}=\frac{d}{d+1}\log\rho_{\a}$. Define 
      \begin{equation}\label{eq:unimodular-rep}
        x_{\a}^{\mrm{u}}=k_{\alpha}^{-1}g_{-t_{\a}}u(b_{\alpha})\Ga
        \in G/\Ga.    
      \end{equation}
    Suppose that~\eqref{eq:hypothesis on a,b,c} holds and let $A_\ast$ for some $\e>0$ and $\d$ be as in~Theorem~\ref{thm: effective single equidist}. Let $\k_\ast=\frac{d+1}{d}\d$. Then,
    
    \begin{equation}\label{eq:eff equidist}
      \int \vp\left(g_t u(\xbf)x_\a^{\mrm{u}} \right)\;d\mu(\xbf) = 
      \int \vp \;\der m_{G/\Ga}
      + O_{\e}\left(\rho_\a^{-A_\ast} \Scal_{\infty,\ell}(\vp) \cdot e^{-\k_\ast t} \right), 
    \end{equation}
    for all $\vp\in  \mrm{B}_{\infty,\ell}^{\infty}(G/\Ga)$, $t>0$, and $\a\in \L^\ast$.
  \end{corollary}
  \begin{proof}
    
    Recalling the notation in~\eqref{eq: simultaneous parametrization}, we have that $x_\a^{\mrm{u}}$ is the image of $ k_\a^{-1} a(1/\rho_\a)u(b_\a)\Ga(1)\in X_\infty(1)$ under the identification $X_\infty(1) \cong G/\Ga$ from Lemma~\ref{lem:identification-space-of-lattices}.
    Then, the corollary follows by Theorem~\ref{thm: effective single equidist}. The explicit value for $\k_\ast$ is obtained by noting that the canonical projection maps $g_{t}\in G$ to $a(e^{(d+1)t/d})$ in $\Gbf_{\infty}$.

  \end{proof}

    \subsection{Proof of Theorem~\ref{intro-thm:equidist} using Theorem~\ref{thm: effective single equidist}}
    Let $\k,\ell,L$ and $p$ be as above.
  Assume that~\eqref{eq:thickness hypothesis} holds with
    \begin{equation}\label{eq:epsilon0}
      \epsilon_0 = \min\seti{1, \frac{\k\epsilon(d)}{d\epsilon(d) + (4\ell+L) p }}.
    \end{equation}
    We show that this implies that~\eqref{eq:hypothesis on a,b,c} holds.
    The proof is purely formal, and relies on the fact that $(\rho_i^s)_{i\in\L}$ form a probability vector. This follows from the open set condition; cf.~\cite{Moran}.
    
    Let $\e>0$ be sufficiently small so that
      \begin{equation}\label{eq:delta0}
        \left(\frac{d\log\rho_{\min}}{\log \l_{\max}}-1\right) \frac{\log\l_{\min}}{s\log\rho_{\max}}
        \leqslant 
        \d_0:=\min\seti{1, \frac{\k\epsilon(d)}{d\epsilon(d) + (4\ell+L) (p+\e) }}.
      \end{equation}
      Such $\e$ exists since the left-hand side is assumed to be strictly less than $\epsilon_0$.

    First, we find bounds on $\sigma,\omicron_\e$ and $\upsilon$.
    Let $p_\e=(p+\e)/\epsilon(d)$.
    We note that $q_\e/(q_\e-1)=2p_\e$. Hence, since $\l_i$ form a probability vector, we get $\l_{\min}^{1/2p_\e}\leqslant r^{\omicron_\e}\leqslant \l_{\max}^{1/2p_\e}$.
    Similarly, we have $r^{-\upsilon}\leqslant \rho_{\min}^{-\ell-L/4}$ and $r^{-\sigma} \leq (\l_{\max}\rho_{\min}^{-d})^{1/2}$.
    It follows that
    \begin{equation*}
      \frac{ \log \l_{\max}}{2p_\e\log r}\leqslant
      \omicron_\e \leqslant  \frac{ \log \l_{\min}}{2p_\e\log r}, \qquad \upsilon\leqslant \frac{(\ell+L/4)\log\rho_{\min}}{\log r}, \qquad
      \sigma\leqslant \frac{\log\l_{\max} - d\log\rho_{\min}}{-2\log r}.
    \end{equation*}
    Put together, and noting that $\log r \leq \log \rho_{\max} <0$, we obtain
    \begin{align}\label{eq:estimate a,b,c}
      \frac{\sigma(\omicron_\e+\upsilon)}{\omicron_\e+\sigma} < \frac{\sigma(\omicron_\e+\upsilon)}{\omicron_\e} &\leqslant 
      \left(\frac{d\log\rho_{\min}}{\log \l_{\max}}-1\right) 
      \left( \frac{\log\l_{\min} +2p_\e(\ell+L/4) \log \rho_{\min}}{2\log r}\right)
      \nonumber\\
      &\leqslant 
      \left(\frac{d\log\rho_{\min}}{\log \l_{\max}}-1\right) 
      \left( \frac{\log\l_{\min} +2p_\e(\ell+L/4) \log \rho_{\min}}{2\log \rho_{\max}}\right).
    \end{align}
    Since $(\rho_i^s)_{i\in\L}$ is a probability vector, we have $\l_{\min}\leq \rho_{\max}^s$. Hence, $\log\l_{\min}/ s\log\rho_{\max} \geq 1$.
    By~\eqref{eq:delta0}, we obtain
    \begin{equation*}
        \frac{d\log \rho_{\min}} {\log \l_{\max}} -1\leqslant\d_0,
    \end{equation*}
    which implies that $\rho_{\min}^d\geq\l_{\max}^{1+\d_0} \geq \l_{\min}^{1+\d_0}$.
    This shows that
    \begin{equation*}
        \frac{\log\rho_{\min}}{\log\rho_{\max}}\leq\frac{1+\d_0}{d}\frac{\log\l_{\min}}{\log\rho_{\max}}.
    \end{equation*}
    Combined with~\eqref{eq:delta0} and~\eqref{eq:estimate a,b,c}, and using that $\d_0\leq 1$ and $s\leq d$, this shows that
    \begin{align*}
        \frac{\sigma(\omicron_\e+\upsilon)}{\omicron_\e+\sigma} &< \left(\frac{d\log\rho_{\min}}{\log \l_{\max}}-1\right) 
        \left( \frac{1}{2}+\frac{p_\e(\ell+L/4)(1+\d_0)}{d} \right) \frac{\log\l_{\min}}{\log\rho_{\max}}\\ &\leqslant s\d_0  \left( \frac{1}{2}+\frac{p_\e(\ell+L/4)(1+\d_0)}{d} \right)
        \\
        &\leqslant \d_0 \left( \frac{d}{2} + 2p_\e(\ell+L/4) \right) \leqslant \frac{\k}{2},
      \end{align*}
    where the last inequality follows by definition of $\d_0$.
    This proves that~\eqref{eq:thickness hypothesis}  $\Longrightarrow$~\eqref{eq:hypothesis on a,b,c}.

    %--------------------------------------------------

\subsection{Set up}     
    The remainder of the section is dedicated to the proof of Theorem~\ref{thm: effective single equidist}.
    
    Let $U$ be an open set satisfying the open set condition for $\Fcal$.
    We fix a smooth non-negative compactly supported function $\tilde{\psi} \leq 1$ on $\R^d$ such that $\tilde{\psi}$ is not identically zero and its support is contained in $U$.
    Let $\mrm{d}\xbf$ denote the Lebesgue measure on $\R^d$ and define a probability measure $\nu$ by
\begin{align}\label{eq: nu}
  \mrm{d}\nu(\xbf) =  \psi(\xbf)\mrm{d}\xbf,\qquad
  \text{where} \qquad
  \psi= \frac{\tilde{\psi}}{\int \tilde{\psi}\mrm{d}\xbf}.
\end{align}
For convenience, we also set
\begin{equation*}
    C_{\nu} = \frac{1}{\int \tilde{\psi}\;d\xbf}.
\end{equation*}

Let $P_\l$ be the operator introduced in Definition~\ref{def: markov operator}. Note that ${f_{\omega}}_{\ast}\nu$ is supported on $f_{\omega}(\mrm{supp}(\psi))$ for all $\omega\in\Lambda^{\ast}$ and therefore there exists $R$ depending on $\Fcal$ and $\psi$ such that $P_{\lambda}^{m}(\nu)$ is supported in the ball $B(\mathbf{0},R)\subset\bR^{d}$ of radius $R$ around the origin; cf.~\cite[Thm.~1]{Hutchinson}. We fix a smooth non-negative function $\psi_{2}$ such that $\mrm{supp}(\psi_{2})\subset B(\mathbf{0},2R)$ and $\psi_{2}|_{B(\mathbf{0},R)} = 1$.

    The implicit constants in our error terms will depend on $\nu$ through the $\mrm{C}^\ell$-norm of its density and the size of its support and the choice of $\psi_{2}$. These in turn can be made to depend only on the set $U$ (and hence only on $\Fcal$) by choosing $\psi$ suitably.
    
    By Theorem~\ref{thm: convergence in L-metric} and Lemma~\ref{lem: rate of convergence in fixed point theorem}, there exists a constant $C \geq 1$, depending only on $\mu, \nu$ and $r$ such that for all $m\in \N$ and all bounded Lipschitz functions $\vp$ on $\R^d$, we have
    \begin{equation}\label{eq: thicken}
         \lvert\mu(\vp)- P_\l^m(\nu)(\vp)\rvert\leqslant C r^{m}\mrm{Lip}(\vp).
    \end{equation}
    We also record the following identity which follows from the definition of $P_\l$:
    \begin{equation} \label{eq: natural MF}
        P_\l^{m+n}(\nu)(\vp) = P_\l^m(\nu)(P_\l^n(\vp)). 
    \end{equation}

%--------------------------------------------------
    \subsection{Invariance by the S-arithmetic random walk and Cauchy-Schwarz}
    Denote by $S:=S(\Fcal)$ the minimal set of primes, along with $\infty$, such that 
    \[ a(\rho_i), k_i, u(b_i) \in \GammaS=\Gbf\left(\Z[S^{-1}]\right),
    \quad i\in \L.
    \]
    We denote by $\Sf\subset S$ the set of finite places.
    We view $\GammaS$ as being diagonally embedded as a lattice in $\Gbf_S$.
    We denote by $x_0 $ and $x_S$ the identity cosets in $X_{\infty}(1)$ and $X_S$ respectively.

    Before we turn to the proof of Theorem~\ref{thm: effective single equidist}, we recall that by~\eqref{eq: class number one} the functions in $\mrm{B}_{\infty,\ell}^{\infty}(X_{\infty}(1))$ are in one-to-one correspondence with the $K_{\mrm{f}}$-invariant functions in $\mrm{B}_{\infty,\ell}^{\infty}(X_S)$, where $K_{\mrm{f}}$ denotes the maximal compact-open subgroup of $\Gbf_{\mrm{f}}$; cf.~Section~\ref{sec: s-arith}. From now on, we identify $\vp$ with the corresponding $K_{\mrm{f}}$-invariant function in $\mrm{B}_{\infty,\ell}^{\infty}(X_{S})$.

    For the proof of Theorem~\ref{thm: effective single equidist}, we fix a word $\a\in \Lambda^{\ast}$, 
    $t>1$ and $\vp\in \mrm{B}_{\infty,\ell}^\infty(X_\infty(1))$, where $\ell$ is as in Proposition~\ref{prop:banana}.
    We assume without loss of generality that $\vp$ is real-valued and
    \begin{equation}\label{eq:vp has mean 0}
        \int_{X_{\infty}(1)} \vp = 0.
    \end{equation}
    
    For simplicity, we also use $h_\a$ and $x_\a$ to denote $(h_\a, \mrm{Id})\in \Gbf_S$ and $(h_\a,\mrm{Id})\GammaS$ respectively, where $\mrm{Id}$ denotes the identity element in $\Gbf$.
    Recall the maps $\g_\w$ defined in~\eqref{def: gamma_w} for $\w\in \L^\ast$ and the operators $\Pcal_\l$.
    Note that $\gamma_{\w}$ and $a(t,1)$ commute. Therefore one obtains from~\eqref{eq:crucialpropertygamma} the following equality, which is key to our proof:
    \begin{equation*}
        \g_\w a(t,1)u(\xbf,\mathbf{0}) \g_\w^{-1} u(b_\w,\mathbf{0}) = a(t,1)u(f_\w(\xbf),\mathbf{0}).
    \end{equation*}
    Since $\g_\w^{-1}u(b_\w,\mathbf{0}) \in \GammaS$, this implies the following key identity:
    \begin{equation}\label{eq: key identity}
        \gamma_\w a(t,1)u(\xbf,\mathbf{0}) x_S 
        = a(t,1) u(f_\w(\xbf),\mathbf{0})x_S.
    \end{equation}

    Given a word $\w$, we denote by $\a\w$ the word obtained by concatenating $\w$ to the end of $\a$.
    We claim that
     \begin{equation}\label{eq: relating identity coset to basepoint}
        a(t,1)u\big(f_{\w}(\xbf),\mathbf{0}\big)x_{\a}=(k_\a^{-1},\mrm{Id})a(t/\rho_\a,1)u\big(f_{\a\w}(\xbf),\mathbf{0}\big)x_S.
    \end{equation}
    Indeed, since $f_\a(\ybf) = \rho_\a O_\a \ybf + b_\a$, we see that
    \begin{equation*}
        u(\ybf) h_\a = k_\a^{-1}a(1/\rho_\a) u(f_\a(\ybf)).
    \end{equation*}
    Applying this identity with $\ybf = f_\w(\xbf)$, we obtain \begin{equation}\label{eq:SimpleRelIdentityCosBasepoint}
      u(f_\w(\xbf),\mathbf{\mathbf{0}}) x_\a = (k_\a^{-1},\mrm{Id})a(1/\rho_\a,1) u(f_{\a\w}(\xbf),\mathbf{0}) x_S
    \end{equation}
    and hence~\eqref{eq: relating identity coset to basepoint} follows from the fact that $a(t,1)$ and $(k_\a,\mrm{Id})$ commute.

    Hence, we obtain the following generalization of~\eqref{eq: key identity}:
    \begin{align*}
        a(t,1)u\big(f_{\w}(\xbf),\mathbf{0}\big)x_{\a}&\overset{\eqref{eq: relating identity coset to basepoint}}{=}(k_\a^{-1},\mrm{Id})a(t/\rho_\a,1)u\big(f_{\a\w}(\xbf),\mathbf{0}\big)x_S\\
        &\overset{\eqref{eq: key identity}}{=}(k_\a^{-1},\mrm{Id})\g_{\a\w}a(t/\rho_\a,1)u\big(\xbf,\mathbf{0}\big)x_S\\
        &\overset{\eqref{eq: key identity}}{=}(k_\a^{-1},\mrm{Id})\g_{\a\w}\g_{\a}^{ -1}a(t/\rho_\a,1)u\big(f_\a(\xbf),\mathbf{0}\big)x_S\\
        &\overset{\eqref{eq:SimpleRelIdentityCosBasepoint}}{=}(k_\a^{-1},\mrm{Id})\g_{\a\w}\g_{\a}^{ -1}(k_\a,\mrm{Id})a(t,1)u(\xbf,\mathbf{0})x_\a.
    \end{align*}
    As $\g_{\a\w}=\g_\a\g_\w$, we therefore get
    \begin{equation}\label{eq: identity with basepoint}
        \left((k_\a^{-1},\mrm{Id})\g_\a  \cdot\g_\w \cdot\g_\a^{-1}(k_\a,\mrm{Id}) \right) \cdot a(t,1)u(\xbf,\mathbf{0})x_\a= a(t,1)u(f_\w(\xbf),\mathbf{0})x_\a.
    \end{equation}

    Recall the operators $\a\cdot\Pcal_\l$ in~\eqref{eq: padic MF alpha}. Define $\Phi:\bR^{d}\to\bC$ by $\Phi(\xbf)=\vp(a(t,1)u(\xbf,\mathbf{0})x_{\a})$. Using the above identification and~\eqref{eq: identity with basepoint}, for all $n\in\bN$ we have
      \begin{align*}
        P_{\lambda}^{n}(\Phi)(\xbf)&=\sum_{\omega\in\Lambda^{n}}\lambda_{\omega}(\Phi\circ f_{\omega})(\xbf)=\sum_{\omega\in\Lambda^{n}}\lambda_{\omega}\vp(a(t,1)u(f_{\omega}(\xbf),\mathbf{0})x_{\a})\\
        &=\sum_{\omega\in\Lambda^{n}}\lambda_{\omega}\vp\left((k_\a^{-1},\mrm{Id})\g_\a\g_\w\g_\a^{-1}(k_\a,\mrm{Id})a(t,1)u(\xbf,\mathbf{0}) x_{\a} \right)\\
        &=(\alpha\cdot\Pcal_{\lambda})^{n}(\vp)(a(t,1)u(\xbf,\mathbf{0}) x_{\a}).
      \end{align*}
    In particular, for every probability measure $\nu$ on $\R^d$, $m,n\in\N$, and $t>0$, one has
    \begin{equation}\label{eq: lift identity}
      \int \vp\left( a(t) u(\xbf)x_\a\right) \;\mrm{d} P_\l^{m+n}(\nu)(\xbf)=\int (\a\cdot\Pcal_\l)^n(\vp)\left( a(t,1) u(\xbf,\mathbf{0}) x_\a \right) \;\mrm{d} P_\l^{m}(\nu)(\xbf).
    \end{equation}
    We fix some $m,n\in\N$ to be chosen towards the end of the proof.

    Note that we may regard $\vp\circ a(t)$ as a function on $\R^d$ by restriction to the closed orbit $U(\R)x_\a$. Moreover, since conjugation by $a(t)$ expands $U(\R)$ by a factor of $t$, we have
    \begin{equation}\label{eq: scaling lipschitz}
      \mrm{Lip}\left(\vp\circ a(t) \vert_{U(\R)x_\a}\right) \leqslant t \cdot \mrm{Lip}\left(\vp \vert_{U(\R)a(t)x_\a } \right) \ll t\cdot \Scal_{\infty,\ell}(\vp).
    \end{equation}

    Combined with the estimate in~\eqref{eq: thicken},
    ~\eqref{eq: scaling lipschitz} implies
    \begin{align*}
              \int \vp\left( a(t) u(\xbf)x_\a\right)\;\mrm{d}\mu 
            =\int \vp\left( a(t)u(\xbf)x_\a \right) \;\mrm{d} P_\l^{m+n}(\nu)
            +O\left(  r^{m+n}t \cdot \Scal_{\infty,\ell}(\vp)\right).
        \end{align*}
        It then follows from Equation~\eqref{eq: lift identity} that
        \begin{align}\label{eq: from mu to nu}
           \int \vp\left( a(t) u(\xbf)x_\a\right)\;\mrm{d}\mu = \int \Theta(\xbf) \;\mrm{d} P_\l^{m}(\nu)(\xbf)
            + O\left(  r^{m+n}t \cdot \Scal_{\infty,\ell}(\vp)\right),
        \end{align}
        where for simplicity, we write
        \begin{equation*}
            \Theta(\xbf) := (\a\cdot\Pcal_\l)^n(\vp)\left(a(t,1)u(\xbf,\mathbf{0})x_\a\right).
        \end{equation*}

    Using the definition of $P_\l^m$ in~\eqref{eq: MF on functions} and the measure $\mrm{d}\nu(\xbf) = \psi(\xbf)\mrm{d}\xbf$, we obtain
    \begin{align*}
        \int \Theta(\xbf) \;\mrm{d} P_\l^{m}(\nu)(\xbf) = \int
        \sum_{\w\in \L^m} \l_\w  \Theta(f_\w(\xbf))\psi(\xbf)\;\mrm{d}\xbf.
    \end{align*}

    For each $m\in \N$ define
    \begin{equation*}
        B_m := \mrm{supp} \left(P_\l^m(\nu)\right).
    \end{equation*}
    Applying a change of variable $\xbf\mapsto f_{\w}^{-1}(\xbf)$ and noting that $\mrm{d}f_\w^{-1}(\xbf) = \rho_{\w}^{-d}\mrm{d}\xbf$ by invariance of the Lebesgue measure under rotations and translations, we obtain
     \begin{align*}
        \int\Theta(\xbf)\;\mrm{d}P_{\lambda}^{m}(\nu)(\xbf)&=\int\Theta(\xbf)\mathbf{1}_{B_{m}}(\xbf)\;\mrm{d}P_{\lambda}^{m}(\nu)(\xbf)\\
        &=\int\sum_{\w\in \L^m} \l_\w  \Theta(f_\w(\xbf))\mathbf{1}_{B_{m}}(f_{\omega}(\xbf))\psi(\xbf)\;\mrm{d}\xbf\\
        &=\int \Theta(\xbf)\mathbf{1}_{B_{m}}(\xbf)\sum_{\w\in \L^m} \l_\w \rho_\w^{-d} \psi(f_\w^{-1}(\xbf))\;\mrm{d}\xbf.
      \end{align*}
      Hence, Cauchy-Schwarz gives
      \begin{equation}\label{eq: Cauchy Schwarz}
        \left(\int\Theta(\xbf)\;\mrm{d}P_{\lambda}^{m}(\xbf)\right)^{2}\leqslant \underbrace{
        \left(\int_{B_{m}}\Theta(\xbf)^{2}\;\mrm{d}\xbf\right)}_{\text{Horospherical term}}
        \underbrace{
        \left(\int\Big(\sum_{\omega\in\Lambda^{m}}\lambda_{\omega}\rho_{\omega}^{-d}(\psi\circ f_{\omega}^{-1})(\ybf)\Big)^{2}\;\mrm{d}\ybf\right)}_{\text{Mass term}}.
      \end{equation}

%--------------------------------------------------
\subsection{The Open Set Condition and the Mass Term}

    The open set condition and the assumption that $\supp(\psi)\subseteq U$ imply that whenever $\a\neq \w\in \L^m$, $\psi\circ f_\w^{-1}$ and $\psi\circ f_\a^{-1}$ have disjoint supports.
    Thus, expanding the squared sum in the mass term yields
    \begin{align*}
        \left( \sum_{\w\in \L^m} \l_\w \rho_\w^{-d} \psi(f_{\w}^{-1}(\xbf)) \right)^2
        =  \sum_{\w\in \L^m} \l_\w^2 \rho_\w^{-2d} \psi(f_{\w}^{-1}(\xbf))^{2}.
    \end{align*}
    Applying the change of variable $\xbf\mapsto f_\w(\xbf)$, the mass term becomes
    \begin{align}\label{eq: sqs to ssq}
        \int \left( \sum_{\w\in \L^m} \l_\w \rho_\w^{-d} \psi(f_{\w}^{-1}(\xbf)) \right)^2\;\mrm{d}\xbf
        = \sum_{\w\in \L^m} \l_\w^2 \rho_\w^{-d} \int \psi(\xbf)^{2}\;\mrm{d}\xbf.
    \end{align}
    Recall that $\psi = C_\nu \tilde{\psi}$, where $\tilde{\psi}\leq 1$.
    Hence, we have $ \psi^2 \leq C_\nu^2 \tilde{\psi}$.
    It follows that 
    \begin{equation}
        \int \psi^2\;\mrm{d}\xbf \leq C_\nu.
    \end{equation}
    Applying Lemma~\ref{lem: average of submul coc is subadditive} with $\t_i = \l_i\rho_i^{-d}$ yields
    \begin{equation*}
         \sum_{\w\in \L^m} \l_\w^2 \rho_\w^{-d}
        = \bigg(  \sum_{i\in\L} \l_i^2\rho_i^{-d}   \bigg)^m.
    \end{equation*}
   Hence, we obtain the following bound on the mass term
   \begin{equation}\label{eq: bound on mass term}
       \int \left( \sum_{\w\in \L^m} \l_\w \rho_\w^{-d} \psi(f_{\w}^{-1}(\xbf)) \right)^2\;\mrm{d}\xbf \leqslant C_\nu 
       \bigg( \sum_{i\in\L} \l_i^2\rho_i^{-d}  \bigg)^m=C_\nu r^{-2\sigma m}.
   \end{equation}
   
   \begin{remark}\label{remark:positivity of a}
   It is worth noting that the sum on the right side of~\eqref{eq: bound on mass term} is always $\geq 1$. Indeed, if $s$ is the Hausdorff dimension of $\Kcal$, then the open set condition implies that $\sum_{i\in\L}\rho_i^s=1$; cf.~\cite{Moran}.
   Moreover, since $s\leq d$ and $\rho_i\leq 1$, we have $\rho_i^d \leq \rho_i^s$ for each $i\in \L$. Note further that, since $(\lambda_{i})_{i\in\Lambda}$ is a probability vector, the sum on the right hand side of~\eqref{eq: bound on mass term} can be interpreted as an expected value with respect to the probability measure determined by $(\lambda_{i})_{i\in\Lambda}$. Jensen's inequality thus implies
   
     \begin{equation*}
       \sum_{i\in\Lambda}\l_i^{2} \rho_i^{-d} \geqslant \frac{1}{\sum_{i\in\Lambda}\l_{i}\l_i^{-1}\rho_i^d }
       \geqslant \frac{1}{\sum_{i\in\L} \rho_i^s } = 1.
     \end{equation*}
   Moreover, the inequality is strict unless $s=d$ and $\l_i = \rho_i^d$ for each $i\in \L$.
   \end{remark}
   
%--------------------------------------------------

\subsection{Uniform Spectral Gaps and the Horospherical Term}

    Recall that $B_k$ denotes the support of the measure $P_\l^k(\nu)$ for $k\in\N$.
    To estimate the horospherical term in~\eqref{eq: Cauchy Schwarz}, recall that for all $k\in \N$,
    \begin{equation*}
        \mathbf{1}_{B_k}(\xbf)\leqslant \psi_{2}(\xbf).
    \end{equation*}
    Let $\mrm{d}\nu_2(\xbf) := \psi_{2}(\xbf)\mrm{d}\xbf$. By positivity of $\Theta^{2}$ it follows that
    \begin{equation}\label{eq:enlarge support}
        \int_{B_m} \Theta(\xbf)^{2} \dx{\xbf}
        \leqslant \int \Theta(\xbf)^{2} \dx{\nu_2(\xbf)}.
        \end{equation}

    By Lemma~\ref{lem:index estimate}, since $\vp$ is $K_{\mrm{f}}$-invariant, the function $(\a\cdot \Pcal_\l)^n(\vp)$ is invariant under a compact open subgroup $K_{\mrm{f}}[N]$ of $K_{\mrm{f}}$.

    We wish to apply Proposition~\ref{prop:banana}.
    By the double coset decomposition~\eqref{eq:congruence}, $(\a\cdot \Pcal_\l)^n(\vp)$ can be regarded as a function on a finite, disjoint union of copies of $X_\infty(N)$, where $\Ga(N)$ denotes the congruence lattice of level $N$.
    We note that if $g = a(t)u(\xbf)h_\a $, then
    \begin{equation*}
      K_{\mrm{f}}[N]g\GammaS=K_{\mrm{f}}[N](g,e_{\mrm{f}})\GammaS=g.K_{\mrm{f}}[N](e_{\infty},e_{\mrm{f}})\GammaS
    \end{equation*}
    is identified with the point $g\Gamma(N)\in X_{\infty}(N)$ in the copy of $X_{\infty}(N)$ containing the identity double coset in $K_{\mrm{f}}[N]\backslash\Gbf_{S}/\GammaS$.
    Moreover, since $g\in \Gbf_\infty^+$, $g\Ga(N)$ is contained in the connected component containing the identity coset inside this copy of $X_\infty(N)$.
    In particular, the integral on the right side of~\eqref{eq:enlarge support} can be seen as an integral over this
    component in this copy of $X_\infty(N)$.
    Let $\Ga^+(N) = \Gbf_\infty^+\cap \Ga(N)$ and $V_N=\sqrt{[\Ga(1):\Ga(N)]}$. Letting $D$ denote the total mass of $\nu_{2}$, Proposition~\ref{prop:banana} implies
    \begin{align}\label{eq:applying banana}
         \frac{1}{D}\int (\a\cdot \Pcal_\l)^n(\vp)^2\big( a(t) u(\xbf)&h_\a \Ga(N)\big) \dx{\nu_2(\xbf)}
         \nonumber=  
         \int (\a\cdot \Pcal_\l)^n(\vp)^2 \;\der m_{\Gbf_\infty^+/\Ga^+(N)} 
         \nonumber\\
         &+   
         O\left(V_N \Scal_{\infty,\ell}\big( (\a\cdot \Pcal_\l)^n(\vp)^2\big) t^{-\k}\max\big\{1, \mrm{inj}_{\Ga(N)}(h_\a\Ga(N))^{-C}\big\}\right),
    \end{align}
    where $m_{\Gbf_\infty^+/\Ga^+(N)}$ is the $\Gbf_\infty^+$-invariant Haar probability measure.
    Here, we used the fact that $\mrm{Lip}(\cdot)+ \Scal_{2,\ell}(\cdot) \ll \Scal_{\infty,\ell}(\cdot)$. 
    
      First, we bound the error terms in~\eqref{eq:applying banana}.
      Since $\Ga(N)$ is a subgroup of $\Gamma(1)$, the injectivity radius at $h_\a\Ga(N)$ in $X_{\infty}(N)$ is bounded from below by the injectivity radius at $x_{\a}=h_\a \Ga(1)$ in $X_{\infty}(1)$.
      It follows that $\mrm{inj}_{\Ga(N)}(h_\a\Ga(N))^{-C} \leq \mrm{inj}_{\Ga(1)}(x_{\a})^{-C} $. 
      Moreover, since $k_\a^{-1}u(b_\a)$ is uniformly bounded in $\Gbf_\infty$, we can find $C_0\geq 1$ such that $\mrm{inj}_{\Ga(1)}(x_{\a})^{-C}\ll \rho_\a^{-C_0}$. Noting that $\mrm{inj}^{-C}_{\Ga(1)} \gg 1$, we get
      \begin{equation}\label{eq:height-trivial-estimate}
        \max\big\{1, \mrm{inj}_{\Ga(N)}(h_\a\Ga(N))^{-C}\big\}
        \ll \mrm{inj}_{\Ga(1)}(x_{\a})^{-C} \ll \rho_\a^{-C_0}.
      \end{equation}

    In order to bound $\Scal_{\infty,\ell}( (\a\cdot \Pcal_\l)^n(\vp)^2)$, note that the Archimedean component of \[(k_\a^{-1},\mrm{Id})\g_\a\g_\w\g_\a^{-1}(k_\a,\mrm{Id})\] is equal to $ k_\w  a(\rho_\w)$; cf.~\eqref{def: gamma_w} for the definition of $\gamma_{\w}$.
    Moreover, by $M_\infty$-invariance of our chosen norm on the Lie algebra of $\Gbf_\infty$, we have $\lVert\mrm{Ad}( k_\w a(\rho_\w) )\rVert_{\infty}= \rho_\w^{-1}$.
    Hence, by Lemma~\ref{lem:Sobolev}, we have
    \begin{align}\label{eq:estimate Sobolev norm of operator}
        \Scal_{\infty,\ell}( (\a\cdot \Pcal_\l)^n(\vp)^2) &\ll \Scal_{\infty,\ell}( (\a\cdot \Pcal_\l)^n(\vp))^2
        \leqslant \left(\sum_{\w\in \L^n} \l_\w 
        \norm{\mrm{Ad}( k_\w  a(\rho_\w))}_{\infty}^\ell  \Scal_{\infty,\ell}(\vp)\right)^2 \nonumber\\
        &= \Scal_{\infty,\ell}(\vp)^2 \left(\sum_{\w\in \L^n} \l_\w 
        \rho_\w^{-\ell}  \right)^2.
    \end{align}
    By Lemma~\ref{lem:index estimate}, we have 
    \begin{equation}\label{eq:volume of cover}
        V_N \ll \rho_\a^{-A/2}\rho_{\min}^{-nL/2}.
    \end{equation}

    The next step is to bound the main term in~\eqref{eq:applying banana} by applying the estimate on the spectral gap of $\a\cdot\Pcal_\l$, Proposition~\ref{prop: spectral gap}.
    To this end, we wish to lift the integral in the main term from $\Gbf_\infty^+/\Ga^+(N)$ to $X_S$. We do so via a second application of the double coset decomposition~\eqref{eq:congruence}.
    Let
    \begin{equation*}
         C(S,d) = [\Gbf_S:\Gbf_S^+].
    \end{equation*} 
    Then, $C(S,d)$ is an upper bound on the number of $\Gbf_\infty^+$ orbits on the double quotient  $K_{\mrm{f}}[N] \backslash X_S$ by Proposition~\ref{prop:components-as-quotients}.
    The Haar measure on $K_{\mrm{f}}[N] \backslash X_S$ is the convex combination (with equal weights) of the $\Gbf_\infty^+$-invariant probability measures on each of these orbits.
    Hence, it follows by positivity that
    \begin{equation}\label{eq:apply bound on conn compo}
        \int
        (\a\cdot \Pcal_\l)^n(\vp)^2 \;\der m_{\Gbf_\infty^+/\Ga^+(N)}
        \leqslant C(S,d)
        \int_{K_{\mrm{f}}[N] \backslash X_S} (\a\cdot \Pcal_\l)^n(\vp)^2 = 
         C(S,d) \int_{X_S} (\a\cdot \Pcal_\l)^n(\vp)^2.
    \end{equation}
    In the second equality, we used the invariance of $(\a\cdot \Pcal_\l)^n(\vp)^2$ by $K_{\mrm{f}}[N]$.

    Recall the constant $p$ provided by Proposition~\ref{prop:summable} so that
    $\norm{\xi_\Gbf}_{\ell^{p+\e}(\GammaS)} <\infty$, for all $\e>0$.
    Fix some $\e>0$ so that our hypothesis~\eqref{eq:hypothesis on a,b,c} holds.
    Let $\th_\e$ denote the H\"older conjugate of $(p+\e)/\epsilon(d)$, where $\epsilon(d)$ is given by~\eqref{eq:epsilon}.
    
    Note that the $\Gbf_{\infty}$-invariant probability measure on $X_\infty(1)$ agrees with the push-forward of the $\Gbf_{S}$-invariant probability measure under the canonical projection with respect to $K_{\mrm{f}}$. It thus follows from Lemma~\ref{lem: character spectrum} that $\vp\in \mrm{L}_{00}^{2}(X_S)^{K_{\mrm{f}}}$.
    Hence, Proposition~\ref{prop: spectral gap} implies that
    \begin{equation}\label{eq: application spectral gap}
      \begin{aligned}
        \int_{X_S} (\a\cdot \Pcal_\l)^n(\vp)^2
        &\ll  \Scal_{\infty,\ell}(\vp)^2 \left( \sum_{i\in\L}\l_i^{q_\e} \right)^{2n/{q_\e}} \lVert\xi_\Gbf^{\epsilon(d)}\rVert_{\ell^{(p+\e)/\epsilon(d)}} \\
        &\ll_{\e,S,d} \Scal_{\infty,\ell}(\vp)^2 \left( \sum_{i\in\L}\l_i^{q_\e} \right)^{2n/{q_\e}},
      \end{aligned}
    \end{equation}
    where $q_\e=2\th_\e/(\th_\e+1) >1$. 
    
    Combining all the above estimates, we obtain the following bound on the horospherical term:
    \begin{align*}
        \text{Horospherical term }
        \ll_{\e,\nu,S,d} \Scal_{\infty,\ell}(\vp)^2 \rho_\a^{-A/2-C_0}\left(
         \left( \sum_{i\in\L}\l_i^{q_\e} \right)^{2n/q_\e} + t^{-\k} \rho_{\min}^{-nL/2}
         \left(\sum_{\w\in \L^n} \l_\w 
        \rho_\w^{-\ell}  \right)^2
        \right).
    \end{align*}
    
    By Lemma~\ref{lem: average of submul coc is subadditive}, applied with $\t_i=\rho_i^{-\ell}$, we have $\sum_{\w\in \L^n} \l_\w \rho_\w^{-\ell} = (\sum_{i\in \L} \l_i 
        \rho_i^{-\ell})^n$.
    Hence, recalling the constants $\sigma,\omicron_\e$, and $\upsilon$ defined in~\eqref{def: a, b,c}, we see that the last factor in the above estimate becomes $r^{2n\omicron_\epsilon}+t^{-\k}r^{-2n\upsilon}$.
    For convenience, let $\omicron=\omicron_\e$ and define $\t$ by the equation
    \begin{equation*} 
        t = r^{-\t}.
    \end{equation*}
    Let $A_\ast=(C_0+A/2)/2$.
    By combining the above estimate on the horospherical term with~\eqref{eq: Cauchy Schwarz} and~\eqref{eq: bound on mass term}, the estimate in~\eqref{eq: from mu to nu} becomes
    \begin{align}\label{eq:combined rates}
        \int \vp(a(t)u(\xbf)x_\a)\;d\mu \ll_{\nu,\e,S,d} 
        \Scal_{\infty,\ell}(\vp)
        \rho_\a^{-A_\ast}\left(
        r^{m+n-\t}  + 
        r^{-\sigma m}
        \sqrt{r^{2n\omicron} + r^{\k \t-2n\upsilon}}
        \right).
    \end{align}
    To balance these rates, we choose $n$ to be the largest natural number so that $2n\omicron\leq \k\t-2n\upsilon$.
    We then choose $m$ to be the largest integer satisfying $m+n-\t \leq -\sigma m +\omicron n$. 
    Hence, $n$ and $m$ are given by
    \begin{equation}\label{eq:choice of m,n}
        n = \left\lfloor \frac{\k \t }{2(\omicron+\upsilon)}\right\rfloor, \qquad m =  \left\lfloor \frac{1}{1+\sigma}\left(\tau+(\omicron-1)\left\lfloor\frac{\k\t}{2(\omicron+\upsilon)}\right\rfloor\right)\right\rfloor. 
      \end{equation}

    Note that $n\r\infty$ as $t\r\infty$.
    In Lemma~\ref{lem:m positive, m,n go to infinity} below, we show that $m\to \infty $ as $t\r\infty$ under our hypotheses. In particular, $m\geq 0$ for all large enough $t$. 
    
    With these choices, the right-hand side of~\eqref{eq:combined rates} tends to $0$ as $t\r\infty$ whenever
    \begin{equation}
      \d:= \frac{ \k  }{2(\omicron+\upsilon)}+   \frac{1}{1+\sigma}
      \left(1+ \frac{(\omicron-1)\k}{2(\omicron+\upsilon)} \right) -1 > 0.
    \end{equation}
    Rearranging, we obtain that $\d>0$ if and only if
      \begin{equation*}
        \frac{2 \sigma(\omicron+\upsilon)}{\omicron+\sigma} < \k,
      \end{equation*}
    which is satisfied by our hypothesis.
    Thus, the constant $\d$ satisfies the conclusion of the theorem.

    \begin{lem}\label{lem:m positive, m,n go to infinity}
      Our choice of $m$ in~\eqref{eq:choice of m,n} implies that $m\r\infty$ as $t\r\infty$.
    \end{lem}
    \begin{proof}
    It suffices to show that $1+ \frac{(\omicron-1)\k}{2(\omicron+\upsilon)}>0$.
    Suppose not and note that this implies that $2\omicron+\omicron\k\leq\k-2\upsilon$.
    In particular, this implies that $\k>2\upsilon$, since $(\k+2)\omicron>0$. On the other hand, by Jensen's inequality, we have
    \begin{equation*}
        r^{-\upsilon} = \rho_{\min}^{-L/4}\sum_{i\in\L}\l_i \rho_i^{-\ell} \geqslant 
        \rho_{\min}^{-L/4}\left( \sum_{i\in\L} \l_i\rho_i\right)^{-\ell} = \rho_{\min}^{-L/4} r^{-\ell}\geqslant r^{-\ell}. 
    \end{equation*}
    Since $\ell = d(d+1)/2$, cf.~\eqref{eq:value of ell},
    it follows that $\upsilon\geq \frac{d(d+1)}{2} $. We thus get that $\k>d(d+1)$, contrary to our assumption that $\k\leq d(d+1)$ in~\eqref{eq:kappa is not big}. Thus, we conclude that $1+ \frac{\k(\omicron-1)}{2(\omicron+\upsilon)}>0$ as desired.
    \end{proof}
%--------------------------------------------------

\section{Equidistribution of Random Walks}
\label{sec:BQ}

The goal of this section is to prove Theorem~\ref{cor:BQ}. 
The argument is similar to the proof of Theorem~\ref{thm: effective single equidist} and hence we omit some of the details.
The main difference is that we appeal to the equidistribution of translates of rational points on pieces of horocycles in place of the equidistribution of absolutely continuous measures.

The following is the precise substitute for Proposition~\ref{prop:banana} needed for the proof.
\begin{prop}\label{prop:rational-points}
There exist $\s>0, \ell\in\N$ and $c\geq 1$ such that the following holds. 
  For every $p,m,N\in\N$, and $\vp\in \mrm{B}_{\infty,\ell}^\infty(X_\infty(N))$, the following holds for all $x\in X_\infty(N)$:

  \begin{align*}
   p^{-m} \sum_{0\leq k<p^m} &\vp\left(a(p^m)u(kp^{-m}) x \right) \\
   &= \int \vp \;\der m_{\Gbf_\infty^+ \cdot x}
   + 
    O\Big( \sqrt{[\Ga(1):\Ga(N)]}
    \Scal_{\infty,\ell}(\vp) \max\{1,\mrm{inj}_{\Ga(N)}(x)^{-c}\}  \cdot p^{-\s m} \Big).
  \end{align*}
  
\end{prop}

\begin{proof}
This statement is obtained in~\cite[Prop.~5.3]{sltwo} for $x$ being the identity coset. We outline the needed adjustments to the proof and omit the details. The generalization to points with non-periodic horocycle orbits follows the same reasoning needed when considering pieces of periodic orbits instead of the full orbit in \textit{loc.~cit}.
The explicit dependence on $N$ follows by using Proposition~\ref{prop:banana} in place of~\cite[Eq.~(16)]{sltwo}.
\end{proof}

Let $\Fcal$ be a missing digit IFS, cf.~Definition~\ref{def: p cantor}, with digit set $\L\subseteq \Delta:= \seti{0,\dots,p-1}$ for some $p\in\N$. We will assume that $p$ is an odd prime to obtain the best convergence rate available through our methods. The same argument works for general $p\in\N$ at the cost of worsening the equidistribution rate via the use of Proposition~\ref{prop: spectral gap} instead of Proposition~\ref{prop:cantor spectral gap}.

Denote by $s$ the Hausdorff dimension of the attractor $\Kcal$ of $\Fcal$. 
Let $0<\varrho_0<1$ be a small parameter whose value is determined at the end of the proof and assume that $s>1-\varrho_0$.
 
Let $\l$ be the uniform probability vector on $\L$. Then, $\l_i=|\L|^{-1} = p^{-s}$ for all $i\in \L$~\cite{Moran}.
Let $S$ be the set of places for $\Q$ consisting of $\infty$ along with the prime divisors of $p$. 
Fix $\vp\in \mrm{B}^{\infty}_{\infty,\ell}(X_\infty(1))$ with integral $0$.

Fix some $\a\in \L^\ast$ and set $h_\a=a(\rho_\a^{-1})u(b_\a)$.
As before, we will abuse notation and denote by $x_\a$ both the coset of $h_\a$ in $X_\infty(N)$ and of $(h_\a,e_{\mrm{f}})$ in $X_S$.
Recall the measure $\nu$ defined in~\eqref{eq:BQ notation}.
We shall show that for some $\k_0>0$, we have
\begin{equation}\label{eq:BQ general basepoints}
    \int \vp \;\der(\nu^{\ast n}\ast \d_{x_\a}) = O_{\vp,\a,p}( p^{-\k_0 n}).
\end{equation}

The image of $\nu$ in $X_\infty(1)$, also denoted $\nu$, under the isomorphism $X_\infty(1)\cong \mrm{SL}_2(\R)/\mrm{SL}_2(\Z)$, cf.~Lemma~\ref{lem:identification-space-of-lattices}, satisfies
\begin{align*}
    \int \vp \;\der(\nu^{\ast n}\ast \d_{x_\a})
    = \sum_{\w\in \L^n} \l_\w \vp( a(t_n) u(b_\w)x_\a),
\end{align*}
where $t_n=p^{n}$.
Recall that $b_\w = f_\w(\mathbf{0})$. Then, denoting by $\d_{\mathbf{0}}$ the Dirac mass at $\mathbf{0}$, we observe that
\begin{equation*}
    \sum_{\w\in\L^n} \l_\w \d_{b_\w}  = P_\l^n(\d_{\mathbf{0}}).
\end{equation*}
Let $0<\e<1/2$ be a parameter to be chosen later and define
\begin{equation}\label{eq:choice-scale}
    k=\lfloor\e n\rfloor, \qquad m=n-k. 
\end{equation}
Applying the key identity~\eqref{eq: lift identity} with $\nu=\d_{\mathbf{0}}$, we obtain

\begin{align}\label{eq:applying key identity}
    \sum_{\w\in \L^n} \l_\w \vp( a(t_n) u(b_\w)x_\a) &= 
    \int \vp(a(t_n) u(\xbf) x_\a)\;\der P_\l^n(\d_{\mathbf{0}})(\xbf)
    \nonumber\\
    &= \int (\a\cdot\Pcal_\l)^k(\vp)(a(t_n) u(\xbf)x_\a)\;\der P_\l^{m}(\d_{\mathbf{0}})(\xbf).
\end{align}

Denote by $\chi_{\Kcal}$ the indicator function of $\Kcal$. For $\w\in \Delta^m$, we denote by $b_\w$ those rationals in $[0,1)$ with denominator $p^m$ and a numerator whose digit expansion mod $p$ is given by $\w$.
In particular, we have $\seti{b_\w:\w\in\Delta^m} = \seti{k/p^m: 0\leq k<p^m}$ and $\Kcal\cap\seti{b_\w:\w\in\Delta^m}=\seti{b_\w:\w\in\Lambda^m}$.

Hence, the last integral in~\eqref{eq:applying key identity} can be rewritten as follows: 

\begin{align*}
    \int (\a\cdot\Pcal_\l)^k(\vp)(a(t_n) u(\xbf)x_\a)\;\der P_\l^{m}(\d_{\mathbf{0}})
    =p^{-sm} \sum_{\w\in \Delta^{m}}  (\a\cdot\Pcal_\l)^k(\vp)(a(t_n) u(b_\w)x_\a) \chi_{\Kcal}(b_\w).
\end{align*}

By Cauchy-Schwarz, we get
\begin{align*}
    \bigg(\sum_{\w\in \Delta^{m}}  (\a\cdot\Pcal_\l)^k(\vp)(a(t_n) u(b_\w)x_\a) &\chi_{\Kcal}(b_\w)\bigg)^2 \\
    &\leqslant p^{sm}
    \sum_{\w\in \Delta^{m}}  (\a\cdot\Pcal_\l)^k(\vp)^2(a(t_n) u(b_\w)x_\a)  \\
    &=p^{sm}
    \sum_{\w\in \Delta^{m}}\big((\a\cdot\Pcal_\l)^k(\vp)^2\circ a(t_k)\big)(a(t_m) u(b_\w)x_\a).
\end{align*}

Arguing as in the proof of Theorem~\ref{thm: effective single equidist} using Lemma~\ref{lem:index estimate}, we can regard the last sum above as taking place in $X_\infty(N)$, for a suitable $N$.
Let $\Phi =(\a\cdot \Pcal_\l)^k(\vp)^2\circ a(t_k)$. Then the last sum is obtained by summing the values of $\Phi$ over the rational points with denominator $p^{m}$ on the horocycle orbit through the basepoint $x_\a$ after expansion by $a(p^m)$.

By effective equidistribution of those points on $X_\infty(N)$, cf.~Proposition~\ref{prop:rational-points}, and arguing as for~\eqref{eq:applying banana} and~\eqref{eq:apply bound on conn compo}, we obtain 
\begin{align*}
    p^{-m} &\sum_{\w\in \Delta^{m}} (\a\cdot \Pcal_\l)^k(\vp)^2  \circ a(t_k) (a(t_m) u(b_\w)x_\a) \\
    &\ll_p \int_{X_S}   (\a\cdot\Pcal_\l)^k(\vp)^2 
    + O\left( V_N
    \Scal_{\infty,\ell}\big((\a\cdot\Pcal_\l)^k(\vp)^2\circ a(t_k)\big)\; t_m^{-\s}\; \max\seti{1,\mrm{inj}_{\Ga(N)}(x_\a)^{-c}} \right),
\end{align*}
where $ V_N=\sqrt{\Ga(1):\Ga(N)}$, and we used the invariance of the Haar measure on $X_S$ under $a(t_k)$.
The dependence on $p$ in the implied constant is through the index $[\Gbf_S:\Gbf_S^+]$.

Using Lemma~\ref{lem:Sobolev}, we can estimate the error term as follows, cf.~\eqref{eq:estimate Sobolev norm of operator}:
\begin{align*}
    \Scal_{\infty,\ell}\big((\a\cdot\Pcal_\l)^k(\vp)^2\circ a(t_k)\big) \ll t_k^{\ell}\cdot \Scal_{\infty,\ell}((\a\cdot\Pcal_\l)^k(\vp)^2)
    \ll t_k^{3\ell}\cdot \Scal_{\infty,\ell}(\vp)^2.
\end{align*}

Let $\th=\min\{s,25/32\}$.
By taking $\varrho_0<1/2$, we have $s>1/2$.
Hence, applying the spectral gap estimate for the operators $\a\cdot\Pcal_\l$, Proposition~\ref{prop:cantor spectral gap}, we obtain 
\begin{align*}
    \int_{X_S}   (\a\cdot\Pcal_\l)^k(\vp)^2 \ll_\e p^{-(\th-\e)k}\Scal_{\infty,\ell}(\vp)^2 \leqslant p^{-(1/2-\e)k}\Scal_{\infty,\ell}(\vp)^2.  
\end{align*}

By Lemma~\ref{lem:index estimate} and arguing as in~\eqref{eq:volume of cover}, we have $V_N\ll p^{3|\a| +3k/2} $, where $|\a|$ is the integer satisfying $\a\in \L^{|\a|}$.
Combining all the estimates, along with~\eqref{eq:choice-scale}, and the facts $\mrm{inj}_{\Ga(1)}\leq\mrm{inj}_{\Ga(N)}$ and $\mrm{inj}^{-1}_{\Ga(1)} \gg 1$, we obtain
\begin{align*}
    \int \vp \;\der(\nu^{\ast n}\ast \d_{x_\a}) &\ll_{p,\e} p^{3|\a|/2}
    p^{\frac{(1-s)m}{2}}
    \sqrt{p^{-(1/2-\e)k} + p^{(3\ell+3/2) k - \s m}}
    \;\Scal_{\infty,\ell}(\vp)\; \mrm{inj}_{\Ga(1)}(x_\a)^{-c/2}\\
    &\ll_{p} 
     p^{3|\a|/2}
    p^{\frac{(1-s)(1-\e)n}{2}} 
    \sqrt{p^{-(1/2-\e)\e n+1} + p^{((3\ell+3/2+\s) \e-\s) n  } }
    \;\Scal_{\infty,\ell}(\vp) \;\mrm{inj}_{\Ga(1)}(x_\a)^{-c/2}.
  \end{align*}
  Setting $-(1/2-\e)\e  =(3\ell+3/2+\s) \e-\s$, yields a quadratic equation in $\e$ with one positive root given as follows: letting $\omicron=3\ell+2+\s$, then
  \begin{equation*}
      \e = \frac{\omicron-\sqrt{\omicron^2-4\s}}{2}.
  \end{equation*}
  Note that we may assume that $\s$ is small enough so that $\e<1/2$.
  Hence, the above estimate becomes
  \begin{align*}
       \int \vp \;\der(\nu^{\ast n}\ast \d_{x_\a}) 
       \ll_{p,\e}
     p^{3|\a|/2}
    p^{\frac{((1-s)(1-\e) - (1/2-\e)\e)n }{2}} \;\Scal_{\infty,\ell}(\vp) \;\mrm{inj}_{\Ga(1)}(x_\a)^{-c/2}.
  \end{align*}
Then, noting that $\a$ is fixed, the estimate above tends to $0$ when
\begin{align*} 
    1-s < \varrho_0:=  \frac{(1/2-\e)\e}{1-\e}.
\end{align*}
This concludes the proof.

\begin{remark}\label{rem:random walks}
The above proof extends readily to IFS in higher dimensions which generate products of copies of missing digit sets.
Moreover, with some additional effort, one can handle non-uniform probability vectors. The method is limited however to such special types of IFS, compared to the ones addressed by Theorem~\ref{thm: effective single equidist}.
First, by Theorem~\ref{thm: convergence in L-metric}, the measures $P_\l^n(\d_{\mathbf{0}})$ converge exponentially fast, with speed $p^{-n}$, towards the Hausdorff measure on $\Kcal$. However, the Lipschitz constant of the functions $\vp(a(t_n)u(\cdot)x_\a)$ is roughly $p^n$, which prevents us from deducing Theorem~\ref{cor:BQ} from Theorem~\ref{thm: effective single equidist}. This is also the reason we appeal to the equidistribution of rational points, Proposition~\ref{prop:rational-points}, instead. For a general IFS, there is no natural analog of Proposition~\ref{prop:rational-points} for a ``completed set'' of the translation vectors $\seti{b_\w:\w\in \L^n}$, i.e., an analog of the full set of rational points of denominator $p^n$.
\end{remark}

%--------------------------------------------------
\section{Reduction To Dynamics}\label{sec: reduction}
In this section, which largely follows~\cite{KleinbockMargulis-Loglaws}, we set up some notation and background which allows us to connect $\psi$-approximability to cusp excursions, i.e., to homogeneous dynamics. 
The connection between Khintchine's Theorem on $\R$ and the geodesic flow on the modular surface was first observed in~\cite{Sullivan}, where it was attributed to David Kazhdan.
Throughout the remainder of the article, we let
    \begin{equation*}
        G=\mrm{SL}_{d+1}(\R),\qquad \Gamma = \mrm{SL}_{d+1}(\Z).
    \end{equation*}

    \begin{lem}[Lemma 8.3,~\cite{KleinbockMargulis-Loglaws}]
    \label{lem: psi to r}
    Suppose $\psi:[1,\infty)\r\R_+$ is a continuous non-increasing function.
    Then, there exists a unique continuous function $r=r_\psi: [t_0,\infty) \r \R$, where $t_0 =  - \frac{d}{d+1}\log\psi(1)$,
    such that
    \begin{enumerate}[(i)]
        \item The function
        \begin{equation}\label{eq: lamda(t)}
            \l(t) = t-r(t)
        \end{equation}
        is strictly increasing and tends to $\infty$ as $t\r\infty$. Moreover, $\l(t_0) = 0$.
    
        \item The function
        \begin{equation}\label{eq: L(t)}
            L(t) = t+ dr(t)
        \end{equation}
        is non-decreasing.
        
        \item The functions $\l(t)$ and $L(t)$ are related by the following identity:
            \begin{equation}\label{eq: lamda to L}
                \psi^d\left(e^{\l(t)}\right) = e^{-L(t)}, \qquad t\geq t_0.
            \end{equation}

        \item The function $r$ is weakly monotone in the following sense: for every $t_2 \geq t_1 \geq t_0$, we have
        \begin{equation}\label{eq: weak monotone}
            r(t_2) - r(t_1) \geqslant \frac{-1}{d}(t_2-t_1).
        \end{equation}
        Moreover, the function $\l$ satisfies the following growth property for all $t_2\geq t_1 \geq t_0$:
        \begin{equation}\label{eq: lamda growth}
             \l(t_2) - \l(t_1) \leqslant \frac{d+1}{d}(t_2- t_1).
        \end{equation}
    
    \end{enumerate}
     
    \end{lem}
    
    \begin{proof}
    Items (i)-(iii) follow by~\cite[Lemma 8.3]{KleinbockMargulis-Loglaws} with $x_0=1$ in the notation in \textit{loc.~cit}. 
    The assertion $\l(t_0)=\log x_0=0$ follows from their proof.
     The last item now follows immediately from properties~\eqref{eq: lamda(t)}-\eqref{eq: lamda to L} via elementary manipulation. Let $t_{2}\geq t_{1}\geq t_{0}$. As $\psi$ is non-increasing and strictly positive,~\eqref{eq: L(t)} and~\eqref{eq: lamda to L} yield
    \begin{equation*}
        1\geq e^{-(L(t_{2})-L(t_{1}))}=e^{-(t_{2}+dr(t_{2})-t_{1}-dr(t_{1}))}
    \end{equation*}
    and thus~\eqref{eq: weak monotone} follows from monotonicity of the exponential function. For~\eqref{eq: lamda growth} one calculates
    \begin{align*}
        \l(t_{2})-\l(t_{1})&=(t_{2}-t_{1})-\big(r(t_{2})-r(t_{1})\big)\leq (t_{2}-t_{1})+\frac{1}{d}(t_{2}-t_{1})\\
        &=\frac{d+1}{d}(t_{2}-t_{1}).
    \end{align*} 
    \end{proof}
    
    We record a corollary of the above lemma which we use frequently throughout our arguments.
    \begin{corollary}\label{cor:growth of tn}
    Let $\psi, t_0$ and $\l$ be as in Lemma~\ref{lem: psi to r}. For each $n\in \N$, let $t_n$ be such that $e^{\l(t_n)} = 2^n$. Then, 
    \[ t_n \geqslant t_0
    + n \frac{  d \log 2}{d+1}. \]
    \end{corollary}
    
    \begin{proof}
      Lemma~\ref{lem: psi to r}(iv) implies that 
    $t_{n+1} - t_n \geqslant d\log 2/ (d+1)$. 
    The corollary follows by induction.
    \end{proof}

    Recall that the map $g\Ga \mapsto g\Z^{d+1}$ provides an identification of $X=G/\Ga$ with the space of unimodular lattices in $\R^{d+1}$.
    A subgroup $L$ of a lattice $\Delta$ in $\R^{d+1}$ is primitive if $L = \Delta\cap \R\cdot L $, where $\R\cdot L$ is the $\R$-span of $L$.
    We also recall the norms defined in Section~\ref{sec: IFS}.
    We define a function $d_1:X\r[1,\infty)$ as follows:
    \begin{align} \label{defn: phi}
    d_1(g\Ga) =  \max \seti{\norm{v}^{-1}: v\in g\Z^{d+1}-\seti{\mathbf{0}} }.
    \end{align}
    As of Mahler's compactness criterion the function $d_1$ is proper.
    For $\e>0$, define
    \begin{equation}\label{eq: cpt exhaustion}
        \mathscr{C}(\e) := \seti{x\in X: d_1(x) > 1/\e }.
    \end{equation}
    Then, the sets $X\setminus\mathscr{C}(\e)$ form a compact exhaustion of $X$.

    Denote by $m_{G/\Ga}$ the normalized $G$-invariant Haar probability measure on $X$.
    The next ingredient is an estimate on the measure of the sets $\mathscr{C}(\e)$ for the purpose of applying Borel-Cantelli arguments.
    \begin{proposition}[Proposition 7.1,~\cite{KleinbockMargulis-Loglaws}]
    \label{prop: cusp measure}
    There exist constants $C_d, C'_d \geqslant 1$, depending only on $d$ and the choice of norm on $\R^{d+1}$, so that
    \begin{equation*}
        C_d \e^{d+1} - C'_d \e^{2(d+1)} \leqslant m_{G/\Ga}\left(\mathscr{C}(\e)\right) \leqslant C_d \e^{d+1},
    \end{equation*}
    for all $0<\e<1$. In fact, we may take $C_d = \mathfrak{c}_{d+1}/2\zeta(d+1)$, where $\mathfrak{c}_{d+1}$ is the volume of the unit ball in $\R^{d+1}$ in our fixed norm.
    \end{proposition}
    
    The following proposition, due to Kleinbock and Margulis, allows us to approximate the characteristic functions of the sets $\mathscr{C}(\e)$ by smooth functions.
    In the following statement, we identify $\mrm{SO}_d(\R)$ with a subgroup of $\mrm{SL}_{d+1}(\R)$ via the map
    \begin{equation}\label{eq:embedding O_d in SL_d+1}
        O \mapsto \begin{pmatrix}
          O & \mathbf{0} \\ \mathbf{0} & 1
        \end{pmatrix}. 
    \end{equation}

    \begin{proposition}
    \label{prop: approximate cusp}
    For every $\eta>0$ and $\ell\in\N$, there exists a constant $S_\eta\geq 1$ such that for every $\e>0$, there are non-negative functions $\vp_\e,\vp^+_\e\in \mrm{B}_{\infty,\ell}^\infty(G/\Ga)$ such that
    \begin{enumerate}
        \item\label{item:K-invariance} If the norm defining $d_1$ (cf.~\eqref{defn: phi}) is invariant under $\mrm{SO}_{d}(\R)$, then so are $\vp_\e$ and $\vp_\e^+$. 
        \item\label{item:upper and lower estimate} $\vp_\e\leqslant \chi_{\mathscr{C}(\e)} \leqslant \vp^+_\e$.
        \item \label{item:norm of vp_epsilon}$\max\seti{\Scal_{\infty,\ell}(\vp_\e) ,\Scal_{\infty,\ell}(\vp^+_\e)} \leq S_\eta$.
        \item \label{item:bound measure of vp_epsilon}
        $m_{G/\Ga}(\vp^+_\e) \leqslant (1+\eta) C_d \e^{d+1}$ and 
         \begin{equation*}
        \frac{1}{(1+\eta)}\big(C_d \e^{d+1} - C'_d \e^{2(d+1)}\big) \leqslant m_{G/\Ga}(\vp_\e),
        \end{equation*}
    where $C_d$ and $C'_d$ are the constants in Proposition~\ref{prop: cusp measure}.
    \end{enumerate}

    \end{proposition}
    
    \begin{proof}
    The statement is standard, so we only sketch the proof.
    Let $U_\eta$ denote the set of $g\in G$ whose operator norm induced from the norm on $\R^{d+1}$ is at most $a:=(1+\eta)^{1/(d+1)}$.
    In particular, for $x\in \mathscr{C}(\d)$ and $g\in U_\eta$, we have $gx\in \mathscr{C}(a\d)$.
    The construction proceeds by choosing a $\mrm{C}^\infty$-bump function $\psi_\eta$ supported in $U_\eta$, which is right $\mrm{SO}_d(\R)$-invariant and has integral $1$.
    This is possible because the norm is $\mrm{SO}_d(\R)$-invariant.
    The desired functions $\vp_\e$ and $\vp_\e^+$ are then given by convolving respectively the indicator functions of $\mathscr{C}(\e/a)$ and $\mathscr{C}(a\e)$ with $\psi_\eta$.
    The invariance of $\vp_\e$ and $\vp^+_\e$ follows by right invariance of $\psi_\eta$. 
    Note that $m_{G/\Ga}(\vp_\e^{+})=m_{G/\Ga}(\mathscr{C}(a\e))$ and similarly for $\vp_\e$. Therefore, the measure estimates follow by Proposition~\ref{prop: cusp measure}.
    The Sobolev norms of the resulting functions can be bounded in terms of those of $\psi_\eta$ as follows. Let $\a$ be a multi-index and denote by $D^\a$ a differential operator of order $\ell=|\a|$ on $G$, defined using $\a$ in terms of a basis of the Lie algebra. Then, using standard properties of the convolution, we have $D^{\alpha}(\psi_{\eta}\ast\chi_{\mathscr{C}(\e)})=D^{\alpha}(\psi_{\eta})\ast\chi_{\mathscr{C}(\e)}$.
    Hence, for $x\in G/\Ga$,
    \begin{equation*}
    \left\lvert\int_{G}D^{\alpha}(\psi_{\eta})(g^{-1})\chi_{\e}(gx)\der g\right\rvert\leq\lVert\chi_{\mathscr{C}(\e)}\rVert_{\infty}\lVert D^{\alpha}(\psi_{\eta})\rVert_{1}=\lVert D^{\alpha}(\psi_{\eta})\rVert_{1}.
    \end{equation*}
    This shows that $\Scal_{\infty,\ell}(\vp_{\e})\ll \Scal_{1,\ell}(\psi_{\eta})$.
    \end{proof}

   %--------------------------------------------------

\section{The Convergence Theorem}\label{sec: convergence}

    The goal of this section is to obtain an analogue of the convergence part of Khintchine's theorem for measures whose translates become effectively equidistributed, Theorem~\ref{thm: convergence thm}.
    We note that we do not require that the measure in question is self-similar. We use the notation introduced in Section~\ref{sec: reduction}. We also use the subgroups $g_t$ and $u(\xbf)$ defined in~\eqref{eq: simultaneous parametrization}.

    \begin{theorem}[A Convergence Theorem]
    \label{thm: convergence thm}
    Let $\psi:\N \r \R_+$ be a non-increasing function.
     Suppose $\mu$ is a Borel probability measure on $\R^d$ such that $\mu$ satisfies the conclusion of Corollary~\ref{cor:SL equidist} for $\a=\emptyset$ (i.e. for $x_\a^u =\Ga\in G/\Ga$). 
     Then,
     \begin{equation*}
         \sum_{q\in \N} \psi^d(q) < \infty \Longrightarrow      \mu\left( W(\psi) \right) = 0.
     \end{equation*}
    
    \end{theorem}

%--------------------------------------------------

\begin{proof}
 For a lattice $\Delta$ in $\R^{d+1}$, we denote by $P(\Delta)$ the set of primitive vectors in $\Delta$.
In particular, $P(\Z^{d+1})$ consists of $v=(v_1,\dots,v_{d+1}) \in \Z^{d+1}$ such that the greatest common divisor of $v_1,\dots,v_{d+1}$ is 1.
Let $\norm{\cdot}$ denote the sup-norm on $\R^d$ and define the following sets:
\begin{equation}\label{def: A_n}
  A_n(\psi):= \left\lbrace 
    \xbf\in \R^d: \exists (\bfm{p},q)\in P(\Z^{d+1})\text{ s.t.~}
    0<q < 2^{n+1}\text{ and }
    \lVert q\xbf-\bfm{p}\rVert < \psi(2^n)
  \right\rbrace.
\end{equation}
We use $A_n$ to denote $A_n(\psi)$ for simplicity.
Then, we note that if $\xbf\in\R^d$ satisfies~\eqref{eq: psi-approx} for some $q\in [2^n, 2^{n+1})$, then monotonicity of $\psi$ implies that $\xbf\in A_n$.
Hence, we have that
\[W(\psi)\subseteq \limsup_{n\r\infty} A_n.\]
By the Borel-Cantelli Lemma, it suffices to show $\sum_{n\geq 1} \mu(A_n)<\infty$.

We shall view $\psi$ as a continuous function on $[1,\infty)$ by linearly interpolating its values at $\N$.
Let $r(t)$ and $\l(t)$ denote the functions provided by Lemma~\ref{lem: psi to r}.
For each sufficiently large $n\in \N$, we let $t_n$ be such that $e^{\l(t_n)} = 2^n$.
Note that~\eqref{eq: lamda to L} yields
\begin{equation}\label{eq: tn condensates}
  e^{-(d+1)r(t_n)} = 2^n\psi^d(2^n).
\end{equation}

For each $n\in \N$,
define $V_n \subset \R^{d+1}$ by\footnote{The extra factor of $2$ in the bound on $w_{d+1}$ ensures~\eqref{eq:A_n equal V_n}.}
\begin{equation*}
  V_n=\seti{w=(w_1,\ldots,w_{d+1})\in \R^{d+1}:  \norm{(w_1,\dots,w_d)} < e^{-r(t_n)},\quad\lvert w_{d+1}\rvert<2e^{-r(t_{n})}}. 
\end{equation*}
Denote by $\tilde{v}_n$ the indicator function of $V_n$.
Denote by $v_n$ the Siegel transform of $\tilde{v}_n$.
More precisely, $v_n$ is the function on $G/\Ga$ defined by
\begin{equation}\label{eq:Siegel transform}
  v_n(g\Z^{d+1}) = \sum_{w\in P(g\Z^{d+1})} \tilde{v}_n(w).
\end{equation}
It follows from the definitions that for sufficiently large $n\in\bN$, we have
\begin{equation}\label{eq:A_n equal V_n}
  A_n = \seti{\xbf: v_n(g_{t_n}u(\xbf)\Ga) \geq 1}.
\end{equation}

Next, we estimate the measure of $A_n$.
Denote by $2V_n$ the box obtained by scaling the side lengths of $V_n$ by $2$.
We let $v^+_n$ denote the Siegel transform of the indicator function of $2 V_n$.

As the natural representation of $G$ on $\bR^{d+1}$ is continuous and using precompactness of $V_{n}$, there is a neighbourhood $\Theta\subseteq G$ of the identity such that $\Theta V_{n}\subseteq 2V_{n}$.
We take $\th$ to be a non-negative smooth function supported in the interior of $\Theta$ and having integral $1$ with respect to the Haar measure on $G$, where the latter is normalized so that the induced measure on $G/\Ga$ is a probability measure.

Denote by $\tilde{\chi}_n$ and $\tilde{\chi}_n^+$ the indicator functions of the set of $x\in X$ such that $v_n(x) \geq 1$ and $v_n^+(x)\geq 1$ respectively.
We let $\tilde{\vp}^+_n = \th * \tilde{\chi}^+_n$ denote the convolution of $\th$ with $\tilde{\chi}^+_n$.
It follows that $\tilde{\chi}_n \leq \tilde{\vp}^+_n$.
Hence, using Corollary~\ref{cor:SL equidist}, we obtain
\begin{align*}
  \mu(A_n)&= \int \tilde{\chi}_n(g_{t_n}u(\xbf)\Ga)\;d\mu(\xbf)
  \leqslant \int \tilde{\vp}^+_n(g_{t_n}u(\xbf)\Ga)\;d\mu\\
  &\leqslant \int \tilde{\vp}^+_n \;\der m_{G/\Ga} + O(\Scal(\tilde{\vp}^+_n) e^{-\k_\ast t_n}).
\end{align*}

Note that $\Scal(\tilde{\vp}^+_n)\ll\Scal(\th)$ and the implied constant is independent of $n$; cf.~Lemma~\ref{lem:Sobolev}.
Moreover, using Fubini's theorem, the $G$-invariance of the Haar measure on $G/\Ga$, and the fact that $\th$ has integral $1$, we get
\begin{equation*}
    \int \tilde{\vp}_n^+\;\der m_{G/\Ga} =\int\tilde{\chi}_n\;\der m_{G/\Ga}.
\end{equation*}
Hence, since $\tilde{\chi}_{n}^{+}\leq\frac{1}{2}v_{n}^{+}$ by symmetry of $2V_{n}$, Siegel's summation formula~\cite[Equation 25]{Siegel-MVT} implies that
\begin{equation*}
  \int \tilde{\vp}^+_n \;\der m_{G/\Ga} \leqslant \frac{1}{2}\int v^+_n \;\der m_{G/\Ga}
  = \frac{1}{2\zeta(d+1)} \mrm{Vol} (2 V_n) = \frac{2^{d+2}\mathfrak{c}_d}{\zeta(d+1)} e^{-(d+1)r(t_n)},
\end{equation*}
where $\mathfrak{c}_d$ is the volume of the unit ball in $\R^d$.

    Let $n\in\bN$. 
    By Corollary~\ref{cor:growth of tn}, we get that $t_n \geqslant t_0+  \frac{n d \log 2}{d+1}$.
    Combined with~\eqref{eq: tn condensates}, we obtain
    \begin{equation*}
        \sum_{n\geq 1} \mu(A_n) \ll \sum_{n\geq 1} 2^n\psi^d(2^n) + e^{-\s n},
    \end{equation*}
    for some $\s>0$ and where the implied constant depends only on $d$, $\eta$, and $\psi$.
    Summability and monotonicity of $\psi$ then shows that the measures of the sets $A_n$ are summable concluding the proof.
  \end{proof}

%--------------------------------------------------

\section{Effective Double Equidistribution}
\label{sec:double}

    In our application to Diophantine approximation, we need to apply a certain converse of the classical Borel-Cantelli (Prop.~\ref{lem: BC divergence} below).
    As a replacement for the assumption on the independence of the events, we need a decay of correlation estimate, which we deduce from our equidistribution statement. 
    The idea behind this deduction follows similar lines to~\cite[Theorem 1.2]{KleinbockShiWeiss}, where a similar deduction is carried out for translates of absolutely continuous measures.
    
For the remainder of this section, we fix a tuple $(\Fcal,\lambda)$ with self-similar measure $\mu$ and we let $\norm{\cdot}$ denote the associated norms on $\R^d$ and $\R^{d+1}$; cf.~Section~\ref{sec: IFS}.
    We will use the notation from Corollary~\ref{cor:SL equidist} and we denote by $m_{G/\Ga}$ the $G$-invariant Haar probability measure on $G/\Ga$.

    The following is the main result of this section.
    \begin{prop}\label{prop: double equidist}
    Suppose that $\mu$ satisfies the conclusion of Corollary~\ref{cor:SL equidist}.
    Assume further that $\mu$ has null overlaps (cf.~\eqref{def:null overlaps}).
    Then, there exist constants $\d,\e_\ast>0$ and $C_\ast\geq 1$ such that the following holds.
    For all non-negative functions $\vp,\psi\in \mrm{B}_{\infty,\ell}^{\infty}(G/\Ga)$ which are invariant under $\seti{k_i: i\in \L}$, and for all $t\geq s >0$, satisfying
    \begin{equation*}
        t\geqslant C_\ast s \qquad \text{or} \qquad s\leqslant t\leqslant (1+\e_\ast)s,
    \end{equation*}
    we have
    \begin{equation*}
        \int \vp\left(g_tu(\xbf)\Ga \right)\psi\left(g_s u(\xbf)\Ga \right)\;d\mu \leqslant
       \int \vp \;\mrm{d}m_{G/\Ga} \int \psi\left(g_s u(\xbf)\Ga \right)\;d\mu
       +O\big( \Scal_{\infty,\ell}(\vp)\Scal_{\infty,\ell}(\psi) e^{- \d |t-s|}\big).
    \end{equation*}
    \end{prop}

    %--------------------------------------------------
    \subsection{Proof of Proposition~\ref{prop: double equidist} for long range correlations}
    
     This subsection is dedicated to the proof under the assumption $t\geq C_\ast s$ for a suitable $C_\ast\geq 1$. The other case is handled in the next subsection and its proof is much simpler. We remark however that both cases require the effective equidistribution hypothesis.

    To handle the case where the contraction ratios of the IFS are not all the same, we need the notion of complete prefix sets.
    We say $\a\in \L^k$ is a \textbf{prefix} of $\w =(\w_i)_i\in\L^\N$, if $\w = (\a,T^{k}\w)$, where $T:\L^{\bN}\to\L^{\bN}$ is the shift-map given by $(T\w)_{i}=\w_{i+1}$ for $i\in\bN$.
    We say $\a$ is a prefix of a finite word $\w$ if $\w$ is strictly longer than $\a$ and $\w$ is obtained from $\a$ by concatenating a finite word to the end of $\a$.
    Following~\cite{KleinbockLindenstraussWeiss}, we make the following definition.

	\begin{definition}\label{def: prefix}
	We say a finite set $P\subset \L^\ast$ is a \textbf{complete prefix set} if for every $\w\in \L^\N$, there is a unique word $\a\in P$ which occurs as a prefix for $\w$.
	\end{definition}

    Given $0<\e<1$, one can find a complete prefix set $P(\e)$ such that every word $\a\in P(\e)$ satisfies
    \begin{equation}\label{eq:tight ratios}
        \e\rho_{\min} \leq \rho_\a < \e.
    \end{equation}
    For example, $P(\e)$ can be chosen as follows:
    \begin{equation}\label{eq: complete prefix for epsilon}
         P(\e) = \seti{\a\in\L^{\ast}: \a \text{ satisfies~\eqref{eq:tight ratios} and no prefix of } \a \text{ satisfies~\eqref{eq:tight ratios}}}.
    \end{equation}
     One then checks that the sets $P(\e)$ chosen as above are complete prefix sets. We use those sets through the following lemma.
    
    \begin{lem}\label{lem: comp prefix}
    Assume $\mu$ has null overlaps and let $P$ be a complete prefix set. Then, for every continuous function $f$ on $\R^d$,
    \begin{equation*}
        \int f d\mu = \sum_{\a\in P} \int_{\Kcal_\a} f d\mu.
    \end{equation*}
    \end{lem}
    \begin{proof}
    Since $\mu$ has null overlaps, the collection $\seti{\Kcal_\a: \a\in P}$ forms a measurable partition of the support of $\mu$. The lemma follows readily.
    \end{proof}

    Finally, we need the following version of the mean value theorem.
    \begin{lem}\label{lem:first derivative}
    Let $\psi \in \mrm{C}^1(G/\Ga)$. Then, for all $x\in G/\Ga$ and $\vbf\in\R^d$,
    \begin{equation*}
        \left|\psi(u(\vbf)x) - \psi(x)  \right| \ll \Scal_{\infty,1}(\psi)\lVert\vbf\rVert.
    \end{equation*}
    \end{lem}
    \begin{proof}
    Let $X_\vbf \in \mrm{Lie}(G)$ be such that $u(\vbf)=\exp(X_\vbf)$. Then, $\lVert X_\vbf\rVert\ll \lVert\vbf\rVert$. Viewing $X_\vbf$ as a differential operator, we have $X_\vbf\psi(x) = \lim_{t\to 0} (\psi(u(t\vbf)x)-\psi(x))/t$. It follows that
    \begin{equation*}
       \left|\psi(u(\vbf)x) - \psi(x)  \right| = \left|\int_0^1 X_\vbf\psi(u(t\vbf)x)\;dt\right| \leq \norm{X_\vbf\psi}_\infty \ll \Scal_{\infty,1}(\psi)\lVert\vbf\rVert,
    \end{equation*}
    where the last inequality follows by Lemma~\ref{lem:Sobolev}. 
    \end{proof}

    Let $\k_\ast,\ell$ and $A_\ast$ be the constants provided by Corollary~\ref{cor:SL equidist}.
    Let $\k_\sharp = d\k_\ast/(d+1)$.
    Define 
    \begin{equation*}
         w = \frac{\k_\sharp t+ s}{1+ \k_\sharp+A_\ast}.
    \end{equation*}
     To simplify notation, for $g\in G$, we use $\vp(g)$ to denote $\vp(g\Ga)$. We further let
    \begin{equation*}
        \xi_t := g_{dt/(d+1)} = \begin{pmatrix} e^{t/(d+1)} \mrm{Id}_d & \bfm{0}\\ \bfm{0} &  e^{-dt/(d+1)}\end{pmatrix}, 
        \qquad \Scal:=\Scal_{\infty,\ell}.
    \end{equation*}
    Let $\rho_{\min} = \min\seti{\rho_i: i\in\L}$. Let $P:=P(e^{-w})$ be the  complete prefix set defined in~\eqref{eq: complete prefix for epsilon} with $\e=e^{-w}$.
    In what follows, we make repeated use of the fact that the norm $\Scal$ dominates the supremum norms and the Lipschitz constants of the functions $\vp$ and $\psi$.

    Recall that $b_\a = f_\a(\mathbf{0})$. 
    Note further that $\norm{f_\a(\xbf)-f_\a(\ybf)} =\rho_\a\norm{\xbf-\ybf} $ for all $\xbf,\ybf\in\R^d$ and $\a\in\L^\ast$.
    Let $K=\sup_{\xbf\in\Kcal} \norm{\xbf}$.
    By Lemma~\ref{lem:first derivative}, for each $\a \in P$ and for every $\xbf\in \Kcal_\a$,
    \begin{equation}\label{eq: psi almost constant}
      \begin{aligned}
        \left| \psi(\xi_su(\xbf))- \psi(\xi_su(b_\a)) \right|&\ll
        \Scal(\psi) \norm{e^s(\xbf-b_\a) }
        =
        \Scal(\psi) e^s\rho_\a \norm{f_\a^{-1}(\xbf)-\mathbf{0}}\\
        &\leqslant K\Scal(\psi) e^s \rho_\a \ll \Scal(\psi) e^{s-w},
      \end{aligned}
    \end{equation}
    where we used that $f_\a^{-1}(\xbf)\in\Kcal$ since $\xbf\in \Kcal_\a$.
    It follows that
    \begin{align} \label{eq: remove psi}
      \left|   \int_{\Kcal_\a} \vp\left(\xi_tu(\xbf) \right)\psi\left(\xi_s u(\xbf) \right)d\mu(\xbf) -
        \int_{\Kcal_\a} \vp\left(\xi_tu(\xbf) \right)\psi\left(\xi_s u(b_\a) \right)d\mu(\xbf)   \right| \ll \Scal(\vp) \Scal(\psi)e^{s-w}\mu(\Kcal_\a).
    \end{align}
    
    Next, we note that the definition of $f_\a$ and $O_\a$ in~\eqref{eq: composition parameters} implies that $f_\a^{-1}(\xbf) = \rho_\a^{-1}O_\a^{-1}(\xbf-b_\a)$.
    Moreover,
    for every $\mathbf{y}\in \R^d$, we have
    $ k_\a^{-1} u(\mathbf{y}) k_{\a} = u(O_\a^{-1} \mathbf{y})$.
    It follows that
    \begin{align*}
        k_\a^{-1} \xi_t u(\xbf) = \xi_{t+\log \rho_\a} \xi_{-\log\rho_\a} k_\a^{-1} u(\xbf-b_\a)u(b_\a)
        = \xi_{t+\log \rho_\a} u(f_\a^{-1}(\xbf))
        \underbrace{k_\a^{-1} \xi_{-\log\rho_\a} u(b_\a)}_{h_\a}.
    \end{align*}
    We let $h_\a = k_\a^{-1} \xi_{-\log\rho_\a} u(b_\a)$. Lemma~\ref{lem: transformation of self-similar measures} and the invariance of $\vp$ by $k_\a$ imply
    \begin{align}\label{eq: transform measure}
        \int_{\Kcal_\a} \vp(\xi_tu(\xbf))d\mu
        = \int_{\Kcal_\a} \vp(\xi_{t+\log \rho_\a} u(f_\a^{-1}(\xbf)) h_\a )d\mu
        = \mu(\Kcal_\a) \int \vp(\xi_{t+\log \rho_\a}u(\xbf)h_\a)d\mu.
    \end{align}
    Recall we are assuming $\mu$ satisfies the conclusion of Corollary~\ref{cor:SL equidist}.
    Hence, we obtain
    \begin{align}\label{eq: phi equidistributed}
       \int \vp(\xi_{t+\log \rho_\a}u(\xbf)h_\a)d\mu 
        &= \int \vp\;dm_{G/\Ga}+
        O\left(\rho_\a^{-A_\ast} \Scal(\vp) e^{-\k_\sharp(t+\log \rho_\a)} \right) \nonumber\\
        &= \int \vp\;dm_{G/\Ga}+  O\left(\Scal(\vp) e^{A_\ast w-\k_\sharp(t-w)} \right).
    \end{align}
   To combine the above estimates, we note that
    Lemma~\ref{lem: comp prefix} implies 
    \[
    \int \vp\left(\xi_tu(\xbf) \right)\psi\left(\xi_s u(\xbf) \right)\;d\mu
        = \sum_{\a\in P} \int_{\Kcal_\a} \vp\left(\xi_tu(\xbf) \right)\psi\left(\xi_s u(\xbf) \right)\;d\mu
    \]
    Hence, using that $\psi\geq 0$, we obtain
    \begin{align*}
         \int \vp&\left(\xi_tu(\xbf) \right)\psi\left(\xi_s u(\xbf) \right)\;d\mu
         =
         \sum_{\a\in P} \int_{\Kcal_\a} \vp\left(\xi_tu(\xbf) \right)\psi\left(\xi_s u(\xbf) \right)\;d\mu
         \\
         & \overset{\eqref{eq: remove psi}}{\leqslant}
         \sum_{\a\in P}\psi\left(\xi_s u(b_\a) \right)
         \int_{\Kcal_\a} \vp\left(\xi_tu(\xbf) \right)d\mu + O\left(\Scal(\vp) \Scal(\psi)\left(
        e^{s-w}\right)  \right) \\
        &\overset{\eqref{eq: phi equidistributed}}{\leqslant }
        \left( 
        \int \vp\;dm_{G/\Ga}+  O\left(\Scal(\vp) e^{A_\ast w-\k_\sharp(t-w)}   \right)        \right)
        \sum_{\a\in P}\psi\left(\xi_s u(b_\a) \right) \mu(\Kcal_\a)
         + O\left(\Scal(\vp) \Scal(\psi)\left(
        e^{s-w}\right)  \right)
        \\
        &\leqslant
        \int \vp\;dm_{G/\Ga} 
        \sum_{\a\in P}\psi\left(\xi_s u(b_\a) \right) \mu(\Kcal_\a)
       + O\left(\Scal(\vp) \Scal(\psi)\left(
        e^{s-w} +e^{A_\ast w-\k_\sharp(t-w)} \right)  \right).
  \end{align*}
    Define $\d$ and $C_\ast$ by
    \begin{equation}\label{eq: rate of double} \d=\frac{\k_\sharp}{2(1+\k_\sharp+A_\ast)}, \qquad C_\ast = \frac{2A_\ast+\k_\sharp}{\k_\sharp}.
    \end{equation}
    Suppose $t\geq C_\ast s$.
    Then, our choices of $\d$ and $w$ imply that
    \begin{equation*}
        s-w 
        \leq -\d(t-s).
    \end{equation*}
    It follows that
    \begin{equation*}
       \int \vp\left(\xi_tu(\xbf) \right)\psi\left(\xi_s u(\xbf) \right)\;d\mu
       \leqslant
        \int \vp\;dm_{G/\Ga} 
        \sum_{\a\in P}\psi\left(\xi_s u(b_\a) \right) \mu(\Kcal_\a)
       + O\left(\Scal(\vp) \Scal(\psi)
        e^{-\d (t-s)} \right) . 
    \end{equation*}
    To conclude the proof, we note that~\eqref{eq: psi almost constant} implies that
    \begin{equation*}
        \psi(\xi_s u(b_\a)) \leqslant  \frac{1}{\mu(\Kcal_\a)}\int_{\Kcal_\a} \psi(\xi_s u(\xbf))\;d\mu(\xbf) + O\left(\Scal(\psi) e^{s-w} \right),
    \end{equation*}
    for all $\a\in P$.
    Combined with the fact that $s-w\leq -\d(t-s)$ and $\vp\geq 0$, this implies that
    \begin{equation*}
        \int \vp\;dm_{G/\Ga} 
        \sum_{\a\in P}\psi\left(\xi_s u(b_\a) \right) \mu(\Kcal_\a)
        \leqslant
        \int \vp\;dm_{G/\Ga} 
       \int \psi\left(\xi_s u(\xbf) \right) \;d\mu 
       + O\left(\Scal(\vp) \Scal(\psi) e^{-\d (t-s)} \right).
    \end{equation*}
    This concludes the proof.
    
    \subsection{Decay of intermediate range correlations}

   We retain the notation of the previous subsection and
    let 
    \begin{equation*}
        \g = \frac{(d+1)\k'}{2d}, \qquad \ell_{1}=\ell(d+1)/d,
    \end{equation*}
    where $\k'$ is given in~\eqref{eq:kappa-prime}.

    Define a function $\Phi$ on $G/\Ga$ by
    \begin{equation*}
        \Phi(x) =\psi(x)\bigg( \vp(g_{t-s} x) - \int \vp\;dm_{G/\Ga}\bigg).
    \end{equation*}
    Then, we have
    \begin{equation*}
        \int \vp\left(g_tu(\xbf) \right)\psi\left(g_s u(\xbf) \right)\;d\mu(\xbf) = \int \vp\;dm_{G/\Ga}\int\psi(g_su(\xbf))\;d\mu(\xbf)
        + \int\Phi(g_s u(\xbf)\Ga)\;d\mu(\xbf).
    \end{equation*}
    Since $\mu$ satisfies Corollary~\ref{cor:SL equidist}, we get
    \begin{equation*}
         \int\Phi(g_s u(\xbf)\Ga)\;d\mu(\xbf) = 
         \int \Phi\;dm_{G/\Ga}+
         O\left( \Scal(\Phi) e^{-\k_\ast s}\right).
    \end{equation*}
    By properties of the Sobolev norm, Lemma~\ref{lem:Sobolev}, we get
    \begin{equation*}
        \Scal(\Phi) \ll \Scal(\psi) \Scal(\vp) e^{\ell_{1}(t-s)}.
      \end{equation*}
    Recall that $G/\Ga$ is isomorphic to $X_\infty(1)$ by Lemma~\ref{lem:identification-space-of-lattices}. 
    Hence, we may apply bounds on matrix coefficients provided by Proposition~\ref{prop:uniformspectralgap} to get
    \begin{align*}
        \int \Phi\;dm_{G/\Ga} &= \int \psi\cdot \bigg( \vp\circ g_{t-s}  - \int \vp\;dm_{G/\Ga}\bigg) \;dm_{G/\Ga}\\
        &= \int  \psi \cdot \left(\vp\circ g_{t-s} \right) \;dm_{G/\Ga} -
        \int \vp\;dm_{G/\Ga} \int  \psi\;dm_{G/\Ga}= 
         O\left( \Scal (\vp)\Scal(\psi) e^{-\g (t-s)} \right),
    \end{align*}
    where we applied the proposition with $\e=\k'/2$.
    Letting $\e_\ast$ be given by
    \begin{equation*}
      \e_\ast = \frac{\k_\sharp}{\ell_1 +\g},
    \end{equation*}
    we obtain a decay rate of $e^{-\g(t-s)}$ whenever $s\leq t\leq (1+\e_\ast)s$ as desired.

%--------------------------------------------------

\section{A Converse to Borel-Cantelli's Lemma}
\label{sec:BorelCantelli}

    In this section, we obtain a generalization of the Borel-Cantelli Lemma; Proposition~\ref{lem: BC divergence}.
    This result allows us to overcome the lack of strong independence estimates for all pairs of times $t$ and $s$ in Proposition~\ref{prop: double equidist}.

    \begin{prop}
    \label{lem: BC divergence}
     Suppose $E_n$ is a sequence of measurable sets in a probability space $(X,\mu)$.
     Assume that there are constants $D\geq 1$, $0<\s<1$ and $0<a\leq 1/\s$ such that
    \begin{enumerate}
        
        \item \label{item:divergent sum} $\mu(E_n)>0$ for all $n \gg 1$ and $\sum_{n\in \N} \mu(E_n)=\infty$.
        \item \label{item:long range independence} There exist constants $C_\ast,C_\# \geq 1$ and $\e_\ast>0$, such that for all $m,n\in \N$, with $ m \gg 1$ and satisfying 
        \begin{equation*}
            n\geqslant C_\ast m \qquad \text{or}\qquad  m\leqslant n\leqslant (1+\e_\ast)m,
        \end{equation*}
        we have 
        \begin{equation*}
            \mu(E_m\cap E_n)\leqslant C_\# \mu(E_m)\mu(E_n)+ D\left( e^{-\s m}\mu(E_n)+ e^{-\s (n-m)}\right).
        \end{equation*}
        \item \label{item:short range independence} For all $m,n\in\N$ with $1\ll m\leq  n$,
        \[\mu(E_m\cap E_n) \leq D \mu(E_m) \max\seti{\mu(E_n)^\s, 2^{-\s(n-m)}}.\]
        \item \label{item:weak monotone} For all $m ,n\in \N$ with $1\ll m\leq  n\leq m+\lceil-a\log \mu(E_m)\rceil$,
        \begin{equation*}
            \mu(E_n)\leq  D\mu(E_m)^\s.
        \end{equation*}

    \end{enumerate}
    Then, $ \mu(\limsup E_n) \geqslant 1/C_\#$. 
    \end{prop}

    Let $C_\ast\geq 1$ and $\e_\ast>0$ be the constants in the statement.
    The idea is to choose a subset $\Jcal\subset \N$ so that its elements are separated in such a way that we can apply our hypotheses on the decay of correlations while retaining the divergence of the sum of the measures. The main point in the construction below is that a $C_\ast$-adic interval (i.e.~one of the form $[C^k_\ast,C^{k+1}_\ast]$) consists of $O(\log C_\ast/\log (1+\e_\ast))$ many $(1+\e_\ast)$-adic sub-intervals. The pigeonhole principle then allows us to choose only one such $(1+\e_\ast)$-adic sub-interval from within each $C_\ast$-adic interval. This ensures that we only encounter long and intermediate range correlations so that we may apply Hypothesis~\ref{item:long range independence}. For the short range correlations (encountered only within our chosen $(1+\e_\ast)$-adic intervals), we will apply Hypothesis~\ref{item:short range independence}.
    We now carry out the details.
    
    \begin{step}[Choosing a sub-collection]
    By enlarging $C_\ast$, we may assume without loss of generality that 
    \begin{equation*}
          C_\ast = (1+\e_\ast)^{\ell_\ast},
    \end{equation*}
    for some $\ell_\ast \in \N$. For each integer $k\geq 0$, define $S_k$ by
    \begin{equation*}
        S_k = \max \seti{ \sum_{\substack{C_\ast^k(1+\e_\ast)^\ell \leqslant n <  \lfloor C_\ast^k(1+\e_\ast)^{\ell+1}\rfloor } } \mu(E_n) :  \ell\in\N, 0\leqslant \ell < \ell_\ast
         }.
    \end{equation*}
    We define $q_k=C_\ast^k(1+\e_\ast)^\ell$, for some integer $0\leq \ell <\ell_\ast$, to be the starting point of an interval that realizes the maximum $S_k$.
    More precisely, let $q_k $ be chosen so that
    \begin{equation*} 
        S_k = \sum_{\substack{q_k \leqslant n < \lfloor(1+\e_\ast)q_k\rfloor } } \mu(E_n), \qquad q_k = C_\ast^k(1+\e_\ast)^\ell, \qquad 0\leqslant \ell <\ell_\ast.
    \end{equation*}
    We define $\Jcal_k$ by
    \begin{equation*}
        \Jcal_k = \seti{n\in \N: q_k\leqslant n < \lfloor(1+\e_\ast)q_k\rfloor }.
    \end{equation*}
    We note that, since the measures of $E_n$ are not summable, we have
    \begin{align*}
    \infty=\sum_{n\in\N} \mu(E_n)&=  
    \sum_{k\geq 1} \sum_{0\leq \ell <\ell_\ast}
    \sum_{\substack{C_\ast^k(1+\e_\ast)^\ell \leqslant n <  \lfloor C_\ast^k(1+\e_\ast)^{\ell+1}\rfloor }} 
    \mu(E_n)
    \\
       &\leqslant \ell_\ast \sum_{k\geq 1} S_k = 
      \ell_\ast \sum_{n\in \Jcal_k, k\geq 1 } \mu(E_n).
    \end{align*}
    In particular, the measures of the sets indexed by $\cup_k \Jcal_k$ are not summable.
    Moreover, note that
   \begin{equation*}
       \sum_{n\in \Jcal_k, k\geq 1 } \mu(E_n) = \sum_{n\in \Jcal_{2k}, k\geq 1 } \mu(E_n) + \sum_{n\in \Jcal_{2k-1}, k\geq 1 } \mu(E_n).
   \end{equation*}
   Hence, at least one of the above two sums is infinite.
   We assume that
   \begin{equation}\label{eq:divergence along Jcal}
       \sum_{n\in \Jcal } \mu(E_n) =\infty, \qquad  \Jcal := \bigcup_{k\geq 1} \Jcal_{2k}.
   \end{equation}
   The proof in the case where the sum over the sets with odd index diverges is identical.
   Observe that if $n,q\in \Jcal_{2k}$ and $m\in \Jcal_{2j}$ for some $j< k$, so that $q<n$, then
   \begin{equation}\label{eq:separated in Jcal}
       n\geqslant C_\ast m, \qquad q< n<(1+\e_\ast)q.
   \end{equation}

  \end{step}
    
%----------------------------------------------

    \begin{step}[Reduction to independence estimates]
    
    We recall an inequality due to Chung-Erd\H{o}s~\cite{ChungErdos}: for all $M\leq N$ and positive measure sets $F_r$ in a probability space $(\Omega,\mu)$,
    \begin{equation}\label{eq: chung erdos}
        \mu\left (\bigcup_{r=M}^N F_r\right) \geqslant
        \left(\sum_{r=M}^N \mu(F_r) \right)^2/\sum_{M\leq r,s\leq N} \mu(F_r\cap F_s).
    \end{equation}
    
    In what follows, we use the notation $\sum^\star$ to indicate that the sum is restricted to members of the set $\Jcal$.
    We claim that for any fixed sufficiently large $M$
    \begin{equation}\label{eq: average decay}
        \sideset{}{^\star}\sum_{\substack{M\leq m,n\leq N}} \mu(E_m\cap E_n) \leqslant (C_\# +o(1)) 
        \left(\sideset{}{^\star}\sum_{M\leq n\leq N} \mu(E_n) \right)^2, \qquad \text{as } N\to\infty.
    \end{equation}
    To see that this claim implies the Proposition, note that this estimate combined with~\eqref{eq: chung erdos} implies that $\mu(\bigcup_{n\geq M, n\in\Jcal} E_n) \geqslant 1/C_\#$, for all large enough $M\in\N$.
    Since $\limsup_{n\in\Jcal} E_n$ is a decreasing intersection of sets of this form in a probability space, we obtain the desired result.
    
    Also note that the conclusion of Proposition~\ref{lem: BC divergence} follows trivially if $ \mu(E_n) =1$ for infinitely many $n$.
    Hence, we may assume for the remainder of the proof that $\mu(E_n)<1$ for all $n$ sufficiently large.
    
    Recall the constants $\s,a>0$ in the statement of the proposition. 
    Fix natural numbers $M < N$ with $M$ large enough so that all the hypotheses hold for $n\geq m \geq M$ and so that $\mu(E_m)<1$ for all $m\geq M$.
    Let $f(m) = \lceil -a \log \mu(E_m)\rceil \geq 1$.
    Recall that by assumption, we have $\s<1$ and $a\leq 1/\s$.
    Let $K = \lceil 1/a\s \rceil \in \N$.
    We will need the following elementary fact
    \begin{equation}
        \sum_{n\geq m} 2^{-\s(n-m)} = \frac{1}{1-2^{-\s}}, \qquad \forall m\in\N.
    \end{equation}
    \end{step}

    \begin{step}[Short range independence]

    Applying Hypothesis~\ref{item:weak monotone} iteratively and using that $\sigma<1$, we find that, for all $m\geq M$ and $k\geq 0$,
    \begin{equation*}
         \mu(E_{m+ k f(m) }) \leqslant D^{k} \mu(E_m)^{\s^k}.
    \end{equation*}
    Hence, for all $m\leq n < m+Kf(m)$, if $0\leq k \leq K-1$ satisfies $kf(m)\leq n-m< (k+1)f(m)$, then
    \begin{equation*}
        \mu(E_n) \leqslant D\mu(E_{m+kf(m)})^\s \leqslant D^{k\s +1} \mu(E_m)^{\s^{k+1}}\leqslant D^{(K-1)\s+1}\mu(E_m)^{\s^{K}}.
    \end{equation*}
    Combined with Hypothesis~\ref{item:short range independence} and using that $\s<1$, we obtain
    \begin{align*}
      \sideset{}{^\star}\sum_{\substack{M\leq m\leq N,\\ m\leq n < m+Kf(m)}}  
      & \mu(E_m \cap E_n)
      \leqslant D 
      \sideset{}{^\star}\sum_{\substack{ M\leq m \leq N}} \mu(E_m) 
      \sideset{}{^\star}\sum_{\substack{m\leq n<  m+Kf(m)}}  
      \big(\mu(E_n)^\s + 2^{-\s(n-m)}\big)
      \\
     &\leqslant D^{K+1} \sideset{}{^\star}\sum_{\substack{ M\leq m \leq N}} \mu(E_m)^{1+\s^{K+1}}Kf(m)
     + \frac{D}{1-2^{-\s}}
      \sideset{}{^\star}\sum_{\substack{ M\leq m \leq N}} \mu(E_m). 
    \end{align*}
    Note that $f(m) \leq 1- a\log \mu(E_m)$.
    Let $\e_0 = \s^{K+1}$ and $C_0 \geq 1$ so that $\log x \leq C_0 x^{\e_0}$ for all $x\geq 1$. Then,
     \begin{align}
      \sideset{}{^\star}\sum_{\substack{M\leq m\leq N,\\ m\leq n < m+Kf(m)}} &
       \mu(E_m \cap E_n)\nonumber\\
      &\leqslant
      D^{K+1}  Ka
      \sideset{}{^\star}\sum_{\substack{ M\leq m \leq N}} \mu(E_m)^{1+\e_0}  \left(\log \frac{1}{\mu(E_m)} +\frac{1}{a}\right)
      + \frac{D}{1-2^{-\s}} 
      \sideset{}{^\star}\sum_{\substack{ M\leq m \leq N}} \mu(E_m) 
      \nonumber\\
      \label{eq: short range}&\ll_{D,a,\s}
     \sideset{}{^\star} \sum_{\substack{ M\leq m \leq N}} \mu(E_m).
   \end{align}

    \end{step}

    \begin{step}[Long range independence]

    To estimate the sum over pairs of events which are separated by more than $Kf(m)$, we use Hypothesis~\ref{item:long range independence}. We first bound the contributions of the error terms. Note that
    \begin{align}\label{eq: error 2}
        \sideset{}{^\star}\sum_{m=M}^{N} \sideset{}{^\star} \sum_{n= m+K f(m)}^N e^{-\s(n-m)}
        \leqslant 
        \sideset{}{^\star}\sum_{m=M}^{N}  \sideset{}{^\star}\sum_{k=Kf(m)}^\infty
        e^{-\s k} \ll \sideset{}{^\star}\sum_{m=M}^{N} e^{\log\mu(E_m)} = 
        \sideset{}{^\star}\sum_{m=M}^{N} \mu(E_m),
    \end{align}
    where the implied constant depends only on $\s$.
    For the other error term in Hypothesis~\ref{item:long range independence}, we have
    \begin{align}\label{eq: error 1}
        \sideset{}{^\star}\sum_{m=M}^{N}  
        \sideset{}{^\star}\sum_{\substack{ n=m+Kf(m)}}^{N} e^{-\s m} \mu(E_n) \leqslant \frac{1}{1-e^{-\s}} \sideset{}{^\star}\sum_{\substack{M\leq m\leq N}} \mu(E_m).
    \end{align} 

    Recall that our choice of $\Jcal$ implies that the pairs $n,m\in \Jcal$ with $m\leq n \leq N$ satisfy the separation inequalities~\eqref{eq:separated in Jcal}. 
    In particular, we may apply Hypothesis~\ref{item:long range independence}, combined with~\eqref{eq: error 2} and~\eqref{eq: error 1}, to obtain
    \begin{equation}\label{eq: long range}
       \sideset{}{^\star}\sum_{m=M}^{N}  
       \sideset{}{^\star}\sum_{n=m+Kf(m)}^{N}
        \mu(E_m\cap E_n) \leqslant C_\# 
        \sideset{}{^\star}\sum_{\substack{M\leq m,n \leq N}}  \mu(E_m) \mu(E_n) + C_1 
        \sideset{}{^\star}\sum_{\substack{M\leq m\leq N}} \mu(E_m),
    \end{equation}
    for some constant $C_1 \geq 1$, depending only on $D$ and $\s$. 
    Finally, by~\eqref{eq:divergence along Jcal}, we have that the sum of the measures of $E_m$ diverges when restricting to $m\in \Jcal$. Therefore, for every fixed $M$,
    \begin{equation}\label{eq: x less x2}
        \sideset{}{^\star}\sum_{M\leq m\leq N} \mu(E_m) = o\left( \left(\sideset{}{^\star}\sum_{M\leq m\leq N} \mu(E_m)\right)^2\right), \qquad \text{as } N\r\infty.
    \end{equation}
    Hence,~\eqref{eq: average decay} follows from combining~\eqref{eq: short range} and~\eqref{eq: long range} with~\eqref{eq: x less x2}, thus concluding the proof.
    \end{step}

%--------------------------------------------------
\section{The Divergence Theorem}
\label{sec:divergence}

    The goal of this section is to obtain an analogue of the divergence part of Khintchine's Theorem for self-similar measures whose translates become effectively equidistributed, Theorem~\ref{thm: divergence thm}.
    Together with Theorem~\ref{thm: convergence thm}, this provides a complete analogue of Khintchine's theorem for this class of measures and completes the proof of Theorem~\ref{thm:main simplified}.

    Throughout this section, we fix $(\Fcal,\lambda)$ and the associated norms on $\R^d$ and $\R^{d+1}$, denoted $\norm{\cdot}$, as in Section~\ref{sec: IFS}. We let $\mu=\mu_{(\Fcal,\l)}$ denote the associated self-similar probability measure.

    \begin{theorem}[A Divergence Theorem]
    \label{thm: divergence thm}
    Let $\psi:\N \r \R_+$ be a non-increasing function and let $\mu$ be as above.
    Assume that $\Fcal$ is irreducible and satisfies the open set condition.
     Assume further that $\mu$ satisfies the conclusion of Corollary~\ref{cor:SL equidist} for functions $\vp$ which are invariant under $\seti{k_{i}:i\in\L}$.
     Then,
    \begin{equation*}
        \sum_{q\in \N} \psi^d(q) = \infty \Longrightarrow      \mu\left( W(\psi) \right) = 1.
    \end{equation*}
    \end{theorem}
    
    \begin{remark}
      It can be shown that if the dimension of the attractor of $\Fcal$ is $>d-1$, then $\Fcal$ is automatically irreducible. In particular, irreducibility holds for any IFS satisfying Hypothesis~\ref{eq:thickness hypothesis} in the introduction in view of our choice of $\epsilon_0\in (0,1)$ in~\eqref{eq:epsilon0}.
    \end{remark}

    Restricting to functions which are invariant under $\seti{k_{i}:i\in\L}$ in the statement of Theorem~\ref{thm: divergence thm} allows us to deduce it as a direct consequence of the following apriori weaker statement.

    \begin{prop}[Main Proposition]\label{prop: divergence thm}
     Let $\psi:\N \r \R_+$ be a non-increasing function.
        There exists a constant $\delta_{\psi,d}>0$ such that for any measure $\mu$ satisfying the hypotheses of Theorem~\ref{thm: divergence thm}, we have
        \begin{equation*}
        \sum_{q\in \N} \psi^d(q) = \infty \Longrightarrow      \mu\left( W(\psi)
        \right) \geq \delta_{\psi,d}.
    \end{equation*}
    \end{prop}

\subsection{Proof of Theorem~\ref{thm: divergence thm} assuming Proposition~\ref{prop: divergence thm}}

Let $\d_\psi>0$ be the constant provided by Proposition~\ref{prop: divergence thm} and set
\begin{equation*}
    \mu_\th := (f_\th)_\ast\mu.
\end{equation*}
By Lemma~\ref{lem:Lebesgue-density}, to show that $\mu(W(\psi))=1$, it suffices
to prove that $\mu_\th(W(\psi))\geq \d_\psi$.
This lower bound in turn follows by Proposition~\ref{prop: divergence thm} upon verifying that the measures $\mu_\th := (f_\th)_\ast\mu$ satisfy its hypotheses for all $\th\in  \L^\ast$.

Fix $\th\in\L^{\ast}$. 
Then, $\mu_\th  $ is fully supported on $\Kcal_\th$.
Moreover, $\mu_\th$ is self-similar with respect to the IFS $\Fcal_\th =\seti{f^\th_i:=f_\th f_i f_\th^{-1}:i\in\L}$ and the same probability vector $(\l_i)$. One also checks that $\mu_\th$ has null overlaps.
We claim that $\mu_\th$ satisfies the conclusion of Corollary~\ref{cor:SL equidist} for functions which are invariant under $\seti{k_i:i\in\L}$

Indeed, let $\a\in \L^n$ for some $n$ and let $\vp\in \mrm{B}^\infty_{\infty,\ell}(G/\Ga)$ be a function that is invariant under $\seti{k_i:i\in\L}$.
The similarity $f^\th_\a $ takes the form $\rho^\th_\a O^\th_\a + b^\th_\a$, where
\begin{equation*}
  \rho^\th_\a = \rho_\a, \quad 
  O^\th_\a = O_\th O_\a O^{-1}_\th, \quad
  b^\th_\a = f_\th f_\a f_\th^{-1}(\mathbf{0}).
\end{equation*}
We let $k_{\a}^{\th}=k_{\th}k_{\a}k_{\th}^{-1}$. If $\th=\emptyset$, we will usually omit the superscript.

We let $t_{\a}=\frac{d}{d+1}\log\rho_{\a}$ and we set $h_\a^\th =(k_{\a}^{\th})^{-1} g_{-t_{\a}} u(b_\a^{\th})$. By~\eqref{eq:unimodular-rep}, we need to verify the conclusion of Corollary~\ref{cor:SL equidist} holds for $\mu=\mu_\th$ and for basepoints of the form $h_{\a}^{\th}\Gamma$. We recall the following key identity, related to~\eqref{eq:crucialpropertygamma}:
\begin{equation}\label{eq:identity SLn}
  u(f_{\a}^{\th}\xbf)(h_{\a}^{\th})^{-1}=k_{\a}^{\th}g_{t_{\a}}u(\xbf).
\end{equation}
This implies that $h_{\a}^{\th}=h_{\th}^{-1}h_{\a}h_{\th}$. Indeed, the key identity implies that for all $\xbf\in\bR^{d}$
\begin{equation*}
  u(f_{\a}^{\th}\xbf)h_{\th}^{-1}h_{\a}^{-1}h_{\th}=k_{\a}^{\th}g_{t_{\a}}u(\xbf)=u(f_{\a}^{\th}\xbf)(h_{\a}^{\th})^{-1}.
\end{equation*}
In particular, the basepoints in Corollary~\ref{cor:SL equidist} that we need to examine for the IFS $\Fcal_\th$ take the form
\begin{equation*}
  x^{\th,\mrm{u}}_\a= (k^\th_\a)^{-1} g_{-t_{\a}} u(b^\th_\a)\Gamma=h_\th^{-1} h_\a h_\th\Gamma.
\end{equation*}
By another application of~\eqref{eq:identity SLn}, we have, for all $t\in\bR$,
\begin{equation}\label{eq:applied identity}
  g_t u(f_\th\xbf) h_\a^\th
  = k_\th g_{t+t_{\th}} u(\xbf)
  h_\a h_\th.
\end{equation}

   Denote by $\th\a$ the word obtained by concatenating $\a$ to the end of $\th$.
   It follows from the definitions that $h_\a h_\th = h_{\th\a}$.
   Hence, since $\mu$ satisfies Corollary~\ref{cor:SL equidist} by hypothesis and since $\vp$ is $k_\th$-invariant, we obtain
   \begin{align*}
       \int \vp \big( g_t u(\xbf) x_{\a}^{\th,\mrm{u}} \big) \;d\mu_\th(\xbf)
       &= \int \vp \big( g_t u(f_\th\xbf)x_{\a}^{\th,\mrm{u}}  \big) \;d\mu 
       = \int \vp(g_{t+t_\th} u(\xbf) h_{\th\a}\Ga) \;d\mu \\
       &= \int \vp \;dm_{G/\Ga} +  O(\rho_{\th\a}^{-A_\ast} \Scal_{\infty,\ell}(\vp) e^{-\k_\ast (t+t_\th)} ).
   \end{align*}
  Finally, we note that $\rho_{\th\a}=\rho_\th\rho_\a$.
   This shows that $\mu_\th$ satisfies the conclusion of Corollary~\ref{cor:SL equidist} for functions invariant by $\seti{k_i:i\in\L}$ and concludes the proof of Theorem~\ref{thm: divergence thm}.
   
   \begin{remark}\label{rem:ergodicity}
   In the case where the IFS is \textit{rational}, one of the referees suggested
   an alternative argument for upgrading from $\mu(W(\psi))\geq \delta_{\psi}$ to $\mu(W(\psi))=1$, which we now sketch.
   Let $\bar{\psi}:\N\to\R_+$ be a non-increasing function satisfying
   \begin{enumerate}
       \item $\sum_{q\geq 1} \bar{\psi}^d(q)\geq 1$.
       \item For all $T>0$, we have $\lim_{q \to \infty} \bar{\psi}(q)/\psi(Tq)=0$.
   \end{enumerate}
   In particular, our arguments show that $\mu(W(\bar{\psi}))\geq \delta_{\psi}$.
    Note further that rationality of the IFS implies that if $x$ is in $W(\bar{\psi})$, then $h(x)$ is in $W(\psi)$ for any $h$ in the set of maps generated by the IFS. Ergodicity of the associated operator $P_\l$ along with positivity of the measures of $W(\psi)$ and $W(\bar{\psi})$ imply that $W(\psi)$ has full measure.
    We thank the referee for this suggestion.
   \end{remark}

%--------------------------------------------------

    \subsection{Preliminary reductions}\label{sec: prelim reductions div}
    
    The remainder of the section is dedicated to the proof of Proposition~\ref{prop: divergence thm}.
    We retain the notation of Section~\ref{sec: convergence} pertaining to the homogeneous space $G/\Ga$.
    
    %--------------------------------------------------
    
    Recall that $P(\Z^{d+1})$ is the set of primitive vectors in $\Z^{d+1}$.
    Let a non-increasing approximation function $\psi_0$ be given so that $\sum_q \psi^d_0(q)=\infty$.
    Let $C_0\geq 1$ be a constant satisfying
    \begin{equation*}
        C_0^{-1} \norm{\cdot}\leq \norm{\cdot}_0 \leq C_0 \norm{\cdot},
    \end{equation*}
    where $\norm{\cdot}_0$ is the sup-norm on $\R^d$.

    \begin{lem}\label{lem: psi_0 small}
     In the proof of Proposition~\ref{prop: divergence thm}, we may assume that
    \begin{equation}\label{eq: psi_0 small}
        \psi_0^d(2^n) \leqslant 1/2^n \qquad(n\in\bN).
    \end{equation}
    \end{lem}
    \begin{proof}
        Suppose that $\psi_0^d(2^n)\geq 2^{-n}$ for some $n\in\N$.
        Then, Dirichlet's Theorem implies that for \emph{every} $\xbf\in\R^d$, there exists $q\in [1,2^n]$ and $\mathbf{p}\in\Z^d$, so that
        \begin{equation*}\label{eq:Dirichlet}
            \norm{q\xbf - \mathbf{p}}_0 \leq 2^{-n/d} \leq \psi_0(2^n).
        \end{equation*}

        In particular, if $\psi_0^d(2^n)\geq 2^{-n}$ for infinitely many $n\in \N$, then $W(\psi_0)=\R^d$, and the conclusion of Proposition~\ref{prop: divergence thm} follows. Hence, we may assume~\eqref{eq: psi_0 small} fails only for finitely many $n$. Since changing finitely many values of $\psi_0$ does not change the set $W(\psi_0)$, we may assume~\eqref{eq: psi_0 small} holds for all $n$.
    \end{proof}
    
    Let $\psi = \psi_0/C_0$ and note that $\sum_q \psi^d(q)=\infty$.
    For $n\in \N$, define
     \begin{equation}\label{def: A_n^ast}
        A_n^\ast(\psi):= \left\lbrace 
             \xbf\in \R^d: \exists (\bfm{p},q)\in P(\Z^{d+1})\text{ s.t.~}
             2^{n-1} \leq |q| < 2^n\text{ and }
             \lVert q\xbf-\bfm{p}\rVert < \psi(2^n)
        \right\rbrace.
    \end{equation}
    We then define $W^\ast(\psi)$ to be
    \begin{equation*}
        W^\ast(\psi) = \limsup_{n\r\infty} A_n^\ast(\psi).
    \end{equation*}
    By monotonicity of $\psi$, we have $W^\ast(\psi) \subseteq W(\psi_0)$,\footnote{Recall that $W(\psi_0)$ was defined in~\eqref{eq: psi-approx} using the sup-norm on $\R^d$.} and hence it suffices to show that $W^\ast(\psi)$ has full measure. 
    By Lemma~\ref{lem: psi_0 small} and using $C_{0}\geq 1$, we can assume
    without loss of generality that 
    \begin{equation}\label{eq: psi small}
        \psi^d(2^n) \leqslant 1/2^n \qquad (n\in\N).
    \end{equation}

    The remainder of this section is dedicated to verifying the hypotheses of Proposition~\ref{lem: BC divergence} 
    for the sets $A_n^\ast(\psi)$, which we denote $A_n^\ast$ for simplicity.

    It will be convenient for our arguments to also have a lower estimate on $\psi$; cf.~Lemma~\ref{lem: quasi-independence}. This is done in the following lemma.
    \begin{lem}\label{lem:psi large}
      In the proof of Proposition~\ref{prop: divergence thm}, we may assume that 
      \begin{equation}\label{eq: psi large}
         \psi^d(q) \geqslant \frac{1}{q \log^{1.1}q}.
    \end{equation}
    
    \end{lem}
    
    \begin{proof}
      Let $\psi_2^d(q) = 1/(q\log^{1.1} q)$, $\psi_3=\max \seti{\psi,\psi_2}$. Suppose that $\mu(W(\psi_3))=1$.
      By summability of $\psi_2^{d}$ and the Convergence Theorem  (Theorem~\ref{thm: convergence thm}), we have $\mu(W(\psi_2))=0$. Using $W(\psi_3)=W(\psi)\cup W(\psi_2)$, it follows that $\mu(W(\psi))=1$.
     
    \end{proof}
    
    Hence, throughout the remainder of the proof, we assume that~\eqref{eq: psi large} holds.

%--------------------------------------------------
%--------------------------------------------------
    Similarly to Section~\ref{sec: convergence}, we view $\psi$ as a continuous function on $[1,\infty)$ by linearly interpolating its values at $\N$.
    Denote by $r(t)$ the function obtained from $\psi$ by Lemma~\ref{lem: psi to r}. Let $\l(t)$ be the strictly increasing function provided by the same lemma.
    Define an increasing sequence of times $t_n$ by
    \begin{equation}\label{eq: t_n}
        e^{\l(t_n)} = 2^n.
    \end{equation}

    Let $\ell \in \N$ and $\k_\ast>0$ be the constants provided by Corollary~\ref{cor:SL equidist}.
    In order to simplify notation, we let $\Scal=\Scal_{\infty,\ell}$.
        
    Define the function $d_1$ in~\eqref{defn: phi} using our fixed norm on $\R^{d+1}$ and set
    \begin{equation}\label{eq:eta}
        \eta = \left(\frac{4}{3}\right)^{1/(d+2)}-1.
    \end{equation}
    Applying Proposition~\ref{prop: approximate cusp} with this $\eta$, we obtain, for each $\e>0$, functions $\vp_\e$ and $\vp_\e^+$ with uniformly bounded Sobolev norms.
    For each $n\in \N$, we let
    \begin{equation}\label{eq: cusp approximations}
        \vp_n := \vp_{e^{-r(t_n)}}, \qquad \vp^+_n := \vp^+_{e^{-r(t_n)}}, \qquad \chi_n := \chi_{\mathscr{C}(e^{-r(t_n)})},
    \end{equation}
    where $\chi_{\mathscr{C}(e^{-r(t_n)})}$ denotes the indicator function of $\mathscr{C}(e^{-r(t_n)})$.

      Viewing $\mrm{SO}_d(\R)$ as a subgroup of $G$ via the embedding in~\eqref{eq:embedding O_d in SL_d+1}, we see it leaves the norm on $\R^{d+1}$ invariant; cf.~Section~\ref{sec: IFS}. In particular, the functions $\vp_n$ and $\vp_n^+$ are invariant by $\mrm{SO}_d(\R)$ by Proposition~\ref{prop: approximate cusp}. Hence, we may apply our hypothesis that $\mu$ satisfies Corollary~\ref{cor:SL equidist} for functions invariant by $\seti{k_i:i\in\L}\subset \mrm{SO}_d(\R)$ to $\vp_n$ and $\vp_n^+$.

%----------------------------------------------------------
    \subsection{Divergence of the sum of measures}\label{sec:divergence-of-sum}
    Our first lemma allows us to verify the first hypothesis of Proposition~\ref{lem: BC divergence}.
    
    \begin{lem}\label{lem: measure estimate}
      There exists $C_\Fcal \geq 1$ such that for all $n\in \N$,
      \begin{equation*}
        C_de^{-(d+1)r(t_n)}/6 
        - C'_d e^{-2(d+1)r(t_n)} - C_\Fcal e^{-\k_\ast t_n} \leqslant \mu(A_n^\ast)
        \leqslant
        2C_d e^{-(d+1)r(t_n)} + C_\Fcal e^{-\k_\ast t_n}
        ,
      \end{equation*}
      where $C_d$ and $C'_d$ are the constants in Proposition~\ref{prop: cusp measure}.  
    \end{lem}

    \begin{proof}
      Fix $n\in \N$.
      We begin by proving the lower bound.
      For all $n\in\bN$, we define $U_{n}\subset \R^{d+1}$ by (cf.~Section~\ref{sec: IFS})
      \begin{align*}
        U_n &= \seti{w=(w_1,\ldots,w_{d+1})\in\R^{d+1}:\lVert (w_1,\dots,w_d)\rVert < e^{-r(t_n)}, |w_{d+1}|<e^{-r(t_n)}}.
      \end{align*}
      Similarly define 
      \begin{equation*}
        V_n=\seti{w=(w_1,\ldots,w_{d+1})\in U_n: \lvert w_{d+1}\rvert<2^{-1}e^{-r(t_{n})}}.
      \end{equation*}
      Denote by $\tilde{u}_n$ and $\tilde{v}_n$ the indicator functions of $U_n$ and $V_n$ and by $u_{n}$ and $v_{n}$ the Siegel transforms of $\tilde{u}_n$ and $\tilde{v}_n$ respectively; cf.~\eqref{eq:Siegel transform}.
      
      Consider the following sets:
      \begin{equation*}
        I_n :=\seti{\xbf\in\R^d:\chi_n(g_{t_n}u(\xbf)\Ga)=1}
        = \seti{\xbf\in\R^d:u_n(g_{t_n}u(\xbf)\Ga)\geq 1}.
      \end{equation*}
      The second equality follows from the fact that a lattice admits a non-trivial short vector if and only if it admits a short primitive vector.
      We also define sets $E_n$ by
      \begin{equation*}
        E_n:=\seti{\xbf\in\R^d:v_n(g_{t_n}u(\xbf)\Ga)\geq1}.
      \end{equation*}
      Note that $I_{n}\setminus E_{n}\subseteq A_{n}^{\ast}$
      and therefore
      \begin{equation}\label{eq: a_n is i_n minus e_n}
        \mu(A_n^\ast) \geqslant \mu(I_n) - \mu(E_n).
      \end{equation}
      
      We first bound the measure of $I_n$ from below. By definition we have $\vp_n \leq \chi_n$ and hence $\mu(I_n) \geqslant \int \vp_n(g_{t_n}u(\xbf) \Ga)\;d\mu(\xbf)$.
      Let $C_{\Fcal}\geq 1$ be chosen such that it bounds the implicit constant in Corollary~\ref{cor:SL equidist} from above and such that $\Scal(\vp_n) \leq C_{\Fcal}$ for all $n\in\bN$. Since $\mu$ satisfies Corollary~\ref{cor:SL equidist} by our hypothesis, 
      \begin{equation}\label{eq:equidist_I_n}
       \mu(I_n)\geqslant \int \vp_n(g_{t_n}u(\xbf) \Ga)\;d\mu(\xbf) \geqslant
       \int \vp_n \;\der m_{G/\Ga} - C_\Fcal e^{-\k_\ast t_n}.
      \end{equation}
     Since $\vp_n$ was chosen using Proposition~\ref{prop: approximate cusp}, we obtain
      \begin{equation}\label{eq: bound I_n}
        \mu(I_n) \geqslant \frac{C_d}{1+\eta} e^{-(d+1)r(t_n)} -
        \frac{C'_d}{1+\eta} e^{-2(d+1)r(t_n)} - C_\Fcal e^{-\k_\ast t_n},
      \end{equation}
      where $C_d=\mathfrak{c}_{d+1}/2\zeta(d+1)$ and $C'_d$ are the constants provided by Proposition~\ref{prop: approximate cusp}. Here, $\mathfrak{c}_{d+1}$ is the volume of the unit ball in $\R^{d+1}$ with respect to our norm.
      
      Next, we bound the measure of $E_n$ from above. The idea is similar to the proof of Theorem~\ref{thm: convergence thm}. Using continuity of the natural representation of $G$ on $\bR^{d}$ and the fact that $V_{n}$ is open with compact closure, we can choose a neighbourhood $\Theta_{\eta}\subseteq G$ of the identity such that $\Theta_{\eta}V_{n}\subseteq(1+\eta)V_{n}$. In what follows, $\theta_{\eta}$ is a non-negative smooth function on $G$ with support contained in $\Theta_{\eta}$ and of integral $1$ with respect to the Haar measure on $G$, which we normalize so that the induced measure on $G/\Ga$ is a probability measure.

    We let $v_{n}^{+}$ denote the Siegel transform of the indicator function on $(1+\eta)V_{n}$, and we let $\tilde{\chi}_{n}$ and $\tilde{\chi}_{n}^{+}$ denote the indicator functions on the set of $x\in X$ such that $v_{n}(x)\geq 1$ and $v_{n}^{+}(x)\geq 1$ respectively. Let $\tilde{\varphi}_{n}^{+}=\theta_{\eta}\ast\tilde{\chi}_{n}^{+}$. The argument above then implies that $\tilde{\chi}_{n}\leq\tilde{\vp}_{n}^{+}$. 

 Using Corollary~\ref{cor:SL equidist} it follows that
    \begin{align*}
      \mu(E_{n})&=\int\tilde{\chi}_{n}(g_{t_{n}}u(\xbf)\Ga)\der\mu(\xbf)\leqslant\int\tilde{\vp}_{n}^{+}(g_{t_{n}}u(\xbf)\Ga)\der\mu(\xbf)\\
      &\leqslant \int \tilde{\vp}_{n}^{+} \;\der m_{G/\Ga}
      +O(\Scal(\tilde{\vp}_{n}^{+})e^{-\k_\ast t_{n}}).
    \end{align*}

    By Lemma~\ref{lem:Sobolev}\eqref{item:Sobolev convolution}, we have that $\Scal(\tilde{\vp}^+_n)$ depends only on $\Scal(\th_\eta)$. In particular, by enlarging $C_\Fcal$ if necessary, we may assume that $\Scal(\tilde{\vp}^+_n)\leq C_\Fcal$ for all $n$. Note that by symmetry of norm balls we have $\tilde{\chi}_{n}^{+}\leq\frac{1}{2}v_{n}^{+}$. Using Fubini's and Siegel's theorems, cf.~\cite[Equation 25]{Siegel-MVT}, we find
    \begin{align*}
      \int\tilde{\vp}_{n}^{+} \;\der m_{G/\Ga}
      \leqslant \frac{1}{2}\int v_{n}^{+} \;\der m_{G/\Ga}
      =\frac{1}{2\zeta(d+1)}\mrm{Vol}((1+\eta)V_{n}) = \frac{(1+\eta)^{d+1}\mathfrak{c}_{d+1}}{4\zeta(d+1)}e^{-(d+1)r(t_n)},
    \end{align*}
    where $\mathfrak{c}_{d+1}$ is the volume of the unit ball in $\R^{d+1}$.
    Hence,
    \begin{equation}\label{eq: bound E_n}
      \mu(E_n) \leqslant \frac{C_d(1+\eta)^{d+1}}{2}e^{-(d+1)r(t_n)} + C_\Fcal e^{-\k_\ast t_n},
    \end{equation}
    where $C_d $ is the same constant as in~\eqref{eq: bound I_n}.
  Using the choice of $\eta$, the lower bound in the lemma now follows by combining~\eqref{eq: a_n is i_n minus e_n},~\eqref{eq: bound I_n}, and~\eqref{eq: bound E_n}.
     
  The upper bound follows upon observing that $\mu(A_n^\ast) \leq \int \vp_n^+(g_{t_n} u(\xbf)\Ga)d\mu$, where $\vp_n^+$ was chosen in~\eqref{eq: cusp approximations}. Then one applies Corollary~\ref{cor:SL equidist}, Proposition~\ref{prop: cusp measure}, and the properties of $\vp_{n}^{+}$ as in Proposition~\ref{prop: approximate cusp} to conclude.
\end{proof}

    We note that the lower bound in Lemma~\ref{lem: measure estimate} may not be positive for all $n$, which causes issues for the independence arguments. To this end we would like to restrict ourselves to a set of indices which avoids this problem. In order to do that, we do however need that $r(t_{n}) \to \infty$ as $n\to\infty$.
    
    \begin{lem}\label{lem:rUnbounded}
        We can assume without loss of generality that 
        \begin{equation}\label{eq:rUnbounded}
        \liminf_{n\to\infty} r(t_{n}) = \infty.
        \end{equation}
    \end{lem}
    \begin{proof}
        Suppose that $\liminf_{n\to\infty}r(t_{n}) = R$ for some $0\leq R<\infty$.
        Let $I_{n}$ as in the proof of Lemma~\ref{lem: measure estimate}. One checks that\footnote{ The sets $I_n$ differ from $A_n^\ast$ in removing the lower bound restriction on the denominators $q$.
       }
        \begin{align*}
            I_n = \left\lbrace 
             \xbf\in \R^d: \exists (\bfm{p},q)\in P(\Z^{d+1})\text{ s.t.~}
             0< |q| < 2^n\text{ and }
             \lVert q\xbf-\bfm{p}\rVert < \psi(2^n)
        \right\rbrace.
        \end{align*}
        Thus, $W(\psi)\supseteq \limsup I_n$.
        It is then elementary to check that
        \begin{align*}
            \mu(W(\psi)) \geqslant \limsup_n \mu(I_n).
        \end{align*}
        Let $\vp_{e^{-R}}$ be the smooth function obtained by applying Proposition~\ref{prop: approximate cusp} with $\eta$ as in~\eqref{eq:eta}.
        In particular, recalling the definition of $\vp_n$ in~\eqref{eq: cusp approximations}, the functions $\vp_n$ converge pointwise to $\vp_{e^{-R}}$ along a subsequence.
        Recall by Corollary~\ref{cor:growth of tn} that $t_n\to\infty$.
        Thus the bound~\eqref{eq:equidist_I_n} and the dominated convergence
        theorem yield
        \begin{align*}
            \limsup_n \mu(I_n) \geqslant \int \vp_{e^{-R}}\;\der m_{G/\Ga}>0,
        \end{align*}
        and the conclusion of Proposition~\ref{prop: divergence thm} follows.
    \end{proof}
    
    We will from now on assume~\eqref{eq:rUnbounded}.
    Define $\Gcal_0 \subseteq \N$ by
    \begin{equation}\label{eq: Gcal_0}
        \Gcal_0 = \seti{n\in \N:  C_d e^{-(d+1)r(t_n)}/12 \geqslant
        C'_d e^{-2(d+1)r(t_n)} +C_\Fcal e^{-\k_\ast t_n} + C_\Fcal e^{-\d n}
        },
    \end{equation}
    where $\d= \k_\ast d \log 2/(d+1)$ and $C_d$ and $C'_d$ are the constants in Lemma~\ref{lem: measure estimate}.
    Using Lemma~\ref{lem: measure estimate}, one obtains that for all $n\in\Gcal_{0}$
    \begin{equation}\label{eq: measure comparable to r}
    C_d e^{-(d+1)r(t_n)} /12\leqslant   \mu(A_n^\ast) \leqslant 4C_d e^{-(d+1)r(t_n)}.
    \end{equation}

    \begin{lem}\label{lem:G_0 cofinite}
        The set $\N\setminus \Gcal_0$ is finite.
    \end{lem}
    \begin{proof}
      As of Corollary~\ref{cor:growth of tn}, the last two terms in the defining
      inequality in~\eqref{eq: Gcal_0} are $O_{\Fcal,\psi,d}(e^{-\delta n})$. On
      the other hand, by definition of $t_{n}$, by~\eqref{eq: psi large}, and by
      Lemma~\ref{lem: psi to r} we have 
      \begin{equation}\label{eq: r condensation}
        e^{-(d+1)r(t_{n})} = 2^{n}\psi^{d}(2^{n})
      \end{equation}
      and therefore $e^{-(d+1)r(t_{n})}\gg n^{-1.1}$ for all $n\in\bN$. Thus we obtain
      \begin{equation*}
        C_{d}e^{-(d+1)r(t_{n})}/24 \geq C_{\Fcal}(e^{-\k_\ast t_n} + e^{-\d n})
      \end{equation*}
      for sufficiently large $n$. The claim now follows from combining this
      with~Lemma~\ref{lem:rUnbounded}.
    \end{proof}

     \begin{corollary}\label{cor: sum mu An star div}
  $\sum_{n\in\Gcal_{0}} \mu(A_n^\ast) =\infty$.
\end{corollary}

\begin{proof}
  Since by hypothesis $\psi$ is monotone and satisfies $\sum_{q\geq 1}
  \psi^{d}(q) =\infty$, it follows from~\eqref{eq: r condensation} that
  \begin{equation*}
    \sum_{n\in \N} e^{-(d+1)r(t_n)} = \infty.
  \end{equation*}
  The claim thus follows from Lemma~\ref{lem:G_0 cofinite} and~\eqref{eq:
    measure comparable to r}.
\end{proof}

%--------------------------------------------------

\subsection{Weak Quasi-independence}

    The goal of this subsection is to verify the Hypothesis~\ref{item:short range independence} of Proposition~\ref{lem: BC divergence} regarding the weak quasi-independence hypothesis of the sets $A_n^\ast$ in the short range.

     \begin{prop}\label{prop: weak quasi-independence}
     
  There exist constants $C\geq 1$ and $\g>0$ such that for all $m,n\in \Gcal_0$ with $m\leq  n$,
  \begin{equation*}
      \mu(A_m^\ast \cap A_n^\ast) \leqslant C \rho_{\min}^{-\g} \mu(A_m^\ast) \max\seti{\mu(A_n^\ast)^{\g/d}, 2^{-\g(n-m)}},
  \end{equation*}
  where $\rho_{\min}$ is the smallest contraction ratio of the IFS $\Fcal$. 
  \end{prop}
  
  \begin{remark}
  The constant $\rho_{\min}$ is not an intrinsic parameter to $\mu$ and hence is distinguished from the constant $C$ in the statement. 
  \end{remark}

    We remark that the proof of Proposition~\ref{prop: weak quasi-independence} relies on a doubling estimate for the measures of the sets $A_n^\ast$; cf.~Proposition~\ref{prop: doubling}.
    This step in turn relies on our effective equidistribution hypothesis. Additionally, a key ingredient in the proof is the following simplex lemma, whose idea is attributed to Davenport~\cite{Davenport-simplex}. This observation has found numerous applications in problems in Diophantine approximation.

\begin{prop}\label{prop: simplex}
Suppose $B \subset \R^d$ is a ball of radius $r>0$ in our fixed norm and let $N\geq 1$. Denote by $\Qcal(N)$ the set of all the rational points $\bfm{p}/q$ satisfying $0<|q| < N$ and $\bfm{p}\in \Z^d$. Assume that $\mrm{Vol}(B) < \frac{1}{d! N^{d+1}}$.
Then, there exists an affine hyperplane $\Lcal$ such that $B\cap \Qcal(N) \subset \Lcal$.
\end{prop}

\begin{proof}
The lemma is well-known and we include a proof for completeness. If $B\cap \Qcal(N)$ contains $d$ or fewer points, then the lemma follows in this case.
Otherwise, suppose that $\bfm{p}_i/q_i \in B\cap \Qcal(N)$, $1\leq i\leq d+1$, are distinct points which are not contained in any proper affine hyperplane. Denote by $\Delta$ the $d$-dimensional simplex with vertices given by the points $\bfm{p}_i/q_i$.
Then, $\Delta$ is contained in $B$ and hence
\begin{equation}\label{eq: simplex upper}
    |\Delta| \leqslant \mrm{Vol}(B),
\end{equation}
where $|\Delta|$ denotes the volume of $\Delta$.

For each $i$, write $\bfm{p}_i = (p_{i,1},\dots,p_{i,d})$.
The product of $d!$ and the volume of $\Delta$ equals the absolute value of the determinant of the matrix whose $i^{th}$ row is $(1,p_{i,1}/q_i,\dots, p_{i,d}/q_i)$. It follows that
\begin{equation}\label{eq: simplex lower}
     d!\lvert\Delta\rvert\geqslant \frac{1}{\lvert q_1 \cdots q_{d+1}\rvert}\geqslant\frac{1}{N^{d+1}}.
\end{equation}
We obtain a contradiction upon combining~\eqref{eq: simplex upper} and~\eqref{eq: simplex lower} with our hypothesis that $\mrm{Vol}(B) < \frac{1}{d! N^{d+1}}$.
\end{proof}

The next ingredient is the following doubling property of the measures of the sets $A_k^\ast$.

\begin{prop}\label{prop: doubling}
For every $A\geq 1$ there is $A^{\prime}\geq 1$ such that
\begin{equation}\label{eq: general doubling}
    \mu(A_m^\ast(A\psi)) \leqslant A^{\prime}\mu(A_m^\ast(\psi)),
\end{equation}
whenever $m\in \Gcal_0$.
\end{prop}

\begin{proof}
 Let $\tilde{\psi} = A\psi$ and let $\tilde{r}$ and $\tilde{\l}$ be the corresponding functions provided by Lemma~\ref{lem: psi to r}.
    Let $\tilde{t}_m$ be chosen so that $e^{\tilde{\l}(\tilde{t}_m)} = 2^m $.
    By Lemma~\ref{lem: measure estimate}, we have
    \begin{equation*}
        \mu(A_m^\ast(\tilde{\psi})) \leqslant
        2C_d  e^{-(d+1)\tilde{r}(\tilde{t}_m)}
        + C_\Fcal e^{-\k_\ast \tilde{t}_m}. 
    \end{equation*}
    It then follows from the relation $e^{-(d+1)r(t)} = e^{\l(t)}\psi^{d}(e^{\l(t)})$ (which also holds for $\tilde{r}$ and $\tilde{\l}$ in place of $r$ and $\l$ respectively) that 
    \begin{equation*}
        e^{-(d+1)\tilde{r}(\tilde{t}_m)} = 2^m A\psi^{d}(2^m) = Ae^{-(d+1)r(t_m)}
        \leqslant A\frac{12}{C_d} \mu(A_m^\ast(\psi)),
        \end{equation*}
    where we used~\eqref{eq: measure comparable to r} and the fact that $m\in \Gcal_0$ for the last inequality.
    Moreover, again using Lemma~\ref{lem: psi to r}(iv) and an induction argument, there is a constant $\tilde{\tau}_{0}$ depending solely on $A$ and $\psi$ such that $\tilde{t}_m \geqslant\tilde{\tau}_{0}+m d\log 2 /(d+1)$.
    Since $m\in \Gcal_0$, it follows that
    \begin{equation*}
         C_\Fcal e^{-\k_\ast \tilde{t}_m} \leqslant\frac{e^{-\k_\ast\tilde{\tau}_{0}}C_d}{12} e^{-(d+1)r(t_m)} \leqslant
         e^{-\k_\ast\tilde{\tau}_{0}}\mu(A_m^\ast(\psi)).
    \end{equation*}
   We, thus, obtain
   \begin{equation}\label{eq: doubling}
       \mu(A_m^\ast(\tilde{\psi})) \leqslant \left[24 A + e^{-\k_\ast\tilde{\tau}_{0}} \right] \mu(A_m^\ast(\psi)).
    \end{equation}
\end{proof}

\begin{proof}[Proof of Proposition~\ref{prop: weak quasi-independence}]
 For $y\in \R^d$ and $r\geq 0$, we write $B(y,r)$ for the ball around $y$ of radius $r$ in our fixed norm $\norm{\cdot}$ on $\R^d$.
   For all $k\in \N$, we can write $A_k^\ast$ as a union of boxes as follows:
   \begin{equation}\label{eq: union of boxes}
       A_k^\ast = \bigcup_{2^{k-1}\leq |q| < 2^k} 
       \bigcup_{\bfm{p}\in \Z^d: (\bfm{p},q)\in P(\Z^{d+1})} B\left(\frac{\bfm{p}}{q}, \frac{\psi(2^k)}{\lvert q\rvert}\right),
   \end{equation}
   where for $\bfm{p}=(p_1,\dots,p_d)$, we use $(\bfm{p},q)$ to denote the vector $(p_1,\dots,p_d,q)$.
   We denote by $\Zcal(k)$ the set of centers $\bfm{p}/q$ of the boxes in the union in~\eqref{eq: union of boxes}.
   Denote by $\Lambda$ the finite index set of the maps in the IFS $\Fcal$.
   We let $K$ denote the diameter of the fractal $\Kcal$ (in the metric induced by $\norm{\cdot}$) and denote by $K_{\ast}$ the maximum of $1$ and $K$.

   Denote by $\mathfrak{c}_d$ the volume of the unit ball in $\R^d$ in the norm $\norm{\cdot}$.
   Fix $m\leq n$, with $m,n\in \Gcal_0$, and define
   \begin{equation}\label{eq: case 1 n close}
       \d = \min\seti{\frac{\psi(2^m)}{2^m} , 2^{-\frac{d+1}{d} n}}.
   \end{equation}
   The monotonicity of $\psi$ implies that $\psi(2^m)/2^m \geq \psi(2^n)/2^n$.
   Moreover, in view of~\eqref{eq: psi small}, we have that
   $ 2^{-\frac{d+1}{d} n} \geq \psi(2^n)/2^n$
   and thus we have
   \begin{equation}\label{eq: delta big}
       \frac{\psi(2^n)}{2^n} \leqslant \d.
   \end{equation}
   Let $P(\d)$ be the complete prefix set defined in~\eqref{eq: complete prefix for epsilon}, with $\e=\d$.
   By Lemma~\ref{lem: comp prefix}, we have
   \begin{align}\label{eq: sum over alpha}
       \mu(A_m^\ast \cap A_n^\ast) &= \sum_{\a\in P(\d)} 
       \mu(\Kcal_\a \cap A_m^\ast \cap A_n^\ast) 
       \leqslant \sum_{\substack{\a\in P(\d),\\
       \Kcal_\a \cap A_m^\ast \neq\emptyset}}
       \mu(\Kcal_\a \cap A_n^\ast).
   \end{align}
   Fix $\a\in P(\d)$. Note that the diameter of $\Kcal_\a$ is $\leq K_{\ast}\rho_\a$.
   Hence, using~\eqref{eq: delta big}, for any ball $B_{\a}$ of radius $5K_{\ast}\d$
   and center in $\Kcal_\a$, we have
   \begin{equation*}
       \seti{\xbf \in \R^d: d(\xbf,\Kcal_\a) \leqslant \frac{2\psi(2^n)}{2^{n-1}}} \subseteq B_{\a},
   \end{equation*}
   where $d$ denotes the distance in the metric induced by $\norm{\cdot}$.
   Let $B(\a,n)=\Zcal(n) \cap B_{\a}$.
   We note that if $\Kcal_\a$ meets a box $B(\bfm{p}/q,\psi(2^n)/\lvert q\rvert)$ for some $\bfm{p}/q\in \Zcal(n)$, then $\bfm{p}/q \in B(\a,n)$.
   Recall that there is $M_{d}\in\bN$ depending only on $d$ such that $B_\a$ can
   be covered using at most $M_{d}$ balls of volume strictly smaller than
   $2^{-(d+1)n}/d!$.
   It follows from Proposition~\ref{prop: simplex} that there is a collection $\mf{L}_\a =
   \seti{\Lcal_i}$ consisting of at most $M_{d}$ hyperplanes so that
   \begin{equation*}
       B(\a,n) \subset \bigcup_{\Lcal_i\in\mf{L}_\a} \Lcal_i.
   \end{equation*}
   In particular, letting $\varepsilon=\psi(2^n)/2^{n-1}$, this shows that for all $\pbf/q\in\Zcal(n)$, we have
   \begin{equation}\label{eq: neighbourhood hyperplane}
     \Kcal_{\alpha}\cap B\left(\frac{\pbf}{q},\frac{\psi(2^{n})}{\lvert q\rvert}\right)\subseteq \bigcup_{\Lcal_i\in\mf{L}_\a} \Kcal_{\alpha}\cap\Lcal_{i}^{(\varepsilon)},
   \end{equation}
   where $\Lcal_i^{(\e)}$ is the open $\e$-neighborhood of $\Lcal_i$.
   
   Since $\Fcal$ is irreducible and satisfies the open set condition, Proposition~\ref{prop: friendly} shows that $\mu$ is $(C,\g)$-absolutely decaying for some $C\geq 1$ and $\g>0$.
   Combined with~\eqref{eq: sum over alpha} and~\eqref{eq: neighbourhood hyperplane}, this yields
   \begin{align}\label{eq: apply friendly}
       \mu(A_m^\ast \cap A_n^\ast) &\leqslant
       \sum_{\substack{\a\in P(\d),\\
       \Kcal_\a \cap A_m^\ast \neq\emptyset}} 
       \sum_{\Lcal_i\in\mf{L}_\a} \mu(\Kcal_\a \cap \Lcal_i^{(\e)})
       \leqslant C 
       \sum_{\substack{\a\in P(\d),\\
        \Kcal_\a \cap A_m^\ast \neq\emptyset}}
       \lvert \mf{L}_\a \rvert\left(\frac{\e}{\rho_\a}\right)^\g \mu(\Kcal_\a)\nonumber \\
       &\leqslant \frac{C_1}{\rho_{\min}^\g } \left(\frac{\psi(2^n)/2^n}{\d}\right)^\g 
       \sum_{\substack{\a\in P(\d),\\
       \Kcal_\a \cap A_m^\ast \neq\emptyset}}
       \mu(\Kcal_\a),
   \end{align}
   where we set $C_1 = 2^\g C M_{d}$.

   The next ingredient is to observe that if $\a\in P(\d)$ and $\Kcal_\a$ intersects $A_m^\ast$ non-trivially, then we have that $\Kcal_\a$ is contained in $A_m^\ast((K_\ast+1)\psi)$ by choice of $\delta$, where $A_m^\ast((K_\ast+1)\psi)$ is as in~\eqref{def: A_n^ast} with the function $(K_\ast+1)\psi$ in place of $\psi$.
   Hence, by Proposition~\ref{prop: doubling}, there exists $A^{\prime}\geq 1$, depending only on $\psi$ and $K_\ast$, such that
   \begin{equation}\label{eq: apply doubling}
       \sum_{\substack{\a\in P(\d),\\
       \Kcal_\a \cap A_m^\ast \neq\emptyset}}
       \mu(\Kcal_\a) \leqslant A^{\prime}
       \mu(A_m^\ast).
   \end{equation}
   
   We can now conclude the proof of Proposition~\ref{prop: weak quasi-independence}.
   First we note that as $\psi$ is by assumption non-increasing and as $m\leq n$, we have
   \begin{equation}\label{eq: magic identity}
     \frac{\psi(2^n)/2^n}{\psi(2^m)/2^m}
      \leqslant 2^{-(n-m)}.  
   \end{equation}
   This estimate takes care of the case $\d = \psi(2^m)/2^m$.
   For the second case, we recall that $e^{\l(t_{n})}=2^{n}$. Hence~\eqref{eq: lamda(t)}-\eqref{eq: lamda to L} imply that
   \begin{equation*}
     \frac{\psi(2^n)/2^{n}}{2^{-(d+1)n/d}}  = 2^{n/d}\psi(2^n) = e^{-\frac{d+1}{d}r(t_n)}.
   \end{equation*}
   Hence, it follows by~\eqref{eq: measure comparable to r}, since $n\in \Gcal_0$, that
   \begin{equation}\label{eq: apply effective equidistribution}
       \left(\frac{\psi(2^n)/2^n}{2^{-(d+1)n/d}}\right)^\g \leqslant\left( \frac{12}{C_d} \mu(A_n^\ast)\right)^{\g/d}.
   \end{equation}
   The lemma follows upon combining~\eqref{eq: apply friendly},~\eqref{eq: apply doubling},~\eqref{eq: magic identity}, and~\eqref{eq: apply effective equidistribution}.

\end{proof}

%--------------------------------------------------

\subsection{Quasi-independence and Weak Monotonicity}\label{sec:QI-and-WM}

    The goal of this subsection is to verify Hypotheses~\ref{item:long range independence} and~\ref{item:weak monotone} of Proposition~\ref{lem: BC divergence}.
    
    \begin{lem}\label{lem: quasi-independence}
    There exist constants $C'_\ast\geq 1$ and $\s,\e'_\ast>0$ such that the following holds.
    For all sufficiently large $m,n\in \Gcal_0$ satisfying
    \begin{equation}\label{eq:separation for n and m}
        n\geqslant C'_\ast m, \qquad \text{or} \qquad m\leqslant n \leqslant (1+\e'_\ast) m,
    \end{equation}
    we have
    \begin{equation*}
        \mu(A_m^\ast \cap A_n^\ast) \leqslant 576
        \mu(A_m^\ast) \mu(A_n^\ast) + C_\Fcal^{\prime} \left( e^{-\s m} \mu(A_n^\ast) + e^{-\s (n-m)}\right)
    \end{equation*}
    for some constant $C_\Fcal^{\prime} \geq 1$.
    \end{lem}
    
    \begin{proof}
    We start with an outline of the argument. Similarly to the convergence case,
    we approximate $\mu(A_{k}^{\ast})$ by the average with respect to $\mu$ of a
    smooth function on $G/\Gamma$ over pushed periodic horocycle at time $t_{k}$
    and then apply Proposition~\ref{prop: double equidist} in order to
    derive the desired inequality. In order to do this, we first need to find
    $C'_\ast$ and $\e'_\ast$ so that the separation of $m$ and $n$ implies the
    separation for $t_{m}$ and $t_{n}$ as required by Proposition~\ref{prop:
    double equidist}. The rest of the proof is then concerned with deriving the
    upper bound in Lemma~\ref{lem: quasi-independence} using the assumed
    equidistribution of the fractal measure on expanding horospheres. The main terms 
    coming from the equidistribution give rise to the constant $576$
    and -- via Proposition~\ref{lem: BC divergence} -- thus to the uniform lower
    bound in Proposition~\ref{prop: divergence thm}.

    Let $m\leq n$ be in $\Gcal_0$. Using our hypotheses on $\mu$, we choose
    constants $C_\ast$ and $\e_\ast$ as provided by the conclusion of
    Proposition~\ref{prop: double equidist}. 
    
    First, we choose parameters $C_\ast'\geq 1$ and $\e_\ast'>0$ so that the separation assumption~\eqref{eq:separation for n and m} implies the corresponding inequalities for $t_n$ and $t_m$.

    By Lemma~\ref{lem: psi to r}, we have
    \begin{equation*}
      t_n- L(t_n) = -dr(t_n) = -d(t_n-\l(t_n)).
    \end{equation*}
    Since $\l(t_n)=n\log 2$ and $L(t_n) = -\log \psi^d(2^n)$, we get
    \begin{equation*}
      t_n = \frac{d\log 2}{d+1} n - \frac{1}{d+1} \log\psi^d(2^n).
    \end{equation*}
    It follows that
    \begin{equation}\label{eq:formula for tn-tm}
      t_n - t_m = \frac{d\log 2}{d+1} (n -m) + \frac{1}{d+1} \log \frac{\psi^d(2^m)}{\psi^d(2^n)}.
    \end{equation}
    Recall that $\psi^d(2^m)\leq 1/2^m $ and $\psi^d(q) \geq 1/q\log^{1.1}q$ by~\eqref{eq: psi small} and~\eqref{eq: psi large} respectively.
    Moreover, by Corollary~\ref{cor:growth of tn}, we have that
    \begin{equation}\label{eq:t_m grows linearly}
    t_m \geqslant t_0 +m d\log 2/(d+1) \qquad (m\in\N).    
    \end{equation}

    Let $\e_\ast' = d\e_\ast/2(d+1)$ and suppose that $n-m\leq \e'_\ast m$.
    Then,
    \begin{align*}
      t_n - t_m &\leqslant (n-m)\log 2+ 
      \frac{1.1}{d+1}\log n\nonumber\\
      &\leqslant m \e'_\ast\log 2+
      \frac{1.1}{d+1}\log ((1+\e'_\ast)m)\nonumber\\
      &\leqslant \e_\ast(t_m-t_0)/2 +
      \frac{1.1}{d+1}\log ((1+\e'_\ast)m).
    \end{align*}
    In view of~\eqref{eq:t_m grows linearly}, we have that
    \begin{equation*}
        \frac{1.1}{d+1}\log ((1+\e'_\ast)m)
      - \e_\ast t_0/2 \leqslant \e_\ast t_m/2,
    \end{equation*}
    for all $m$ large enough.
    Hence, it follows that $t_n-t_m\leq \e_\ast t_m$ for $m$ sufficiently large.

    Now, suppose $n\geq C'_\ast m$ with $C'_\ast$ still to be determined. Using that $t_n\geq t_0 + nd\log 2/(d+1)$, it follows that for large enough $m$,
    \begin{equation*}
        t_n \geqslant \frac{dC'_\ast \log 2  }{2(d+1)} m. 
    \end{equation*}
    Arguing as above using~\eqref{eq:formula for tn-tm} to estimate $t_m-t_0$, we find that $t_m \leqslant m$ whenever $m$ is large enough.
    Choosing $C'_\ast$ to be large enough, depending on $C_\ast$, we see that $t_n\geqslant C_\ast t_m$ when $m\gg 1$.

    We now proceed to applying Proposition~\ref{prop: double equidist}.
    Let $k\in\bN$. Using the notation in~\eqref{eq: cusp approximations}, we note that if $\th_k$ is the indicator function of $A_k^\ast$, then $\th_k(\xbf) \leqslant \chi_k(g_{t_k}u(\xbf)\Ga)$ for all $\xbf\in\bR^{d}$.
    Moreover, since $\chi_k \leq \vp_k^+$ for all $k\in\N$, we obtain
    \begin{align}\label{eq:bound intersection}
        \mu(A_m^\ast \cap A_n^\ast) \leqslant 
        \int \vp_m^+(g_{t_m}u(\xbf)\Ga) \vp_n^+(g_{t_n}u(\xbf)\Ga) \;d\mu(\xbf).
    \end{align}
    Recall that Proposition~\ref{prop: approximate cusp}\eqref{item:norm of vp_epsilon} implies $\Scal(\vp_k^+)\ll 1$, uniformly over $k\in\N$.
    Hence, by Proposition~\ref{prop: double equidist}, there exist constants $\d>0$ and $\tilde{C}_\Fcal\geq1$ such that
    \begin{equation}\label{eq: apply double equidist}
      \begin{aligned}
        \int \vp_m^+(g_{t_m}u(\xbf)\Ga) &\vp_n^+(g_{t_n}u(\xbf)\Ga) \;d\mu(\xbf)\\
        &\leqslant \int\vp_m^+(g_{t_m}u(\xbf)\Ga)\;d\mu 
        \int \vp_n^+ \;\der m_{G/\Ga}
        + \tilde{C}_\Fcal e^{-\d |t_n - t_m|}.
      \end{aligned}
    \end{equation}

    Next, we find an upper bound for the right side of~\eqref{eq: apply double equidist} as in the conclusion of the lemma.
    By~\eqref{eq:formula for tn-tm} and monotonicity of $\psi$, we have that
    \begin{equation*}
        t_n-t_m \geqslant \frac{d\log 2}{d+1}(n-m).
    \end{equation*}
    Hence, for $\sigma\leq\frac{d}{d+1}\d\log 2$, we obtain
    \begin{equation}\label{eq: spacing of t_i's}
      e^{-\d|t_n - t_m|} \leqslant e^{-\sigma|n-m|}.
    \end{equation}
   Moreover, by definition of $\vp_n^+$ in~\eqref{eq: cusp approximations} and Proposition~\ref{prop: approximate cusp}\eqref{item:bound measure of vp_epsilon}, we have that
   \begin{equation*}
        \int \vp_n^+ \;\der m_{G/\Ga}
        \leqslant (1+\eta)C_d e^{-(d+1)r(t_n)}\leqslant 2C_d e^{-(d+1)r(t_n)}.
   \end{equation*}
    Hence, since $n\in\Gcal_0$, we may apply~\eqref{eq: measure comparable to r} to get that
    \begin{equation}\label{eq:bound vp_n by A_n}
      \int \vp_n^+ \;\der m_{G/\Ga} \leqslant 24\mu(A_{n}^{\ast}).
    \end{equation}

    To bound the term $\int\vp_m^+(g_{t_m}u(\xbf)\Ga)\;d\mu $, we use the
    effective equidistribution hypothesis on $\mu$. After possibly enlarging the
    constant $\tilde{C}_\Fcal$ in order to subsume the implicit constant in
    Corollary~\ref{cor:SL equidist}, we get
    \begin{align*}
        \int\vp_m^+(g_{t_m}u(\xbf)\Ga)\;d\mu 
        &\leqslant \int \vp_m^+ \;dm_{G/\Ga} + 
        \tilde{C}_\Fcal e^{-\k_\ast t_m}. 
    \end{align*}
    Arguing as above, since $m\in \Gcal_0$, we see that
    \begin{equation}\label{eq: non-equidist factor}
        \int\vp_m^+(g_{t_m}u(\xbf)\Ga)\;d\mu
        \leqslant 24\mu(A_m^\ast)+\tilde{C}_\Fcal e^{-\k_\ast t_m},
    \end{equation}
    
    Finally, using~\eqref{eq:t_m grows linearly} once again, there is a constant $A\geq 1$, depending only on $\psi$ and $d$, such that
    \begin{equation}\label{eq: usual bound t_m}
      e^{-\k_\ast t_m}\leq Ae^{-\sigma m}
    \end{equation}
    for any $\sigma\leq\k_\ast d\log 2/(d+1)$.

    Let $\s = \min\seti{\d,\k_\ast}d\log 2/(d+1)$ and $C_\Fcal' = 24A\tilde{C}_\Fcal$.
    Combining the estimates~\eqref{eq:bound intersection}--\eqref{eq: usual bound t_m}, we obtain
    \begin{align*}
        \mu(A_m^\ast \cap A_n^\ast)
        &\leqslant
        \left(24\mu(A_m^\ast)+
        A\tilde{C}_\Fcal  e^{-\s m}
        \right)
        24\mu(A_n^\ast)
        + \tilde{C}_\Fcal e^{-\s(n-m)}
        \\
        &\leqslant 576
        \mu(A_m^\ast)\mu(A_n^\ast)
        + C'_\Fcal
        \left(e^{-\s m} + e^{-\s(n-m)} \right),
    \end{align*}
    where we used the trivial bound $\mu(A_n^\ast)\leq 1$.
    The lemma follows.
    \end{proof}
    
    %--------------------------------------------------
    
    Since the function $r(t)$ may fail to be monotone, the measures of the sets $A_n^\ast$ may also fail to decrease monotonically to $0$.
    The next lemma allows us to control this failure of monotonicity in short intervals of natural numbers.

  \begin{lem}\label{lem:weak monotone}
    Let $\gamma\geq 0$ be arbitrary.
    For every $0<\e<0.1$, there exists a constant $C_{d,\gamma}\geq 1$, so that for all $m\leq n\in\Gcal_{0}$,
    \begin{equation*}
      n - m \leqslant \gamma+\e r(t_m) \Longrightarrow \mu(A_n^\ast) \leqslant C_{d,\gamma} \mu(A_m^\ast)^{\s},
    \end{equation*}
    where $\s = 1- 0.1/(d+1)$.
  \end{lem}
  \begin{proof}
    
    Fix $\e \in (0,0.1)$ and assume that
    $n-m \leqslant \g+\e r(t_{m})$.
    We would like to bound $t_n-t_m$ from above.
    Arguing similarly to the proof of Lemma~\ref{lem: quasi-independence}, using~\eqref{eq:formula for tn-tm}, we find
    \begin{align*}
      t_n - t_m &\leqslant \frac{(\g+\e r(t_m))d\log 2}{d+1}+ 
      \frac{1}{d+1}\log \frac{\psi^d(2^m)}{\psi^d(2^n)}.
    \end{align*}
    Since $2^k\psi^d(2^k)=e^{-(d+1)r(t_k)}$ for all $k\in\N$ and $\e\leq 0.1$, we obtain
    \begin{align*}
        t_n - t_m &\leqslant \g+0.1 r(t_m)+ r(t_n)-r(t_m).
    \end{align*}   
    Recall that $r(t_n) -r(t_m) \geq -(t_n-t_m)/d$ by~\eqref{eq: weak monotone}.
    Hence, we obtain 
    \begin{equation}\label{eq:growth of r}
      r(t_n) -r(t_m) \geqslant   -(\g+0.1 r(t_m))/(d+1).
    \end{equation}
    
    Moreover, since $m,n\in\Gcal_0$,~\eqref{eq: measure comparable to r} implies $C_d e^{-(d+1)r(t_m)} \leq 8\mu(A_m^\ast)$ and $\mu(A_n^\ast) \leq 4C_d e^{-(d+1)r(t_n)}$.
    Combined with~\eqref{eq:growth of r}, we obtain
    \begin{equation*}
      \mu(A_n^\ast) \leqslant C_{d,\gamma} \mu(A_m^\ast)^\s ,
    \end{equation*}
    where $C_{d,\g} = 32 C_d^{1-\s}  e^\g$ and $\s = 1 -0.1/(d+1)$.
    This completes the proof.    
  \end{proof}

  \begin{corollary}\label{cor:weak monotone}
    The collection $\{A_{n}^{\ast}:n\in\bN\}$ satisfies Hypothesis~\ref{item:weak monotone} of Proposition~\ref{lem: BC divergence}.
  \end{corollary}
  \begin{proof}
    Let $m\in \N$. It suffices to verify the hypothesis holds when $m$ is large enough.
    As $\Gcal_{0}$ is cofinite, we can assume without loss of generality that $m\in\Gcal_{0}$. Hence by~\eqref{eq: measure comparable to r} we have
    \begin{equation*}
      -\frac{1}{d+1}\log\mu(A_{m}^{\ast})\leqslant r(t_{m})-\frac{1}{d+1}\log\Big(\frac{C_{d}}{12}\Big).
    \end{equation*}
    Given $\varepsilon\in(0,1)$, let $a_{\varepsilon}=\frac{\varepsilon}{d+1}$ and 
    \begin{equation*}
        \gamma=\max\left\{0,-\frac{1}{d+1}\log\Big(\frac{C_{d}}{12}\Big)\right\}.
    \end{equation*}
    Then, for all sufficiently large $m\in\Gcal_{0}$ and for all $n\geq m$ we have
    \begin{equation*}
      n-m\leq \lceil -a_{\varepsilon}\log\mu(A_{m}^{\ast})\rceil\implies n-m\leq\gamma+1+ \varepsilon r(t_{m}).
    \end{equation*}
    Hence, the Corollary follows from Lemma~\ref{lem:weak monotone}.
  \end{proof}

  %--------------------------------------------------
 
\subsection{Proof of Proposition~\ref{prop: divergence thm}}\label{sec:pf-divergence}
The results of this section verify Hypotheses~\ref{item:divergent sum}-\ref{item:weak monotone} of Proposition~\ref{lem: BC divergence} for the sequence of events $A_n^\ast$.
In particular, Lemma~ \ref{lem: quasi-independence} shows we may take $C_\#=576$ in the notation of Proposition~\ref{lem: BC divergence}.
Hence, we get that $\mu\left( W(\psi)\right) \geqslant 1/576$ as
claimed.

%--------------------------------------------------

 %--------------------------------------------------
 \appendix
 
 \section{Spectral Gap for Missing Digit Cantor sets}
  \label{sec:Cantor spectral gap}
  
  The goal of this section is to prove Theorem~\ref{thm:intro Cantor} providing a stronger version of our Khintchine and equidistribution theorems in the special case of missing digit Cantor sets. This is done by weakening the hypothesis~\eqref{eq:hypothesis on a,b,c} in Theorem~\ref{thm: effective single equidist}.
  A key input is a sharper estimate on the spectral gap of the operators $\Pcal_\l$, Proposition~\ref{prop:cantor spectral gap}.
 Additionally, we take advantage of the equal contraction ratios to show that, in fact, the Sobolev norm (not just the $\mrm{L}^2$-norm) of a suitable variant of the operators $\Pcal_\l^n$ decays in $n$. 
  Finally, we require a sharper form of Proposition~\ref{prop:banana} due to Str\"ombergsson as well as bounds towards Selberg's eigenvalue conjecture by Kim-Sarnak.

  First, we recall the definition of a missing digit Cantor set.
  
    \begin{definition}\label{def: p cantor}
    A set $\Kcal \subset [0,1]$ is a \textbf{missing digit Cantor set} if there exists a prime number $p\geq 3$ and $\emptyset\neq \L \subseteq \seti{0,\dots,p-1}$ such that $\Kcal$ consists of those $x\in [0,1]$ whose digits in their base $p$ expansion all belong to $\L$.
    A \textbf{missing digit IFS} (with attractor $\Kcal$) is defined as follows: 
    \begin{equation}
        \mc{F} = \seti{f_i(x) = \frac{x+i}{p} : i\in \L}.  
    \end{equation}
    \end{definition}

    Throughout the remainder of this section, we fix a missing digit Cantor set $\Kcal$ in base $p$ and digit set $\L$ along with its associated missing digit IFS $\Fcal$.

    In particular, in our notation, $\rho=\rho_i = 1/p$ and $b_i = i/p$. 
    One checks that this IFS satisfies the open set condition.
    In particular, we have
    \begin{equation*}
        s:=\dim_H(\Kcal) = \log |\L|/\log p.
    \end{equation*}
    By~\cite{Moran}, the $s$-dimensional Hausdorff measure of $\Kcal$ is positive and finite. We denote by $\mu$ the restriction of this measure to $\Kcal$, normalized to be a probability measure. 
    By~\cite{Hutchinson}, $\mu$ is the self-similar measure associated to the probability vector $\l_i=\rho^s, i\in\L$.

    The following is the precise form of Theorem~\ref{thm:intro Cantor}.
    \begin{theorem}\label{thm:cantor equidist}
    The conclusions of Theorem~\ref{thm:main simplified} and Theorem~\ref{thm: effective single equidist} hold for $\mu$ as above whenever 
    \begin{equation}\label{eq:thickness for p cantor sets}
        s > \cutoff.
    \end{equation}
    \end{theorem}
    
    Since we showed that Theorem~\ref{thm: effective single equidist} implies Theorem~\ref{thm:main simplified}, we only need to verify that in this special case the former holds under the condition~\eqref{eq:thickness for p cantor sets}.
    
    A wasteful step in the proof of Theorem~\ref{thm: effective single equidist} is~\eqref{eq:estimate Sobolev norm of operator}.
    To improve this estimate, we introduce slightly different operators than $\a\cdot\Pcal_\l$ which take advantage of the equal contraction ratios.
    For $\w\in\L^n$, we define
    \begin{equation}
      \t_\w = u(\mathbf{0},-b_\w) a(1,\rho_\w) =  u(\mathbf{0},-b_\w) a(1,p^{-n}).
    \end{equation}
    Note that $\t_\w$ has trivial Archimedean component. For $\a\in\L^\ast$, let $\a\cdot\Qcal_\l$ denote the averaging operator defined analogously to $\a\cdot\Pcal_\l$ in~\eqref{eq: padic MF alpha} with $\t_\w$ in place of $\g_\w$. 
    Note that $a(\rho_\w,1)\t_\w=\g_\w$; cf.~\eqref{def: gamma_w}. 
    In particular, for any function $\vp$ on $X_S$ and every $n\in\N$, we have
    \begin{equation}\label{eq:Qcal vs Pcal}
        (\a\cdot\Pcal_\l)^n(\vp)(x) = (\a\cdot\Qcal_\l)^n(\vp)\big(a(p^{-n},1)x\big).
    \end{equation}
    
\subsection{Sharper version of Proposition~\ref{prop: spectral gap}}  
The following result provides a sharper rate of decay of the operator norm of $\Pcal_\l^n$. It holds without restrictions on the dimension of the Cantor set.
    \begin{proposition}\label{prop:cantor spectral gap}
    Let $\e>0$ and $\delta_{\e}=\frac{25}{32}-2\e$.    
    For all $n\in \N$, $\a\in \L^\ast$, and for every smooth $K_{\mrm{f}}$-invariant function $\vp\in \mrm{L}^2_{00}(X_S)$, we have
    
    \begin{equation*}
        \big\lVert(\a\cdot \Pcal_\l)^n (\vp)\big\rVert^2_{\mrm{L}^2} \ll_{\e,p,s} \Scal_{2,1}(\vp)^2 (p^{-(s-\e)n}+p^{-\delta_{\e} (n+\lvert\alpha\rvert)}).
    \end{equation*}
    
    The same estimate holds for $\Qcal_\l$ in place of $\Pcal_\l$.
    \end{proposition}
  
    \begin{proof}
      Note that in view of bounds towards the Generalized Ramanujan Conjectures (GRC) for $\mrm{SL}_2$ in~\cite[Proposition 2]{KimSarnak}, for $K_{\mrm{f}}$-invariant functions, one can take the bound in Corollary~\ref{cor:matrix coeffs smooth} to be $\xi_{\Gbf}^{25/32}$ instead of $\xi_{\Gbf}^{1/2}$ (GRC predicts the exponent should be $1$); cf.~\cite[Lem.~9.1]{Venkatesh}. In what follows, we let $m:=\lvert\a\rvert$. Given $\w\in\Lambda^{n}$, in analogy to the proof of Proposition~\ref{prop: spectral gap}, we denote $\gamma_\w^\a=\gamma_\a\gamma_\w\gamma_\a^{-1}$ and similarly $\tau_\w^\a=\gamma_\a\tau_\w\gamma_\a^{-1}$. Expanding $(\a\cdot\Pcal_\l)^n$ according to \eqref{eq: padic MF alpha}, it follows from Corollary~\ref{cor:matrix coeffs smooth} and Proposition~\ref{prop:boundHarishChandra} that
      \begin{align*}
        \big\lVert(\a\cdot \Pcal_\l)^n(\vp)\big\rVert^2  &= 
        \sum_{\eta,\w\in \L^n} \l_\eta\l_\w \langle \g_\eta^\a \vp,\g_\w^\a \vp  \rangle
        \ll \Scal_{2,1}(\vp)^2 \sum_{\eta,\w\in \L^n} \l_\eta\l_\w \xi_\Gbf^{25/32}\left(\g_\w^\a(\g_\eta^\a)^{-1}\right ) \\
        &= \Scal_{2,1}(\vp)^2 \sum_{\eta,\w\in \L^n} \l_\eta\l_\w \xi_\Gbf^{25/32}\big(u(p^{-m}(b_\eta-b_\w))\big)\\
        &\ll_\e \Scal_{2,1}(\vp)^2 \sum_{\eta,\w\in \L^n} \l_\eta\l_\w
        \big\lVert u(p^{-m}(b_\eta-b_\w))\big\rVert_p^{-\frac{25}{64}+\e},
      \end{align*}
      where $\lVert{u(p^{-m}(b_\eta-b_\w))\rVert}_p$ denotes the norm of the adjoint action of $u(p^{-m}(b_\eta-b_\w))$ on the Lie algebra of $\Gbf(\Q_p)$. 
      Note further that the above estimate holds for $\Qcal_\l$ since $\g_\w^\a(\g_\eta^\a)^{-1}=\t_\w^\a(\t_\eta^\a)^{-1}$.
      Hence, it suffices to bound the above average.
      
      To calculate the adjoint norm, we find a polar decomposition of $u(p^{-m}(b_\eta-b_\w))$.
      Note that for all $\xbf\in \Q_p$ with $|\xbf|_p> 1$, we have 
      \begin{equation*}
        \begin{pmatrix}
          1 & 0 \\ \frac{-1}{\xbf +1} &1
        \end{pmatrix}
        u(\xbf)
        \begin{pmatrix}
          1 & 0 \\ 1 & 1
        \end{pmatrix} 
        \begin{pmatrix}
          1 & \frac{-\xbf}{\xbf+1}\\ 0 & 1
        \end{pmatrix}
        = \begin{pmatrix}
          1+\xbf & 0\\ 0& \frac{1}{\xbf+1}
        \end{pmatrix}.
      \end{equation*}
      Then, since $|1/(\xbf+1)|_p<1$ and $|\xbf/(\xbf+1)|_p = 1$, we obtain
      \begin{equation*}
        u(\xbf)=m_1 a\big((1+\xbf)^2\big) m_2
      \end{equation*}
      for some $m_1,m_2\in \Gbf(\Z_p)$ (recall that $\Gbf=\mrm{PGL}_2$).
      Using~\eqref{eq:invariance of norm}, we get
      \begin{equation*}
        \norm{u(\xbf)}_p = \big\lVert a\big((\xbf+1)^2\big)\big\rVert_p = |\xbf|_p^2.
      \end{equation*}
      It follows that
      \begin{equation*}
        \big\lVert(\a\cdot \Pcal_\l)^n(\vp)\big\rVert^2  \ll_{\e} 
        \Scal_{2,1}(\vp)^2 \left(\sum_{\eta=\w\in \L^n} \l_\eta\l_\w 
          + p^{-(\frac{25}{32}-2\e)m}\sum_{\eta\neq \w\in \L^n} \l_\eta\l_\w |b_\eta-b_\w|_p^{-\frac{25}{32}+2\e} 
        \right).
      \end{equation*}
      
      Fix some $\e>0$ and let $\d = 25/32-2\e$.
      For each $\eta\neq \w \in \L^n$, define
      \begin{equation*}
        d(\eta,\w) = \max\seti{1\leq i\leq n: \eta_i \neq \w_i}.
      \end{equation*}
      Recall that $b_j = j/p$ for all $j\in \L$.
      Let $\eta\neq\w\in\L^n$ and let $d=d(\eta,\w)$.
      A simple calculation then shows that
      \begin{align*}
        b_{\eta}-b_\w = \sum_{i=1}^n (b_{\eta_i}-b_{\w_i}) p^{-i+1}
        = \sum_{i=1}^{d} (\eta_i-\w_i) p^{-i}
        = p^{-d} \sum_{i=1}^{d} (\eta_i-\w_i) p^{d-i}.
      \end{align*}
      By definition, we have $\eta_d\neq \w_d$. This implies that the integer $\sum_{i=1}^{d} (\eta_i-\w_i) p^{d-i}$ is coprime to $p$, i.e., a unit in $\bZ_{p}=\{\xbf\in\bQ_{p}:\lvert\xbf\rvert_{p}\leq 1\}$. Thus, it follows that
      \begin{align*}
        \sum_{\eta\neq\w\in \L^n} \l_\eta \l_\w \left|b_\eta-b_\w \right|_p^{-\d}
        = \sum_{\eta\neq\w\in \L^n} \l_\eta \l_\w p^{-\d d(\eta,\w)}
        = \sum_{\eta\in \L^n} \l_\eta \sum_{j=1}^n p^{-\d j} 
        \sum_{\substack{\w\in \L^n\\ d(\eta,\w)=j}} \l_\w.
      \end{align*}
      
      We now specialize to the case where $\l$ is the uniform probability vector with weight $1/|\L|$. Then, for each $\eta\in \L^n$, we have
      \[\l_\eta = |\L|^{-n} = p^{-sn}, \]
      where $s = \log |\L|/\log p$.
      Hence, we obtain
      \begin{align*}
        \sum_{\eta\neq\w\in \L^n} \l_\eta \l_\w \left|b_\eta-b_\w \right|_p^{-\d}
        &= p^{-2sn} \sum_{j=1}^n p^{-\d j} \sum_{\eta\in \L^n}
        \underbrace{\lvert\{\w\in \L^n: d(\eta,\w)=j\}\rvert}_{=\lvert\Lambda\rvert^{j-1}(\lvert\Lambda\rvert-1)}\\
        &\leqslant  p^{-2sn} \sum_{j=1}^n p^{-\d j} \sum_{\eta\in \L^n} |\L|^{j}=  p^{-2sn} \sum_{j=1}^n p^{-\d j} |\L|^{n+j}.
      \end{align*}
      If $s\neq \d$, then, using that $|\L| = p^s$, we obtain
      \begin{align*}
        \sum_{\eta\neq\w\in \L^n} \l_\eta \l_\w \left|b_\eta-b_\w \right|_p^{-\d}
        &\leqslant p^{-sn+s-\d} \sum_{j=0}^{n-1} p^{(s-\d) j} 
        = p^{-sn+s-\d} \frac{p^{(s-\d)n}- 1}{ p^{(s-\d)} -1}
        \leqslant  \frac{p^{s-\d}}{\lvert p^{s-\d}-1\rvert}\big(p^{-\d n}+p^{-sn}\big).
      \end{align*}
      Otherwise, if $s=\d$, we get a bound of the form $n p^{-s n}$.
      Finally,
      we note that
      \begin{equation*}
        \sum_{\eta=\w\in \L^n} \l_\eta \l_\w = p^{-sn}.
      \end{equation*}

    \end{proof}

   In the proof of Theorem~\ref{thm:cantor equidist}, we will need an estimate on the decay of the $\mrm{L}^4$-norm of the operators $\a\cdot\Qcal_\l$. We deduce this estimate in the following corollary. 
     \begin{corollary}\label{cor:spectral gap in Lq}
   For all $q\geq 2$, $n\in \N$, $\a\in \L^\ast$, and for every bounded smooth $K_{\mrm{f}}$-invariant function $\vp\in \mrm{L}^2_{00}(X_S)$, we have
    \begin{equation*}
          \big\lVert(\a\cdot \Qcal_\l)^n (\vp)\big\rVert^q_{\mrm{L}^q} \ll_{\e,p,s,q} \Scal_{2,1}(\vp)^2 \norm{\vp}_\infty^{q-2}\cdot 
          p^{-2\omicron_\e n },
    \end{equation*}
    for every $\e>0$, where $2\omicron_\e=\min\seti{25/32,s}-\e$.
   \end{corollary}
   \begin{proof}
   The case when $q=2$ is exactly Proposition~\ref{prop:cantor spectral gap}.
   Hence, we may assume $q>2$.
   Let $\mu_S$ denote the $\Gbf_S$-invariant probability measure on $X_S$. Using Fubini's Theorem one checks that
   \begin{align*}
       \big\lVert(\a\cdot \Qcal_\l)^n (\vp)\big\rVert^q_{\mrm{L}^q}
       &=\int_0^\infty \mu_S(x: |(\a\cdot \Qcal_\l)^n (\vp)(x)|^q>t)\;dt
       \\
       &= \int_0^{\norm{\vp}^q_\infty} \mu_S(x: |(\a\cdot \Qcal_\l)^n (\vp)(x)|^q>t)\;dt, 
   \end{align*}
   where we used that $\norm{\a\cdot\Qcal_\l^n(\vp)}_\infty\leq \norm{\vp}_\infty$.
   Hence, by Proposition~\ref{prop:cantor spectral gap} and Chebychev's inequality, we have for all $t>0$,
   \begin{align*}
       \mu_S(x: |(\a\cdot \Qcal_\l)^n (\vp)(x)|^q>t)
       \ll_{\e,p,s} \Scal_{2,1}(\vp)^2
          p^{-2\omicron_\e n } t^{-2/q}.
   \end{align*}
   Hence, since $q>2$, we obtain
   \begin{align*}
       \big\lVert(\a\cdot \Qcal_\l)^n (\vp)\big\rVert^q_{\mrm{L}^q}
       \ll_{\e,p,s} \Scal_{2,1}(\vp)^2    p^{-2\omicron_\e n }
       \norm{\vp}_\infty^{q-2} \frac{q}{q-2}.
   \end{align*}
   \end{proof}

   \subsection{Sharper version of Proposition~\ref{prop:banana}}  
    The following result provides a sharper value of $\k$ constituting the rate of equidistribution of horospherical measures on congruence covers.
    
    \begin{prop}[Prop.~3.1,~\cite{Strombergsson-JMD}]
    \label{prop:Strombergsson}
    
    Let $\Delta\leq \Ga(1)$ be a congruence lattice and $X_\Delta= \Gbf_\infty/\Delta$.
    Then, for every $\vp\in \mrm{B}_{2,\lval}^{\infty}(X_\Delta)$, $x\in X_\Delta$ and $t\geq 1$,
    \begin{align*}
        \int_0^1 \vp(a(t) u(\xbf)x)\;\der \xbf = \int \vp \;\der m_{\Gbf_\infty^+ \cdot x}
        + O\big(V_\Delta\cdot \Scal_{2,\lval}(\vp) \cdot t^{-\k}\cdot \mathscr{Y}^{1/2}_{\Delta}(x)\big),
    \end{align*}
    where $V_\Delta = \sqrt{[\Ga(1):\Delta]}$, $\mathscr{Y}_{\Delta}$ is a positive proper function on $X_\Delta$ and if $\l_1\in (0,1/4)$ is a uniform lower bound on the non-zero eigenvalues of the Laplacian on $X_\Delta$ for all $\Delta$, then
    \begin{equation*}
         \k = \frac{1-\sqrt{1-4\l_1}}{2}. 
    \end{equation*}
    The implied constant is independent of $\Delta$.
    \end{prop}
    
    \begin{proof}
      The statement in~\cite[Prop.~3.1]{Strombergsson-JMD} is stated in a slightly different form, we outline the needed modifications.
      First, the results in \textit{loc.~cit.}~are stated for quotients of $\mrm{SL}_2(\R)$. Recall that $\Gbf_\infty^+$ is the image of $\mrm{SL}_2(\R)$ inside $\Gbf_\infty$ and is a normal subgroup of index $2$. In particular, for each $\Delta$, $X_\Delta$ consists of at most two connected components, each of which is isomorphic to $\mrm{SL}_2(\R)/\Delta'$, where $\Delta'\leq\SL_{2}(\bZ)$ is a congruence lattice.
      We define $\mathscr{Y}_{\Delta}$ to be $\mathscr{Y}_{\Delta'}$ (in the notation of~\cite[Eq.~(11)]{Strombergsson-JMD}) on each of the connected components of $X_\Delta$.
      
      The measure on $X_\Delta$ defining the $\mrm{L}^2$-Sobolev norms $\norm{\cdot}_{W_k}$ in \textit{loc.~cit.~}has total mass $\asymp V_\Delta^2$. In particular, this norm is equivalent to $V_\Delta \cdot\Scal_{2,k}$ for all $k\in\N$; cf.~discussion following~\cite[Eq.~(9)]{Strombergsson-JMD}.
      Note further that the statement is made for long horocycle orbits starting from a point $p$. The above statement is obtained from this result with $T=t$ and with $pa(T^{-1})$ in place of $p$ in the notation in \textit{loc.~cit} using standard conjugation relations of $a(t)$ and $u(\xbf)$.

      Next, we note that the implied constant in~\cite[Prop.~3.1]{Strombergsson-JMD} can be made independent of $\Delta$.
      The dependence on the lattice comes from~\cite[Lem.~2.1]{Strombergsson-JMD}. Note that the bounds in~\cite[Lem.~2.2, 2.3]{Strombergsson-JMD} are not needed for our weaker error term $t^{-\k}\mathscr{Y}_{\Delta}^{1/2}(x)$.
      
      The dependence in~\cite[Lem.~2.1]{Strombergsson-JMD} arises from a choice of an injectivity radius to allow for a thick-thin decomposition of $\mathbb{H}^2/\Delta'$ in order to apply the Sobolev embedding theorem (cf.~the choice of $\epsilon$ in the proof of~\cite[Lem.~2.1]{Strombergsson-JMD} given in~\cite[Lem.~5.3]{FlaminioForni}). As $X_\Delta$ are all covers of $X_\infty(1)\cong\mrm{SL}_2(\R)/\mrm{SL}_2(\Z)$, a choice of an injectivity radius in $X_\infty(1)$ works for all of $X_\Delta$.
      
    Hence, the error term can be obtained by applying~\cite[Lem.~2.1]{Strombergsson-JMD} to Burger's integral formula in~\cite[Eq.~(23)]{Strombergsson-JMD} combined with the estimates on the height function in~\cite[pg 303]{Strombergsson-JMD} and the estimates on the intertwining operators given in~\cite[Eq.~(22)]{Strombergsson-JMD} (or~\cite[pg.~791]{Burger-effectiveHorocycles} with $\a=\k$ in the notation of~\cite{Burger-effectiveHorocycles}) as is done in~\cite{Strombergsson-JMD}. 
    One uses~\cite[Lem.~2.2]{Strombergsson-JMD} to ensure pointwise convergence of the last integral in~\cite[Eq.~(23)]{Strombergsson-JMD} to a bounded continuous function as is done towards the end of the proof so that the above bounds apply.
    
    Finally, we note that the order $\lval$ Sobolev norm in the statement (as opposed to $\Scal_{2,4}$ in \textit{loc.~cit.}) arises from only applying the bounds of~\cite[Lem.~2.1]{Strombergsson-JMD} in the proof of~\cite[Prop.~3.1]{Strombergsson-JMD}.

    \end{proof}

 \subsection{Proof of Theorem~\ref{thm:cantor equidist}}
    We outline the needed modifications of the proof of Theorem~\ref{thm: effective single equidist} in this setting. We retain the notation in that proof, in particular the constants $a,b_\e$ and $c$ in the statement of Theorem~\ref{thm: effective single equidist}.
    We begin by noting that the average contraction ratio $r$ is $p^{-1}$ in the case at hand.

    Since missing digit Cantor sets satisfy the open set condition with the open set $(0,1)$, the proof of Theorem~\ref{thm: effective single equidist} shows that we can take the absolutely continuous measure $\nu$ to be the Lebesgue measure on the unit interval.
    In this case, the Mass Term in~\eqref{eq: Cauchy Schwarz} takes the form
    \begin{equation*}
        \text{Mass Term} = p^{2\sigma m}, \qquad 2\sigma=1-s.
    \end{equation*}

    By Lemma~\ref{lem:index estimate} and using~\eqref{eq:Qcal vs Pcal}, $(\a\cdot \Qcal_\l)^n(\vp)$ can be regarded as a function on $\Gbf_\infty/\Delta$, for some congruence lattice $\Delta$.
    Hence, we may apply Proposition~\ref{prop:Strombergsson} in place of Proposition~\ref{prop:banana} to obtain the following replacement of~\eqref{eq:applying banana}:
    \begin{align}\label{eq:applying Strombergsson}
         \int_0^1 (\a\cdot \Pcal_\l)^n(\vp)^2\big( a(t) u(\xbf)h_\a \Delta\big) \dx{\xbf}
         &=  \int_0^1 (\a\cdot \Qcal_\l)^n(\vp)^2\big( a(tp^{-n}) u(\xbf)h_\a \Delta\big) \dx{\xbf}
         \nonumber\\
         &=\int (\a\cdot \Qcal_\l)^n(\vp)^2  \;\der m_{\Gbf_\infty^+/\Delta^+}
         \nonumber\\
         &\qquad+   
         O\left(V_{\Delta} \Scal_{2,\ell}\big( (\a\cdot \Qcal_\l)^n(\vp)^2\big) p^{\k n}t^{-\k} \mathscr{Y}^{1/2}_{\Delta}(h_\a\Delta)\right),
    \end{align}
    where $\ell=3$, $\Delta^+=\Gbf_\infty^+\cap \Delta$. Here, we use the fact that $h_\a\in \Gbf_\infty^+$ so that $\Gbf_\infty^+\cdot h_\a\Delta \cong \Gbf_\infty^+/\Delta^+$.

    Note further that, by~\cite[Eq.~(11)-(13)]{Strombergsson-JMD}, $\mathscr{Y}_{\Delta}(x)\leq \mathscr{Y}_{\Ga(1)}(x) \ll \norm{\mrm{Ad}_g} $, where $g\in \Gbf_\infty$ is any representative of $x$ and $\norm{\mrm{Ad}_g}$ denotes the norm of its adjoint action.
    In particular, the estimate $\mathscr{Y}^{1/2}_{\Delta}(x_\a)\ll \rho_\a^{-C_0}$ holds for a suitable $C_0\geq 1$ in place of the estimate~\eqref{eq:height-trivial-estimate}.

    The key point in introducing the operators $\Qcal_\l$ is as follows. Since multiplication by elements of $\Gbf_\mrm{f}$ commutes with differential operators on $\Gbf_\infty$, one checks using Lemma~\ref{lem:Sobolev} that
    \begin{equation*}
        \Scal_{2,\ell}\big( (\a\cdot \Qcal_\l)^n(\vp)^2\big)
        \ll \Scal_{4,\ell}\big( (\a\cdot \Qcal_\l)^n(\vp)\big)^2.
    \end{equation*}
    Moreover, note that $\Dcal\vp$ has mean $0$ for any differential operator $\Dcal$.
    This can be checked by induction on the degree of the operator using the dominated convergence theorem, invariance of the Haar measure, and the limit definition of Lie derivatives; cf.~proof of Lemma~\ref{lem:first derivative}.
    In particular, Lemma~\ref{lem: character spectrum} implies that the lift of $\Dcal\vp$ to $X_S$ belongs to $\mrm{L}^2_{00}(X_S)$.
    
    Hence, by Corollary~\ref{cor:spectral gap in Lq}, applied with $q=4$, for any $\Dcal$ of degree $\leq \ell$, we have
    \begin{align*}
        \norm{(\a\cdot \Qcal_\l)^n(\Dcal \vp)}_{\mrm{L^4}}^4 \ll_{\e,p,s}
        \Scal_{\infty,1}(\Dcal\vp)^4 p^{-2\omicron_\e n}
        \leqslant \Scal_{\infty,\ell+1}(\vp)^4 p^{-2\omicron_\e n},
    \end{align*}
    where
    \begin{equation*}
        2\omicron_\e := \min\seti{25/32,s }-\e.
    \end{equation*}
    It follows that
    \begin{align*}
        \Scal_{2,\ell}\big( (\a\cdot \Qcal_\l)^n(\vp)^2\big)
        \ll
        \Scal_{4,\ell}\big( (\a\cdot \Qcal_\l)^n(\vp)\big)^2 
        \ll_{\e,p,s} \Scal_{\infty,\ell+1}(\vp)^2 p^{-\omicron_\e n}.
    \end{align*}

    Additionally, by Lemma~\ref{lem:index estimate}, the congruence lattice $\Delta $ can be chosen so that $V_\Delta \ll p^{3|\a|+3n/2}$.
    Finally, estimating the main term in~\eqref{eq:applying Strombergsson} using Proposition~\ref{prop:cantor spectral gap}, we obtain the following sharper bound on the horospherical term:
    \begin{align*}
        \text{Horospherical term }
        \ll_{\e,p} \Scal_{\infty,\ell+1}(\vp)^2 p^{(3+C_0)|\a|}\left(
         p^{-2\omicron_\e n} +  p^{2n\upsilon-\k\t}
        \right),
    \end{align*}
    where
    \begin{equation*}
        2\upsilon= 3/2+\k-\omicron_\e, \qquad p^\t = t.
    \end{equation*}
    By known bounds towards Selberg's eigenvalue conjecture due to~\cite[Proposition 2]{KimSarnak}, we can take $\l_1\geq 975/4096$ in Proposition~\ref{prop:Strombergsson}.
    In particular, we may take $\k = 25/64$.
    
By combining the above estimates and balancing the rates as is done in the proof of Theorem~\ref{thm: effective single equidist} (cf.~discussion following~\eqref{eq:combined rates}), we see that the conclusion of that theorem holds in our setting if
     \begin{equation*}
        2\sigma(\omicron_\e+\upsilon) < \k(\sigma+\omicron_\e)
        \Longleftrightarrow
        (1-s)(\omicron_\e+3/2)<2\k \omicron_\e,
    \end{equation*}
    for some $\e>0$ and with our choices of $\sigma,\omicron_\e,\upsilon$ and $\k$ as above.
    This condition is in turn satisfied under our hypothesis~\eqref{eq:thickness for p cantor sets} as can be shown by a direct calculation.

  \subsection{A version of Lebesgue density}
  
  In this subsection, we verify the version of Lebesgue density theorem for Bernoulli measures on symbolic spaces used in the proof of Lemma~\ref{lem:Lebesgue-density}. 
  
  Let $\L$ be a finite set and $\l$ be a probability vector on $\L$.
  For $\a\in \Sigma := \L^\N$ and $k\in \N$, denote by $\Sigma(\a,k)$ the cylinder set given by the prefix of $\a$ of length $k$ and denote this prefix by $\a|_k$.
  We endow $\L$ with the discrete topology and $\Sigma$ with the associated product topology.
  
  \begin{lem}\label{lem:Leb density for cylinders}
    Suppose $B\subseteq \Sigma$ is a Borel set. Then, for $\l^\N$-almost every $x\in B$, 
    \begin{equation*}
        \lim_{k\to\infty} \frac{\l^\N(B\cap \Sigma(x,k))}{\l^\N(\Sigma(x,k))} = 1.
    \end{equation*}
  \end{lem}
    
    \begin{proof}
    We deduce this result from the corresponding well-known Lebesgue density theorem for Radon measures on the real line.
    Let $p=2|\L|$ and consider the auxiliary IFS given by
    \begin{equation*}
        \Fcal = \seti{ f_i(x) = (x+i)/p: 0\leq i\leq p-1, i \text{ is even} }.
    \end{equation*}
    Let $\Kcal$ be its attractor and note that the images of $\Kcal$ under distinct maps in $\Fcal$ are disjoint.
    Let $\pi:\Sigma\to \R$ be the coding map defined by $\pi(\a) = \lim_{k\to\infty} f_{\a|_k}(0)$ and $\mu =\pi_\ast\l^\N$ be the self-similar measure. 
    Then, $\pi$ is a homeomorphism onto its image $\Kcal$; cf.~\cite[Thm. 3.1.(3) and Thm. 4.4.(4)]{Hutchinson}.
    Hence, it suffices to show that
    \begin{equation*}
        \lim_{k\to\infty} \frac{\mu(\pi(B)\cap\pi( \Sigma(\a,k)))}{\mu(\pi(\Sigma(\a,k)))} = 1
    \end{equation*}
    for $\l^\N$-almost every $\a\in B$. Let $\a\in B$. 
    To relate the images of cylinder sets under $\pi$ to intervals in $\R$, one first checks that $\pi(\Sigma(\a,k))$ is contained in the image of $[0,1]$ under $f_{\a|_k}$.
    Hence, by definition of $\Fcal$, given any $\b\in\Sigma$ such that $\Sigma(\b,k)\neq \Sigma(\a,k)$, the distance between $\pi(\Sigma(\b,k))$ and $\pi(\Sigma(\a,k))$ is at least $p^{-k}$. It follows that 
    \begin{equation*}
        \pi(\Sigma(\a,k))  = B(\pi(\a), p^{-k})\cap\Kcal,
    \end{equation*}
    where $B(\pi(\a), p^{-k})$ denotes the open interval around $\pi(\a)$ of radius $p^{-k}$. 
    It follows by Lebesgue's density theorem for Radon measures on $\R$ that 
        \begin{equation*} 
        \lim_{k\to\infty} \frac{\mu(\pi(B)\cap\pi( \Sigma(\a,k)))}{\mu(\pi(\Sigma(\a,k)))} = 
        \lim_{k\to\infty} \frac{\mu(\pi(B)\cap B(\pi(\a),p^{-k}))}{\mu(B(\pi(\a),p^{-k}))} = 1,
    \end{equation*}
    for $\l^\N$-almost every $\a\in B$. Note that we are allowed to use open balls in this application of Lebesgue density since $\mu$ is non-atomic. Indeed, it suffices to note that $\l^\N(\Sigma(\beta,k))\leq \l_{\max}^k\to 0$ for any $\beta\in \Sigma$, where $\l_{\max}$ denotes the largest component of $\l$.
    \end{proof}
%--------------------------------------------------

\section{Congruence quotients}
 
The goal of this appendix is to give proofs of several facts presented in Section~\ref{sec: s-arith} and used in the proof of Theorem~\ref{thm: effective single equidist}. 
In Corollary~\ref{cor:S-arith-correspondence}, we establish the correspondence between compact-open subgroups of $\Gbf(\bA_{\mrm{f}})$ and principal congruence subgroups of $\Gbf(\bZ)$ which underlies the double coset decomposition~\eqref{eq: class number one}. In Proposition~\ref{prop:components-as-quotients}, we prove the uniform bound on the number of connected components of $\biquotient{\GammaS}{\Gbf_{S}}{K_{\Sf}[N]}$ used in~\eqref{eq:apply bound on conn compo}. At the end of this section, we will define general congruence subgroups; this extension is immediate but we include it for completess.

\subsection{Integral structures}
We begin by making an explicit choice of the integral structure on $\Gbf$ used to define congruence groups.
Given a ring $R$, we let
\begin{equation*}
  V_{R}=\Mat_{d+1}(\bZ)\otimes_{\bZ}R.
\end{equation*}
Then $V_{R}$ is an $R$-algebra which is a free $R$-module of rank $(d+1)^{2}$. The algebra $V_{R}$ allows us to realize $\Gbf(k)$ as a linear group whenever $k$ is a field.
More explicitly, we fix a faithful $k$-representation of $\Gbf$ by choosing the standard basis $\Ecal_{d+1}$ of $V_{k}$ which gives rise to an isomorphism $\GL(V_{k})\cong\GL_{(d+1)^{2}}(k)$ and we define
\begin{equation*}
  \Gbf=\{g\in\GL_{(d+1)^{2}}:\forall u,v\in\Ecal_{d+1}\quad g(uv)=(gu)(gv)\}.
\end{equation*}
In what follows, we let $\Phi:\GL_{d+1}\to\Gbf$ denote the $k$-representation given by
\begin{equation}\label{eq:representation}
\Phi(x)(v)=xvx^{-1}\quad(x\in\GL_{d+1},v\in\Mat_{d+1}).    
\end{equation}
By the Skolem-Noether theorem we have $\Phi(\GL_{d+1}(k))=\Gbf(k)$ for any field $k$ and in particular $\Gbf(k)\cong\GL_{d+1}(k)/k^{\times}$, where we identify $k^{\times}$ with the scalar diagonal matrices in $\GL_{d+1}(k)$.
We record the following consequence of the above discussion which is used to apply the results of~\cite{GMO}.
\begin{lem}\label{lem:connectedness}
  The group $\Gbf$ is a connected group over $k$.
\end{lem}
\begin{proof}
  As $\Gbf$ is an affine $k$-group, we only have to prove that it is connected. Recall that $\GL_{d+1}$ is an irreducible affine $\bQ$-group. To this end, note that $\GL_{d+1}$ is the principal open set defined by the polynomial $\det(x_{ij})$, i.e.,
  \begin{equation*}
    \GL_{d+1}(k)=\big\{(x_{ij})\in k^{(d+1)^{2}}:\det(x_{ij})\neq 0\big\}.
  \end{equation*}
  This is a Zariski-open subset of affine space. As affine space is irreducible, every open subset of affine space is irreducible and hence $\GL_{d+1}(k)$ is irreducible.
  In particular, it follows that $\GL_{d+1}$ is a connected group over $k$; cf.~\cite[Prop.~I.1.2]{Borel1991}

  By the Skolem-Noether theorem, $\Gbf$ is therefore  the image of a connected group under the morphism~\eqref{eq:representation} and as morphisms map Zariski-connected sets to Zariski-connected sets, the claim follows.
\end{proof}
In what follows, we let $D=(d+1)^{2}$. We identify $\Gbf(k)$ with its image in $\GL_{D}(k)$ given by the basis $\Ecal_{d+1}$.
\begin{definition}
  Let $k$ be a field and let $R\hookrightarrow k$ a subring. Then
  \begin{equation*}
    \Gbf(R)=\Gbf(k)\cap\GL_{D}(R).
  \end{equation*}
\end{definition}
We denote by $\Vcal_{\mrm{f}}\subseteq\bN$ the set of finite rational primes and we let $\Vcal=\Vcal_{\mrm{f}}\cup\{\infty\}$. The following definition of adelic points and integral adelic points of a $\bQ$-group is formulated for a general algebraic $\bQ$-subgroup $\Hbf$ of $\GL_{D}$. It encompasses in particular the cases $\Hbf=\Gbf$ and $\Hbf=\GL_{D}$.
\begin{definition}
  Let $S\subseteq\Vcal$, $\Sf=S\setminus\{\infty\}$, and $\Hbf\leq\GL_{D}$ be a $\bQ$-subgroup. We set
  \begin{align*}
    \Hbf(\bZ_{\Sf})&=\prod_{p\in\Sf}\Hbf(\bZ_{p}),\\
    \Hbf(\bQ_{\Sf})&=\left\{(g_{p})_{p\in\Sf}\in\prod_{p\in\Sf}\Hbf(\bQ_{p}):g_{p}\in\Hbf(\bZ_{p})\text{ for all but finitely many }p\in\Sf\right\}.
  \end{align*}
  If $\infty\in S$, then $\Hbf(\bQ_{S})=\Hbf(\bR)\times\Hbf(\bQ_{\Sf})$.
  If $\Sf=\Vcal_{\mrm{f}}$, we set $\Hbf(\widehat{\bZ})=\Hbf(\bZ_{\Sf})$, $\Hbf(\bA_{\mrm{f}})=\Hbf(\bQ_{\Sf})$, $\Hbf(\bA)=\Hbf(\bR)\times\Hbf(\bA_{\mrm{f}})$, and $\Hbf(\bR\times\widehat{\bZ})=\Hbf(\bR)\times\Hbf(\widehat{\bZ})$.
\end{definition}
\subsection{The fundamental compact-open subgroups}\label{sec:fundamentalcompactopen}
Given $p\in\Vcal_{\mrm{f}}$ and $v\in\bN$ we define a map
\begin{equation*}
  \pi_{p,v}:\Mat_{D}(\bZ_{p})\to\Mat_{D}(\bZ/p^{v}\bZ)
\end{equation*}
by coordinate-wise reduction mod $p^{v}$. This map clearly defines a ring homomorphism and for any $x\in\Mat_{D}(\bZ_{p})$ we have
\begin{equation*}
  (\det\circ\pi_{p,v})(x)\equiv\det(x)\mod p^{v}.
\end{equation*}

As $\bZ_{p}^{\times}$ and $(\bZ/p^{v}\bZ)^{\times}$ consist precisely of the elements whose projections mod $p^{v}$ do not vanish, this induces a group homomorphism $\pi_{p,v}:\GL_{D}(\bZ_{p})\to\GL_{D}(\bZ/p^{v}\bZ)$. For the sake of completeness, we argue that it is surjective. To this end one notes that $\bZ/p^{v}\bZ$ is a semi-local ring, so that $\SL_{D}(\bZ/p^{v}\bZ)$ is generated by elementary matrices; cf.~\cite[Thm.~4.3.9]{Hahn1989}. Therefore $\pi_{p,v}$ restricts to an epimorphism from $\SL_{D}(\bZ_{p})$ to $\SL_{D}(\bZ/p^{v}\bZ)$. Now one uses that
\begin{equation*}
  \GL_{D}(\bZ/p^{v}\bZ)\cong(\bZ/p^{v}\bZ)^{\times}\ltimes\SL_{D}(\bZ/p^{v}\bZ),
\end{equation*}
where $(\bZ/p^{v}\bZ)^{\times}$ identifies with the set of matrices of the form
\begin{equation*}
  \left\{\begin{pmatrix}
      a & 0 \\ 0 & \mrm{Id}_{D-1}
    \end{pmatrix}: a\in(\bZ/p^{v}\bZ)^{\times}\right\}.
\end{equation*}
In what follows, we will denote $L_{p}[p^{v}]=\ker\pi_{p,v}$ and $L_{p}[1]=\GL_{D}(\bZ_{p})$. We set \[K_{p}[p^{v}]=L_{p}[p^{v}]\cap\Gbf(\bZ_{p}).\]

\begin{lem}
  The family $\{L_{p}[p^{v}]\}$ is a basis of open neighbourhoods of the identity in $\GL_{D}(\bZ_p)$. In particular, the family $\{K_{p}[p^{v}]:v\in\bN\}$ is a basis of open neighbourhoods of the identity in $\Gbf(\bZ_{p})$.
\end{lem}
\begin{proof}
  The group $L_{p}[p^{v}]\leq\GL_{D}(\bZ_{p})$ is closed and has finite index, therefore it is open. The topology on $\GL_{D}(\bZ_{p})$ is induced by the metric $\lVert\cdot\rVert_{p}$ on $\Mat_{D}(\bZ_{p})$ given by
  \begin{equation*}
    \lVert x\rVert_{p}=\max\{\lvert x_{i,j}\rvert_{p}:1\leq i,j\leq D\}\quad(x\in\Mat_{D}(\bZ_{p})).
  \end{equation*}
  Let $v\in\bN$ and $x\in\Mat_{D}(\bZ_{p})$, then
  \begin{equation*}
    \lVert\mrm{Id}_{D}-x\rVert_{p}\leq p^{-v}\iff x\in\mrm{Id}_{D}+p^{v}\Mat_{D}(\bZ_{p}).
  \end{equation*}
  In particular, the collection
  \begin{equation*}
    \{\mrm{Id}_{D}+p^{v}\Mat_{D}(\bZ_{p}):v\in\bN\}
  \end{equation*}
  is a basis of open neighbourhoods of the identity in $\GL_{D}(\bZ_{p})$.
\end{proof}

Given $\Sf\subseteq\Vcal_{\mrm{f}}$, we denote by $\Ical_{\Sf}$ the set of natural numbers whose prime factorization involves only primes contained in $\Sf$.
Given $N\in\Ical_{\Sf}$, we define $(v_{p}(N))_{p\in\Ical_{\Sf}}$ by $N=\prod_{p\in\Sf}p^{v_{p}(N)}$. We set
\begin{equation*}
  L_{\Sf}[N]=\prod_{p\in\Sf}L_{p}[p^{v_{p}(N)}]
\end{equation*}
and $K_{\Sf}[N]=L_{\Sf}[N]\cap\Gbf(\bZ_{\Sf})$. If $\Sf=\Vcal_{\mrm{f}}$, we write $L_{\mrm{f}}[N]$ and $K_{\mrm{f}}[N]$ for $L_{\Sf}[N]$ and $K_{\Sf}[N]$ respectively.

\begin{corollary}\label{cor:compact open nbhd basis}
  The family $\{L_{\Sf}[N]:N\in\Ical_{\Sf}\}$ forms a basis of compact open neighbourhoods of the identity in $\GL_{D}(\bZ_{\Sf})$. In particular, the family $\{K_{\Sf}[N]:N\in\Ical_{\Sf}\}$ forms a basis of compact open neighbourhoods of the identity in $\Gbf(\bZ_{\Sf})$.
\end{corollary}
\begin{proof}
  The groups $L_{\Sf}[N]$ are open by definition of the product topology. Compactness follows from Tychonov's theorem. In order to prove that they form a neighbourhood basis, let $V\subseteq\GL_{D}(\bZ_{\Sf})$ be an open neighbourhood of the identity. Then there is a finite set $T_{\mrm{f}}\subseteq\Sf$ and for all $p\in T_{\mrm{f}}$ an open neighbourhood $V_{p}$ of the identity in $\GL_{D}(\bZ_{p})$ such that
  \begin{equation*}
    \prod_{p\in T_{\mrm{f}}}V_{p}\times\prod_{p\in\Sf\setminus T_{\mrm{f}}}\GL_{D}(\bZ_{p})\subseteq V.
  \end{equation*}
  Given $p\in T_{\mrm{f}}$, let $v_{p}\in\bN$ be such that $L_{p}[p^{v_{p}}]\subseteq V_{p}$ and define $N=\prod_{p\in T_{\mrm{f}}}p^{v_{p}}$. Then, $L_{\Sf}[N]\subseteq V$ by definition.
\end{proof}

\subsection{Principal congruence subgroups}
Similar to what was done in Section~\ref{sec:fundamentalcompactopen}, we can define for any $N\in\bN$, with $N\geq 2$, the group homomorphism $\varpi_{N}:\GL_{D}(\bZ)\to\GL_{D}(\bZ/N\bZ)$ given by projection mod $N$. Note that $\varpi_{N}$ is \emph{not} surjective. We let $\Lambda(N)=\ker\varpi_{N}$ and $\Gamma(N)=\Lambda(N)\cap\Gbf(\bZ)$. We also define $\Lambda(1)=\GL_{D}(\bZ)$ and $\Gamma(1)=\Gbf(\bZ)$.
\begin{definition}
  Let $N\in\bN$. The family $\{\Gamma(N):N\in\bN\}$ is called the family of principal congruence subgroups. A subgroup $\Delta\leq\Gamma(1)$ is a congruence subgroup if it contains a principal congruence subgroup.
\end{definition}

Let $\Sf\subseteq\Vcal_{\mrm{f}}$. In what follows, we will view
\begin{equation*}
  \bZ[\Sf^{-1}]=\bZ[\tfrac{1}{p}:p\in\Sf]
\end{equation*}
as a subring of $\bQ_{\Sf}$ by embedding it diagonally. Similarly, $\GL_{D}(\bZ[\Sf^{-1}])$ and $\Gbf(\bZ[\Sf^{-1}])$ become subgroups of $\GL_{D}(\bQ_{\Sf})$ and $\Gbf(\bQ_{\Sf})$.

We are now ready to prove the first main result of this section.
\begin{proposition}\label{prop:congruenceviacompactopen-S-arith}
  Let $\Sf\subseteq\Vcal_{\mrm{f}}$ and $N\in\Ical_{\Sf}$. Then, $\Gamma(N)=K_{\Sf}[N]\cap\Gbf(\bZ[\Sf^{-1}])$.
  In particular, for all $N$, we have
  \begin{equation*}
    \Gamma(N)=K_{\mrm{f}}[N]\cap\Gbf(\bQ).
  \end{equation*}
\end{proposition}
\begin{proof}
  It suffices to show that 
  \begin{equation*}
    \Lambda(N)=L_{\Sf}[N]\cap\GL_{D}(\bZ[\Sf^{-1}]).
  \end{equation*}
  If $g\in\Lambda(N)$, i.e., $g\in\GL_{D}(\bZ)$ and $g\equiv\mrm{Id}_{D}\mod N$, then clearly for all $p|N$ we have $g\equiv\mrm{Id}_{D}\mod p^{v_{p}(N)}$. As $\det g\in\{\pm1\}\subseteq\bZ_{p}^{\times}$ for all $p\in\Sf$, we have $g\in\GL_{D}(\bZ_{p})$ for all $p\in\Sf$. Combining these two facts, we obtain that $\Lambda(N)\subseteq L_{\Sf}[N]\cap\GL_{D}(\bZ[\Sf^{-1}])$.

  Before we turn to the opposite inclusion, we note that
  \begin{equation*}
    \GL_{D}(\bZ_{\Sf})\cap\GL_{D}(\bZ[\Sf^{-1}])=\GL_{D}(\bZ).
  \end{equation*}
  This, in particular, implies the result in the special case $N=1$. The inclusion $\GL_{D}(\bZ)\subseteq\GL_{D}(\bZ_{\Sf})\cap\GL_{D}(\bZ[\Sf^{-1}])$ is clear. For the opposite inclusion, one first notes that $\bZ[\Sf^{-1}]\cap\bZ_{\Sf}=\bZ$ and hence $\GL_{D}(\bZ_{\Sf})\cap\GL_{D}(\bZ[\Sf^{-1}])\subseteq\Mat_{D}(\bZ)$.
  Let $g\in \GL_{D}(\bZ_{\Sf})\cap\GL_{D}(\bZ[\Sf^{-1}])$.
  Then, $\det g\in\bZ$. On the other hand $\det g\in\bZ_{\Sf}^{\times}$, i.e., we have $\det g\in\bZ_{p}^{\times}$ for all $p\in\Sf$. This means that $\det g$ is coprime to $p$ for all $p\in\Sf$.
  But, since $\det g\in \bZ[\Sf^{-1}]^\times$, we get $\det g\in\{\pm1\}$.
  It follows that $\GL_{D}(\bZ_{\Sf})\cap\GL_{D}(\bZ[\Sf^{-1}])\subseteq\GL_{D}(\bZ)$.

  Let now $g\in L_{\Sf}[N]\cap\GL_{D}(\bZ[\Sf^{-1}])$. In particular
  \begin{equation*}
    g\in\GL_{D}(\bZ_{\Sf})\cap\GL_{D}(\bZ[\Sf^{-1}])=\GL_{D}(\bZ).
  \end{equation*}
  Therefore reduction mod $N$ is just the standard reduction. By assumption we have for all $p\in\Sf$ that $g\equiv\mrm{Id}_{D}\mod p^{v_{p}(N)}$ and in particular $g\in \L(N)$.
\end{proof}

\subsection{Finiteness of class number and principal congruence subgroups}
Similarly to what we did earlier, we will now regard $\bQ$ as a subfield of the ring $\bA=\bR\times\bA_{\mrm{f}}$ by diagonal embedding. Notice that this embedding differs from the composition of embeddings $\bQ\hookrightarrow\bA_{\mrm{f}}\hookrightarrow\bA$. Similarly, we can view $\Gbf(\bQ)$ as a subgroup of $\Gbf(\bA)$. It was proven by Borel and Harish-Chandra that $\Gbf(\bQ)$ is a lattice in $\Gbf(\bA)$.
\begin{proposition}\label{prop:classnumber}
  Let
  \begin{equation*}
    X_{\bA,1}=\biquotient{\Gbf(\bQ)}{\Gbf(\bA)}{\Gbf(\widehat{\bZ})}.
  \end{equation*}
  Then, $\Gbf_{\infty}$ acts transitively on $X_{\bA,1}$, i.e., $\Gbf_{\infty}\backslash X_{\bA,1}$ is a singleton.
\end{proposition}
\begin{proof}
  We first claim that $\Gbf(\bA)=\Phi(\GL_{d+1}(\bA))$. To this end, let $g\in\Gbf(\bA)$ and using Skolem-Noether choose $x_{p}\in\GL_{d+1}(\bQ_{p})$, $p\in\Vcal_{\mrm{f}}$, such that $g_{p}=\Phi(x_{p})$. By definition we have that $g_{p}\in\Gbf(\bZ_{p})$ for all but finitely many $p\in\Vcal_{\mrm{f}}$. Recall that $g_{p}\in\Gbf(\bZ_{p})$ implies that we can assume $x_{p}\in\GL_{d+1}(\bZ_{p})$; cf.~the proof of Lemma~\ref{lem:identification-space-of-lattices}. It follows in particular that $g=\Phi(x)$ for some $x\in\GL_{d+1}(\bA)$.
  
  Let $g\in\Gbf(\bA)$ arbitrary and choose $x\in\GL_{d+1}(\bA)$ such that $g=\Phi(x)$. By~\cite[Prop.~8.1]{Rapinchuk1994}, we know that $\GL_{d+1}$ has class number one, i.e., $x=k\gamma$, where $k\in\GL_{d+1}(\bR\times\widehat{\bZ})$ and $\gamma\in\GL_{d+1}(\bQ)$. In particular, we have $g=\Phi(k)\Phi(\gamma)$. Note that $\Phi(\gamma)\in\GL_{D}(\bQ)$ by rationality of the representation $\Phi$. Note that
  \begin{equation*}
    k_{p}\Mat_{d+1}(\bZ_{p})k_{p}^{-1}=\Mat_{d+1}(\bZ_{p})
  \end{equation*}
  and hence $\Phi(k_{p})\in\GL_{D}(\bZ_{p})$ for all $p\in\Vcal_{\mrm{f}}$, i.e., $\Phi(k)\in\GL_{D}(\bR\times\widehat{\bZ})$. Therefore
  \begin{equation*}
    g\in\Gbf(\bR\times\widehat{\bZ})\Gbf(\bQ).
  \end{equation*}
\end{proof}
\begin{proposition}\label{prop:correspondence}
  Let $N\in\bN$. Then, the double quotient
  \begin{equation*}
    X_{\bA,N}=\biquotient{\Gbf(\bQ)}{\Gbf(\bA)}{K_{\mrm{f}}[N]}
  \end{equation*}
  is a finite union of $\Gbf_{\infty}$-orbits. Let $x\in X_{\bA,N}$, then
  \begin{equation*}
    \Gbf_{\infty}.x\cong\quotient{\Gamma(N)}{\Gbf_{\infty}}
  \end{equation*}
  as $\Gbf_{\infty}$-spaces, i.e., $X_{\bA,N}$ is a disjoint union of finitely many copies of $\Gbf_{\infty}/\Gamma(N)$.
\end{proposition}
\begin{proof}
  Recall that $K_{\mrm{f}}[1]=\Gbf(\widehat{\bZ})$ by definition. By Proposition~\ref{prop:classnumber}, we know that $\Gbf_{\infty}\backslash X_{\bA,1}$ is a singleton. As $K_{\mrm{f}}[N]\leq K_{\mrm{f}}[1]$ is a finite index subgroup, the finiteness of $\Gbf_{\infty}\backslash X_{\bA,N}$ follows immediately.

  For the second part, using Proposition~\ref{prop:classnumber}, let $\Rcal_{D,N}\subseteq\Gbf(\widehat{\bZ})$ be a set of representatives for the double quotient $(\Gbf_{\infty}\times K_{\mrm{f}}[N])\backslash\Gbf(\bA)/\Gbf(\bQ)$. After possibly multiplying $x$ by an element in $\Gbf_{\infty}$, we can assume that $x=K_{\mrm{f}}[N]\eta\Gbf(\bQ)$ for some $\eta\in \Rcal_{D,N}$. We will show that in this case
  \begin{equation*}
    \Stab_{\Gbf_{\infty}}(x)=\Gamma(N).
  \end{equation*}

  Let $g_{\infty}\in\Gbf_{\infty}$, then
  \begin{align*}
    g_{\infty}\cdot K_{\mrm{f}}[N]\eta\Gbf(\bQ)=K_{\mrm{f}}[N]\eta\Gbf(\bQ)&\iff\exists\gamma\in\Gbf(\bQ)\exists k\in K_{\mrm{f}}[N],\;(g_{\infty},\eta)=(\gamma,k\eta\gamma)\\
    &\iff g_{\infty}\in\Gbf(\bQ)\cap\eta^{-1}K_{\mrm{f}}[N]\eta.
  \end{align*}
  Now note that $K_{\mrm{f}}[N]$ is the kernel of the group homomorphism
  \begin{equation*}
    \Psi:\Gbf(\widehat{\bZ})\to\prod_{p|N}\Gbf(\bZ/p^{v_{p}(N)}\bZ), \qquad \Psi((k_{p})_{p\in\Vcal_{\mrm{f}}})=(k_{p}\mod p^{v_{p}(N)}\bZ)_{p|N}.
  \end{equation*}
  In particular, $K_{\mrm{f}}[N]$ is a normal subgroup and thus $\eta^{-1}K_{\mrm{f}}[N]\eta=K_{\mrm{f}}[N]$. Hence, if $x\in X_{\bA,N}$ is arbitrary, letting $\eta\in\Rcal_{D,N}$ be such that $x=K_{\mrm{f}}[N]\eta\Gbf(\bQ)$, Proposition~\ref{prop:congruenceviacompactopen-S-arith} implies that
  \begin{equation*}
    g_{\infty}\in\Stab_{\Gbf_{\infty}}(x)\iff g_{\infty}\in\Gbf(\bQ)\cap\eta^{-1}K_{\mrm{f}}[N]\eta=\Gamma(N).
  \end{equation*}
\end{proof}

\subsection{Correspondence in the S-arithmetic setup}
We deduce analogous results to those obtained in the previous section for quotients of $\Gbf_S$. In particular, the decomposition in~\eqref{eq:congruence} follows by Corollary~\ref{cor:S-arith-correspondence}. First, we need the following.

\begin{corollary}\label{cor:adelic-to-S-arithmetic}
  Let $\Sf\subseteq\Vcal_{\mrm{f}}$ finite and let $S=\{\infty\}\cup\Sf$. Then,
  \begin{equation*}
    \Gbf(\bR\times\bZ_{\Sf})\GammaS=\Gbf(\bQ_{S}).
  \end{equation*}
\end{corollary}
\begin{proof}
  Let
  \begin{equation*}
    M_{\Sf}=\prod_{p\in\Vcal_{\mrm{f}}\setminus\Sf}\Gbf(\bZ_{p}).
  \end{equation*}
  We first claim that
  \begin{equation*}
    \biquotient{\Gbf(\bQ)}{\Gbf(\bA)}{M_{\Sf}}\cong\quotient{\GammaS}{\Gbf(\bQ_{S})}
  \end{equation*}
  as $\Gbf(\bQ_{S})$-spaces. To this end, we note that $\Gbf(\bQ_{S})$ acts transitively on the left hand side by Proposition~\ref{prop:classnumber}. Denote by $x_{0}$ the identity coset in the left-hand side double quotient. In particular, it remains to show that
  \begin{equation*}
    \Stab_{\Gbf(\bQ_S)}(x_{0})=\GammaS.
  \end{equation*}
  Let $e$ denote the identity in $\prod_{p\not\in S}\Gbf(\bZ_p)$ and $g\in \Stab_{\Gbf(\bQ_S)}(x_{0})$.
  Arguing as in the proof of Proposition~\ref{prop:congruenceviacompactopen-S-arith}, we get that $g\in M_{\Sf}\cap\Gbf(\bQ)=\GammaS$ as desired.
 
\end{proof}

The following is the analogue of Proposition~\ref{prop:correspondence} and implies~\eqref{eq: class number one}.
\begin{corollary}\label{cor:S-arith-correspondence}
  Let $N\in\bN$ and $\Sf\subseteq\Vcal_{\mrm{f}}$ such that $v_{p}(N)\neq0\implies p\in\Sf$. Let $\Rcal_{D,N}$ be as in the proof of Proposition~\ref{prop:correspondence}. Then, the projection $\Rcal_{D,N,\Sf}$ of $\Rcal_{D,N}$ to $K_{\Sf}[1]$ is a set of representatives of the $\Gbf_{\infty}$-orbits in
  \begin{equation*}
    X_{\bQ_{S},N}=\biquotient{\GammaS}{\Gbf_{S}}{K_{\Sf}[N]}.
  \end{equation*}
  Moreover, the map $\Rcal_{D,N}\to\Rcal_{D,N,\Sf}$ is a bijection and $X_{\bQ_{S},N}$ is a disjoint union of $\lvert\Rcal_{D,N}\rvert$-many copies of $\Gbf_{\infty}/\Gamma(N)$.
\end{corollary}
\begin{proof}
  Note that $K_{\mrm{f}}[N]=K_{\Sf}[N]\times M_{\Sf}$. Therefore, as in the proof of Corollary~\ref{cor:adelic-to-S-arithmetic}, we obtain
  \begin{equation*}
    X_{\bQ_{S},N}\cong\rquotient{K_{\Sf}[N]}{(\biquotient{\Gbf(\bQ)}{\Gbf(\bA)}{M_{\Sf}})}\cong\biquotient{\Gbf(\bQ)}{\Gbf(\bA)}{K_{\mrm{f}}[N]}.
  \end{equation*}
  Looking at these isomorphisms more explicitly, it is easy to check that $\Rcal_{D,N,\Sf}$ is a set of representatives which is in one-to-one correspondence with $\Rcal_{D,N}$. We leave the rest of the proofs to the reader.
\end{proof}

\subsection{Counting connected components}
In this section, we aim at finding a uniform bound on the number of connected components of
\begin{equation*}
  X_{\bQ_{S},N}=\biquotient{\GammaS}{\Gbf_{S}}{K_{\Sf}[N]},
\end{equation*}
independently of $N$, where we assume that $\Sf\subseteq\Vcal_{\mrm{f}}$ is finite, $S=\Sf\cup\{\infty\}$, and $N\in\Ical_{\Sf}$. The main result is Proposition~\ref{prop:components-as-quotients}.

We first need a lemma about $\Gbf_{S}^{+}$-orbits. 
Given $p\in\Vcal_{\mrm{f}}$ and $n\in\bN$, let $S_{p}[p^{n}]\leq\SL_{d+1}(\bZ_{p})$ be the kernel of the homomorphism $\SL_{d+1}(\bZ_{p})\to\SL_{d+1}(\bZ/p^{n}\bZ)$ given by reduction mod $p^{n}$. For $n=0$, we let $S_{p}[p^{n}]=\SL_{d+1}(\bZ_{p})$. Recall the representation $\Phi:\GL_{d+1}\to\Gbf$ defined in~\eqref{eq:representation}.
A calculation shows that $\Phi(S_{p}[p^{n}])\subseteq K_{p}[p^{n}]$. In what follows, we let $K_{p}[p^{n}]^{+}=\Phi(S_{p}[p^{n}])$ and, given $N\in\Ical_{\Sf}$, we define
\begin{equation*}
  K_{\Sf}[N]^{+}=\prod_{p\in\Sf}K_{p}[p^{v_{p}(N)}]^{+}\subseteq K_{\Sf}[N].
\end{equation*}
In fact, $K_{\Sf}[N]^{+}$ is normal in $\Gbf(\bZ_{\Sf})$. To this end, one checks that $S_{p}[p^{n}]\lhd\GL_{d+1}(\bZ_{p})$ is a normal subgroup and then uses that
\begin{equation*}
  \Gbf(\bZ_{p})=\Phi(\GL_{d+1}(\bZ_{p}))
\end{equation*}
as already argued in the proof of Lemma~\ref{lem:identification-space-of-lattices}.

Similarly, we note that $\Phi(\SL_{d+1}(\bZ[\Sf^{-1}]))\subseteq\GammaS$, and we will denote
\begin{equation*}
  \Gamma_{S}^{+}=\Phi(\SL_{d+1}(\bZ[\Sf^{-1}])).
\end{equation*}
As $\SL_{d+1}$ has the strong approximation property, we find
\begin{equation}\label{eq: strong-approximation-SL}
  \Gbf_{S}^{+}=(\Gbf_{\infty}^{+}\times K_{\Sf}[N]^{+})\Gamma_{S}^{+}
\end{equation}
for all $N\in\Ical_{\Sf}$; cf.~\cite[Thm.~7.12]{Rapinchuk1994}.

\begin{lem}\label{lem:Gplus-orbits}
  Let $\Sf\subseteq\Vcal_{\mrm{f}}$, $S=\Sf\cup\{\infty\}$, $N\in\Ical_{\Sf}$,
  and $x\in \Gbf_{S}/\Gamma_{S}$. Then,
  \begin{equation*}
    \Gbf_{S}^{+}.x=(\Gbf_{\infty}^{+}\times K_{\Sf}[N]^{+}).x.
  \end{equation*}
\end{lem}
\begin{proof}
  This follows immediately from Corollary~\ref{cor:adelic-to-S-arithmetic} and~\eqref{eq: strong-approximation-SL}. To this end let $g\in\Gbf_{\infty}\times K_{\Sf}[N]$ and $\eta\in\Gbf(\bZ_{\Sf})$ such that $x=g\eta\GammaS$. As $K_{\Sf}[N]^{+}$ is normal in $\Gbf(\bZ_{\Sf})$, it follows that
  \begin{equation}\label{eq:orbitGplus}
    \Gbf_{S}^{+}.x=g\eta\Gbf_{S}^{+}\GammaS=g\eta(\Gbf_{\infty}^{+}\times K_{\Sf}[N]^{+})\GammaS=(\Gbf_{\infty}^{+}\times K_{\Sf}[N]^{+}).x.
  \end{equation}
\end{proof}

In what follows, we let
\begin{equation*}
  \Gbf_{\mrm{char},S}=\biquotient{\GammaS}{\Gbf_{S}}{\Gbf_{S}^{+}}
\end{equation*}
and we note that $\Gbf_{\mrm{char}}$ is a finite set whose cardinality is bounded by the index $[\Gbf_S:\Gbf_S^+]$.
\begin{proposition}\label{prop:components-as-quotients}
  We have
  \begin{equation*}
    \rquotient{\Gbf_{\infty}^{+}}{X_{\bQ_{S},N}}\cong\biquotient{\GammaS}{\Gbf_{S}}{K_{\Sf}[N]\Gbf_{S}^{+}}\cong\rquotient{K_{\Sf}[N]}{\Gbf_{\mrm{char},S}}.
  \end{equation*}
  In particular, given a finite set $ S\subseteq\Vcal$ containing $\infty$, the number of connected components of $X_{\bQ_{S},N}$ is at most $[\Gbf_S:\Gbf_S^+]$.
\end{proposition}
\begin{proof}
  For the first bijection, we recall from the proof of Lemma~\ref{lem:Gplus-orbits} that $K_{\Sf}[N]^{+}\leq K_{\Sf}[N]$ and therefore \eqref{eq:orbitGplus} yields that for all $g\in\Gbf_{S}$
  \begin{align*}
    (\Gbf_{\infty}^{+}\times K_{\Sf}[N])g\GammaS&=K_{\Sf}[N](\Gbf_{\infty}^{+}\times K_{\Sf}[N]^{+})g\GammaS\\
    &=K_{\Sf}[N]\Gbf_{S}^{+}g\GammaS.
  \end{align*}
  The second bijection follows by definition of $\Gbf_{\mrm{char},S}$.
  The last part of the statement follows by definition of $\Gbf_{\mrm{char},S}$.
\end{proof}

The following corollary follows immediately.
\begin{corollary}\label{cor:components-general}
  Let $W\subseteq K_{\Sf}$ be an open subgroup. Then the quotient $W\backslash\Gbf_{S}/\GammaS$ is a union of at most $[\Gbf_{S}:\Gbf_{S}^{+}]$-many $\Gbf_{\infty}^{+}$-orbits.
\end{corollary}

%--------------------------------------------------
\section{KAK-decomposition and norms}
\label{sec:KAK decomposition}
In this section we will introduce a function on $\Gbf(\bQ_{v})$ which measures the size of an element. These functions are used in Section~\ref{sec:summability} to define norm-balls on $\Gbf(\bQ_{S})$. The main input is the KAK-decomposition, which is well-known for $v=\infty$. We only discuss the case where $v$ is a finite place of $\bQ$. In what follows, $p$ is a natural prime.

\begin{lem}\label{lem:padic supnorm}
  Let $D\in\bN$ and denote by $\lVert\cdot\rVert_{p}:\Mat_{D}(\Qp)\to[0,\infty)$ the norm defined for $x=(x_{ij})$ by
  \begin{equation*}
      \lVert x\rVert_{p}=
      \max\{\lvert x_{ij}\rvert_{p}:1\leq i,j\leq D\}.
  \end{equation*}
   Then, $\lVert\cdot\rVert_{p}$ is an operator norm. Moreover, for all $k\in\GL_{D}(\Zp)$ we have $\lVert k\rVert_{p}=1$.
\end{lem}
\begin{proof}
  The fact that $\lVert\cdot\rVert_{p}$ is an operator norm follows from the ultrametric property of the $p$-adic absolute value. We leave the details to the reader. For the second part we note that for any $x\in\GL_{D}(\Zp)$ we have $\lvert\det x\rvert_p=1$ and $\lVert x\rVert_{p}\leq 1$. On the other hand we have $\lvert\det x\rvert_{p}\leq\lVert x\rVert_{p}^{D}$ by the ultrametric property. Hence $\lVert k\rVert_{p}=1$ for all $k\in\GL_{D}(\Zp)$.
\end{proof}
    In the following discussion, Lemma~\ref{lem:padic supnorm} is used in the form of the following corollary.
\begin{corollary}\label{cor:trivialcorollary}
  Let $D\in\bN$, $x\in\Mat_{D}(\Qp)$ and $k_{1},k_{2}\in\GL_{D}(\Zp)$. Then
  \begin{equation*}
    \lVert k_{1}xk_{2}\rVert_{p}=\lVert x\rVert_{p}.
  \end{equation*}
\end{corollary}
\begin{proof}
  Since operator norms are submultiplicative, we have
  \begin{equation*}
    \lVert x\rVert_{p}\leq\lVert k_{1}^{-1}\rVert_{p}\lVert k_{1}xk_{2}\rVert_{p}\lVert k_{2}^{-1}\rVert_{p}=\lVert k_{1}xk_{2}\rVert_{p}\leq\lVert k_{1}\rVert_{p}\lVert x\rVert_{p}\lVert k_{2}\rVert_{p}=\lVert x\rVert_{p}.
  \end{equation*}
\end{proof}

\begin{lem}\label{lem:KAK2x2}
  Let $x\in\GL_{2}(\bQ_{p})$. Then there are $m_{1},m_{2}\in\GL_{2}(\bZ_{p})$ and $n,k\in\bZ$ such that
  \begin{equation*}
    x=m_{1}\begin{pmatrix}
      p^{n} & 0 \\ 0 & p^{k}
    \end{pmatrix}m_{2}.
  \end{equation*}
\end{lem}
\begin{proof}
  Let $x=(\begin{smallmatrix}a & b \\ c & d\end{smallmatrix})$. Assume that $\lVert x\rVert_{p}=\lvert a\rvert_{p}$ and write $a=p^{m}u$, $b=p^{n}v$ and $c=p^{\ell}w$ with $u,v,w\in\Zp^{\times}$. Then
  \begin{equation*}
    \begin{pmatrix}
      1 & 0 \\ -p^{\ell-m}\frac{w}{u} & 1 
    \end{pmatrix}
    \begin{pmatrix}
      a & b \\ c & d 
    \end{pmatrix}
    \begin{pmatrix}
      1 & -p^{n-m}\frac{v}{u} \\ 0 & 1 
    \end{pmatrix}=
    \begin{pmatrix}
      a & 0 \\ 0 & d^{\prime}
    \end{pmatrix}.
  \end{equation*}
  As we assumed that $a$ was maximal, it follows that the two unipotent matrices lie in $\GL_{2}(\Zp)$. Denote them by $m_{1},m_{2}$ respectively. The resulting diagonal matrix is of the form
  \begin{equation*}
    \begin{pmatrix}
      a & 0 \\ 0 & d^{\prime}
    \end{pmatrix}=
    \begin{pmatrix}
      u & 0 \\ 0 & z
    \end{pmatrix}
    \begin{pmatrix}
      p^{m} & 0 \\ 0 & p^{k}
    \end{pmatrix}
  \end{equation*}
  with $u,z\in\Z_{p}^{\times}$, and thus we have shown that
  \begin{equation*}
    x=m_{1}^{-1}\begin{pmatrix}
      u & 0 \\ 0 & z
    \end{pmatrix}
    \begin{pmatrix}
      p^{n} & 0 \\ 0 & p^{k}
    \end{pmatrix}
    m_{2}^{-1}.
  \end{equation*}
  As $m_{1}\mathrm{diag}(u,z)\in\GL_{2}(\bZ_p)$, we obtain the claim under the assumption that $\lVert x\rVert_{p}=\lvert a\rvert_{p}$. If $\lVert x\rVert_{p}\neq\lvert a\rvert_{p}$, then we distinguish two cases. If $\lVert x\rVert_{p}=\lvert d\rvert_{p}$, then we conjugate $x$ by the matrix $m=(\begin{smallmatrix} 0 & 1 \\ 1 & 0 \end{smallmatrix})$, so that the maximal entry comes to lie in the top left corner. If $\lVert x\rVert_{p}=\lvert b\rvert_{p}$ or $\lVert x\rVert_{p}=\lvert c\rvert_{p}$, then it is a property of non-archimedian absolute values that for either $x(\begin{smallmatrix} 1 & 0 \\ 1 & 1 \end{smallmatrix})$ or $(\begin{smallmatrix} 1 & 1 \\ 0 & 1 \end{smallmatrix})x$ the top left entry will be maximal. These operations all follow from multiplying $x$ with matrices in $\GL_{2}(\Zp)$ and thus the claim is proven.
\end{proof}

In what follows, given $x\in\GL_{d+1}(\bQ_{p})$, we denote by $[x]\in\Gbf(\bQ_{p})$ its image under $\Phi$ as introduced in the proof of Proposition~\ref{prop:classnumber}.
\begin{proposition}\label{prop:generalKAK}
  Let $p$ be a finite rational prime and assume that $g\in\Gbf(\Qp)$. Then there exist uniquely determined nonnegative integers $n_{1}\geq\cdots\geq n_{d}$ as well as elements $k_{1},k_{2}\in\Gbf(\Zp)$ such that
  \begin{equation*}
    g=k_{1}[\mathrm{diag}(p^{-n_{1}},\ldots,p^{-n_{d}},1)]k_{2}.
  \end{equation*}
\end{proposition}
\begin{proof}
  As $[\GL_{d+1}(\bZ_{p})]=\Gbf(\bZ_{p})$, cf.~the proof of Proposition~\ref{prop:classnumber}, it suffices to prove the existence of a decomposition of any element $x\in\GL_{d+1}(\Qp)$ into a product of the form $k_{1}Dk_{2}$, where $D$ is a diagonal matrix whose non-zero entries are powers of $p$ for a set of exponents uniquely determined (with multiplicity) by $x$. If this is the case, we can use elements in $\GL_{d+1}(\Zp)$ to arrange the diagonal entries in decreasing order with respect to the $p$-adic absolute value. Furthermore, there will be an element $\omega$ in the center of $\GL_{d+1}(\Qp)$ such that $\omega^{-1}D$ is of the form required by the proposition and we note that $[\omega^{-1}D]=[D]$.
  
  So let $x\in\GL_{d+1}(\Qp)$ be arbitrary and in view of Lemma~\ref{lem:KAK2x2} assume that $d\geq2$. For the existence of a decomposition of $x$, we apply elements in $\GL_{d+1}(\Zp)$ so that $\lVert x\rVert_{p}=\lvert x_{11}\rvert_{p}$. Then one can use the copies of $\SL_{2}(\Zp)$ in $\GL_{d+1}(\Qp)$ associated with spans of pairs of the standard basis to reduce the matrix $x$ to a matrix of the form
  \begin{equation*}
    x^{\prime}=\begin{pmatrix}
      x_{11} & 0 \\
      0 & y
    \end{pmatrix},
  \end{equation*}
  where $y\in\GL_{d}(\Qp)$ and $\lVert y\rVert_{p}\leq\lvert x_{11}\rvert_{p}$. Using $\Qp=\bigsqcup_{n\in\Z}p^{n}\Zp^{\times}$ and multiplication by a diagonal matrix with diagonal entries contained in $\Zp^{\times}$, we can assume without loss of generality that $x_{11}=\lVert g\rVert_{p}^{-1}$ in the expression for $x^{\prime}$ obtained above. Now we proceed by induction on $d$.
\end{proof}

Let us give a more intrinsic interpretation of Proposition~\ref{prop:generalKAK}. In what follows, we consider
\begin{equation*}
  \liepgl_{d+1}(\Qp)=\{v\in\mathrm{Mat}_{d+1}(\Qp):\tr (v)=0\}.
\end{equation*}
Then $\Ad:\Gbf(\Qp)\to\GL(\liepgl_{d+1}(\Qp))$ given by 
\begin{equation*}
  \Ad_{g}(v)=gvg^{-1}\qquad(v\in\liepgl_{d+1}(\Qp))
\end{equation*}
is a well-defined, faithful representation. We let $\lVert\cdot\rVert_{p}$ be the norm on $\GL(\liepgl_{d+1}(\bQ_{p}))$ induced by the operator norm $\lVert\cdot\rVert_{p}$ on $\GL(\Mat_{d+1}(\bQ_{p}))$ via restriction of the isomorphism to $\liepgl_{d+1}(\bQ_{p})$. To this end we note that any isomorphism of $\liepgl_{d+1}(\bQ_{p})$ extends trivially to the center of $\Mat_{d+1}(\bQ_{p})$.
\begin{lem}\label{lem:KAKvsOpNorm}
  Let $k_{1},k_{2}\in\Gbf(\Zp)$ and let $n_{1}\geq\cdots\geq n_{d}$ nonnegative integers. Set
  \begin{equation*}
    g=k_{1}[\mathrm{diag}(p^{-n_{1}},\ldots,p^{-n_{d}},1)]k_{2}.
  \end{equation*}
  Then $\lVert\Ad_{g}\rVert_{p}=\lVert\Ad_{g^{-1}}\rVert_{p}=p^{n_{1}}$.
\end{lem}
\begin{proof}
  By Proposition~\ref{prop:generalKAK} and Corollary~\ref{cor:trivialcorollary}, we can assume without loss of generality that $g=[\mathrm{diag}(p^{-n_{1}},\ldots,p^{-n_{d}},1)]$. We let $n_{d+1}=0$. An elementary calculation shows that the operator $\Ad_{g}$ acting on $\Mat_{d+1}(\bQ_{p})$ is diagonalizable with eigenvalues $p^{n_i-n_j}$ for $:1\leq i,j\leq d+1$.
  Therefore $\Ad_{g}$ has maximal eigenvalue $p^{-n_{1}}$ and $\lVert\Ad_{g}\rVert_{p}=p^{n_{1}}$. As $\sigma(\Ad_{g})$ is symmetric under multiplicative inversion, we also have $\lVert\Ad_{g^{-1}}\rVert_{p}=p^{n_{1}}$.
\end{proof}

%--------------------------------------------------

\section{Modular character on the Borel subgroup} 
\label{sec:modular char}
The goal of this section is to prove~\eqref{eq:determineshellexponent}. Recall first that the $p$-adic value satisfies that for all $x\in\Qp$ and for all $f\in\compact(\Qp)$
\begin{equation*}
  \lvert x\rvert_{p}\int_{\Qp}f(xt)\der m_{\Qp}(t)=\int_{\Qp}f(t)\der m_{\Qp}(t),
\end{equation*}
where $m_{\Qp}$ denotes any choice of a Haar measure on $\Qp$. This implies that up to normalization the Haar measure on $\Qp^{\times}$ is given by
\begin{equation*}
  \int_{\Qp^{\times}}f(y)\der m_{\Qp^{\times}}(y)=\int_{\Qp^{\times}}\frac{f(y)}{\lvert y\rvert_{p}}\der m_{\Qp}(y)\quad(f\in\compact(\Qp^{\times})).
\end{equation*}
We let $B_{p}=\Bbf(\Qp)$ and note that $B_{p}=A_{p}U_{p}$, where $A_{p}$ denotes the image of the diagonal subgroup of $\GL_{d+1}(\Qp)$ and $U_{p}$ is the (injective) image of the subgroup of upper triangular unipotent matrices. Note that $U_{p}$ is homeomorphic to $\Qp^{d^{\prime}}$ with $d^{\prime}=\frac{1}{2}d(d+1)$ and that the push forward $m_{U_{p}}$ of the Haar measure on $\Qp^{d^{\prime}}$ to $U_{p}$ defines a Haar measure on $U_{p}$. Therefore one finds that for a matrix $b=au\in B_{p}$ with $u\in U_{p}$ and $a=\mathrm{diag}(a_{1},\ldots,a_{d},1)$ we have
\begin{align*}
  \der m_{B_{p}}^{\mathrm{left}}(b)&\propto\der m_{U_{p}}(u)\prod_{i=1}^{d}\frac{\der m_{\Qp}(a_{i})}{\lvert a_{i}\rvert_{p}^{d+2-i}},&
  \der m_{B_{p}}^{\mathrm{right}}(b)&\propto\der m_{U_{p}}(u)\prod_{i=1}^{d}\frac{\der m_{\Qp}(a_{i})}{\lvert a_{i}\rvert^{i}}.
\end{align*}

As the modular function satisfies
$
  \der m_{B_{p}}^{\mathrm{right}}(b)\propto\delta_{\Bbf}(b)\der m_{B_{p}}^{\mathrm{left}}(b)
$,
we find
\begin{equation*}
  \delta_{\Bbf}(b)=\prod_{i=1}^{d}\lvert a_{i}\rvert_{p}^{d+2-2i}.
\end{equation*}
Equation~\eqref{eq:determineshellexponent} now follows by plugging in $a_{\mathbf{n}}$ for $b$.
%--------------------------------------------------

\section{Integrability of Matrix Coefficients}\label{sec:integrability-eta}
In this section, we show that $\eta_{S}\in \mrm{L}^{v(d)+\varepsilon}(\Gbf_{S})$. Given $v\in S$, let $\eta_{v}:\Gbf(\bQ_{v})\to(0,\infty)$ denote
\begin{equation*}
  \eta_{v}(g_{v})=\lVert g_{v}\rVert^{-1}_{v}
  \quad(g_{v}\in\Gbf(\bQ_{v})).
\end{equation*}
As $\eta_{\Gbf}$ is the product of the various $\eta_{v}$, $v\in S$, and as the Haar measure on $\Gbf_{S}$ is the product measure, it suffices to show that $\eta_{v}\in \mrm{L}^{v(d)+\varepsilon}(\Gbf_{v})$ for all places $v$ of $\bQ$.
\subsection{Integrability in the Archimedean place}\label{sec:integrability-of-eta-infty}
We recall that the Haar measure on $\Gbf_{\infty}$ is explicitly given by the formula
\begin{equation*}
  \int_{\Gbf_{\infty}}f(g)\der g=\int_{K}\int_{A_{\infty}^{+}}\int_{K}f(k_{1}ak_{2})\prod_{\alpha\in\Sigma_{\infty}^{+}}(\sinh\alpha (\log a))^{m_{\alpha}}\der k_{1}\der a\der k_{2}\quad(f\in \mrm{C}_{c}(\Gbf_{\infty})),
\end{equation*}
where $m_{\alpha}$ denotes the multiplicity of the positive root $\alpha\in\Sigma_{\infty}^{+}$ and where the Haar measure on $A$ is the push-forward of the Lebesgue measure on the Cartan subalgebra under the exponential map; cf.~\cite[Prop.~5.28]{Knapp1986}. Note that in our situation $m_{\alpha}=1$ for all $\alpha\in\Sigma_{\infty}^{+}$. By definition of the hyperbolic sine we have
\begin{equation*}
  \prod_{\alpha\in\Sigma_{\infty}^{+}}\sinh\alpha (\log a)\leq e^{\beta(\log a)}\quad(a\in A_{\infty}^{+}),
\end{equation*}
where $\beta=\sum_{\alpha\in\Sigma_{\infty}^{+}}\alpha$. We recall that for any $\alpha\in\Sigma_{\infty}^{+}$ there are $1\leq i<j\leq d+1$ such that for all $a=\mathrm{diag}(a_{1},\ldots,a_{d+1})\in A_{\infty}^{+}$, we have $ \alpha(\log a)=\log a_{i}-\log a_{j}$.
Hence, we get
\begin{equation*}
  \beta(\log a)=\sum_{i=1}^{d+1}(d+2-2i)\log a_{i}.
\end{equation*}
As argued in the proof of Lemma~\ref{lem:padicshells} and recalling that we parametrize $A_{\infty}^{+}$ such that $a_{d+1}=1$, we thus find that $\beta(\log a)\leq v(d)\log a_{1}$.
Using the lower bound in~\eqref{eq:squeeze}, it follows that $ \eta_{\infty}(k_{1}ak_{2})\leq a_{1}^{-1}$ and thus
\begin{align*}
  \int_{\Gbf_{\infty}}\eta_{\infty}(g)^{q}\der g&=\int_{K}\int_{A_{\infty}^{+}}\int_{K}\eta_{\infty}(k_{1}ak_{2})^{q}\prod_{\alpha\in\Sigma_{\infty}^{+}}\sinh\alpha(\log a)\der k_{1}\der a\der k_{2}
  \ll_{\vartheta}\int_{A_{\infty}^{+}}a_{1}^{-q+v(d)}\der a.
\end{align*}
By definition of $A_{\infty}^{+}$ and recalling that the Haar measure on the connected component of the diagonal subgroup of $\Gbf_{\infty}$ is the push-forward of the Lebesgue measure on the Lie algebra
\begin{equation*}
  \mathfrak{a}\cong\{(t_{1},\ldots,t_{d},t_{d+1}):t_{1}+\cdots+t_{d+1}=0\}
\end{equation*}
under the exponential map, we have for any $s>0$ that
\begin{equation*}
\int_{A_{\infty}^{+}}a_{1}^{-s}\der a=\int_{0}^{\infty}\int_{t_{d}}^{\infty}\cdots\int_{t_{2}}^{\infty}e^{-st_{1}}\der t_{1}\cdots\der t_{d}=\frac{1}{s^{d}}.
\end{equation*}
Hence, whenever $q>v(d)$, then 
\begin{equation*}
  \int_{\Gbf_{\infty}}\eta_{\infty}(g)^{q}\der g<\infty.
\end{equation*}

\subsection{Integrability in the finite places}
In what follows, we note that
\begin{equation*}
  \Gbf(\Qp)=\bigsqcup_{n\in\bN_{0}}S_{n}, \qquad S_{n}=\{g\in\Gbf(\Qp):\lVert g\rVert_{p}=p^{n}\}.
\end{equation*}
Moreover, Lemma~\ref{lem:padicshells} yields that for all $n\in\bN_{0}$ we have
\begin{align*}
  \int_{\Gbf(\Qp)}\eta_{p}(g)^{q}\der g=\sum_{n\in\bN_{0}}\int_{S_{n}}\eta_{p}(g)^{q}\der g \ll_{d}
  \sum_{n\in\bN_{0}} p^{-nq}\vol{S_{n}}\ll_{\varepsilon,d}
  \sum_{n\in\bN_{0}}p^{n(-q+v(d)+\varepsilon)}.
\end{align*}
It follows that $\eta_{p}\in \mrm{L}^{v(d)+\varepsilon}(\Gbf(\bQ_{p}))$
whenever $q>v(d)$.
\subsection{Integrability of matrix coefficients}\label{sec:integrability-xi}
This argument was mentioned in the proof of Proposition~\ref{prop:uniformspectralgap} and it follows readily from Section~\ref{sec:integrability-of-eta-infty}. Recall from Proposition~\ref{prop:boundHarishChandra} that
\begin{equation*}
  \xi_{\infty}(g)^{2/(1-2\varepsilon)}\ll_{\varepsilon,d}\eta_{\infty}(g)\quad(g\in\Gbf_{\infty}).
\end{equation*}
Hence for all $0<\varepsilon<\frac{1}{2}$ and for $q>\frac{2v(d)}{1-2\varepsilon}$ we find
\begin{equation*}
  \int_{\Gbf_{\infty}}\xi_{\infty}(g)^{q}\der g\ll_{\varepsilon,d}\int_{\Gbf_{\infty}}\eta_{\infty}(g)^{(\frac{1}{2}-\varepsilon)q}\der g<\infty
\end{equation*}
as $(\frac{1}{2}-\varepsilon)q>v(d)$ and $\eta_{\infty}\in \mrm{L}^{v(d)+\varepsilon}(\Gbf_{\infty})$.
\begin{remark}
  The same argument works for finite places.
\end{remark}

%----------------------------------------------------------

%----Bibliography------------------------------
\def\cprime{$'$} \providecommand{\bysame}{\leavevmode\hbox
  to3em{\hrulefill}\thinspace}
\providecommand{\MR}{\relax\ifhmode\unskip\space\fi MR }
\providecommand{\MRhref}[2]{%
  \href{http://www.ams.org/mathscinet-getitem?mr=#1}{#2} }
\providecommand{\href}[2]{#2}

\bibliographystyle{amsalpha}

\end{document}